\documentclass[a4paper]{amsart}

\usepackage[T1]{fontenc}
\usepackage[utf8]{inputenc}
\usepackage[english]{babel}
\usepackage{lmodern}
\usepackage{exscale}
\usepackage[babel]{microtype}
\usepackage{amsmath, amssymb, mathtools, mathrsfs}

\usepackage{csquotes}
\usepackage[backend=bibtex,style=alphabetic, sorting=nyt, maxnames=10, maxalphanames=5, doi=false,isbn=true,url=false,eprint=false]{biblatex}
\renewbibmacro{in:}{%
  \ifentrytype{article}{}{\printtext{\bibstring{in}\intitlepunct}}}

\AtEveryBibitem{\clearfield{issn}}
\AtEveryCitekey{\clearfield{issn}}
\AtEveryBibitem{\clearfield{month}}

\usepackage{tikz}
\usetikzlibrary{matrix,shapes,arrows,calc,3d,decorations,decorations.pathmorphing,through,cd}
\usepackage{todonotes}

\usepackage[numbered]{bookmark}
\usepackage{hyperref}
\usepackage{amsthm}
\hypersetup{
  colorlinks,
  urlcolor = {blue!80!black},
  linkcolor = {red!70!black},
  citecolor = {green!50!black}
}

\theoremstyle{plain}
  \newtheorem{thm}{Theorem}
  \newtheorem{defi}[thm]{Definition}
\newtheorem{conj}[thm]{Conjecture}
  \newtheorem{prop}[thm]{Proposition}
  
  \newtheorem{cor}[thm]{Corollary}
  
  \newtheorem{lemma}[thm]{Lemma}
\theoremstyle{definition}
  
  \newtheorem{rem}[thm]{Remark}

\tikzset{ext/.style={circle, draw,inner sep=1pt},int/.style={circle,draw,fill,inner sep=1pt},nil/.style={inner sep=1pt}}
\tikzset{exte/.style={circle, draw,inner sep=3pt},inte/.style={circle,draw,fill,inner sep=3pt}}
\tikzset{diagram/.style={matrix of math nodes, row sep=3em, column sep=2.5em, text height=1.5ex, text depth=0.25ex}}
\tikzset{diagram2/.style={matrix of math nodes, row sep=0.5em, column sep=0.5em, text height=1.5ex, text depth=0.25ex}}

\newcommand{\alg}[1]{\mathfrak{{#1}}}

\newcommand{\ad}{{\text{ad}}}

\newcommand{\widebar}[1]{{\overline{#1}}}

\newcommand{\Q}{\mathbb{Q}}

\newcommand{\Graphs}{{\mathsf{Graphs}}}
\newcommand{\pdGraphs}{{\mathsf{Graphs}}^*}
\newcommand{\BVGraphs}{{\mathsf{BVGraphs}}}

\newcommand{\sll}{\mathfrak{sl}}

\newcommand{\mF}{\mathcal{F}}

\newcommand{\op}{\mathcal}

\newcommand{\Lie}{\mathsf{Lie}}

\newcommand{\FM}{\mathsf{FM}}
\newcommand{\FFM}{\mathsf{FFM}}

\newcommand{\bpm}{\begin{pmatrix}}
\newcommand{\epm}{\end{pmatrix}}

\newcommand{\GC}{\mathsf{GC}}

\newcommand{\hotimes}{\mathbin{\hat\otimes}}

\DeclareMathOperator{\gr}{gr}
\DeclareMathOperator{\Tw}{Tw}

\DeclareMathOperator{\Hom}{Hom}

\DeclareMathOperator{\Harr}{Harr}

\DeclareMathOperator{\Conf}{Conf}
\DeclareMathOperator{\Def}{Def}
\DeclareMathOperator{\id}{id}

\DeclareMathOperator{\colim}{colim}

\newcommand{\pdGraphsone}{\mathsf{Graphs}^*_{(1)}}
\newcommand{\fraktone}{\mathfrak{t}^{\mathrm{non-fr}}_{(1)}}
\newcommand{\frakrone}{\mathfrak{r}^{\mathrm{non-fr}}_{(1)}}
\newcommand{\Zone}{Z^{\mathrm{non-fr}}_{(1)}}

\newcommand{\frakb}{\mathfrak{b}}

\newcommand{\beq}[1]{\begin{equation}\label{#1} 
}
\newcommand{\eeq}{\end{equation}}

\newcommand{\minitadp}
{{
\begin{tikzpicture}[baseline=-.55ex,scale=.4, every loop/.style={}]
 \node[circle,draw,fill,inner sep=.5pt] (a) at (0,0) {};
 \draw (a) edge[loop] (a);
\end{tikzpicture}}}

\newcommand{\zzeroo}
{{
\begin{tikzpicture}[baseline=-.55ex,scale=.7, every loop/.style={}]
 \node[circle,draw,fill,inner sep=.8pt] (a) at (0,0) {};
 \node[circle,draw,dashed,inner sep=1.5pt] (b) at (-0.6,0.6) {$\omega$};
 \draw (a) edge[dashed] (b);
\end{tikzpicture}}}

\newcommand{\zzeroab}
{{
\begin{tikzpicture}[baseline=-.55ex,scale=.7, every loop/.style={}]
 \node[circle,draw,fill,inner sep=.8pt] (a) at (0,0) {};
 \node[circle,draw,dashed,inner sep=1.5pt] (b) at (-0.6,0.6) {${\alpha_i}$};
 \node[circle,draw,dashed,inner sep=1.5pt](c) at (0.6,0.6) {${\beta_i}$};
 \draw (a) edge[dashed] (b);
 \draw (a) edge[dashed] (c);
\end{tikzpicture}}}

\newcommand{\triabc}
{{
\begin{tikzpicture}[baseline=-.55ex,scale=.7, every loop/.style={}]
 \node[circle,draw,fill,inner sep=.8pt] (a) at (0,0) {};
 \node[circle,draw,dashed,inner sep=1.5pt] (b) at (-0.6,0.6) {${\alpha}$};
 \node[circle,draw,dashed,inner sep=1.5pt](c) at (0.6,0.6) {${\beta}$};
 \node[circle,draw,dashed,inner sep=1.5pt](d) at (0,-0.84) {${\gamma}$};
 \draw (a) edge[dashed] (b);
 \draw (a) edge[dashed] (c);
 \draw (a) edge[dashed] (d);
\end{tikzpicture}}}

\newcommand{\treegraph}
{{
\begin{tikzpicture}[baseline=-.55ex,scale=.7, every loop/.style={}]
 \node[circle,draw,fill,inner sep=.8pt] (a1) at (0,0) {};
 \node[circle,draw,fill,inner sep=.8pt] (a2) at (1,0) {};
 \node (q) at (2,0) {$\dots$};
 \node[circle,draw,fill,inner sep=.8pt] (a3) at (3,0) {};
 \node[circle,draw,fill,inner sep=.8pt] (a4) at (4,0) {};
 \node[circle,draw,fill,inner sep=.8pt] (a5) at (5,0) {};
 \node (q2) at (1.7,0) {};
 \node (q3) at (2.3,0) {};
 \node[circle,draw,dashed,inner sep=1.5pt] (b1) at (-0.6,0.6) {${\alpha}$};
 \node[circle,draw,dashed,inner sep=1.5pt](c1) at (-0.6,-0.6) {${\beta}$};
 \node[circle,draw,dashed,inner sep=1.5pt](b2) at (1,1) {${\alpha}$};
 \node[circle,draw,dashed,inner sep=1.5pt](b3) at (3,1) {${\alpha}$};
 \node[circle,draw,dashed,inner sep=1.5pt](b4) at (4,1) {${\alpha}$};
 \node[circle,draw,dashed,inner sep=1.5pt](b5) at (5.6,0.6) {${\alpha}$};
  \node[circle,draw,dashed,inner sep=1.5pt](c2) at (5.6,-0.6) {${\beta}$};
 \draw (a1) edge[dashed] (b1);
 \draw (a1) edge[dashed] (c1);
 \draw (a2) edge[dashed] (b2);
 \draw (a3) edge[dashed] (b3);
 \draw (a4) edge[dashed] (b4);
 \draw (a5) edge[dashed] (b5);
 \draw (a5) edge[dashed] (c2);
 \draw (a1) to (a2);
 \draw (a3) to (a4);
 \draw (a4) to (a5);
 \draw (a2) to (q2);
 \draw (a3) to (q3);
\end{tikzpicture}}}

\newcommand{\smcluster}
{{
\begin{tikzpicture}[baseline=-.55ex,scale=.7, every loop/.style={}]
 \node[circle,draw,inner sep=1.5pt] (a) at (0,0) {};
\end{tikzpicture}}}

\newcommand{\smclusterseries}
{{
\begin{tikzpicture}[baseline=-.55ex,scale=.7, every loop/.style={}]
 \node[circle,draw,inner sep=1.5pt] (a) at (0,0) {};
   \node[] (b) at (0.55,0) {$\cdots$};
     \node[circle,draw,inner sep=1.5pt] (c) at (1.,0) {};
\end{tikzpicture}}}

\newcommand{\smclusterbra}
{{
\begin{tikzpicture}[baseline=-.55ex,scale=.7, every loop/.style={}]
 \node[circle,draw,inner sep=1.5pt] (a) at (0,0) {};
 \node[circle,draw,inner sep=1.5pt] (b) at (0.5,0) {};
 \draw (a) edge[] (0.25,-0.25);
 \draw (b) edge[] (0.25,-0.25);
 \draw (0.25,-0.25) edge[] (0.25,-0.5);
\end{tikzpicture}}}

\newcommand{\examplePhi}
{{
\begin{tikzpicture}[baseline=-.55ex,scale=.7, every loop/.style={}]
 \node[circle,draw,inner sep=1.5pt] (a) at (0,0) {};
 \node[circle,draw,inner sep=1.5pt] (b) at (1,0) {};
 \node[circle,draw,inner sep=1.5pt] (c) at (2,0) {};
 \node[circle,draw,inner sep=1.5pt] (d) at (3,0) {};
 \node[circle,draw,inner sep=1.5pt] (e) at (4,0) {};
 \draw[dashed] (a) to[bend right] (b);
 \draw[dashed] (b) to[bend right] (c);
 \draw (b) to[bend right] (a);
 \draw(c) to[bend right] (b);
\draw (e) to[out=110,in=180] (4,1);
\draw (e) to[out=70,in=0] (4,1);
\draw[dashed] (e) to[out=290,in=0] (4,-1);
\draw[dashed] (e) to[out=250,in=180] (4,-1);
\node[circle,draw,densely dotted,inner sep=.5pt] (f) at (3,1) {$\alpha_1$};
\node[circle,draw,dashed,inner sep=.5pt] (g) at (3,-1) {$a_1$};
\node[circle,draw,densely dotted,inner sep=.5pt] (h) at (1,1) {$\beta_2$};
\node[circle,draw,dashed,inner sep=.5pt] (i) at (1,-1) {$b_2$};
\draw[densely dotted] (f) to (d);
 \draw[dashed] (g) to (d);
 \draw[densely dotted] (b) to (h);
 \draw[dashed] (b) to (i);
\node[circle,draw,inner sep=1.5pt] (j1) at (5,0) {};
\node[circle,draw,inner sep=1.5pt] (j2) at (6,0) {};
\end{tikzpicture}}}

\newcommand{\examplePhii}
{{
\begin{tikzpicture}[baseline=-.55ex,scale=.7, every loop/.style={}]
 \node[circle,draw,inner sep=1.5pt] (a) at (0,0) {};
 \node[circle,draw,inner sep=1.5pt] (b) at (1,0) {};
 \node[circle,draw,inner sep=1.5pt] (c) at (2,0) {};
 \node[circle,draw,inner sep=1.5pt] (d) at (3,0) {};
 \node[circle,draw,inner sep=1.5pt] (e) at (4,0) {};
\node[circle,draw,densely dotted,inner sep=.5pt] (f) at (2,1) {$\alpha_1$};
\node[circle,draw,dashed,inner sep=.5pt] (g) at (2,-1) {$a_1$};
\node[circle,draw,densely dotted,inner sep=.5pt] (h) at (1,1) {$\beta_2$};
\node[circle,draw,dashed,inner sep=.5pt] (i) at (1,-1) {$b_2$};
\draw[densely dotted] (f) to (c);
 \draw[dashed] (g) to (c);
 \draw[densely dotted] (b) to (h);
 \draw[dashed] (b) to (i);
\node[circle,draw,inner sep=1.5pt] (j1) at (5,0) {};
\node[circle,draw,densely dotted,inner sep=.5pt] (f2) at (-1,1) {$\alpha_3$};
\node[circle,draw,dashed,inner sep=.5pt] (g2) at (0,-1) {$\nu$};
\node[circle,draw,densely dotted,inner sep=.5pt] (h2) at (0,1) {$\beta_3$};
 \draw[dashed] (a) to (g2);
 \draw[densely dotted] (a) to (f2);
 \draw[densely dotted] (a) to (h2);
\end{tikzpicture}}}

\newcommand{\brabvgraphs}
{{
\begin{tikzpicture}[baseline=-.55ex,scale=.7, every loop/.style={}]
 \node[circle,draw,inner sep=1.5pt] (a) at (0,0) {};
 \node[circle,draw,inner sep=1.5pt] (b) at (0.75,0) {};
 \draw (b) to[bend right] (a);
\end{tikzpicture}}}

\newcommand{\prodbvgraphs}
{{
\begin{tikzpicture}[baseline=-.55ex,scale=.7, every loop/.style={}]
 \node[circle,draw,inner sep=1.5pt] (a) at (0,0) {};
 \node[circle,draw,inner sep=1.5pt] (b) at (0.75,0) {};
\end{tikzpicture}}}

\newcommand{\deltabvgraphs}
{{
\begin{tikzpicture}[baseline=-.55ex,scale=.7, every loop/.style={}]
 \node[circle,draw,inner sep=1.5pt] (a) at (0,0) {};
 \draw (a) to[out=110,in=180] (0,0.75);
\draw (a) to[out=70,in=0] (0,0.75);
\end{tikzpicture}}}

\newcommand{\mcbvtwo}
{{
\begin{tikzpicture}[baseline=-.55ex,scale=.7, every loop/.style={}]
 \node[circle,draw,inner sep=1.5pt] (a) at (0,0) {};
 \node[circle,draw,inner sep=1.5pt] (b) at (0.75,0) {};
 \draw (a) edge[] (0.375,-0.375);
 \draw (b) edge[] (0.375,-0.375);
 \draw (0.375,-0.375) edge[] (0.375,-0.75);
 \draw[dashed] (-0.375,0) arc (180:360:0.75);
\end{tikzpicture}}}

\newcommand{\clu}
{{
\begin{tikzpicture}[baseline=-.55ex,scale=.7, every loop/.style={}]
 \node[circle,draw,inner sep=1.5pt] (a) at (0,0) {};
 \draw[dashed] (-0.25,0) arc (180:360:0.25);
\end{tikzpicture}}}

\newcommand{\aclu}
{{
\begin{tikzpicture}[baseline=-.55ex,scale=.7, every loop/.style={}]
 \node[circle,draw,inner sep=1.5pt] (a) at (0,0) {};
 \draw[dashed] (-0.25,0) arc (180:360:0.25);
 \node[circle,draw,dashed,inner sep=.5pt] (b) at (0,-1) {$\alpha$};
 \draw[dashed] (0,-0.25) to (b);
\end{tikzpicture}}}

\newcommand{\cclu}
{{
\begin{tikzpicture}[baseline=-.55ex,scale=.7, every loop/.style={}]
 \node[circle,draw,inner sep=1.5pt] (a) at (0,0) {};
 \draw[dashed] (-0.25,0) arc (180:360:0.25);
 \node[circle,draw,dashed,inner sep=.5pt] (b) at (0,-1) {$c$};
 \draw[dashed] (0,-0.25) to (b);
\end{tikzpicture}}}

\newcommand{\pclu}
{{
\begin{tikzpicture}[baseline=-.55ex,scale=.7, every loop/.style={}]
 \node[circle,draw,inner sep=1.5pt] (a) at (0,0) {};
 \draw[dashed] (-0.25,0) arc (180:360:0.25);
 \node[circle,draw,dashed,inner sep=.5pt] (b) at (0,-1) {$p$};
 \draw[dashed] (0,-0.25) to (b);
\end{tikzpicture}}}

\newcommand{\qclu}
{{
\begin{tikzpicture}[baseline=-.55ex,scale=.7, every loop/.style={}]
 \node[circle,draw,inner sep=1.5pt] (a) at (0,0) {};
 \draw[dashed] (-0.25,0) arc (180:360:0.25);
 \node[circle,draw,dashed,inner sep=.5pt] (b) at (0,-1) {$q$};
 \draw[dashed] (0,-0.25) to (b);
\end{tikzpicture}}}

\newcommand{\aiclu}
{{
\begin{tikzpicture}[baseline=-.55ex,scale=.7, every loop/.style={}]
 \node[circle,draw,inner sep=1.5pt] (a) at (0,0) {};
 \draw[dashed] (-0.25,0) arc (180:360:0.25);
 \node[circle,draw,dashed,inner sep=.5pt] (b) at (0,-1) {$a_i$};
 \draw[dashed] (0,-0.25) to (b);
\end{tikzpicture}}}

\newcommand{\ajclu}
{{
\begin{tikzpicture}[baseline=-.55ex,scale=.7, every loop/.style={}]
 \node[circle,draw,inner sep=1.5pt] (a) at (0,0) {};
 \draw[dashed] (-0.25,0) arc (180:360:0.25);
 \node[circle,draw,dashed,inner sep=.5pt] (b) at (0,-1) {$a_j$};
 \draw[dashed] (0,-0.25) to (b);
\end{tikzpicture}}}

\newcommand{\biclu}
{{
\begin{tikzpicture}[baseline=-.55ex,scale=.7, every loop/.style={}]
 \node[circle,draw,inner sep=1.5pt] (a) at (0,0) {};
 \draw[dashed] (-0.25,0) arc (180:360:0.25);
 \node[circle,draw,dashed,inner sep=.5pt] (b) at (0,-1) {$b_i$};
 \draw[dashed] (0,-0.25) to (b);
\end{tikzpicture}}}

\newcommand{\bjclu}
{{
\begin{tikzpicture}[baseline=-.55ex,scale=.7, every loop/.style={}]
 \node[circle,draw,inner sep=1.5pt] (a) at (0,0) {};
 \draw[dashed] (-0.25,0) arc (180:360:0.25);
 \node[circle,draw,dashed,inner sep=.5pt] (b) at (0,-1) {$b_j$};
 \draw[dashed] (0,-0.25) to (b);
\end{tikzpicture}}}

\newcommand{\oclu}
{{
\begin{tikzpicture}[baseline=-.55ex,scale=.7, every loop/.style={}]
 \node[circle,draw,inner sep=1.5pt] (a) at (0,0) {};
 \draw[dashed] (-0.25,0) arc (180:360:0.25);
 \node[circle,draw,dashed,inner sep=.5pt] (b) at (0,-1) {$\nu$};
 \draw[dashed] (0,-0.25) to (b);
\end{tikzpicture}}}

\newcommand{\braclu}
{{
\begin{tikzpicture}[baseline=-.55ex,scale=.7, every loop/.style={}]
 \node[circle,draw,inner sep=1.5pt] (a) at (0,0) {};
 \node[circle,draw,inner sep=1.5pt] (b) at (0.5,0) {};
 \draw (a) edge[] (0.25,-0.25);
 \draw (b) edge[] (0.25,-0.25);
 \draw (0.25,-0.25) edge[] (0.25,-0.5);
 \draw[dashed] (-0.25,0) arc (180:360:0.5);
\end{tikzpicture}}}

\newcommand{\cbraclu}
{{
\begin{tikzpicture}[baseline=-.55ex,scale=.7, every loop/.style={}]
 \node[circle,draw,inner sep=1.5pt] (a) at (0,0) {};
 \node[circle,draw,inner sep=1.5pt] (b) at (0.5,0) {};
 \draw (a) edge[] (0.25,-0.25);
 \draw (b) edge[] (0.25,-0.25);
 \draw (0.25,-0.25) edge[] (0.25,-0.5);
 \draw[dashed] (-0.25,0) arc (180:360:0.5);
 \node[circle,draw,dashed,inner sep=.5pt] (c) at (0.25,-0.8) {$c$};
\end{tikzpicture}}}

\newcommand{\pbraclu}
{{
\begin{tikzpicture}[baseline=-.55ex,scale=.7, every loop/.style={}]
 \node[circle,draw,inner sep=1.5pt] (a) at (0,0) {};
 \node[circle,draw,inner sep=1.5pt] (b) at (0.5,0) {};
 \draw (a) edge[] (0.25,-0.25);
 \draw (b) edge[] (0.25,-0.25);
 \draw (0.25,-0.25) edge[] (0.25,-0.5);
 \draw[dashed] (-0.25,0) arc (180:360:0.5);
 \node[circle,draw,dashed,inner sep=.5pt] (c) at (0.25,-0.8) {$p$};
\end{tikzpicture}}}

\newcommand{\qbraclu}
{{
\begin{tikzpicture}[baseline=-.55ex,scale=.7, every loop/.style={}]
 \node[circle,draw,inner sep=1.5pt] (a) at (0,0) {};
 \node[circle,draw,inner sep=1.5pt] (b) at (0.5,0) {};
 \draw (a) edge[] (0.25,-0.25);
 \draw (b) edge[] (0.25,-0.25);
 \draw (0.25,-0.25) edge[] (0.25,-0.5);
 \draw[dashed] (-0.25,0) arc (180:360:0.5);
 \node[circle,draw,dashed,inner sep=.5pt] (c) at (0.25,-0.8) {$q$};
\end{tikzpicture}}}

\newcommand{\obraclu}
{{
\begin{tikzpicture}[baseline=-.55ex,scale=.7, every loop/.style={}]
 \node[circle,draw,inner sep=1.5pt] (a) at (0,0) {};
 \node[circle,draw,inner sep=1.5pt] (b) at (0.5,0) {};
 \draw (a) edge[] (0.25,-0.25);
 \draw (b) edge[] (0.25,-0.25);
 \draw (0.25,-0.25) edge[] (0.25,-0.5);
 \draw[dashed] (-0.25,0) arc (180:360:0.5);
 \node[circle,draw,dashed,inner sep=.5pt] (c) at (0.25,-0.8) {$\nu$};
\end{tikzpicture}}}

\newcommand{\pqbraclu}
{{
\begin{tikzpicture}[baseline=-.55ex,scale=.7, every loop/.style={}]
 \node[circle,draw,inner sep=1.5pt] (a) at (0,0) {};
 \node[circle,draw,inner sep=1.5pt] (b) at (0.5,0) {};
 \draw (a) edge[] (0.25,-0.25);
 \draw (b) edge[] (0.25,-0.25);
 \draw (0.25,-0.25) edge[] (0.25,-0.5);
 \draw[dashed] (-0.25,0) arc (180:360:0.5);
 \node[circle,draw,dashed,inner sep=.5pt] (c) at (0.25,-1) {$pq$};
\end{tikzpicture}}}

\newcommand{\cbrabraclu}
{{
\begin{tikzpicture}[baseline=-.55ex,scale=.7, every loop/.style={}]
 \node[circle,draw,inner sep=1.5pt] (a) at (0,0) {};
 \node[circle,draw,inner sep=1.5pt] (b) at (0.5,0) {};
 \node[circle,draw,inner sep=1.5pt] (c) at (1,0) {};
 \draw (a) edge[] (0.25,-0.25);
 \draw (b) edge[] (0.25,-0.25);
 \draw (0.25,-0.25) edge[] (0.5,-0.5);
 \draw (c) edge[] (0.5,-0.5);
  \draw (0.5,-0.75) edge[] (0.5,-0.5);
 \draw[dashed] (-0.25,0) arc (180:360:0.75);
 \node[circle,draw,dashed,inner sep=.5pt] (d) at (0.5,-1.1) {$c$};
\end{tikzpicture}}}

\newcommand{\hairyactson}
{{
\begin{tikzpicture}[baseline=-.55ex,scale=.7, every loop/.style={}]
 \node (a) at (0,0) {$x$};
 \node (b) at (2,0) {$\partial_{y}$};
 \draw[] (a) to[] (b);
\end{tikzpicture}}}

\newcommand{\hairyactsonx}
{{
\begin{tikzpicture}[baseline=-.55ex,scale=.7, every loop/.style={}]
 \node (a) at (0,0) {$x$};
 \node (b) at (2,0) {$\partial_{x}$};
 \draw[] (a) to[] (b);
\end{tikzpicture}}}

\newcommand{\hairyactsonpoincare}
{{
\begin{tikzpicture}[baseline=-.55ex,scale=.7, every loop/.style={}]
 \node (a) at (0,0) {$x^*$};
 \node (b) at (2,0) {$\partial_{y}$};
 \draw[] (a) to[] (b);
\end{tikzpicture}}}

\newcommand{\hairyactsonpoincareomega}
{{
\begin{tikzpicture}[baseline=-.55ex,scale=.7, every loop/.style={}]
 \node (a) at (0,0) {$x^*$};
 \node (b) at (2,0) {$\partial_{\omega}$};
 \draw[] (a) to[] (b);
\end{tikzpicture}}}

\newcommand{\hairyactsonone}
{{
\begin{tikzpicture}[baseline=-.55ex,scale=.7, every loop/.style={}]
 \node (a) at (0,0) {$1$};
 \node (b) at (2,0) {$\partial_{1}$};
 \draw[] (a) to[] (b);
\end{tikzpicture}}}

\newcommand{\hairysigma}
{{
\begin{tikzpicture}[baseline=-.55ex,scale=.7, every loop/.style={}]
 \node (a) at (0,0) {$\sigma(x)$};
 \node (b) at (2,0) {$\partial_{x}$};
 \draw[] (a) to[] (b);
\end{tikzpicture}}}

\newcommand{\hmcvertab}
{{
\begin{tikzpicture}[baseline=-.55ex,scale=.7, every loop/.style={}]
 \node[circle,draw,fill,inner sep=.8pt] (a) at (0,0.25) {};
 \node[] (b) at (0,-0.3) {};
 \draw (a) edge[] (b);
 \node[circle,draw,dashed,inner sep=.5pt] (c) at (-0.5,0.8) {$_{\alpha_i}$};
 \node[circle,draw,dashed,inner sep=.5pt] (d) at (0.5,0.8) {$_{\beta_i}$};
 \node[] (e) at (0,-0.4) {$_{\partial_{\omega}}$};
 \draw (a) edge[dashed] (c);
 \draw (a) edge[dashed] (d);
\end{tikzpicture}}}

\newcommand{\hmcvertomega}
{{
\begin{tikzpicture}[baseline=-.55ex,scale=.7, every loop/.style={}]
 \node[circle,draw,fill,inner sep=.8pt] (a) at (0,0.25) {};
 \node[] (b) at (0,-0.3) {};
 \draw (a) edge[] (b);
 \node[circle,draw,dashed,inner sep=.5pt] (c) at (-0.5,0.8) {$_{\omega}$};
 \node[] (e) at (0,-0.4) {$_{\partial_{\omega}}$};
 \draw (a) edge[dashed] (c);
\end{tikzpicture}}}

\newcommand{\hmcverta}
{{
\begin{tikzpicture}[baseline=-.55ex,scale=.7, every loop/.style={}]
 \node[circle,draw,fill,inner sep=.8pt] (a) at (0,0.25) {};
 \node[] (b) at (0,-0.3) {};
 \draw (a) edge[] (b);
 \node[circle,draw,dashed,inner sep=.5pt] (c) at (-0.5,0.8) {$_{\alpha_i}$};
 \node[] (e) at (0,-0.4) {$_{\partial_{\alpha_i}}$};
 \draw (a) edge[dashed] (c);
\end{tikzpicture}}}

\newcommand{\hmcvertb}
{{
\begin{tikzpicture}[baseline=-.55ex,scale=.7, every loop/.style={}]
 \node[circle,draw,fill,inner sep=.8pt] (a) at (0,0.25) {};
 \node[] (b) at (0,-0.3) {};
 \draw (a) edge[] (b);
 \node[circle,draw,dashed,inner sep=.5pt] (c) at (-0.5,0.8) {$_{\beta_i}$};
 \node[] (e) at (0,-0.4) {$_{\partial_{\beta_i}}$};
 \draw (a) edge[dashed] (c);
\end{tikzpicture}}}

\newcommand{\hmcvertone}
{{
\begin{tikzpicture}[baseline=-.55ex,scale=.7, every loop/.style={}]
 \node[circle,draw,fill,inner sep=.8pt] (a) at (0,0.25) {};
 \node[] (b) at (0,-0.3) {};
 \draw (a) edge[] (b);
 \node[] (e) at (0,-0.4) {$_{\partial_{1}}$};
\end{tikzpicture}}}

\newcommand{\hmclineab}
{{
\begin{tikzpicture}[baseline=-.55ex,scale=.7, every loop/.style={}]
 \node (a) at (0,0) {$\partial_{\alpha_i}$};
 \node (b) at (2,0) {$\partial_{\beta_i}$};
 \draw[] (a) to[] (b);
\end{tikzpicture}}}

\newcommand{\hmclinebb}
{{
\begin{tikzpicture}[baseline=-.55ex,scale=.7, every loop/.style={}]
 \node (a) at (0,0) {$\partial_{\beta_i}$};
 \node (b) at (2,0) {$\partial_{\beta_i}$};
 \draw[] (a) to[] (b);
\end{tikzpicture}}}

\newcommand{\hmclineomega}
{{
\begin{tikzpicture}[baseline=-.55ex,scale=.7, every loop/.style={}]
 \node (a) at (0,0) {$\partial_{1}$};
 \node (b) at (2,0) {$\partial_{\omega}$};
 \draw[] (a) to[] (b);
\end{tikzpicture}}}

\newcommand{\novertexhairy}
{{
\begin{tikzpicture}[baseline=-.55ex,scale=.7, every loop/.style={}]
 \node (a) at (0,0) {$\partial_{x}$};
 \node (b) at (2,0) {$\partial_{y}$};
 \draw[] (a) to[] (b);
\end{tikzpicture}}}

\newcommand{\novertexhairyxx}
{{
\begin{tikzpicture}[baseline=-.55ex,scale=.7, every loop/.style={}]
 \node (a) at (0,0) {$\partial_{x}$};
 \node (b) at (2,0) {$\partial_{x}$};
 \draw[] (a) to[] (b);
\end{tikzpicture}}}

\newcommand{\novertexpoincare}
{{
\begin{tikzpicture}[baseline=-.55ex,scale=.7, every loop/.style={}]
 \node (a) at (0,0) {$\partial_{x}$};
 \node (b) at (2,0) {$\partial_{x^*}$};
 \draw[] (a) to[] (b);
\end{tikzpicture}}}

\newcommand{\hgcomegaomega}
{{
\begin{tikzpicture}[baseline=-.55ex,scale=.7, every loop/.style={}]
 \node (a) at (0,0) {$\partial_{\omega}$};
 \node (b) at (2,0) {$\partial_{\omega}$};
 \draw[] (a) to[] (b);
\end{tikzpicture}}}

\newcommand{\hmcabii}
{{
\begin{tikzpicture}[baseline=-.55ex,scale=.7, every loop/.style={}]
 \node (a) at (0,0) {$\overset{1}{\partial_{\alpha_i}}$};
 \node (b) at (2,0) {$\overset{2}{\partial_{\beta_i}}$};
 \draw[] (a) to[] (b);
 \node (c) at (1,0.3) {$_3$};
\end{tikzpicture}}}

\newcommand{\hmcabjj}
{{
\begin{tikzpicture}[baseline=-.55ex,scale=.7, every loop/.style={}]
 \node (a) at (0,0) {$\overset{1}{\partial_{\alpha_j}}$};
 \node (b) at (2,0) {$\overset{2}{\partial_{\beta_j}}$};
 \draw[] (a) to[] (b);
 \node (c) at (1,0.3) {$_3$};
\end{tikzpicture}}}

\newcommand{\hmcbajj}
{{
\begin{tikzpicture}[baseline=-.55ex,scale=.7, every loop/.style={}]
 \node (a) at (0,0) {$\overset{2}{\partial_{\alpha_j}}$};
 \node (b) at (2,0) {$\overset{1}{\partial_{\beta_j}}$};
 \draw[] (a) to[] (b);
 \node (c) at (1,0.3) {$_3$};
\end{tikzpicture}}}

\newcommand{\dbdaij}
{{
\begin{tikzpicture}[baseline=-.55ex,scale=.7, every loop/.style={}]
 \node (a) at (0,0) {$\overset{1}{\partial_{\beta_i}}$};
 \node (b) at (2,0) {$\overset{2}{\partial_{\alpha_j}}$};
 \draw[] (a) to[] (b);
 \node (c) at (1,0.3) {$_3$};
\end{tikzpicture}}}

\newcommand{\dadbji}
{{
\begin{tikzpicture}[baseline=-.55ex,scale=.7, every loop/.style={}]
 \node (a) at (0,0) {$\overset{1}{\partial_{\alpha_j}}$};
 \node (b) at (2,0) {$\overset{2}{\partial_{\beta_i}}$};
 \draw[] (a) to[] (b);
 \node (c) at (1,0.3) {$_3$};
\end{tikzpicture}}}

\newcommand{\adaijorder}
{{
\begin{tikzpicture}[baseline=-.55ex,scale=.7, every loop/.style={}]
 \node (a) at (0,0) {$\overset{1}{\alpha_i}$};
 \node (b) at (2,0) {$\overset{2}{\partial_{\alpha_j}}$};
 \draw[] (a) to[] (b);
 \node (c) at (1,0.3) {$_3$};
\end{tikzpicture}}}

\newcommand{\bdbjiorder}
{{
\begin{tikzpicture}[baseline=-.55ex,scale=.7, every loop/.style={}]
 \node (a) at (0,0) {$\overset{1}{\beta_j}$};
 \node (b) at (2,0) {$\overset{2}{\partial_{\beta_i}}$};
 \draw[] (a) to[] (b);
 \node (c) at (1,0.3) {$_3$};
\end{tikzpicture}}}

\newcommand{\adbi}
{{
\begin{tikzpicture}[baseline=-.55ex,scale=.7, every loop/.style={}]
 \node (a) at (0,0) {$\alpha_i$};
 \node (b) at (2,0) {$\partial_{\beta_i}$};
 \draw[] (a) to[] (b);
\end{tikzpicture}}}

\newcommand{\bdai}
{{
\begin{tikzpicture}[baseline=-.55ex,scale=.7, every loop/.style={}]
 \node (a) at (0,0) {$\beta_i$};
 \node (b) at (2,0) {$\partial_{\alpha_i}$};
 \draw[] (a) to[] (b);
\end{tikzpicture}}}

\newcommand{\adaij}
{{
\begin{tikzpicture}[baseline=-.55ex,scale=.7, every loop/.style={}]
 \node (a) at (0,0) {$\alpha_i$};
 \node (b) at (2,0) {$\partial_{\alpha_j}$};
 \draw[] (a) to[] (b);
\end{tikzpicture}}}

\newcommand{\bdbji}
{{
\begin{tikzpicture}[baseline=-.55ex,scale=.7, every loop/.style={}]
 \node (a) at (0,0) {$\beta_j$};
 \node (b) at (2,0) {$\partial_{\beta_i}$};
 \draw[] (a) to[] (b);
\end{tikzpicture}}}

\newcommand{\odai}
{{
\begin{tikzpicture}[baseline=-.55ex,scale=.7, every loop/.style={}]
 \node (a) at (0,0) {$\omega$};
 \node (b) at (2,0) {$\partial_{\alpha_i}$};
 \draw[] (a) to[] (b);
\end{tikzpicture}}}

\newcommand{\odbi}
{{
\begin{tikzpicture}[baseline=-.55ex,scale=.7, every loop/.style={}]
 \node (a) at (0,0) {$\omega$};
 \node (b) at (2,0) {$\partial_{\beta_i}$};
 \draw[] (a) to[] (b);
\end{tikzpicture}}}

\newcommand{\bdone}
{{
\begin{tikzpicture}[baseline=-.55ex,scale=.7, every loop/.style={}]
 \node (a) at (0,0) {$\beta_i$};
 \node (b) at (2,0) {$\partial_{1}$};
 \draw[] (a) to[] (b);
\end{tikzpicture}}}

\newcommand{\adone}
{{
\begin{tikzpicture}[baseline=-.55ex,scale=.7, every loop/.style={}]
 \node (a) at (0,0) {$\alpha_i$};
 \node (b) at (2,0) {$\partial_{1}$};
 \draw[] (a) to[] (b);
\end{tikzpicture}}}

\newcommand{\adbij}
{{
\begin{tikzpicture}[baseline=-.55ex,scale=.7, every loop/.style={}]
 \node (a) at (0,0) {$\alpha_i$};
 \node (b) at (2,0) {$\partial_{\beta_j}$};
 \draw[] (a) to[] (b);
\end{tikzpicture}}}

\newcommand{\adbji}
{{
\begin{tikzpicture}[baseline=-.55ex,scale=.7, every loop/.style={}]
 \node (a) at (0,0) {$\alpha_j$};
 \node (b) at (2,0) {$\partial_{\beta_i}$};
 \draw[] (a) to[] (b);
\end{tikzpicture}}}

\newcommand{\bdaij}
{{
\begin{tikzpicture}[baseline=-.55ex,scale=.7, every loop/.style={}]
 \node (a) at (0,0) {$\beta_i$};
 \node (b) at (2,0) {$\partial_{\alpha_j}$};
 \draw[] (a) to[] (b);
\end{tikzpicture}}}

\newcommand{\bdaji}
{{
\begin{tikzpicture}[baseline=-.55ex,scale=.7, every loop/.style={}]
 \node (a) at (0,0) {$\beta_j$};
 \node (b) at (2,0) {$\partial_{\alpha_i}$};
 \draw[] (a) to[] (b);
\end{tikzpicture}}}

\newcommand{\lsek}
{{
\begin{tikzpicture}[baseline=-.55ex,scale=.7, every loop/.style={}]
 \node[circle,draw,inner sep=1.5pt] (a) at (0,0) {$k$};
  \node[circle,draw,densely dotted,inner sep=1.5pt] (b) at (0,1) {${e_q}$};
  \draw (a) edge[densely dotted] (b);
\end{tikzpicture}}}

\newcommand{\lstkl}
{{
\begin{tikzpicture}[baseline=-.55ex,scale=.7, every loop/.style={}]
 \node[circle,draw,solid,inner sep=1.5pt] (a) at (0,0) {$k$};
  \node[circle,draw,solid,inner sep=1.5pt] (b) at (1,0) {$l$};
  \draw (a) edge[draw,solid] (b);
\end{tikzpicture}}}

\newcommand{\lstkk}
{{
\begin{tikzpicture}[baseline=-.55ex,scale=.7, every loop/.style={}]
 \node[circle,draw,solid,inner sep=1.5pt] (a) at (0,0) {$k$};
  \draw (a) edge[draw,solid,loop] (a);
\end{tikzpicture}}}

\usepackage{fullpage}
\usepackage{graphicx}
\usepackage{extpfeil}
\usepackage{enumitem}

\allowdisplaybreaks

\begin{document}

\title{Graph complexes and higher genus Grothendieck-Teichm\"uller Lie algebras}

\author[M. Felder]{Matteo~Felder}
\address{Matteo~Felder: Department of Mathematics, ETH Zurich, Zurich, Switzerland}
\email{matteo.felder@math.ethz.ch}



\begin{abstract}
We give a presentation in terms of generators and relations of the cohomology in degree zero of the Campos-Willwacher graph complexes associated to compact orientable surfaces of genus $g$. The results carry a natural Lie algebra structure, and for $g=1$ we recover Enriquez' elliptic Grothendieck-Teichm\"uller Lie algebra. In analogy to Willwacher's theorem relating Kontsevich's graph complex to Drinfeld's Grothendieck-Teichm\"uller Lie algebra, we call the results higher genus Grothendieck-Teichm\"uller Lie algebras. Moreover, we find that the graph cohomology vanishes in negative degrees.
\end{abstract}

\maketitle

\newcommand{\mapsfrom}{\mathrel{\reflectbox{\ensuremath{\mapsto}}}}

\newcommand{\GT}{\mathrm{GT}}
\newcommand{\GRT}{\mathrm{GRT}}
\newcommand{\Gal}{\mathrm{Gal}}
\newcommand{\lD}{\mathrm{D}}
\newcommand{\Aut}{\mathrm{Aut}}
\newcommand{\Der}{\mathrm{Der}}
\newcommand{\BiDer}{\mathrm{BiDer}}
\newcommand{\Chains}{\mathrm{Chains}}
\newcommand{\K}{\mathbb K}
\newcommand{\FreeOp}{\mathrm{FreeOp}}
\newcommand{\conf}{\mathrm{conf}}
\newcommand{\Diff}{\mathrm{Diff}}
\newcommand{\MC}{\mathsf{MC}}

\newcommand{\Harru}{\mathrm{Harr}^{\mathrm{unred}}}
\newcommand{\Omegau}{\Omega^{\mathrm{unred}}}

\newcommand{\e}{\mathsf{e}}
\newcommand{\ke}{\mathsf{e_2^\text{!`}}}
\newcommand{\bv}{\mathsf{bv}}
\newcommand{\kbv}{\mathsf{bv^\text{!`}}}
\newcommand{\Com}{\mathsf{Com}}
\newcommand{\bvk}{\mathsf{bv^!}}

\newcommand{\Qc}{\mathsf{Q}}
\newcommand{\Tc}{\mathsf{T}}
\newcommand{\Cc}{\mathsf{C}}

\newcommand{\frakg}{\mathfrak{g}}
\newcommand{\frakh}{\mathfrak{h}}
\newcommand{\wOmega}{\widebar\Omega}
\newcommand{\CoDef}{\mathrm{CoDef}}
\newcommand{\BiDef}{\mathrm{BiDef}}
\newcommand{\hL}{\widehat{\mathbb{L}}}

\newcommand{\rell}{\mathfrak{r}^{ell}}
\newcommand{\frakt}{\mathfrak{t}}
\newcommand{\barfrakt}{\overline{\mathfrak{t}}}
\newcommand{\grtell}{\alg{grt}^{ell}}
\newcommand{\grt}{\alg{grt}}

\newcommand{\grtg}{\alg{grt}_{\mathrm{(g)}}}
\newcommand{\fraksg}{\mathfrak{s}_{\mathrm{(g)}}}

\newcommand{\frakr}{\mathfrak{r}}
\newcommand{\fraka}{\mathfrak{a}}
\newcommand{\fraks}{\mathfrak{s}}

\newcommand{\td}{\bar{d}}

\newcommand{\PK}{\widetilde \Phi_{\tilde \kappa}}

\newcommand{\icg}{\mathsf{ICG}}
\newcommand{\BVicg}{\mathsf{BVICG}}
\newcommand{\pdicg}{{}^*\mathsf{ICG}}
\newcommand{\pdBVGraphs}{\mathsf{BVGraphs}^*}
\newcommand{\pdBVicg}{{}^*\mathsf{BVICG}}
\newcommand{\fGraphs}{\mathsf{fGraphs}}
\newcommand{\divGC}{\mathsf{GC}^{div}}
\newcommand{\tcg}{\mathsf{TCG}}

\newcommand{\icgg}{\mathsf{ICG}_{\mathrm{(g)}}}
\newcommand{\tcgg}{\mathsf{TCG}_{\mathrm{(g)}}}
\newcommand{\GCg}{\mathsf{GC}_{\mathrm{(g)}}}

\newcommand{\fGraphsg}{\mathsf{fGraphs}_{\mathrm{(g)}}}

\newcommand{\I}{\mathsf{I}}

\newcommand{\HGC}{\mathsf{HGC}}
\newcommand{\fHGC}{\mathsf{fHGC}}
\newcommand{\HGCg}{\mathsf{HGC}_{(\mathrm{g})}}
\newcommand{\fHGCg}{\mathsf{fHGC}_{(\mathrm{g})}}

\newcommand{\Grag}{\mathsf{Gra}_{(\mathrm{g})}}

\newcommand{\Graphsg}{\mathsf{Graphs}_{(\mathrm{g})}}
\newcommand{\pdGraphsg}{\mathsf{Graphs}_{(\mathrm{g})}^*}

\newcommand{\BVGraphsg}{\mathsf{BVGraphs}_{\mathrm{(g)}}}
\newcommand{\pdBVGraphsg}{\mathsf{BVGraphs}_\mathrm{(g)}^*}

\newcommand{\uniGraphsg}{\mathsf{uniGraphs}_{(\mathrm{g})}}
\newcommand{\uniGraphs}{\mathsf{uniGraphs}}

\newcommand{\gra}{\mathsf{gra}}

\newcommand{\Mog}{\mathsf{Mo}_{(\mathrm{g})}}
\newcommand{\Mo}{\mathsf{Mo}}
\newcommand{\Mogg}{\mathsf{{Mo}}^!_{(\mathrm{g})}}

\newcommand{\LSg}{\mathsf{LS}^*_{(\mathrm{g})}}
\newcommand{\LSgg}{\mathsf{{LS}^*}^!_{(\mathrm{g})}}

\newcommand{\pdLSg}{\mathsf{LS}_{(\mathrm{g})}}

\newcommand{\pdTw}{{}^*\Tw}

\newcommand{\Ve}{V_{\exp}}
\newcommand{\bVe}{\overline{V}_{\exp}}
\newcommand{\bicg}{\overline{\mathsf{ICG}}}

\newcommand{\fraktg}{\mathfrak{t}_{(\mathrm{g})}}
\newcommand{\CE}{\mathrm{CE}}

\newcommand{\tk}{{\tilde \kappa}}

\newcommand{\flD}{\mathsf{D}^{fr}}
\newcommand{\tGC}{\mathsf{tGC}}
\newcommand{\osp}{\mathfrak{osp}}
\newcommand{\lsp}{\mathfrak{sp}}
\newcommand{\spp}{\mathfrak{sp}'}
\newcommand{\gl}{\mathfrak{gl}}

\newcommand{\mflD}{\mathsf{mD}^{fr}}
\newcommand{\mfM}{\mathsf{mM}^{fr}}
\newcommand{\vspan}{\mathrm{span}}

\newcommand{\fG}{\mathsf{fG}}
\newcommand{\actson}{\circlearrowright}

\usetikzlibrary{decorations.pathreplacing,calc}

\setcounter{tocdepth}{1}
\tableofcontents

\section{Introduction}

Configuration spaces of $r$ points in a surface $X$ 
\[
\Conf_r(X)=\{(x_1,\dots,x_r)\in X^{\times r} \ | \ \forall i\neq j \ : \ x_i\neq x_j \}
\]
are classical objects in topology. In the ``local'' case of $X=\mathbb{R}^2$ its compactifications assemble to form a topological operad $\FM_2$, the Fulton-MacPherson-Axelrod-Singer operad (\cite{AS94},\cite{FM94},\cite{GetzlerJones94}). It is weakly-equivalent to the little disks operad and as such it boasts a wide range of applications in topology, algebra and mathematical physics. For $X=\Sigma_g$ a compact orientable surface of genus $g$, the compactification of framed configuration spaces of points leads to an operadic module $\FFM_{\Sigma_g}$ over the framed variant of the Fulton-MacPherson-Axelrod-Singer operad $\FFM_2$. In a recent series of works Idrissi, Campos and Willwacher (\cite{CW}, \cite{Idrissi19}, \cite{CIW}) compute their homotopy type by means of workable combinatorial models. These consist of a cooperadic comodule $\op M$ over a cooperad $\op C$ such that
\[
\op M(r) \simeq \Omega(\FFM_{\Sigma_g}(r)) \text{ and } \op C(r) \simeq \Omega(\FFM_2(r))
\]
are quasi-isomorphisms of differential graded commutative algebras over a field of characteristic zero, which are compatible with the respective cooperadic comodule structures.

Notice that for a fixed number of points, understanding the homotopy type of each individual space $\Conf_r(X)\simeq \FM_2(r)$ goes back to Arnold \cite{Arnold69} and Cohen \cite{Cohen76} for the local case, and Bezrukavnikov \cite{Bezru94} for surfaces of higher genus. Major breakthroughs by Tamarkin \cite{Tamarkin03}, Kontsevich \cite{Kontsevich99}, Lambrechts and Volic \cite{LV14}, and Fresse \cite{FresseBook17} lead to highly non-trivial results on the homotopy type of configuration spaces of points in the local case as an operad, i.e. respecting the algebraic structure tying together the spaces for different numbers of points. In their approach to tackle the (operadic) surface case, which indeed extends to higher dimensional manifolds in great generality (\cite{CILW}, \cite{CDIW}, \cite{FW21}), Idrissi, Campos and Willwacher construct the cooperadic comodule $\op M=\pdBVGraphsg$ over the cooperad $\op C=\pdBVGraphs$, the latter being a resolution of the Batalin-Vilkovisky cooperad and a model for framed configurations of points in the plane (\cite{Getzler94}, \cite{Severa10}, \cite{GP10}). Elements of $\pdBVGraphsg(r)$ are essentially given by linear combinations of diagrams of the form (here $r=4$)
\\
\begin{center}
{{
\begin{tikzpicture}[baseline=-.55ex,scale=.7, every loop/.style={}]
 \node[circle,draw,inner sep=1.5pt] (a) at (0,0) {$1$};
 \node[circle,draw,inner sep=1.5pt] (b) at (1,0) {$2$};
 \node[circle,draw,inner sep=1.5pt] (c) at (2,0) {$3$};
 \node[circle,draw,inner sep=1.5pt] (d) at (3,0) {$4$};
 \node[circle,draw,fill,inner sep=1.5pt] (d1) at (0.5,1) {};
 \node[circle,draw,fill,inner sep=1.5pt] (d2) at (1.5,1) {};
\node[circle,draw,densely dotted,inner sep=.5pt] (f) at (-0.5,1) {$a_1$};
\node[circle,draw,densely dotted,inner sep=.5pt] (g) at (0,2) {$b_2$};
\node[circle,draw,densely dotted,inner sep=.5pt] (h) at (1,2) {$a_3$};
\node[circle,draw,densely dotted,inner sep=.5pt] (j) at (2,2) {$\omega$};
\draw (a) to (d1);
\draw (b) to (d1);
\draw (b) to (d2);
\draw (c) to (d2);
\draw (d1) to (d2);
\draw[densely dotted] (d1) to (g);
\draw[densely dotted] (a) to (f);
\draw[densely dotted] (d2) to (h);
\draw[densely dotted] (d2) to (j);
\draw (d) edge[loop] (d);
\end{tikzpicture}}}
\end{center}
i.e. diagrams with two types of vertices which possibly carry decorations in the reduced cohomology of the surface $\overline H=H^{\geq 1}(\Sigma_g)$. It comes naturally equipped with a combinatorially defined action of the dg Lie algebra $\spp(H^*)\ltimes \GCg$, where $\GCg$ is the Lie algebra whose elements are linear combinations of diagrams of the form
\\
\begin{center}
{{
\begin{tikzpicture}[baseline=-.55ex,scale=.7, every loop/.style={}]
\node[circle,draw,fill,inner sep=1.5pt] (a) at (0,0) {};
\node[circle,draw,fill,inner sep=1.5pt] (b) at (1,1) {};
\node[circle,draw,fill,inner sep=1.5pt] (c) at (-1,1) {};
\node[circle,draw,fill,inner sep=1.5pt] (d) at (0,2) {};
\node[circle,draw,densely dotted,inner sep=.5pt] (f) at (-2,1) {$\alpha_1$};
\node[circle,draw,densely dotted,inner sep=.5pt] (g) at (-0.5,3) {$\beta_2$};
\node[circle,draw,densely dotted,inner sep=.5pt] (h) at (0.5,3) {$\alpha_3$};
\node[circle,draw,densely dotted,inner sep=.5pt] (j) at (2,1) {$\beta_1$};
\draw (a) to (b);
\draw (b) to (c);
\draw (b) to (d);
\draw (a) to (d);
\draw (a) to (c);
\draw (c) to (d);
\draw[densely dotted] (d) to (g);
\draw[densely dotted] (d) to (h);
\draw[densely dotted] (c) to (f);
\draw[densely dotted] (b) to (j);
\end{tikzpicture}}}
\end{center}
with decorations living in $\overline{H}^*$, the dual of the cohomology of the surface.  The Lie bracket on $\GCg$ is also defined in a combinatorial way. Moreover, $\spp(H^*)$ is a Lie subalgebra of $\osp(H^*)$, the Lie algebra consisting of (graded) linear endomorphisms of $H^*$ preserving the natural pairing on $H^*$.

In this paper, we study the deformation theory of the cooperadic comodule $\pdBVGraphsg$. For this we set up a suitable version of the deformation complex $\Def(\pdBVGraphsg)$, inspired by work of Fresse and Willwacher for the cooperad case \cite{FW20}. In analogy to Fresse and Willwacher's construction, we expect the deformation complex to describe the dg vector space underlying the dg Lie algebra of homotopy biderivations of $\pdBVGraphsg$ - however, this remains an open problem at this stage. The deformation complex comes equipped with a morphism of degree one from the dg Lie algebra $\spp(H^*)\ltimes \GCg$ induced by its action on $\pdBVGraphsg$. We find that, via this action and up to one class of degree three, the dg Lie algebra $\spp(H^*)\rtimes \GCg$ describes precisely the type of deformations encoded in $\Def(\pdBVGraphsg)$, up to homotopy. Our first main result reads as follows.
\begin{thm}(Theorem \ref{thm:defGC})\label{thm:Theoremone}
The action of $\spp(H^*)\rtimes \GCg$ on $\pdBVGraphsg$ induces a quasi-isomorphism \[
(\spp(H^*)\ltimes \GCg)[-1]\rightarrow \Def(\pdBVGraphsg)\]
in all degrees except in degree $4$. In degree $4$, the map on cohomology is injective with one-dimensional cokernel.
\end{thm}

The second main result addresses the cohomology of $\spp(H^*)\rtimes \GCg$ in degree zero. Recall the (framed version of the) Lie algebra $\fraktg(n)$ introduced by Bezrukavnikov \cite{Bezru94}. Generated by the elements $x_l^{(i)}, \ y_l^{(i)}, \ t_{ij}$ for $1\leq i,j\leq n$, $1\leq l \leq g$, subject to a set of relations, the collection $\fraktg=\{\fraktg(n)\}_{n\geq 1}$ assembles to form an operadic module over the (framed version of the) Drinfeld-Kohno Lie algebra operad $\frakt_\bv$. Its Chevalley-Eilenberg complex $C(\fraktg)$ carries the structure of a cooperadic $C(\frakt_\bv)$-comodule. It is related to $\pdBVGraphsg$ through a zig-zag of quasi-isomorphisms, thus forming a further model for the spaces of configuration of points. This equivalence enables us to simplify the deformation complex and express its cohomology in non-positive degrees and in degree one in terms of $\fraktg$. More concretely, we explicitly define a Lie subalgebra
\[
Z_{(g)} \subset \Der(\fraktg(2))
\]
of the Lie algebra of derivations of $\fraktg(2)$, and a Lie ideal $B_{(g)}\subset Z_{(g)}$ in terms of generators and relations (see equation \eqref{eq:defZg} for the precise formulas), and set $\frakr_{(g)}:=Z_{(g)}/B_{(g)}$. The second main result then reads as follows.

\begin{thm}(Theorem \ref{thm:GCgvsp})\label{thm:Theoremtwo}
In degree zero, we have
\[
H^0(\spp(H^*)\ltimes \GCg)\cong \frakr_{(g)}.
\]
Moreover,
\[
H^{-1}(\spp(H^*)\ltimes \GCg)=
\begin{cases}
\spp_{-1}(H^*)=\K [1]\oplus \K[1] & \text{ for } g=1\\
0 & \text{ for } g\geq 2
\end{cases}
\]
and for all $i<-1$, 
\[
H^{i}(\spp(H^*)\ltimes \GCg)=H^{i}(\GCg)=0.
\]
In particular, $H^{i}(\GCg)=0$ for all $i<0$ and $g\geq 1$.
\end{thm}

The cohomology in positive degrees is still unknown. We call $\frakr_{(g)}$ the Grothendieck-Teichm\"uller Lie algebra of a surfaces of genus $g$. This is inspired by Willwacher's seminal paper \cite{Willwacher15} on the connection between the Grothendieck-Teichm\"uller Lie algebra, Kontsevich's graph complex and homotopy derivations of the Gerstenhaber operad. Indeed, in that case, elements of the graph complex consist of linear combinations of similar decoration-free diagrams. In degree zero the cohomology is isomorphic to Drinfeld's Grothendieck-Teichm\"uller Lie algebra, and up to one class, this corresponds to the degree zero cohomology of homotopy derivations of the Gerstenhaber operad (which is itself a model for configuration spaces of points in the plane). Additionally, observe that Enriquez generalized Drinfeld's construction of the Grothendieck-Teichm\"uller Lie algebra in the case of $g=1$ \cite{Enriquez14}. Within our approach, notice that $\Sigma_1$ is parallelizable and the framing is thus not necessary. Although the combinatorial models (and the notation, as will become apparent below) need to be modified slightly, the results still carry over accordingly. More importantly though, we recover Enriquez' version of the elliptic Grothendieck-Teichm\"uller Lie algebra $\rell$ (see equation \eqref{eq:defrell} for its explicit definition) in this case.

\begin{cor}(Corollary \ref{cor:GCrell})
For $g=1$, the zeroth cohomology of the graph complex is isomorphic to the elliptic Grothendieck-Teichm\"uller Lie algebra, i.e.
\[
H^0(\spp(H^*)\rtimes \GC_{(1)}^{\minitadp})\cong   \rell.
\]
\end{cor}

\begin{rem}
Notice that the analogy above was announced in some form by Campos and Willwacher's in (\cite{CW}, Section 7). Indeed they suggest that the Grothendieck-Teichm\"uller Lie algebra of a surface $\Sigma_g$ should be defined as the homotopy derivations of a real model of the pair $(\FFM_{\Sigma_g},\FFM_2)$. Moreover, using their techniques, this definition extends to higher-dimensional manifolds (see also \cite{FW21}). Our text is a first step in the direction sketched by Campos and Willwacher. 
\end{rem}

\begin{rem}
In his recent work \cite{Gonzalez20}, Gonzalez defines the Grothendieck-Teichm\"uller group for a surface of genus $g$ by means of operad theory. Whether the corresponding Lie algebra is isomorphic to $H^0(\spp(H^*)\rtimes \GCg)$ (and thus our version of the Grothendieck-Teichm\"uller Lie algebra) is an open problem for $g\geq 2$. It holds in genus one, since in this case, his result also coincides with Enriquez' definition of the Grothendieck-Teichm\"uller group.
\end{rem}

\begin{rem}
While working on this text, we learned from Benjamin Enriquez that his approach in genus one can be extended to surfaces of higher genus - thus providing a third way of defining the Grothendieck-Teichm\"uller Lie algebra for surfaces. Indeed, he had already found the explicit formulas, and he generously shared his important insights with us, also in form of an unpublished manuscript.
\end{rem}

\subsection{Structure of the paper}
In Section \ref{sec:defforcoops}, we recall the deformation complex in the cooperadic case, as introduced by Fresse and Willwacher. This construction is generalized to cooperadic comodules in Section \ref{sec:defformod}. Next, we recall two combinatorial models for the framed configurations of points on surfaces due to Idrissi, Campos and Willwacher, as well as the graph complexes $\GCg$ and its ``hairy'' version $\HGCg$. At this point, we have all the tools to formulate Theorem \ref{thm:Theoremone}. Its proof makes use of a first simplification of the deformation complex, and the relation of $\GCg$ and $\HGCg$ (up to $\spp(H^*)$ and a class in degree three, they are quasi-isomorphic). In Section \ref{sec:tg}, we introduce the Lie algebra $\fraktg(n)$ and address the operadic $\frakt_\bv$-module structure the collection $\fraktg$ carries. As mentioned in the introduction, its Chevalley-Eilenberg complex $C(\fraktg)$ forms itself a model for the framed configurations of points on surfaces - this is the main result of Section \ref{sec:tg}, and indeed simply the operadic upgrade of Bezrukavnikov's work. In Section \ref{sec:grtg}, we recall Enriquez' work in genus one and define the Grothendieck-Teichm\"uller Lie algebra of a surface of genus $g$. In the last two sections, we prove Theorem \ref{thm:Theoremtwo}, and use this result to deduce further properties of the graph complex and $\frakr_{(g)}$.

\subsection{Acknowledgements}
I am very grateful and highly indebted to Thomas Willwacher for suggesting this research topic and for shaping most of my understanding of it over the past couple of years. In particular, he provided the main arguments for the proof of Theorem \ref{thm:Theoremone} by noting that the quasi-isomorphism should factor through the ``hairy'' graph complex. I heartily thank Benjamin Enriquez for many interesting discussions, for generously sharing his unpublished manuscript containing some of the results appearing in this text and for his warm encouragement. I am grateful to Anton Alekseev and Florian Naef for numerous helpful suggestions and their patient support. In particular, I am indebted to Florian Naef for providing the argument for the proof of Lemma \ref{lemma:center}. Furthermore, it is a pleasure to thank Ricardo Campos, Stefano D'Alesio, Julien Ducoulombier, Najib Idrissi and Elise Raphael for many useful exchanges. This work was supported by the grants no. 340243 MODFLAT and no. 678156 GRAPHCPX of the European Research Council (ERC).

\section{Preliminaries}

\subsection{Notations and conventions}

We work over a field $\K$ of characteristic zero. For a differential graded (dg) vector space $V$, the $r$-fold desuspension is denoted by $V[r]$. The desuspension operator $s$ is of degree $-1$, i.e. for a homogeneous element $x\in V$, $sx\in V[1]$ and $|sx|=|x|-1$. We work with cohomological conventions, i.e. all differentials are of degree $+1$. If each graded component of a graded vector space $V$ is finite dimensional, we say that $V$ is of finite type.
\\
\\
Let $(\frakg,d)$ be a dg Lie algebra acting on a dg vector space $(V,d_V)$. A Maurer-Cartan element is a degree one element $m\in g$ satisfying the Maurer-Cartan equation
\[
dm+\frac12[m,m]=0.
\]
We may form the twisted dg Lie algebra $\frakg^m$ equipped with the same bracket as $\frakg$ but with twisted differential $d_m=d+[m,-]$. Similarly, we define the twisted dg vector space $V^m$ by twisting the differential by the action of $m$, i.e. $V^m=(V,d_V+m\cdot -)$. It comes equipped with a dg Lie algebra action of $\frakg^m$.
\\
\\
Our notational conventions regarding symmetric collections (also referred to as $S$-modules), operads and operadic modules mostly follow the textbook by Loday and Vallette \cite{LodayVallette12}. For symmetric collections $M$ and $N$, their composite $M\circ N$ is the symmetric collection
\[
(M\circ N)(n)=\bigoplus\limits_{k\geq 0} M(k) \otimes \left(\bigoplus_{i_1+\dots +i_k=n} N(i_1)\otimes \dots\otimes N(i_k)\otimes_{S_{i_1}\times \dots \times S_{i_k}} \K[S_n] \right).
\]
The equivalence classes in $(M\circ N)(n)$ are thus represented by elements which we denote by
\[
(\mu;\nu_1,\dots,\nu_k;\sigma)
\]
for $\mu \in M(k)$, $\nu_1,\dots,\nu_k$ with $\nu_j\in N(i_j)$ and a shuffle $\sigma \in \mathrm{Sh}(i_1,\dots,i_k)$. When $\sigma=\id$, we omit it in the notation above.

\subsection{Hopf operads and cooperads}

A dg Hopf operad is an operad $\op P$ in the category of dg (counitary coaugmented) cocommutative coalgebras. Thus, it consists of a collection of dg cocommutative coalgebras $\op P(r)$ for $r\geq 1$, each equipped with an action of the symmetric group $S_r$, an identity element $\id \in \op P(1)$ and operadic composition morphisms $\gamma_{\op P}:\op P \circ \op P \rightarrow \op P$ satisfying the natural equivariance, unit and associativity conditions. Moreover, these morphisms all describe morphisms of dg cocommutative coalgebras (see also \cite{Fresse1}, I. Chapter 3). 

For the partial operadic composition morphisms we adopt the notation $\circ_i:\op P(k)\otimes \op P(l) \rightarrow \op P(k+l-1)$ for $1\leq i \leq k$, and $\mu\circ_i \nu := \gamma_{\op P} (\mu;\id,\dots,\id,\nu,\id,\dots,\id)$. It is sometimes more convenient to index the underlying symmetric collection by finite sets $S$, and to write $\op P(S)$ for the space in arity $|S|$, the cardinality of the set $S$. The partial operadic composition formula then reads
\[
\circ_*:\op P(S')\otimes \op P (S'')\rightarrow \op P ((S' \setminus \{*\}) \sqcup S'')
\]
for arbitrary finite sets $S'$ and $S''$. Finally, notice that the composition morphisms can be used to define treewise composition operations
\[
\gamma_T:\bigotimes\limits_{v\in VT} \op P(\text{star}(v)) \rightarrow \op P(S)
\]
where we take a tensor product over the vertex set $VT$ of a rooted tree $T$ whose leafs are labelled by $S$. Here $\text{star}(v)$ stands for the set of incoming edges at the vertex $v$. The notation is borrowed from (\cite{FTW17}, Section 0.2).

A dg Hopf cooperad $\op C$ is a collection of (unitary augmented) dg commutative algebras $\op C(r)$ for $r\geq 1$, each equipped with an $S_r$-action compatible with the algebra structure. Moreover, we have cooperadic cocompositions $\Delta_{\op C}:\op C\rightarrow \op C \circ \op C$ satisfying the relations dual to ones of operads, and which are morphisms of dg commutative algebras. 

We denote the partial cocomposition morphisms by $\Delta_i:\op C(k+l-1)\rightarrow \op C(k)\otimes \op C(l)$, for $i=1,\dots,k$. These also describe morphisms of dg commutative algebras with the right hand side equipped with the diagonal algebra structure. As in the operad case, it is sometimes useful to index the symmetric collection underlying a cooperad by finite sets. In that case, we write
\[
\Delta_*: \op C(S)\rightarrow \op C(S')\otimes C(S'')
\]
for the partial cooperadic cocompositions. Moreover, these may be used to describe the treewise cocomposition operations
\[
\Delta_T: \op C(S) \rightarrow \bigotimes\limits_{v\in VT} \op C(\text{star}(v))=:\op F_T \op C
\]
defined for rooted trees whose leafs are indexed by the finite set $S$.

%

\subsection{(Co)operadic Hopf (co)modules}\label{sec:comodules}

A right operadic dg Hopf module over a Hopf operad $\op P$ consists of a collection of (counitary coaugmented) dg cocommutative coalgebras, $\op X(r)$ for $r\geq 1$, each equipped with an action of the symmetric group $S_r$, and equivariant composition morphisms $\op X \circ \op P\rightarrow \op X$ satisfying the natural compatibility relations with the operad structure. These morphisms are moreover required to be compatible with the respective coalgebra structures. We denote the corresponding partial composition morphism $\circ_i:\op X(k)\otimes \op P(l)\rightarrow \op X(k+l-1)$ for $1\leq i \leq k$ which are also subject to the natural compatibility conditions.

Dually, a right cooperadic dg Hopf comodule $\op M$ over a Hopf operad $\op C$ consists of a collection of (unitary augmented) dg commutative algebras $\op M(r)$ for $r\geq 1$, each equipped with an $S_r$-action compatible with the algebra structure. Moreover, it comes with cocomposition morphisms $\Delta_{\op M}:\op M\rightarrow \op M\circ \op C$ which satisfy the relations dual to the ones of operadic modules, and are morphisms of dg commutative algebras. The corresponding partial cocomposition morphisms are denoted by $\Delta_i:\op M(k+l-1)\rightarrow \op M(k)\otimes \op C(l)$. Moreover, these may be used to describe the treewise cocompositon operations
\[
\Delta_T: \op M(S) \rightarrow \op M (\text{star}(v_0))\otimes \bigotimes\limits_{v\in VT\setminus \{v_0\}}  \op C(\text{star}(v))=:\op F_T(\op M;\op C)
\]
shaped after rooted trees with root $v_0$ whose leafs are indexed by the finite set $S$.

\subsection{Bicomodules over cooperads and bimodules over operads}\label{sec:bicom}
Fix a symmetric collection $ \Xi=\{ \Xi(r)\}_{r\geq 1}$ in the category of graded vector spaces. We recall the natural extension of a right (co)operadic (co)module over a (co)operad. Given an operad $\op P$ in graded vector spaces, we call $\Xi$ a bimodule over the operad $\op P$ if we have left and right equivariant partial composition morphisms
\begin{align*}
\circ_i:\op P(k) \otimes \Xi(l) &\rightarrow \Xi(k+l-1)\\
\circ_i:\Xi(k)\otimes \op P(l)&\rightarrow \Xi(k+l-1)
\end{align*}
satisfying the natural compatibility relations with the operadic structure. 

Finally, given a cooperad $\op C$ in the category of graded vector spaces, $\Xi$ defines a bicomodule over the cooperad $\op C$ if we have left and right partial cocomposition morphisms
\begin{align*}
\Delta_i: \Xi(k+l-1)&\rightarrow \Xi(k)\otimes \op C(l)\\
\Delta_i:\Xi(k+l-1)&\rightarrow \op C(k) \otimes \Xi(l)
\end{align*}
satisfying the natural extension of equivariance, counit, and coassociativity conditions of the cocompositions of cooperads.

\subsection{Symmetric bimodules over Hopf collections}

A symmetric collection $\op P=\{\op P (r)\}_{r\geq 1}$ in the category of graded commutative algebras is called a Hopf collection. Given such a symmetric Hopf collection $\op P$, we call a symmetric collection $\Xi$ in the category of graded vector spaces a symmetric bimodule over $\op P$ if on each $\Xi(r)$ we have the structure of a bimodule over the commutative algebra $\op P(r)$, i.e. there are equivariant symmetric right and left product operations $-\wedge-: \Xi(r) \otimes \op P(r) \rightarrow \Xi(r)$ and $-\wedge-: \op P(r) \otimes \Xi(r) \rightarrow  \Xi(r)$ for all $r\geq 1$ which satisfy the  usual unit and associativity relations for modules over commutative algebras (see also \cite{FW20}, Definition 1.1.5).

\subsection{Cobar construction for cooperads}

Let $\op C$ be a coaugmented cooperad in the category of graded vector spaces. The \emph{unreduced} cobar construction of $\op C$ is the quasi-free operad
\[
\Omega^{\mathrm{unred}} \op C= (\op {F}(\op C [-1]),d_{\Omega\op C})
\]
where $d_{\Omega\op C}$ is the unique operadic derivation of the free operad extending the unreduced infinitesimal decomposition map $\Delta_{(1)}:\op C\rightarrow \op C \circ_{(1)} \op C$ (see \cite{LodayVallette12}, Section 6.1.3). The infinitesimal decomposition can be viewed as the decomposition of an element in $\op C$ into two parts. Notice also that we adopt Loday and Vallette's sign convention and within the definition of $d_{\Omega\op C}$ we use $\Delta_s: \K s^{-1}\rightarrow \K s^{-1}\otimes \K s^{-1}$, $\Delta_s(s^{-1}):=-s^{-1}\otimes s^{-1}$ (see \cite{LodayVallette12}, Section 6.5.5). Moreover, we denote by $\iota:\op C \rightarrow \Omegau \op C$ the natural (universal twisting) morphism given by the composition
\[
\iota:\op C \rightarrow \op C[-1] \hookrightarrow \op {F}(\op C[-1]).
\]
Its \emph{reduced} version $\Omega \op C$ (to which we refer to as the cobar construction) is defined as
\[
\Omega \op C = (\mathcal{F}(\overline{ \op C}[-1]) , d_{\Omega\op C})
\]
where $\overline {\op C}$ denotes the coaugmentation coideal of $\op C$. Here, $d_{\Omega\op C}$ extends the reduced infinitesimal decomposition map $\Delta_{(1)}: \overline{\op C} \rightarrow \overline{\op C} \circ_{(1)} \overline{\op C}$. In this case, we have
\[
\iota:\overline{\op C} \rightarrow \overline{\op C}[-1] \hookrightarrow \op {F}(\overline{\op C} [-1]).
\]
\begin{rem}
We will often use that the free operad $\op F \op C$ allows an expansion in terms of trees, i.e.
\[
\op F \op C= \bigoplus\limits_{T} \op F_T \op C
\]
where the sum runs over rooted trees, and the summands $\op F_T \op C$ are treewise tensor products indexed by the corresponding vertex set.
\end{rem}

\subsection{Symmetric bimodule structure on the cobar construction for Hopf cooperads}

For a Hopf cooperad $\op C$ both versions of the cobar construction define symmetric bimodules over the Hopf collection $\op C$ (\cite{FW20}, Proposition 1.3.3). For instance the right action on $(\Omega\op C) (r)$ is
\[
(\Omega \op C)(r)\otimes \op C(r)\xrightarrow{1\otimes \Delta_{\op C}} (\Omega\op C)(r)\otimes (\mathcal{F} \op C)(r)\xrightarrow{\mathcal{F}(\mu)} (\Omega\op C)(r)
\]
i.e. it is given by applying the cooperadic cocomposition on $\op C(r)$ before using the commutative algebra structure on each vertex in the tensor product indexed by trees. 

\subsection{Cobar construction for cooperadic comodules}

The cobar construction of a right cooperadic comodule $\op M$ over a cooperad $\op C$ in the category of dg vector spaces is the right operadic $\Omega \op C$-module $\Omega \op M$ whose underlying symmetric sequence is
\[
\Omega \op M=\op M \circ \Omega\op C
\]
and for which the differential $d_{\Omega\op M}$ is the unique derivation extending 
\[
\op M\xrightarrow{\Delta_{\op M}} \op M\circ \op C\xrightarrow {\id_{\op M}\circ \iota} \op M\circ \Omega(\op C)
\]
where $\Delta_{\op M}$ denotes the reduced infinitesimal decomposition map. As such it is given by the formula (see for instance \cite{LodayVallette12}, Prop.6.3.19)
\[
d_{\Omega\op M}=d_{\op M}\circ' \id_{\Omega\op C}+\id_{\op M}\circ' d_{\Omega\op C}+(\id_{\op M}\circ  \gamma_{\Omega \op C})((\id_{\op M}\circ \iota)\circ \Delta_{\op M} )\circ \id_{\Omega \op C}).
\]

The unreduced version $\Omegau \op M=\op M \circ \Omegau \op C$ is defined analogously by using the unreduced infinitesimal decomposition map instead of the reduced one.

\subsection{Symmetric bimodule structure on the cobar construction for Hopf comodules}\label{sec:symmbim}

For a Hopf cooperadic comodule $\op M$ over a Hopf cooperad $\op C$ the cobar constructions $\Omega\op M$ and $\Omegau\op M$ form symmetric bimodules over the Hopf collection $\op M$. For instance the left product is  given by 
\[
-\wedge -: \op M(r)\otimes(\Omega \op M)(r)\xrightarrow{ \Delta_{\op M}\otimes 1} \op F(\op M ; \op C)(r)\otimes (\Omega\op M)(r)\xrightarrow{\mathcal{F}(\mu)} (\Omega\op M)(r)
\]
where
\[
\op F (\op M;\op C)=\bigoplus\limits_T \op F_T (\op M; \op C)
\]
(see Section \ref{sec:comodules}) and $\mathcal F(\mu)$ denotes the commutative product applied to each vertex in the tensor product indexed by the vertex set of the corresponding tree decomposition. We refer to (\cite{FW20}, Proposition 1.3.3) for the proof. It generalizes to the cooperadic comodule case.

\subsection{Complete Hopf operads and filtered Hopf cooperads}\label{sec:filtered}

We will mostly work within the category of complete dg vector spaces. The objects of this category are dg vector spaces, say $V$, equipped with a descending filtration $V=F_0 V\supset F_1 V\supset \dots \supset F_s V\supset \dots$ of subcomplexes such that $V=\lim_s V/F_s V$. For the morphisms of this category we take filtration preserving morphisms. In this case, the completed tensor product $V\hotimes W=\lim_s V\otimes W /F_s(V\otimes W)$, where $F_s(V\otimes W)=\sum_{k+l=s} F_k V\otimes F_l W$ defines a symmetric monoidal structure. We refer to (\cite{FresseBook17}, II. Chapter 13) for a more extensive study of this category.

Indeed some of the Hopf operads we will consider live in the category of (counitary coaugmented cocommutative) coalgebras in complete dg vector spaces. We thus assume that in each arity $\op P(r)$ is a counitary coaugmented cocommutative coalgebra, with coproduct $\Delta:\op P(r)\rightarrow \op P(r)\hotimes \op P(r)$ landing in the completed tensor product. Moreover, the composition morphisms are also defined on the completed tensor product, i.e. $\circ_i: \op P(k)\hotimes \op P(l)\rightarrow \op P(k+l-1)$. Furthermore, we assume the our operads come with a coaugmentation $\eta:\Com \rightarrow \op P$, where we consider the commutative operad as a complete Hopf operad by defining $F_0\Com(r) =\Com(r)=\K$, $F_1\Com(r)=0$ for all $r\geq 1$. Operads of this type are called coaugmented complete Hopf operads. 

In the dual case we consider filtered Hopf cooperads. In addition to being a Hopf cooperad, we require that in each arity, the dg commutative algebras $\op C(r)$ are equipped with an increasing filtration $0=F^{-1}\op C(r)\subset \dots \subset F^{s}\op C(r) \subset \dots \subset \colim_s F^s \op C(r) =\op C(r)$ with the algebra unit $1\in F^0\op C(r)$. The filtration is moreover compatible with the product, i.e. $F^p \op C(r)\cdot F^q \op C(r) \subset F^{p+q} \op C(r)$ for any $p,q\geq 0$, and with the cooperadic cocompositions, that is $\circ_i^* (F^s \op C(k+l-1))\subset \sum_{p+q=s} F^p\op C(k) \otimes F^q \op C(l)$. Similarly as above, we may assume that the cooperad is endowed with an augmentation $\eta^*:\Com^c\rightarrow \op C$, where this time we regard $\Com^c$ as the filtered Hopf cooperad endowed with the trivial filtration ($F^0\Com^c(r)=\K$, for all $r\geq 1$). In this case, we refer to the cooperad as a augmented filtered Hopf cooperad.

By the following result of Fresse and Willwacher, the notions above are dual to each other.

\begin{prop}(\cite{FW20}, Proposition 2.1.4)\label{prop:filtered}
Let $\op C$ be an augmented filtered Hopf cooperad. If the subquotients $F^s \op C(r)/F^{s-1} \op C(r)$ are vector spaces of finite rank degree-wise, the collection $\op C^c=\{\op C^c(r)\}_{r\geq 1}$ consisting of the linear duals in each arity inherits the structure of a coaugmented complete Hopf operad. The complete filtration on $\op C^c(r)$ is given by $F_s  \ \op C^c(r) := \ker ( \op C^c(r) \rightarrow (F^{s-1}\op C(r))^c)$.
\end{prop}

All of the notions above extend to right (co)operadic Hopf (co)modules. Indeed coaugmented complete right operadic Hopf modules are right operadic modules in counitary coaugmented cocommutative coalgebras in the category of complete dg vector spaces, while augmented filtered right cooperadic Hopf comodules are right cooperadic comodules in the category of filtered unitary augmented commutative algebras. The proof of Fresse and Willwacher's result carries over to the right (co)operadic Hopf (co)module setting.

\begin{prop}(\cite{FW20})\label{prop:filteredmod}
Let $\op M$ be an augmented filtered cooperadic Hopf comodule over an augmented filtered cooperad $\op C$. If the respective subquotients $F^s \op M(r)/F^{s-1} \op M(r)$ and $F^s \op C(r)/F^{s-1} \op C(r)$ are vector spaces of finite rank degree-wise, the collection $\op M^c=\{\op M^c(r)\}_{r\geq 1}$ consisting of the linear duals in each arity inherits the structure of a coaugmented complete operadic Hopf module over the coaugmented complete Hopf operad $\op C^c$ equipped with the structure from Proposition \ref{prop:filtered}. The complete filtration on $\op M^c(r)$ is given by $F_s  \ \op M^c(r) := \ker ( \op M^c(r) \rightarrow (F^{s-1}\op M(r))^c)$.
\end{prop}

\subsection{Harrison complexes for commutative algebras}

Let $A$ be an augmented unitary commutative algebra. The \emph{unreduced} Harrison complex $\Harru(A)$ is explicitly given by
\[
\Harru A:=(\mathbb{L}^c(A[1]),d_{\Harr})
\]
where $d_{\Harr}$ is the unique coderivation of the cofree Lie coalgebra $\mathbb{L}^c$ whose projection onto its cogenerators $A[1]$ reduces to the map $\mathbb{L}_2^c(A[1])\rightarrow A[1]$ induced by the product $\mu$ on $A$. Here $\mathbb{L}_2^c(A[1])$ denotes the component of homogeneous weight two the cofree Lie coalgebra. Its \emph{reduced} version $\Harr A$ (which we will simply call the Harrison complex) is defined similarly as
\[
\Harr A :=(\mathbb{L}^c (IA[1]), d_{\Harr} )
\]
where $IA$ denotes the augmentation ideal. Again, slightly abusing notation, $d_{\Harr}$ is induced by the product $\mu$ on $IA$.

Dually, we define the \emph{unreduced} Harrison complex of a complete coaugmented counitary cocommutative coalgebra $C$
\[
\Harru C:=(\hL(C[-1]), d_{\Harr})
\]
where $\hL$ denotes the completed free Lie algebra and $d_{\Harr}$ is the unique derivation extending $C[-1]\rightarrow \hL_2(C[-1])$ induced by the coproduct on $C$. Explicitly, if for $c\in C$ the (unreduced) coproduct is given by 
\[
\Delta(c)= c'\otimes c''
\]
using Sweedler notation, then the Harrison differential is
\[
d_{\Harr}(s^{-1} c)=- (-1)^{|c'|}\frac12 [s^{-1}c',s^{-1}c''].
\]
Its \emph{reduced} version is
\[
\Harr C:=(\hL(IC[-1]), d_{\Harr}).
\]
Here $IC$ is the coaugmentation coideal of $C$, while $d_{\Harr}$ is the unique derivation extending the reduced coproduct on $C$.

\subsection{Harrison complexes for Hopf operads and Hopf cooperads}

For $\op C$ a Hopf cooperad, $\op C(r)$ defines a unitary augmented commutative algebra and we extend the constructions above to $\op C$ arity-wise, i.e. we set for $r\geq 1$
\begin{align*}
(\Harru \op C)(r)&:=\Harru(\op C(r))\\
(\Harr \op C)(r)&:=\Harr(\op C(r)).
\end{align*}
The definitions extend to filtered Hopf cooperads. Dually, we set
\begin{align*}
(\Harru \op P)(r)&:=\Harru(\op P(r))\\
(\Harr \op P)(r)&:=\Harr(\op P(r))
\end{align*}
for a coaugmented complete Hopf operad $\op P$. Recall that in (\cite{FW20}, Proposition 2.1.7) it is shown that, if the complete Hopf operad $\op P=\op C^c$ arises as the dual of an augmented filtered Hopf cooperad $\op C$ as in Proposition \ref{prop:filtered} (i.e. the subquotients are of finite type), then
\[
\Harr (\op C^c(r))=(\Harr \op C(r))^c
\]
as complete dg Lie algebras (see \cite{FresseBook17}, II. Chapter 13). Let us recall one further result on Harrison complexes. We refer to \cite{FW20} for the proofs that
\begin{itemize}
\item $\Harru \op C$ forms a bicomodule over the Hopf cooperad $\op C$ and
\item $\Harru \op P$ forms a bimodule over the complete Hopf operad $\op P$.
\end{itemize}
For instance, the right cocomposition morphisms are given explicitly by the composition
\[
\Delta_r^i: \mathbb L^c(\op C(k+l-1)[1])\xrightarrow{\Delta_i} \mathbb{L}^c(\op C(k)[1]\otimes \op C(l))\xrightarrow\mu\ \mathbb{L}^c(\op C(k)[1])\otimes \op C(l)
\]
where we first apply the cooperadic cocomposition $\Delta_i$ on each tensor factor of the cofree Lie algebra, before multiplying the factors in the graded commutative algebra $\op C(l)$. We denote their sum by $\Delta_r:=\sum_i \Delta_r^i$. The left coproducts are defined similarly, and again denote their sum by $\Delta_l$. In the dual case, the right composition products
\[
\hL( \op P(k)[-1])\hotimes \op P(l)\xrightarrow{\Delta} \hL( \op P(k)[-1]\hotimes \op P(l))\xrightarrow{\circ_i} \hL(\op P(k+l-1)[-1])
\]
are obtained by first taking the coproduct $\Delta$ in the complete cocommutative coalgebra $\op P(l)$ and distributing the factors on the tensor factors of the free Lie algebra, before performing the operadic composition $\circ_i$ on each factor.

\subsection{Harrison complex for cooperadic comodules and operadic modules}

We may extend the construction above to right Hopf cooperadic comodules. Indeed, since for such a comodule $\op M$, its component of arity $r\geq 1$ defines a unitary augmented commutative algebra, we simply set
\begin{align*}
(\Harru \op M)(r):=\Harru(\op M(r))\\
(\Harr \op M)(r):=\Harr(\op M(r)).
\end{align*}
The definitions extend to filtered cooperadic Hopf comodules. Dually, we have
\begin{align*}
(\Harru \op X)(r):=\Harru(\op X(r))\\
(\Harr \op X)(r):=\Harr(\op X(r))
\end{align*}
for a coaugmented complete right Hopf operadic module $\op X$. Moreover, since these definitions only make use of the underlying Hopf collections, the proof that
\[
\Harr (\op M^c(r))=(\Harr \op M(r))^c
\]
carries over from the cooperad case, whenever $\op M$ is a filtered Hopf cooperadic comodule as in Proposition \ref{prop:filteredmod} (i.e. the subquotients are all of finite type). Notice also that the results on the Harrison complexes for (co)operads extend to the comodule case, and we have that
\begin{itemize}
\item $\Harr \op M$ forms a  right comodule over $\op C$ and
\item $\Harr \op X$ forms a right module over $\op P$.
\end{itemize}
The right (co)composition (co)products are given by similar formulas as in the (co)operadic case, but with the (co)operadic (co)compositions replaced by the right (co)operadic (co)module (co)compositions.
\section{Deformation complex for Hopf cooperads}\label{sec:defforcoops}

In this section we review the construction of the deformation complex from \cite{FW20}. In the next section we will define its generalization to the cooperadic comodule case. Let $\op C$ and $\op D$ be dg Hopf cooperads and $\phi:\op C\rightarrow \op D$ be a morphism of dg Hopf cooperads. Assume furthermore that $\op C$ is filtered, in the sense of Section \ref{sec:filtered}, and that the respective subquotients $F^s\op C(r)/ F^{s-1} \op C(r)$ are vector spaces of finite rank degree-wise (i.e. of finite type). Moreover, assume that for all $r\geq 1$, $\op D(r)$ is of finite type, and moreover concentrated in non-negative degrees, i.e. $\op D(r)^d=0$ for $d<0$. To view $\op D$ as a filtered dg Hopf cooperad, equip $\op D(r)$ with the trivial filtration, i.e. $0=F^{-1}\op D(r)\subset F^0 \op D(r)=\op D(r) =F^s\op D(r)$ for all $s\geq 1$.

\subsection{Convolution Lie algebras}

Consider the dg vector space
\[
\frakg_\mathsf{op} := \Hom_S(\op C,\Omegau \op D)=\prod_{r\geq 1} \Hom_{S_r}(\op C(r),(\Omegau\op D) (r))
\]
whose component in arity $r$ consists of $S_r$-equivariant linear morphisms - equipped with the usual differential which, for homogeneous $f\in \Hom_{S} (\op C,  \Omegau \op D)$, is given by
\[
df=  d_{\Omega\op D} f + d_{\op D} f - (-1)^{|f|} f d_{\op C}.
\]
Endow it further with the convolution product induced by the cooperadic cocomposition on $\op C$ and the operadic composition on $\Omegau \op D$. The graded commutator of the convolution product induces a dg Lie algebra structure on $\frakg_\mathsf{op}$. The following result is well-known (see for instance \cite{LodayVallette12}, Section 6.5).

\begin{lemma}
The composition
\[
m_{\mathsf{op}}=\iota \circ \phi: \op C\rightarrow \op D \rightarrow \op D[-1]\rightarrow \Omegau \op D
\]
defines a Maurer-Cartan element in $\frakg_\mathrm{op}$.
\end{lemma}

Next, we define another convolution type dg Lie algebra. For this, consider the dg vector space
\[
\frakg_\Lie (r) := \Hom_{S_r} (\Harru \op C(r), \op D(r))
\]
with the usual differential which, for homogeneous $f\in \Hom_{S_r} (\Harru \op C(r), \op D(r))$ is given by
\[
df=  d_{\op D} f  - (-1)^{|f|} f (d_{\op C} + d_{\Harr\op C}).
\]
Additionally, the Lie coalgebra and the graded commutative algebra structure on $\Harr \op C(r)$ and $\op D(r)$, respectively, enables us to define a convolution Lie bracket on $\frakg_\Lie(r)$. In this setting, we have a result analogous to the one above.

\begin{lemma}
Let $\pi:\Harru \op C(r)\rightarrow \op C(r)[1]$ denote the natural projection. The composition
\[
m_r:=\phi\circ s^{-1}\circ \pi: \Harru\op C(r)\rightarrow \op C(r)[1]\rightarrow \op C(r)\rightarrow \op D(r)
\]
defines a Maurer-Cartan element in $\frakg_\Lie(r)$.
\end{lemma}

\begin{proof}
See Lemma \ref{lemma:mr} below.
\end{proof}

Furthermore, we set $\frakg_{\Lie}:=\prod_{r\geq 1} \frakg_{\Lie}(r)$ and $m_{\Lie}:=\sum_{r\geq 1} m_r$.

\subsection{Deformation complex for Hopf cooperads}

Consider the dg vector space
\[
\Hom_{S} ((\Harru\op C) [-1],\Omegau\op D)=\prod_{r\geq 1} \Hom_{S_r}((\Harru\op C(r))[-1],(\Omegau\op D) (r))
\]
equipped with the usual differential which, for homogeneous $F\in \Hom_{S} (\Harru\op C[-1],\Omegau\op D)$, is given by
\[
dF=(d_\op D+ d_{\Omega\op D}) F  - (-1)^{|F|}F (d_{\op C} + d_{\Harr\op C}).
\]

\begin{lemma}
The dg vector space $\Hom_S(\Harru\op C[-1], \Omegau\op D)$ is a module over the dg Lie algebra $\frakg_{\mathsf{op}}$. 
\end{lemma}

\begin{proof}
For homogeneous $f\in \Hom_S(\op C,\Omegau \op D)$ and $F\in \Hom_S(\Harru\op C[-1],\Omegau\op D)$ the action $f\cdot F$ is given by
\begin{align*}
\Harru \op C[-1] \xrightarrow{\Delta_r+\Delta_l}\Harru \op C[-1] \otimes \op C\oplus \op C \otimes \Harru \op C[-1] \xrightarrow{F \otimes f+f\otimes F} \Omegau\op D \otimes \Omegau\op D \xrightarrow{\gamma_{\Omega \op D}} \Omegau \op D
\end{align*}
It follows from the compatibility of both the bicomodule structure of $\Harru(\op C)$ over $\op C$ and the operadic composition on $\Omegau\op D$ with the respective differentials that
\[
d(f\cdot F)=(df)\cdot F +(-1)^{|f|} f\cdot (dF).
\]
Moreover, the coassociativity of the cooperadic cocomposition, together with the compatibility of the cooperadic and graded commutative algebra structure on $\op C$, as well as the associativity of the operadic composition on $\Omegau \op D$ imply that $\frakg_{\mathsf{op}}$ acts as a graded Lie algebra on $\Hom_S(\Harru\op C[-1], \Omegau\op D)$.
\end{proof}

Twisting the differential on $\Hom_S(\Harru\op C[-1], \Omegau\op D)$ by the action of the Maurer-Cartan element $m_{\mathsf{op}}$ yields an additional term in the differential which we denote by $\partial_{\mathsf{op}}$.

\begin{lemma}
For each $r\geq 1$, the dg vector space $\Hom_{S_r}((\Harru\op C(r))[-1], \Omegau\op D(r))$ is a module over the dg Lie algebra $\frakg_{\Lie}(r)$. 
\end{lemma}

\begin{proof}
For homogeneous $f_r\in \Hom_S(\Harru\op C(r), \op D(r))$ and $F_r\in \Hom_{S_r}(\Harru\op C[-1](r),(\Omegau\op D)(r))$ the action $f_r\cdot F_r$ is given by
\begin{align*}
\Harru \op C(r)[-1] \xrightarrow{\text{Lie cobracket}} \Harru \op C(r)[-1] \otimes \Harru \op C(r)
\xrightarrow{F_r \otimes f_r} (\Omegau\op D)(r) \otimes \op D(r) \xrightarrow{\wedge} (\Omegau \op D)(r)
\end{align*}
As the action of $\op D(r)$ on $(\Omegau \op D)(r)$ intertwines the differentials, and the Harrison differential is a coderivation on $\mathbb{L}^c(\op C(r)[1])$, the action of $\frakg_{\Lie}(r)$ is compatible with the differentials. Moreover, it follows from the dual Jacobi identity and the fact that $\op D(r)$ acts on $(\Omegau\op D)(r)$ as a graded commutative algebra that the action of $\frakg_\Lie(r)$ is compatible with the graded Lie algebra structure.
\end{proof}

Twisting the differential by the action of the Maurer-Cartan element $m_{\Lie}$ yields an additional term in the differential which we denote by $\partial_{\Lie}$.

\begin{lemma}[\cite{FW20}, Construction 1.3.7]
The action of $m_{\mathsf{op}}$ and $m_\Lie$ anti-commute, i.e. we have $\partial_\mathsf{op}\partial_\Lie +\partial_\Lie \partial_\mathsf{op}=0$. We may thus consider the twisted dg vector space
\[
(\Hom_{S} (\Harru\op C [-1],\Omegau\op D),d+\partial_\mathsf{op}+\partial_\Lie).
\]
\end{lemma}

Restricting to elements of the reduced Harrison complex $\Harr \op C\subset \Harru \op C$ and projecting onto the reduced cobar construction $\Omega \op D$ yields the graded linear subspace
\[
\Hom_S((\Harr \op C)[-1],\Omega \op D) \subset \Hom_{S} ((\Harru\op C) [-1],\Omegau\op D).
\]
\begin{prop}[\cite{FW20}, Construction 1.3.7]
The differential restricts to the subspace $\Hom_S((\Harr \op C)[-1],\Omega \op D)$, i.e. 
\[
(\Hom_S((\Harr \op C)[-1],\Omega \op D), d+\partial_\mathsf{op}+\partial_\Lie)
\]
defines a subcomplex.
\end{prop}

We refer to this subcomplex as the deformation complex associated to the morphism $\phi:\op C\rightarrow \op D$, and denote it by
\[
\Def(\op C\xrightarrow \phi \op D):=(\Hom_S((\Harr \op C)[-1],\Omega \op D), d+\partial_\mathsf{op}+\partial_\Lie).
\]

\section{Deformation complex for Hopf cooperadic comodules}\label{sec:defformod}

Next, we shall extend Fresse and Willwacher's construction to right cooperadic dg Hopf comodules. For this, let $(\op M,\op C)$ and $(\op N, \op D)$ be right cooperadic dg Hopf comodules $\op M$ and $\op N$ together with their respective dg Hopf cooperads $\op C$ and $\op D$. Assume furthermore, that $\op M$ and $\op C$ are filtered in the sense of Section \ref{sec:filtered}, and that for all $r\geq 1$, $\op N(r)$ and $\op D(r)$ are of finite type. To view $(\op N, \op D)$ as a filtered right cooperadic dg Hopf comodule, equip both $\op N(r)$ and $\op D(r)$ with the trivial filtration, i.e. $0=F^{-1}\op N(r)\subset F^0 \op N(r)=\op N(r) =F^s\op N(r)$ for all $s\geq 1$, and $0=F^{-1}\op D(r)\subset F^0 \op D(r)=\op D(r) =F^s\op D(r)$ for all $s\geq 1$. Additionally, fix a morphism of filtered right cooperadic dg Hopf comodules, i.e. a pair 
\[
(\Phi: \op M\rightarrow \op N , \ \phi:\op C\rightarrow \op D)
\]
respecting all of the algebraic structures.

\subsection{A convolution Lie algebra for Hopf cooperadic comodules}

Consider the dg Lie algebra
\[
\frakg_{\mathsf{mod}}(r):=\Hom_{S_r}(\Harru \op M(r),\op N(r)).
\]
It is indeed a dg Lie algebra, since, as in the case of $\frakg_\Lie (r)$ the dg Lie algebra structure only depends on the dg commutative algebra structures on $\op M(r)$ and $\op N(r)$ (and not on the comodule structures). For the same reason, we recover the following statement.

\begin{lemma}\label{lemma:mr}
The composition
\[
m_r^{\mathsf{mod}}=\Phi\circ s^{-1}\circ \pi: \Harru\op M(r)\rightarrow \op M(r)[1]\rightarrow \op M(r)\rightarrow \op N(r)
\]
defines a Maurer-Cartan element in $\frakg_\mathsf{mod} (r)$.
\end{lemma}

The proof is given in the next section. Moreover, we set $\frakg_{\mathsf{mod}}:=\prod_{r\geq 1} \frakg_{\mathsf{mod}}(r)$ and $m_{\mathsf{mod}}:=\sum_{r\geq 1} m_r^{\mathsf{mod}}$.

\subsection{Deformation complex for Hopf cooperadic comodules}

Consider the dg vector space 
\[
\Hom_{S} (\Harru\op M ,\Omegau\op N)=\prod_{r\geq 1} \Hom_{S_r}(\Harru\op M(r),(\Omegau\op N) (r)).
\]
equipped with the usual differential which, for homogeneous $f\in \Hom_{S} (\Harru\op M[-1],\Omegau\op N)$, is given by
\[
df=(d_\op N+ d_\Omega) f  - (-1)^{|f|}f (d_{\op M} + d_{\Harr\op M}).
\]

\begin{lemma}
The dg vector space $\Hom_S(\Harru\op M, \Omegau\op N)$ is a module over the dg Lie algebra $\frakg_{\mathsf{op}}$. 
\end{lemma}

\begin{proof}
For homogeneous $f\in \frakg_\mathsf{op}$ and $F\in \Hom_S(\Harru\op M, \Omegau\op N)$ the action $f\cdot F$ is defined through the composition
\[
\Harru \op M \xrightarrow{\Delta_r}\Harru \op M \otimes \op C \xrightarrow{F \otimes f} \Omegau\op N \otimes \Omegau\op D \xrightarrow{\gamma_{\Omega \op N}} \Omegau \op N.
\]
It follows from the compatibility of both the  right comodule structure of $\Harru(\op M)$ over $\op C$ and the operadic $\Omegau\op D$-module structure on $\Omegau\op N$ with the respective differentials that 
\[
d(f\cdot F)=(df)\cdot F +(-1)^{|f|} f\cdot (dF).
\]
Moreover, the coassociativity of the right $\op C$-comodule structure on $\Harru \op M$, together with the compatibility of the cooperad and graded commutative algebra structure on $\op C$, as well as the associativity of the right operadic $\Omegau \op D$-module structure on $\Omegau \op N$ imply that $\frakg_\mathsf{op}$ acts as a graded Lie algebra on $\Hom_S(\Harru\op M, \Omegau\op N)$.
\end{proof}

A similar reasoning as in the case of $\frakg_\Lie$ works for $\frakg_{\mathsf{mod}}$.

\begin{lemma}
The dg vector space $\Hom_S(\Harru\op M, \Omegau\op N)$ is a module over the dg Lie algebra $\frakg_{\mathsf{mod}}$. 
\end{lemma}

\begin{proof}
For homogeneous $f_r\in \Hom_S(\Harru\op M(r), \op N(r))$ and $F_r\in \Hom_{S_r}(\Harru\op M (r),\Omegau\op N(r))$ the action $f_r\cdot F_r$ is given by
\begin{align*}
\Harru \op M(r) \xrightarrow{\text{Lie cobracket}} \Harru \op M(r) \otimes \Harru \op M(r) \xrightarrow{F_r \otimes f_r} \Omegau\op N(r) \otimes \op N(r) \xrightarrow{\wedge} \Omegau \op N(r)
\end{align*}
As the action of $\op N(r)$ on $\Omegau \op N(r)$ intertwines the differentials, and the Harrison differential is a coderivation on $\mathbb{L}^c(\op M(r)[1])$, the action of $\frakg_{\mathsf{mod}}(r)$ is compatible with the differentials. Moreover, it follows from the dual Jacobi identity and the fact that $\op N(r)$ acts on $\Omegau\op N(r)$ as a graded commutative algebra that the action of $\frakg_{\mathsf{mod}}(r)$ is compatible with the graded Lie algebra structure.
\end{proof}

\begin{rem}
For each $r\geq 1$ we may equip $\frakg_{\mathrm{op}}(r)=\Hom_{S_r} (\op C(r),(\Omegau\op D)(r))$ with the complete descending filtration 
\[
F_s \ \frakg_{\mathrm{op}}(r):=\ker\Big(\Hom_{S_r}(\op C(r),(\Omegau\op D)(r))\rightarrow \Hom_{S_r} ( F^s\op C(r),(\Omegau\op D)(r))\Big)
\]
defined through the restriction map induced by the inclusion $F^s\op C(r)\subset \op C(r)$. In particular, we have an isomorphism
\[
\frac{F_s \ \frakg_{\mathrm{op}}(r)}{F_{s+1} \ \frakg_{\mathrm{op}}(r)}\cong \Hom_{S_r} (\frac{F^s\op C(r)}{F^{s+1}\op C(r)}, (\Omegau\op D)(r))
\]
Moreover, we view each arity of the cobar construction $(\Omegau\op D)(r)$ as a complete dg vector space by endowing it with the trivial filtration, i.e. $F^0 (\Omegau\op D)(r)=(\Omegau\op D)(r) \supset F^1(\Omegau\op D)(r)=0$. Also, by Proposition \ref{prop:filtered} $\op C^c(r)$ defines a complete dg vector space. The finiteness assumptions on the filtration on $\op C(r)$ and the fact that $\op D(r)$ is non-negatively graded and of finite type ensure that we have an isomorphism of complete dg vector spaces
\[
\frakg_{\mathrm{op}}(r) \longleftarrow (\Omegau \op D)(r)\hotimes_{S_r}\op C^c(r). 
\]
Indeed, on the level of associated graded complexes both side decompose into a direct product of finite dimensional subspaces (one for each total degree) on which the morphism above restricts to an isomorphism. Furthermore, taking the direct product over all arities, the compatibility of the filtration on $\op C$ and its cooperad structure implies that the respective filtrations and the morphism are compatible with the natural Lie algebra structures - on the right hand side, the Lie bracket is given on homogeneous elements by
\[
[f_1\otimes c_1,f_2\otimes c_2]=(-1)^{|f_2||c_1|} \sum_{i=1}^{r_1} f_1\circ_i f_2 \otimes c_1 \circ_i c_2 -  (-1)^{(|f_1|+|c_1|)(|f_2|+|c_2|)} \sum_{i=1}^{r_2} (-1)^{|f_1||c_2|} f_2\circ_i f_1 \otimes c_2 \circ_i c_1
\]
where $f_i \otimes c_i$ is of arity $r_i$. The isomorphism is thus an isomorphism of complete dg Lie algebras.

Similarly, the filtration on the Harrison complex $\Harru \op C(r)$ induced by the filtration on $\op C(r)$ allows us to define the complete descending filtration 
\[
F_s \ \frakg_{\mathrm{mod}}(r):=\ker\Big(\Hom_{S_r}(\Harru \op M(r),\op N(r))\rightarrow \Hom_{S_r} ( F^s \Harru \op M(r),\op N(r))\Big)
\]
on $\frakg_{\mathrm{mod}}(r)$. Moreover, notice that the finiteness assumptions on $\op M(r)$ imply that the subquotients $F^s\Harr \op M(r)/F^{s-1}\Harru \op M(r)$ are all of finite type. We obtain an isomorphism of complete dg Lie algebras
\[
\frakg_{\mathrm{mod}}(r) \longleftarrow \op N(r) \hotimes_{S_r} \Harru \op M^c(r).
\]
Remark that the Lie bracket on the right reads
\[
[n_1 \otimes m_1, n_2 \otimes m_2]=(-1)^{|n_2||m_1|} n_1 \cdot  n_2 \otimes [m_1,m_2].
\]
Finally, the analogous construction with $\Omegau \op N$ replacing $\op N$ yields an isomorphism of complete dg vector spaces
\[
\Hom_{S} (\Harru \op M,\Omegau \op N) \longleftarrow  \Omegau\op N \hotimes_S \Harru \op M^c.
\]
The right hand side comes equipped with dg Lie algebra actions from $\op N(r) \hotimes_{S_r} \Harru \op M^c(r)$ and $(\Omegau \op D)(r)\hotimes_{S_r}\op C^c(r)$ which are compatible with the respective filtrations and the actions of $\frakg_{\mathrm{mod}}$ and $\frakg_{\mathrm{op}}$ on $\Hom_{S} (\Harru \op M,\Omegau \op N)$ via the isomorphisms above. For homogeneous elements $y\otimes x \in \Omegau\op N \hotimes_S \Harru \op M^c$, $n\otimes m \in \op N(r) \hotimes_{S_r} \Harru \op M^c(r)$ and $d\otimes c \in (\Omegau \op D)(r)\hotimes_{S_r}\op C^c(r)$, we have
\[
(n\otimes m)\cdot (y\otimes x) = (-1)^{|y||m|} n\wedge y \otimes [m,x]
\]
where $\wedge$ denotes the product defining the symmetric bimodule structure of $\Omegau \op N$ over $\op N$ (see Section \ref{sec:symmbim}), and
\[
(d\otimes c)\cdot (y\otimes x) = \sum_{i} (-1)^{(|y|+|x|)|c|+|y||d|} y\circ_i d \otimes x\circ_i c.
\]
Here $y\circ_i d$ is describes the operadic module structure of $\Omegau N$ over $\Omegau D$, while $x\circ_i c$ is given by the right module structure of the symmetric collection $\Harru \op M^c$ over the Hopf operad $\op C^c$ (see Section \ref{sec:bicom}).
\end{rem}

\begin{rem}
Under the identifications above, we write 
\[
\sum_b s^{-1}d_b \otimes c^b \in \Omegau \op D\hotimes_S \op C^c
\]
for the element corresponding to $m_{\mathsf{op}}=\iota\circ \phi \in \frakg_{\mathsf{op}}$. Here $\{c^b\}_b$ are the elements dual to any lift $\{c_b\}_b$ of the bases for the various subquotients of finite type of $\op C$, and $d_b$ describes the image in $\op D$ of $c_b$ under $\phi$. Similarly, we identify $m_\mathsf{mod}$ with 
\[
\sum_a (-1)^{|n_a|} n_a \otimes s^{-1}m^a \in \op N \hotimes_{S} \Harru \op M
\]
where $\{m^a\}_a$ are the linear duals to any lift $\{m_b\}_b$ of the bases for the various subquotients of finite type of $\op M$, and $n_a$ denotes the image in $\op N$ of $m_a$ under $\Phi$.
\end{rem}

We use this equivalence to prove Lemma \ref{lemma:mr}.

\begin{proof}[Proof of Lemma \ref{lemma:mr}]
We identify 
\[
\Hom_{S_r}(\Harru \op M(r),\op N(r))\cong \op N(r)\hotimes_{S_r} \Harr \op M^c(r)
\]
and set $m_r=\sum_a (-1)^{|n_a|} n_a \otimes s^{-1}m^a$ to be the element corresponding to $\Phi\circ s^{-1}\circ \pi$. Summing over repeated indices, we find by direct computation,
\begin{align*}
d m_r^\mathrm{mod}&= (-1)^{|n_a|} d_{\op N}(n_a)\otimes s^{-1} m^a + (-1)^{2|m_a|+1} m_a \otimes s^{-1} d_{\op M^c} (m^a) + (-1)^{2|n_a|} n_a \otimes d_{\Harr\op M^c} (s^{-1}m^a)\\
&=- (-1)^{|{m^a}'|}n_a\otimes \frac12 [ s^{-1}{c^m}',s^{-1}{c^m}''] 
\end{align*}
where $(-1)^{|n_a|} d_{\op N}(n_a)\otimes s^{-1} m^a + (-1)^{2|n_a|+1} n_a \otimes s^{-1} d_{\op M^c}( m^a)=0$ since $\Phi\circ s^{-1}\circ \pi$ is a morphism of complexes. On the other hand,
\begin{align*}
[(-1)^{|d_a|} n_a\otimes s^{-1}m^a, (-1)^{|n_b|} n_b\otimes s^{-1} m^b]&=(-1)^{|n_a|+|n_b|+|s^{-1} m^a||n_b|}  n_a\wedge n_b \otimes [s^{-1}m^a,s^{-1}m^b]\\
&= (-1)^{|n_a|+|m^a||n_b| + |m^a|}  (n_a\wedge n_b \otimes m^a\otimes m^b ) (-\otimes [s^{-1} - ,s^{-1} -])\\
&= (n_p \otimes {m^p}'\otimes {m^p}'')( -\otimes [s^{-1} - ,s^{-1} -])\\
&=(-1)^{|{m^p}'|}  n_p \otimes [s^{-1}{m^p}', s^{-1} {m^p}'']
\end{align*}
where the relation
\begin{equation}\label{eq:phidgcamap}
n_p\otimes {m^p}'\otimes {m^p}''=(-1)^{|n_b||m^a|} n_a\wedge n_b\otimes m^a\otimes m^b
\end{equation}
follows from $m_r^{\mathrm{mod}}=\Phi\circ s^{-1}\circ \pi$ being a morphism of dg commutative algebras.
\end{proof}

\begin{lemma}
The action of $m_{\mathsf{op}}$ and $m_{\mathsf{mod}}$ anti-commute, i.e. we have $\partial_\mathsf{op}\partial_\mathsf{mod} +\partial_\mathsf{mod} \partial_\mathsf{op}=0$.
\end{lemma}

\begin{proof}
We identify 
\begin{align*}
\Hom_{S} (\Harru\op M,\Omegau\op N)&\cong  \Omegau \op N\hotimes_S \Harru \op M^c\\
\frakg_{\mathsf{op}}&\cong  \Omegau \op D \hotimes_S \op C^c\\
\frakg_{\mathsf{mod}}&\cong  \op N \hotimes_S \Harru \op M^c. 
\end{align*}
Let $  y\otimes x\in \Omegau \op N \hotimes_S \Harru \op M^c $, $m_{\mathsf{mod}}= \sum_a (-1)^{|n_a|}n_a \otimes s^{-1} m^a$ and $m_{\mathsf{op}}=\sum_b s^{-1}d_b\otimes c^b$. We sum over repeated indices and find
\begin{align*}
\partial_{\mathsf{mod}} \partial_{\mathsf{op}}(  y\otimes  x)&= (-1)^{|n_a|} (n_a\otimes s^{-1} m^a) \cdot (s^{-1}d_b\otimes c^b)\cdot ( y\otimes  x)\\
&=\sum_i  (-1)^{|n_a|+|y|+|c^b||x|} (n_a\otimes s^{-1} m^a)\cdot (y\circ_i s^{-1}d_b \otimes  x\circ_i c^b)\\
&=\sum_i (-1)^{|c^b||x|+|m^a||y|+|d_b|+|m^a||d_b|} (n_a\wedge (y\circ_i s^{-1}d_b))\otimes  [s^{-1}m^a,x\circ_i c^b]
\end{align*}
Again, let us continue the computation up to a sign. We take care of those below.
\begin{align*}
&\pm (n_a\wedge (y\circ_i s^{-1}d_b))\otimes [s^{-1}m^a,x\circ_i c^b]\\
&=\pm (n_a'\wedge y) \circ_i (d_a''\wedge s^{-1}d_b))\otimes  [s^{-1}m^a,x\circ_i c^b]\\
&=\pm (n_a'\otimes d_a''\otimes m^a)(-\wedge y)\circ_i(-\wedge s^{-1}d_b) \otimes  [s^{-1}-,x\circ_i c^b]\\
&=\pm (n_p\otimes d_q\otimes m^p\circ_i c^q) (-\wedge y)\circ_i(-\wedge s^{-1}d_b) \otimes  [s^{-1}-,x\circ_i c^b]\\
&=\pm (n_p\wedge y)\circ_i(d_q\wedge s^{-1}d_b) \otimes  [s^{-1} (m^p\circ_i c^q),x\circ_i c^b]\\
&=\pm (n_p\wedge y)\circ_i s^{-1} (d_q\wedge d_b) \otimes  [s^{-1} (m^p\circ_i c^q),x\circ_i c^b]\\
&=\pm (d_q\wedge d_b\otimes c^q\otimes c^b)((n_p\wedge y)\circ_i s^{-1} - \otimes  [s^{-1} (m^p\circ_i -),x\circ_i -])\\
&=\pm (d_e\otimes {c^e}'\otimes {c^e}'')((n_p\wedge y)\circ_i s^{-1} - \otimes [s^{-1} (m^p\circ_i -),x\circ_i -])\\
&=\pm (n_p\wedge y)\circ_i s^{-1} d_e \otimes[s^{-1} (m^p\circ_i {c^e}'),x\circ_i {c^e}'']\\
&=\pm (n_p\wedge y)\circ_i s^{-1} d_e \otimes [s^{-1} m^p,x]\circ_i c^e
\end{align*}
To go from the third to the fourth line, we use the identity
\[
n_a'\otimes d_a'' \otimes m^a=(-1)^{|d_q||m^p|} n_p\otimes d_q \otimes m^p \circ_i c^q
\]
reflecting the fact that the pair $(\Phi, \phi)$ defines a morphism of right cooperadic comodules, while going from the seventh line to the eighth, we again apply the relation \eqref{eq:phidgcamap}. The respective signs are given by
\begin{align*}
\text{line 1: }& |c^b||x|+|m^a||y|+|d_b|+|m^a||d_b|\\
\text{line 2: }&|c^b||x|+|m^a||y|+|d_b|+|m^a||d_b|+|d_a''||y|\\
\text{line 3: }&1+|c^b||x|+|d_b|\\
\text{line 4: }&1+|c^b||x|+|d_b|+|d_q||m^p|\\
\text{line 5: }&1+|c^b||x|+|d_b|+|d_q||m^p|+|d_q||y|+|m^p||y|+|m^p||d_b|+|c^q||y|+|c^q||d_b|\\
\text{line 6: }&1+|c^b||x|+|d_b|+|d_q||m^p|+|d_q||y|+|m^p||y|+|m^p||d_b|+|c^q||y|+|c^q||d_b|+|d_q|\\
\text{line 7: }&1+|m^p||y|+|c^q||d_b|\\
\text{line 8: }&1+|m^p||y|\\
\text{line 9: }&1+|m^p||y|+|d_e||n_p|+|d_e||y|+|d_e|+(|{c^e}'|+|{c^e}''|)(|n_p|+|y|+|m^p|)+|{c^e}''||x|\\
\text{line 10: }&1+|m^p||y|+|d_e||n_p|+|d_e||y|+|d_e|+|c^e|(|n_p|+|y|+|m^p|+|x|)
\end{align*}
Thus, we conclude that
\begin{align*}
&\sum_i(-1)^{1+|m^p||y|+|d_e||n_p|+|d_e||y|+|d_e|+|c^e|(|n_p|+|y|+|m^p|+|x|) }(n_p\wedge y)\circ_i s^{-1} d_e \otimes  [s^{-1} m^p,x]\circ_i c^e\\
&=(-1)^{1+|m^p||y|+|n_p|+|y|} (s^{-1}d_e\otimes c^e) \cdot  ( n_p\wedge y)\otimes  [s^{-1} m^p,x]\\
&=(-1)^{|n_p|+1} (s^{-1}d_e\otimes c^e) \cdot (n_p\otimes s^{-1} m^p)\cdot (y \otimes x)\\
&=-\partial_{\mathsf{op}}\partial_{\mathsf{mod}} (  y\otimes x).
\end{align*}
\end{proof}

Thus, we may consider the twisted dg vector space
\[
(\Hom_{S} (\Harru\op M ,\Omegau\op N),d+\partial_\mathsf{op}+\partial_\mathsf{mod}).
\]
Restricting to elements of the reduced Harrison complex $\Harr \op M\subset \Harru \op M$ and projecting onto the reduced cobar construction $\Omega \op N$ yields first of all the graded linear subspace
\[
\Hom_S(\Harr \op M,\Omega \op N) \subset \Hom_{S} (\Harru\op M ,\Omegau\op N).
\]
\begin{prop}\label{prop:subcomplex}
The differential restricts to the subspace $\Hom_S(\Harr \op M,\Omega \op N)$, i.e. 
\[
(\Hom_S(\Harr \op M,\Omega \op N), d+\partial_\mathsf{op}+\partial_\mathsf{mod})
\]
defines a subcomplex.
\end{prop}

\begin{proof}
We need to show that the differential preserves the respective (co)augmentation (co)ideals. In the case of the augmentation ideal, notice the following calculation. Since we work with the unreduced coproduct, the Harrison differential produces terms of the form
\[
(1\otimes d_{\Harr\op M^c} )( y\otimes x)=  (-1)^{|y|} y \otimes d_{\Harr \op M^c}(x)= - (-1)^{|y|} y\otimes   \ad_{s^{-1} 1} (x)+\dots.
\]
Here it is important that $x\in \Harr \op M^c$ is an element of the reduced Harrison complex, otherwise the second equality does not hold. On other hand, the twisted differential $\partial_\mathsf{mod}$ yields the terms
\[
\partial_\mathsf{mod} ( y\otimes x) = (1\otimes s^{-1}1)\cdot (y\otimes x)+\dots= \  (-1)^{|y|} (1\wedge y)\otimes  [s^{-1}1,x] +\dots = (-1)^{|y|} y \otimes  [s^{-1}1,x] +\dots
\]
By the graded Jacobi identity these terms cancel. For the coaugmentation coideal of $\op D$, consider
\begin{align*}
&(d_{\Omega \op N} \otimes 1) ( y\otimes x)= d_{\Omega \op N} (y) \otimes x= -(-1)^{|y|}\sum_i (y;\id,\dots,\underbrace{s^{-1}\id}_{i},\dots,\id)\otimes \ x+\dots \\
= \ & -(-1)^{|y|} \sum_i y \circ_i s^{-1}\id \otimes \ x+ \dots
\end{align*}

On the other hand, applying $\partial_\mathsf{op}$ gives
\begin{align*}
&\partial_\mathsf{op} ( y\otimes x)=(s^{-1}\id \otimes \id)\cdot (y\otimes x)+\dots=\sum_i (-1)^{|y|} y\circ_i s^{-1}\id \otimes \ x\circ_i \id+\dots\\
= \ & (-1)^{|y|} \sum_i y\circ_i s^{-1}\id \otimes \ x+\dots
\end{align*}

\end{proof}

We refer to this subcomplex as the deformation complex associated to the morphism $\Phi:\op M\rightarrow \op N$, and denote it by
\[
\Def(\op M\xrightarrow \Phi \op N):=(\Hom_S(\Harr \op M,\Omega \op N), d+\partial_\mathsf{op}+\partial_\mathsf{mod}).
\]

\subsection{A natural morphism induced by a Lie algebra action}\label{sec:Lieaction}

Let $\frakg$ be a complete dg Lie algebra, $\op C$ a filtered dg Hopf cooperad and $\op M$ a filtered right dg Hopf cooperadic $\op C$-comodule both satisfying the finiteness assumptions from Proposition \ref{prop:filteredmod}. Assume further that $\frakg$ acts on $\op M$ in a way which is compatible with the filtered dg Hopf $\op C$-comodule structure. More precisely, we assume that each homogeneous element $g\in \frakg$ of degree $|g|$ defines for each $r\geq 1$ a linear morphism $D_g:\op M(r)\rightarrow \op M(r)$ of degree $|g|$ making the diagrams
\begin{center}
 \begin{tikzcd}[row sep=large, column sep=10ex]
  \op M(k+l-1) \arrow[r,"D_g"]\arrow[d,"\circ_i"] & \op M(k+l-1) \arrow[d,"\circ_i"]   \\
  \op M(k)\otimes \op C(l)\arrow[r, "D_g\otimes \id"] & \op M(k)\otimes \op C(l) 
 \end{tikzcd}
 \end{center}
 and 
 \begin{center}
 \begin{tikzcd}[row sep=large, column sep=20ex]
  \op M(r)\otimes \op M(r) \arrow[r,"D_g\otimes \id +\id \otimes D_g"]\arrow[d,"\mu"] & \op M(r)\otimes \op M(r) \arrow[d,"\mu"]   \\
  \op M(r) \arrow[r, "D_g"] & \op M(r)
 \end{tikzcd}
 \end{center}
commute (i.e. $D_g$ defines a coderivation with respect to the cooperadic comodule structure, and a derivation with respect to the dg commutative algebra structure in each arity). Furthermore, we assume that the respective filtrations are compatible, i.e. for $g\in F_s\frakg$ and $m\in F^t \op M(r)$, we have $D_g(m)\in F^{t-s} \op M(r)$. We refer to such morphisms as biderivations and we denote by $\mathrm{BiDer}(\op M)\subset \Hom(\op M,\op M)$ the space of biderivations. It carries a natural dg Lie algebra structure. In particular, we have a morphism of dg Lie algebras
\begin{align*}
\frakg &\rightarrow \mathrm{BiDer}(\op M)\\
g&\mapsto D_g.
\end{align*}
Notice moreover that in this case $D_g$ defines a co-Lie coderivation on the Harrison complex arity-wise (since $\Harr \op M (r)$ is cofree) which in addition is compatible with the right comodule structure over $\op C$.

Returning to the deformation complex above, notice that the morphism $\Phi$ induces a preferred element of degree one. Indeed, for
\[
\tilde \Phi: \Harr \op M \xrightarrow {\pi} \op M[1] \xrightarrow{s^{-1}} \op M \xrightarrow \Phi \op N \xrightarrow {\iota} \Omega \op N
\]
we have the following result.

\begin{prop}\label{prop:Lieaction}
The action of $\frakg$ on $\op M$ defines the morphism of complexes
\begin{align*}
W:\frakg[-1] &\rightarrow \Def(\op M\xrightarrow \Phi \op N)\\
s^{-1} g&\mapsto \tilde \Phi \circ D_g.
\end{align*}
\end{prop}

\begin{proof}
Equivalently, $\frakg$ acts on $\op M^c$, respecting the complete right dg Hopf operadic $\op C^c$-module structure. Extend this action by derivation to $\Harr \op M^c$, and furthermore bilinearly to $\Omega \op N \hotimes_S \Harr \op M^c $. We identify $W(s^{-1}g)=\tilde \Phi \circ D_g$ with 
\[
(-1)^{|n_a|+|m^a||g|} n_a \otimes s^{-1} (g.m^a)
\]
and $W(d(s^{-1}g))=-W(s^{-1}(dg))=-\tilde \Phi \circ D_{dg}$ with
\[
-(-1)^{|n_a|+|m_a||dg|} n_a \otimes s^{-1} ((dg).m^a)=-(-1)^{|m_a||g|} n_a \otimes s^{-1} ((dg).m^a)
\]
since $|n_a|+|m^a|=0$. As a part of $dW(g)$, we then compute
\begin{align*}
&(-1)^{|n_a|+|m^a||g|} d_{\op N}  (n_a)\otimes s^{-1} (g.m^a)  + (-1)^{2|n_a|+|m^a||g|+1} n_a\otimes s^{-1} d_{\op M^c}( g . m^a))\\
= \ & (-1)^{|n_a|+|m^a||g|} d_{\op N}  (n_a)\otimes s^{-1} (g.m^a) + (-1)^{2|n_a|+|m^a||g|+1} \left( n_a\otimes s^{-1} ((d g) . m^a) + (-1)^{|g|} n_a\otimes s^{-1} ( g . d_{\op M^c} m^a) \right)\\
= \ & - (-1)^{|m^a||g|} n_a\otimes s^{-1} ((d g) . m^a)
\end{align*}
as in particular the fact that $\tilde \Phi$ describes a morphism of complexes yields
\[
(-1)^{|n_a|+|m^a||g|} d_{\op N}  (n_a)\otimes s^{-1} (g.m^a) + (-1)^{|m^a||g|+1+|g|} n_a\otimes s^{-1} ( g . d_{\op M^c} m^a)=0.
\]
We need to show that all other terms appearing in $d W(g)$ vanish. For instance, we have
\[
(-1)^{|n_a|+|m^a||g|} d_{\Omega \op N} (n_a)\otimes s^{-1} (g.m^a)= -\sum_i (-1)^{|d_a''|+|m^a||g|} (n_a' \circ_i s^{-1} d_a'') \otimes s^{-1}(g. m^a)
\]
while on the other hand
\begin{align*}
(-1)^{|n_a|+|m^a||g|} \partial_\mathsf{op}   (n_a\otimes  s^{-1}(g. m^a))&= (-1)^{|n_a|+|m^a||g|} (s^{-1} d_b\otimes c^b )\cdot(  n_a\otimes  s^{-1}(g. m^a))\\
&=\sum_i (-1)^{|n_a||d_b|+|m^a||g|+|c^b||g|+|c^b|} n_a\circ_i s^{-1}d_b \otimes s^{-1} (g. m^a)\circ_i c^b \\
&=\sum_i (-1)^{|n_a||d_b|+|m^a||g|+|c^b||g|+|c^b|} n_a\circ_i s^{-1}d_b \otimes s^{-1} (g. (m^a \circ_i c^b)) \\
&=\sum_i (-1)^{|n_a||d_b|} (n_a\otimes d_b \otimes m^a\circ_i c^b)( - \circ_i s^{-1} -)\otimes s^{-1} (g .-)\\
&=\sum_i  (n_e'\otimes d_e'' \otimes m^e )( - \circ_i s^{-1} -)\otimes s^{-1}(g.-)\\
&=\sum_i (-1)^{|d_e''|+|m^e||g|} (n_e' \circ_i s^{-1} d_e'')\otimes s^{-1} (g. m^e).
\end{align*}
Similarly,
\begin{align*}
(-1)^{2|n_a|+|m^a||g|} n_a\otimes d_{\Harr \op M^c} s^{-1}(g. m^a)&= (-1)^{|m^a||g|} n_a\otimes g. d_{\Harr \op M^c} s^{-1}m^a\\
& =   (-1)^{(|{m^a}'|+|{m^a}''|)|g|+1+|{m^a}'|} n_a\otimes \frac12  g. [s^{-1} {m^a}' ,s^{-1} {m^a}'']
\end{align*}
whereas,
\begin{align*}
(-1)^{|n_a|+|m^a||g|}\partial_{\mathsf{mod}} (n_a \otimes s^{-1}(g. m^a))&= (-1)^{|n_b|+|n_a|+|m^a||g|}(n_b \otimes s^{-1}m^b)\cdot (n_a\otimes s^{-1}(g. m^a)) \\
&= (-1)^{|n_b|+|m^b||n_a|+|m^a||g|} n_b\wedge  n_a \otimes [ s^{-1} m^b,s^{-1} (g.m^a)]\\
&= (-1)^{|n_b|+|m^b||n_a|+|m^a||g|+|m^b||g|} n_b\wedge  n_a \otimes \frac12 g.[ s^{-1} m^b,s^{-1} m^a]\\
&= (-1)^{|m^b||n_a|} (n_b\wedge n_a\otimes m^b\otimes m^a) (- \otimes \frac12 g. [s^{-1} - ,s^{-1} -])\\
&= (n_e\otimes {m^e}'\otimes {m^e}'') (- \otimes \frac12 g.[s^{-1} - ,s^{-1} -])\\
&= (-1)^{|{m^e}''||g|+|{m^e}'||g|+|{m^e}'|} n_e \otimes \frac12 g. [s^{-1}{m^e}' ,s^{-1} {m^e}'' ].
\end{align*}
The statement follows.

\end{proof}

\section{Gerstenhaber and Batalin-Vilkovisky operads}\label{sec:BV}
In this section we recall the several notions associated to the Gerstenhaber and the Batalin-Vilkovisky operads. In the next section, we will consider modules over these operads. Denote by $\e_2$ the operad governing Gerstenhaber algebras. It is a quadratic operad generated by two binary operations, a commutative product $\mu \in \e_2(2)$ of degree $0$, and a symmetric Lie bracket $\lambda \in \e_2(2)$ of degree $-1$, subject to the usual compatibility condition. Recall further that as symmetric sequences $\e_2=\Com\circ \Lie\{1\}$, and that the operad $\e_2$ is Koszul. Its Koszul dual cooperad is $\e_2^\text{!`}=\e_2^c\{2\}$ \cite{GetzlerJones94}. Its dual cooperad $\e_2^c$ computes the cohomology of the little discs operad $\lD_2$. Its arity $r$ component is given by the graded commutative algebra \cite{Arnold69}
\[
\e_2^c(r)=H^*(\lD_2(r);\K)\cong S(w^{(ij)})_{1\leq i\neq j\leq r }/(w^{(ij)}-w^{(ji)} \ ; \ w^{(ij)}w^{(jk)}+w^{(jk)}w^{(ki)}+w^{(ki)}w^{(ij)}).
\]
with $|w^{(ij)}|=1$.

%

Denote by $\bv$ the operad governing Batalin-Vilkovisky algebras. It is generated by a unary operation $\Delta \in \bv(1)$ of degree -1, and two binary operations, a commutative product $\mu \in \bv_2(2)$ of degree $0$, and a symmetric Lie bracket $\lambda\in \bv(2)$ of degree $-1$. In this case, its dual cooperad $\bv^c$ computes the cohomology of framed little disk operad $\flD_2$. In arity $r$, it is given by the graded commutative algebra \cite{Getzler94}
\[
\bv^c(r)=H^*(\flD_2(r);\K) \cong S(w^{(ij)})_{1\leq i, j\leq r }/(w^{(ij)}-w^{(ji)} \ ; \ w^{(ij)}w^{(jk)}+w^{(jk)}w^{(ki)}+w^{(ki)}w^{(ij)}).
\]
with again all generators of degree one. Note that $\bv$ is not a quadratic operad. Nevertheless, Vallette developed a Koszul duality theory for operads which contain quadratic and linear relations (\cite{Vallette07}, see also the appendix of \cite{HomBV12}). In \cite{HomBV12}, G\'{a}lvez-Carrillo, Tonks and Vallette use these techniques to obtain an explicit cofibrant resolution of the operad $\bv$. For this, define $\kbv$ as the quadratic cooperad
\[
\kbv:= \K[\delta]\circ \ke=\K[\delta]\circ \e^c_2\{2\}=\K[\delta]\circ \Com^c\{2\}\circ \Lie^c\{1\}
\]
where $\delta$ is of degree $-2$ and $\K[\delta]\cong T^c(\delta)$ is the cofree coalgebra generated by $\delta$, thought of as a cooperad concentrated in arity one. As in \cite{HomBV12}, we denote generic element in $\kbv(n)$ by
\[
\delta^s \otimes (L_1\wedge\dots\wedge L_k)
\]
with $L_i\in \Lie^c\{1\}(n_i)$ for $i=1,\dots,k$ and $n_1+\dots+n_k=n$. Note that the degree of this element is
\[
-2s+2\cdot (1-k)+(1-n_1)+\dots+(1-n_k)=-2s+2\cdot (1-k)+k-n=-2s+2-k-n.
\] 
The cooperad $\kbv$ may be endowed with the differential $d_\kbv$ acting as
\[
d_\kbv(\delta^s \otimes (L_1\wedge\dots\wedge L_k))=\delta^{s-1}\otimes \sum\limits_{i=1}^k (-1)^{\epsilon_i} L_1\wedge\dots\wedge L_i' \wedge L_i{''} \wedge \dots\wedge L_k.
\]
Here $\epsilon_i=|L_1|+\dots+|L_{i-1}|=1-n_1+\dots+1-n_{i-1}$ and $L_i'\wedge L_i{''}$ is Sweedler-type notation for the image of $L_i$ under the binary part of the cooperadic cocomposition on $\Lie^c\{1\}$
\[
\Lie^c\{1\}\rightarrow \Lie^c\{1\}(2)\otimes (\underbrace{\Lie^c\{1\}\otimes \Lie^c\{1\}}_{\substack{\rotatebox[origin=c]{90}{$\in$} \\ L_i' \wedge L_i''}}).
\]
The dg cooperad $(\kbv,d_{\kbv})$ is called the Koszul dual cooperad of the operad $\bv$. Indeed, in \cite{HomBV12} the authors show that the cobar construction $\Omega(\kbv)$ of the Koszul dual cooperad $\kbv$ is a resolution of $\bv$, i.e. there is a quasi-isomorphism of operads
\[
\Omega(\kbv)\xrightarrow{\sim} \bv.
\]
By Koszul duality the above is equivalent to the existence of a quasi-isomorphism of operads
\[
\Omega(\bv^c)\xrightarrow{\sim} (\kbv)^c=:\bvk
\]
where $\bvk$ is equipped with the dual differential $d_\bvk:=d_\kbv^*$. The morphism is induced by the linear map of degree zero, $\kappa: \bv^c\rightarrow \bvk$ which is non-zero on generators only, and sends
\begin{align*}
\mu^c &\mapsto s^{-1}\lambda\\
\lambda^c &\mapsto s^{-1}\mu\\
\Delta^c &\mapsto s\delta^*.
\end{align*}
In particular, let 
\[
y\in \bigotimes\limits_{v\in VT}\overline{\bv^c} (\mathrm{star}(v))[-1]
\]
describe an element in $\Omega(\bv^c)$. In this case, the morphism above may be expressed as the composition
\[
(\gamma_T\circ \mathcal{F}_T(\kappa))(y) \in \bvk
\]
where $\mathcal{F}_T(\kappa)$ describes the tensorwise application of $\kappa$ to the tensor product of the vertex set of $T$. Finally, note that, as $S$-modules,
\[
\bvk=\K[\delta^*]\circ (\ke)^c=\K[\delta^*]\circ \e_2\{-2\}=\K[\delta^*]\circ \Com\{-2\}\circ \Lie\{-1\}.
\]

\subsection{Module structures over $\bv^c(r)$}

For any $r\geq 1$, both $(\Omega(\bv^c))(r)$ and $(\bvk)(r)$ are modules over the graded commutative algebra $\bv^c(r)$. For instance, the right action on $(\Omega(\bv^c))(r)$
\[
(\Omega(\bv^c))(r)\otimes \bv^c(r)\xrightarrow{1\otimes \Delta_{\bv^c}} (\Omega(\bv^c))(r)\otimes (\mathcal{F}(\bv^c))(r)\xrightarrow{\mathcal{F}(\mu)} \Omega(\bv^c)(r)
\]
is given by applying the cooperadic cocomposition on $\bv^c(r)$ before using the commutative algebra structure on each vertex in the tensor product indexed by rooted trees. On the other hand, the right action on $\bvk(r)$ is defined by requiring the morphism
\[
\Omega(\bv^c)(r)\rightarrow \bvk(r)
\]
to be equivariant with respect to the actions. In fact, it is thus induced by the following right action 
\begin{align*}
q_1:\bvk(1)\otimes \bv^c(1)&\rightarrow \bvk(1)\\
q_2:\bvk(2)\otimes \bv^c(2)&\rightarrow \bvk(2)
\end{align*}
of $\bv^c(1)$ and $\bv^c(2)$ on $\bvk(1)$ and $\bvk(2)$, respectively,
\begin{align*}
q_1(\delta^* \otimes \Delta) = & 0\\
q_2(\mu\otimes \mu^c)= & \mu\\
q_2(\lambda\otimes\mu^c)=& \lambda\\
q_2(\mu\otimes \lambda^c)=& 0\\
q_2(\lambda\otimes \mu^c)=&\lambda.
\end{align*}
In arity $r\geq 3$, the action is defined by a sum, indexed by rooted trees with $r$ leaves whose vertices are bi- or trivalent, of compositions of the sort $q^T_r:\bvk(r)\otimes \bv^c(r)\rightarrow \bvk(r)$ given by
\[
q^T_r: \bvk(r)\otimes \bv^c(r)\xrightarrow{\Delta_T\otimes \Delta_T} \bigotimes\limits_{v\in VT}\left( \bvk(\mathrm{star}(v))\otimes \bv^c(\mathrm{star}(v))\right)\xrightarrow{\mathcal{F}_T(q)} \bigoplus\limits_{v\in VT} \bvk(\mathrm{star}(v)) \xrightarrow{\nabla_T} \bvk(r).
\]
The left actions are defined analogously.

\section{Graph (co)operads, (co)modules and complexes}

In this section, we define the combinatorial tools on which we will apply the notions of the previous sections. More precisely, we recall the two models for the framed configuration spaces of points on surfaces introduced by Campos and Willwacher \cite{CW}, and Campos, Idrissi, and Willwacher \cite{CIW}. Moreover, we give the construction of the combinatorially defined dg Lie algebra $\GCg$ and its ``hairy'' variant $\HGCg$.

\subsection{The operad $\Graphs$}
The operad $\Graphs$ was introduced by Kontsevich \cite{Kontsevich99}. Let $r\geq 1$. Elements of $\Graphs(r)$ are given by formal series of (isomorphism classes of) undirected graphs with 
\begin{itemize}
\item $r$ labelled vertices, numbered by $1,\dots,r$, called external vertices,
\item an arbitrary number of indistinguishable vertices, called internal vertices
\end{itemize} 
subject to the conditions that
\begin{itemize}
\item internal vertices are of valence at least three,
\item every connected component contains at least one external vertex.
\end{itemize}
Additionally, each graphs comes with an ordering on the set of edges, defined up to signed permutation. The degree of a graph with $k$ edges and $N$ internal vertices is
\[
2N-k.
\]
The differential is denoted by $d_s$ and given by a ``vertex splitting'' operation. In the case of the internal vertices, this amounts to replacing each vertex (one at the time, and taking the sum over all vertices) by two internal vertices connected by an edge, and summing over all possible ways of reconnecting the edges previously attached to the ``splitted'' vertex to the two new vertices in such a way that the valence condition is respected. External vertices are split into an internal and an external vertex connected by an edge. Again we require the internal vertex to remain at least trivalent, when we sum over all possible ways of reconnecting the edges previously attached to the splitted vertex. The new edge is placed last in the new ordering of the edges.

\begin{figure}[h]
\centering
{{
\begin{tikzpicture}[baseline=-.55ex,scale=.7, every loop/.style={}]
\node[circle,draw,inner sep=1.5pt] (a) at (0,0) {$i$};
\node (d1) at (0.2,1) {};
\node (d2) at (-0.6,1) {};
\node (d3) at (-0.2,1) {};
\node (d4) at (0.6,1) {};
\draw (a) to (d1);
\draw (a) to (d2);
\draw (a) to (d3);
\draw (a) to (d4);
\node (aa) at (-1,0) {$d_s$};
\node (aa) at (1,0) {$= \sum$};
\node[circle,draw,inner sep=1.5pt] (a2) at (2.5,0) {$i$};
 \node[circle,draw,fill,inner sep=1.5pt] (b2) at (2.5,0.75) {};
\node (d12) at (2.9,1.4) {};
\node (d22) at (1.8,1) {};
\node (d32) at (2.1,1.4) {};
\node (d42) at (3.2,1) {};
\draw (b2) to (d12);
\draw (a2) to (d22);
\draw (b2) to (d32);
\draw (a2) to (d42);
\draw (a2) to (b2);
\node[circle,draw,fill,inner sep=1.5pt] (a3) at (7,0) {};
\node (d13) at (6.5,0.75) {};
\node (d23) at (6.5,-0.75) {};
\node (d33) at (7.5,0.75) {};
\node (d43) at (7.5,-0.75) {};
\draw (a3) to (d13);
\draw (a3) to (d23);
\draw (a3) to (d33);
\draw (a3) to (d43);
\node (aa) at (6,0) {$d_s$};
\node (aa) at (8,0) {$= \sum$};
\node[circle,draw,fill,inner sep=1.5pt] (a23) at (9,0) {};
 \node[circle,draw,fill,inner sep=1.5pt] (b23) at (10,0) {};
\node (d123) at (8.5,0.75) {};
\node (d223) at (8.5,-0.75) {};
\node (d323) at (10.5,0.75) {};
\node (d423) at (10.5,-0.75) {};
\draw (a23) to (d123);
\draw (a23) to (d223);
\draw (b23) to (d323);
\draw (b23) to (d423);
\draw (a23) to (b23);
\end{tikzpicture}}}
\caption{The differential $d_s$ is defined by the vertex splitting operations above.}\label{fig:ds}
\end{figure}

The collection $\Graphs(r)=\{\Graphs(r)\}_{r\geq 1}$ assembles to form an operad. The symmetric group acts by permuting the labels of the external vertices, while the operadic composition on $\Graphs$ is given by insertion at external vertices. More precisely, $\Gamma_1\circ_j \Gamma_2$ is obtained by inserting $\Gamma_2$ in the $j$-th external vertex of $\Gamma_1$ and summing over all possible ways of connecting the ``loose'' edges which were previously attached to vertex $j$ to the (internal and external) vertices of $\Gamma_1$.

\begin{figure}[h]
\centering
{{
\begin{tikzpicture}[baseline=-.55ex,scale=.6, every loop/.style={}]
 \node[circle,draw,inner sep=1.5pt] (a) at (0,0) {$1$};
 \node[circle,draw,inner sep=1.5pt] (b) at (1,0) {$2$};
\draw (a) edge [bend left=45] (b);
\node (aa) at (2,0) {$\circ_2$};
 \node[circle,draw,inner sep=1.5pt] (a2) at (0+3,0) {$1$};
 \node[circle,draw,inner sep=1.5pt] (b2) at (1+3,0) {$2$};
 \node[circle,draw,inner sep=1.5pt] (c2) at (2+3,0) {$3$};
 \node[circle,draw,fill,inner sep=1.5pt] (d22) at (1+3,1) {};
\draw (b2) to (d22);
\draw (c2) to (d22);
\draw (a2) to (d22);
\node (aaa) at (6,0) {$=$};
\node[circle,draw,inner sep=1.5pt] (a3) at (7,0) {$1$};
\node[circle,draw,inner sep=1.5pt] (b3) at (8,0) {$2$};
\node[circle,draw,inner sep=1.5pt] (c3) at (9,0) {$3$};
\node[circle,draw,inner sep=1.5pt] (d3) at (10,0) {$4$};
\node[circle,draw,fill,inner sep=1.5pt] (d233) at (9,1) {};
\draw (b3) to (d233);
\draw (c3) to (d233);
\draw (d3) to (d233);
\draw (a3) edge [bend left=45] (b3);
\node (aaaa) at (11,0) {$+$};
\node[circle,draw,inner sep=1.5pt] (a4) at (12,0) {$1$};
\node[circle,draw,inner sep=1.5pt] (b4) at (13,0) {$2$};
\node[circle,draw,inner sep=1.5pt] (c4) at (14,0) {$3$};
\node[circle,draw,inner sep=1.5pt] (d4) at (15,0) {$4$};
\node[circle,draw,fill,inner sep=1.5pt] (d2334) at (14,1) {};
\draw (b4) to (d2334);
\draw (c4) to (d2334);
\draw (d4) to (d2334);
\draw (a4) edge [bend left=60] (c4);
\node (aaaaa) at (16,0) {$+$};
\node[circle,draw,inner sep=1.5pt] (a45) at (17,0) {$1$};
\node[circle,draw,inner sep=1.5pt] (b45) at (18,0) {$2$};
\node[circle,draw,inner sep=1.5pt] (c45) at (19,0) {$3$};
\node[circle,draw,inner sep=1.5pt] (d45) at (20,0) {$4$};
\node[circle,draw,fill,inner sep=1.5pt] (d23345) at (19,1) {};
\draw (b45) to (d23345);
\draw (c45) to (d23345);
\draw (d45) to (d23345);
\draw (a45) edge [bend left=45] (d45);
\node (aaaaaa) at (21,0) {$+$};
\node[circle,draw,inner sep=1.5pt] (a456) at (22,0) {$1$};
\node[circle,draw,inner sep=1.5pt] (b456) at (23,0) {$2$};
\node[circle,draw,inner sep=1.5pt] (c456) at (24,0) {$3$};
\node[circle,draw,inner sep=1.5pt] (d456) at (25,0) {$4$};
\node[circle,draw,fill,inner sep=1.5pt] (d233456) at (24,1) {};
\draw (b456) to (d233456);
\draw (c456) to (d233456);
\draw (d456) to (d233456);
\draw (a456) edge [bend left=30] (d233456);
\end{tikzpicture}}}
\caption{The operadic composition in $\Graphs$ is given by insertion at the external vertices.}\label{fig:operadiccomp}
\end{figure}

\begin{rem}
Notice that one can define more general variants of this operad, denoted $\Graphs_n$, one for each $n\geq 0$, and that here we are only considering the version for $n=2$, i.e. $\Graphs=\Graphs_2$. The subscript defines the degrees that are assigned to the intenral vertices (degree $n$) and edges (degree $n-1$).
\end{rem}

In the case of its dual, the cooperad $\pdGraphs$, we only consider finite linear combinations of graphs satisfying the same valence and connectivity conditions as in $\Graphs$. The differential on $\pdGraphs(r)$ is the adjoint of the vertex splitting operation. It is thus given by ``edge contraction'', that is, two internal vertices connected by an edge are merged into one internal vertex, while an internal and an external vertex connected by an edge are merged to an external vertex. Edges between external vertices are not contracted. The cooperadic cocompositions dual to the operadic compositions in $\Graphs$ are described by subgraph extractions.

\begin{rem}\label{rem:graphsfintype}
For $r\geq 1$, the dg vector space $\pdGraphs(r)$ carries a natural graded commutative algebra structure. The product is obtained by identifying external vertices. It is compatible with the differential. Equipped with the increasing filtration
\[
F^p\pdGraphs=\{\Gamma \ | \ 2\# \text{ edges } \leq p\}
\]
$\pdGraphs$ defines a filtered dg Hopf cooperad satisfying the finiteness conditions from Proposition \ref{prop:filtered}. Indeed, for fixed arity and degree, fixing the number of edges leaves us with only finitely many diagrams. Thus, the subquotients
\[
F^s \pdGraphs(r)/F^{s-1}\pdGraphs(r)
\]
are of finite type. Dually, $\Graphs$ defines a complete dg Hopf operad.
\end{rem}

\begin{figure}[h]
\centering
{{
\begin{tikzpicture}[baseline=-.55ex,scale=.7, every loop/.style={}]
 \node[circle,draw,inner sep=1.5pt] (a) at (0,0) {$1$};
 \node[circle,draw,inner sep=1.5pt] (b) at (1,0) {$2$};
 \node[circle,draw,inner sep=1.5pt] (c) at (2,0) {$3$};
 \node[circle,draw,fill,inner sep=1.5pt] (d1) at (0,1) {};
 \node[circle,draw,fill,inner sep=1.5pt] (d2) at (1,1) {};
\draw (a) to (d1);
\draw (b) to (d1);
\draw (b) to (d2);
\draw (d1) to (d2);
\draw (d2) to (a);
\node (aa) at (3,0) {$\wedge$};
 \node[circle,draw,inner sep=1.5pt] (a2) at (0+4,0) {$1$};
 \node[circle,draw,inner sep=1.5pt] (b2) at (1+4,0) {$2$};
 \node[circle,draw,inner sep=1.5pt] (c2) at (2+4,0) {$3$};
 \node[circle,draw,fill,inner sep=1.5pt] (d22) at (1+4.5,1) {};
\draw (b2) to (d22);
\draw (c2) to (d22);
\draw (a2) to (d22);
\draw (b2) to (c2);
\node (aa) at (7,0) {$=$};
 \node[circle,draw,inner sep=1.5pt] (a3) at (8,0) {$1$};
 \node[circle,draw,inner sep=1.5pt] (b3) at (9,0) {$2$};
 \node[circle,draw,inner sep=1.5pt] (c3) at (10,0) {$3$};
 \node[circle,draw,fill,inner sep=1.5pt] (d13) at (8,1) {};
 \node[circle,draw,fill,inner sep=1.5pt] (d23) at (9,1) {};
  \node[circle,draw,fill,inner sep=1.5pt] (d233) at (10,1) {};
\draw (a3) to (d13);
\draw (b3) to (d13);
\draw (b3) to (d23);
\draw (d13) to (d23);
\draw (b3) to (d233);
\draw (c3) to (d233);
\draw (a3) to (d233);
\draw (a3) to (d23);
\draw (b3) to (c3);
\end{tikzpicture}}}
\caption{The graded commutative product on $\pdGraphs$ is given by identifying external vertices.}\label{fig:dgca}
\end{figure}

\subsection{The operad $\BVGraphs$}
Let $\I$ denote the operadic ideal of $\Graphs$ spanned by graphs containing at least one tadpole at an internal vertex. The differential preserves $\I$. To see this, note that when we split an internal vertex to which a tadpole is attached, the tadpole is either preserved (we remain in $\I$) or it produces a double edge (and the graph is thus zero). Furthermore, for any $r\geq 1$, $\I(r)$ describes a coideal in the dg commutative coalgebra $\Graphs(r)$ (see also \cite{Campos17}). We adopt the notation $\BVGraphs$ for the operadic quotient
\[
\BVGraphs:=\Graphs / \I .
\]
The considerations above imply that $\BVGraphs$ is a dg Hopf operad. 

\begin{rem}
In the dual setting, denote by $\pdBVGraphs(r)$ the graded subalgebra of $\pdGraphs (r)$ spanned by finite linear combinations of graphs without tadpoles at internal vertices. It is preserved by the differential since the only way one might produce a tadpole by contracting an edge is when there are two vertices which are connected by two edges. By symmetry reasons, however, such graphs are zero. Moreover, the cooperadic cocompositions can only create tadpoles at external vertices and not at internal vertices. We conclude that therefore the collection $\pdBVGraphs$ forms a dg Hopf cooperad.
\end{rem}

The operad $\Graphs$ was introduced by Kontsevich in order to prove the formality of the little disks operad (\cite{Tamarkin03}, \cite{Kontsevich99}). Indeed, the little disks operad being weakly equivalent to the Fulton-MacPherson-Axelrod-Singer operad (\cite{AS94}, \cite{FM94}, \cite{GetzlerJones94}, \cite{Sinha04}) $\FM_2$, defined as the compactifications of the configuration spaces of points in $\mathbb{R}^2$, by iterated real bordification, the cohomology of the little disks cooperad $e_2^c$ is related to the dg commutative algebra of differential forms via the zig-zag of quasi-isomorphisms of cooperads below. The result extends to the case of framed configuration spaces of points, and the corresponding framed version of the Fulton-MacPherson-Axelrod-Singer operad, $\FFM_2$, whose cohomology is computed by the Batalin-Vilkovisky cooperad (\cite{GP10}, \cite{Severa10}).

\begin{prop}\label{prop:formality}(\cite{Kontsevich99}, \cite{LV14}, \cite{Willwacher15}, \cite{FresseBook17})
There are zig-zags of quasi-isomorphism of dg Hopf cooperads
\begin{align*}
\e_2^c&\leftarrow \pdGraphs \rightarrow \Omega_{\mathrm{PA}}(\FM_2)\\
\bv^c&\leftarrow \pdBVGraphs \rightarrow \Omega_{\mathrm{PA}}(\FFM_2).
\end{align*}
where $\Omega_{\mathrm{PA}}(...)$ denote piecewise semi-algebraic forms (\cite{HLTV}, \cite{KS00}).
\end{prop}
In arity $r\geq 1$, the left morphisms are given by the projection of diagrams with no internal vertices onto the generators of the dg commutative algebras $\e_2^c(r)$ and $\bv^c(r)$, respectively. More precisely, the graph with one edge connecting the external vertices $i$ and $j$ is sent to $\omega_{ij}$, and diagrams with internal vertices are mapped to zero. The right maps are defined through complicated Feynman type integration techniques.

\subsection{The right operadic $\BVGraphs$-module $\BVGraphsg$.}

We recall several results from \cite{CW} on a first combinatorial model for configuration spaces of points on surfaces. Note that most definitions and results of this section generalize to compact manifolds of higher dimension.

\subsection{The right operadic module $\Graphsg'$}
Let $g\geq 1$ and $\Sigma:=\Sigma_g$ be a compact orientable surface of genus $g$. Let $\{1,a_1,b_1,\dots,a_g,b_g,\nu\}$ of degrees $|1|=0,\ |a_i|=|b_i|=1$ for all $i$, and $|\nu|=2$ denote a basis for its cohomology $H:=H^\bullet (\Sigma;\K)$. We equip the space $H$ with the canonical paring $\langle -,-\rangle$ of degree $-2$. Explicitly, $\langle a_i,b_i \rangle=-\langle b_i,a_i\rangle=-1$ and $\langle \nu,1\rangle=\langle 1,\nu\rangle=1$. To render the notation more concise, we occasionally denote the basis of $H$ by $\{e_1,\dots,e_{2g+2}\}$, rather than as above, and write $\{e_q^*\}_q$ for the Poincar\'e-dual basis. The diagonal element is then
\[
\Delta=\nu\otimes 1 + 1\otimes \nu - \sum\limits_i (a_i\otimes b_i-b_i\otimes a_i)=\sum\limits_q \langle e_q,e_q^*\rangle \ e_q\otimes e_q^*.
\]
We denote by $\overline H=H^{\geq 1}(\Sigma)$ the reduced cohomology. For the linear dual $H^*=(H^\bullet(\Sigma;\K))^*$, we denote its basis by $\{1:=1^*,\alpha_1,\beta_1,\dots, \alpha_g,\beta_g,\omega:=\nu^*\}$. Similar as above, we sometimes abbreviate this basis by $\{f_q \}_q$.

In \cite{CW}, Campos and Willwacher define a right operadic dg module $\Graphsg'$ over the operad $\Graphs$. Elements of $\Graphsg'(r)$ are formal series of (isomorphism classes of) graphs with
\begin{itemize}
\item{$r$ labelled vertices, numbered by $1,\dots,r$, called external vertices,}
\item{an arbitrary number of indistinguishable vertices, called internal vertices,}
\end{itemize}
all of which possibly carry decorations in $\overline{H}^*$. As in the case of $\Graphs$, every connected component contains at least one external vertex. However, we have no restriction on the valence of the internal vertices. The valence of a vertex consists of the number of edges and decorations incident at that vertex. Moreover, odd decorations are included in the orientation data, meaning that there is a linear order on the set of edges and decorations of odd degrees. Graphs with different orderings are identified up to sign. The degree of a graph with $k$ edges, $N$ internal vertices and decorations $\{f_{i_1},\dots,f_{i_m}\}$ is
\[
2N-k+\sum\limits_{l=1}^m |f_{i_l}|.
\]
The differential on $\Graphsg'$ is the sum of two operations. The first, denoted $d_s$ is given by a vertex splitting operation which differs only slightly from the one in $\Graphs$. It splits
\begin{itemize}
\item{internal vertices into two internal vertices connected by an edge, and}
\item{external vertices into an internal and an external vertex connected by an edge.}
\end{itemize}
In both cases, we sum over all possible ways of reconnecting the edges and decorations previously attached to the splitted vertex with no restrictions on the valence. The second, which we denote by $d_p$ is given by the sum over all pairs of decorations, say $f_i, f_j \in H^*$ (additionally, assume that, if both are of odd degree, $f_i$ comes first and $f_j$ second in the ordering of the edges and odd decorations), removing the decorations and replacing them by an edge connecting vertices the respective vertices instead, and finally multiplying the result by the pairing of their linear duals $\langle e_i,e_j \rangle\in \K$.

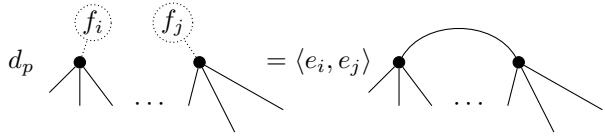
\begin{figure}
\centering
{{
\begin{tikzpicture}[baseline=-.55ex,scale=.7, every loop/.style={}]
\node (aa) at (-1.1,0) {$d_p$};
\node[circle,draw,fill,inner sep=1.5pt] (a) at (0,0) {};
\node (b) at (-0.75,-0.75) {};
\node (c) at (0,-1) {};
\node (d) at (0.75,-1) {};
\node[circle,draw,densely dotted,inner sep=.5pt] (f) at (0.25,0.75) {$f_i$};
\draw (a) to (b);
\draw (a) to (c);
\draw (a) to (d);
\draw[densely dotted] (a) to (f);
\node (aaa) at (1.35,-0.8) {$\cdots$};
\node[circle,draw,fill,inner sep=1.5pt] (a2) at (2.25,0) {};
\node (b2) at (2,-1) {};
\node (c2) at (3,-1.5) {};
\node (d2) at (4,-1) {};
\node[circle,draw,densely dotted,inner sep=.5pt] (f2) at (1.75,0.75) {$f_j$};
\draw (a2) to (b2);
\draw (a2) to (c2);
\draw (a2) to (d2);
\draw[densely dotted] (a2) to (f2);
\node (aaaa) at (4.5,0) {$= \langle e_i,e_j\rangle$};
\node[circle,draw,fill,inner sep=1.5pt] (a3) at (6,0) {};
\node (b3) at (-0.75+6,-0.75) {};
\node (c3) at (0+6,-1) {};
\node (d3) at (0.75+6,-1) {};
\draw (a3) to (b3);
\draw (a3) to (c3);
\draw (a3) to (d3);
\node (aaa3) at (1.35+6,-0.8) {$\cdots$};
\node[circle,draw,fill,inner sep=1.5pt] (a23) at (2.25+6,0) {};
\node (b23) at (2+6,-1) {};
\node (c23) at (3+6,-1.5) {};
\node (d23) at (4+6,-1) {};
\draw (a23) to (b23);
\draw (a23) to (c23);
\draw (a23) to (d23);
\draw (a3) edge[bend left=60] (a23);
\end{tikzpicture}}}
\caption{The differential $d_p$ is defined through the pairing operation. Here $f_i, f_j\in H^*$ and $e_i,e_j\in H$. }\label{fig:pairing}
\end{figure}

\begin{rem}
We tacitly consider all vertices to be decorated by $1^*$. Thus a vertex decorated by $\omega$ is connected by the differential to all vertices.
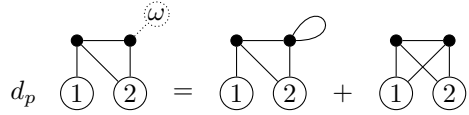
\begin{figure}[h]
\centering
{{
\begin{tikzpicture}[baseline=-.55ex,scale=.7, every loop/.style={min distance=10mm,in=0,out=60,looseness=10}]
 \node (a0) at (-1,0) {$d_p$};
 \node[circle,draw,inner sep=1.5pt] (a) at (0,0) {$1$};
 \node[circle,draw,inner sep=1.5pt] (b) at (1,0) {$2$};
 \node[circle,draw,fill,inner sep=1.5pt] (d1) at (0,1) {};
 \node[circle,draw,fill,inner sep=1.5pt] (d2) at (1,1) {};
\node[circle,draw,densely dotted,inner sep=.5pt] (h) at (1.5,1.5) {$\omega$};
\draw (a) to (d1);
\draw (b) to (d1);
\draw (b) to (d2);
\draw (d1) to (d2);
\draw[densely dotted] (d2) to (h);
\node (aa) at (2,0) {$=$};
 \node[circle,draw,inner sep=1.5pt] (a2) at (3,0) {$1$};
 \node[circle,draw,inner sep=1.5pt] (b2) at (4,0) {$2$};
 \node[circle,draw,fill,inner sep=1.5pt] (d12) at (3,1) {};
 \node[circle,draw,fill,inner sep=1.5pt] (d22) at (4,1) {};
\draw (a2) to (d12);
\draw (b2) to (d12);
\draw (b2) to (d22);
\draw (d12) to (d22);
\draw[loop] (d22) to (d22);
\node (aaa) at (5,0) {$+$};
 \node[circle,draw,inner sep=1.5pt] (a3) at (6,0) {$1$};
 \node[circle,draw,inner sep=1.5pt] (b3) at (7,0) {$2$};
 \node[circle,draw,fill,inner sep=1.5pt] (d13) at (6,1) {};
 \node[circle,draw,fill,inner sep=1.5pt] (d23) at (7,1) {};
\draw (a3) to (d13);
\draw (b3) to (d13);
\draw (b3) to (d23);
\draw (d13) to (d23);
\draw (d23) to (a3);
\end{tikzpicture}}}
\caption{We consider all vertices to be decorated by $1^*$. Thus, $\omega$ is connected to all vertices by the differential. Notice that double edges (two edges connecting the same two vertices) leads to an odd symmetry, and the diagrams are zero.}\label{fig:omegaconnects}
\end{figure}

\end{rem}
The symmetric group $S_n$ acts on $\Graphsg'(r)$ by permuting the labels of the external vertices. Furthermore, the right operadic module structure of $\Graphsg'$ over $\Graphs$ is given by insertion at the external vertices. 

Additionally, $\Graphsg'(r)$ carries the analogous dg cocommutative coalgebra structure as $\Graphs(r)$. In summary, the pair
\[
(\Graphsg',\Graphs)
\]
describes a dg Hopf operad and right operadic dg Hopf module over it.

Let $\I_{\mathrm{(g)}}$ denote the operadic submodule of $\Graphsg'$ spanned by graphs containing at least one tadpole at an internal vertex. Since the right operadic action of $\Graphs$ (and in particular of the operadic ideal $\I\subset \Graphs$) is given by insertion at the external vertices, it clearly preserves $\I_{\mathrm{(g)}}$. Moreover, as with $\I$, the vertex splitting differential preserves $\I_{\mathrm{(g)}}$, and because the pairing operation cannot decrease the number of tadpoles, $\I_{\mathrm{(g)}}$ is closed under the action of the differential. Furthermore, $\I_{\mathrm{(g)}}(r)$ forms a coideal in the dg cocommutative coalgebra $\Graphsg(r)$. It follows that the quotient
\[
\BVGraphsg':=\Graphsg'/ \I_{\mathrm{(g)}}
\]
forms a right operadic dg Hopf module over $\BVGraphs$. Similar to $\BVGraphs$, we represent elements of $\BVGraphsg'$ by formal series of graphs which have no tadpoles at the internal vertices.

\subsection{The dg Lie algebra $\GCg'$}

Campos and Willwacher further define the dg Lie algebra $\GCg'$. Its elements consist of formal series of (isomorphism classes of) connected graphs with a single type of undistinguishable vertices, possibly carrying decorations in $\overline{H}^*$. A graph in $\GCg'$ with $k$ edges, $N$ vertices and $\{f_{i_1},\dots f_{i_m}\}$ decorations has degree
\[
1+2N-k+\sum\limits_{l=1}^m {|f_{i_l}|}.
\]
The differential on $\GCg'$ consists again of similar vertex splitting and pairing operations as above. By abuse of notation, we refer to them again by $d_s$ and $d_p$. They act precisely as they did on internal vertices in the case of $\Graphsg'$, that is,
\begin{itemize}
\item a vertex is split into two vertices connected by an edge, and we sum over all possible ways of reconnecting the edges and decorations previously attached to the splitted vertex,
\item each pair of decorations is replaced by an edge connecting the respective vertices, and the result is multiplied by the pairing of their linear duals - again, a vertex decorated by $\omega$ is connected to all vertices by the differential.
\end{itemize}
The Lie bracket is defined through the same pairing operation. Given $\gamma_1,\gamma_2\in \GCg'$, their Lie bracket is given by summing over all pairs of decorations, one in $\gamma_1$ and the other in $\gamma_2$, replacing them by an edge connecting the respective vertices, and multiplying the result by the pairing of their linear duals. 

\begin{figure}[h]
\centering
{{
\begin{tikzpicture}[baseline=-.55ex,scale=.7, every loop/.style={}]
\node (aa) at (-0.75,0) {$\Bigg [$};
\node (a1) at (0,0.5) {};
\node (a2) at (0,-0.5) {};
\node[circle,draw,inner sep=2pt] (a) at (0,0) {$\gamma_1$};
\node[circle,draw,fill,inner sep=1.5pt] (b) at (1,0) {};
\draw (a) to (b);
\draw (a1) to  (b);
\draw (a2) to (b);
\node[circle,draw,densely dotted,inner sep=.5pt] (c) at (1.25,0.75) {$f_i$};
\draw[densely dotted] (b) to (c);
\node (aaa) at (1.5,0) {$,$};
\node (aaaa) at (4.75,0) {$\Bigg ]=\langle e_i,e_j \rangle $};
\node (a12) at (3,0.5) {};
\node (a22) at (3,-0.5) {};
\node[circle,draw,inner sep=2pt] (a2) at (3,0) {$\gamma_2$};
\node[circle,draw,fill,inner sep=1.5pt] (b2) at (2,0) {};
\draw (a2) to (b2);
\draw (a12) to  (b2);
\draw (a22) to (b2);
\node[circle,draw,densely dotted,inner sep=.5pt] (c2) at (1.75,-0.75) {$f_j$};
\draw[densely dotted] (b2) to (c2);
\node (a13) at (6.5,0.5) {};
\node (a23) at (6.5,-0.5) {};
\node[circle,draw,inner sep=2pt] (a3) at (6.5,0) {$\gamma_1$};
\node[circle,draw,fill,inner sep=1.5pt] (b3) at (7.5,0) {};
\draw (a3) to (b3);
\draw (a13) to  (b3);
\draw (a23) to (b3);
\node (a124) at (9.5,0.5) {};
\node (a224) at (9.5,-0.5) {};
\node[circle,draw,inner sep=2pt] (a24) at (9.5,0) {$\gamma_2$};
\node[circle,draw,fill,inner sep=1.5pt] (b24) at (8.5,0) {};
\draw (a24) to (b24);
\draw (a124) to  (b24);
\draw (a224) to (b24);
\draw (b3) to (b24);
\end{tikzpicture}}}
\caption{The Lie bracket on $\GCg'$ is given via the pairing operation. Here $f_i, f_j\in H^*$ and $e_i,e_j\in H$. }\label{fig:dp}
\end{figure}
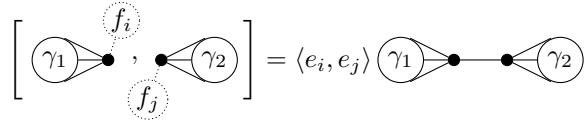

\subsection{Twisted versions}
Let $I_{(g)}\subset \GCg'$ be the Lie ideal spanned by graphs which contain at least one tadpole. Since the differential cannot reduce the number of tadpoles, it also defines a subcomplex. Denote by 
\[
\GCg''=\GCg'/I_{(g)}
\]
the quotient dg Lie algebra. Next, consider the graph
\[
z:= \zzeroo \ + \ \sum\limits_{i=1}^g \ \zzeroab
\]
It defines a Maurer-Cartan element in the quotient dg Lie algebra $\GCg''$ \cite{CW}. 
\begin{defi}
The graph complex of a surface $\Sigma$ of genus $g\geq 1$ is the dg Lie algebra
\[
\GCg:=(\GCg'')^{z}.
\]
In the case $g=1$, the graph $z$ is already a Maurer-Cartan element in $\GCg'$, and we set
\[
\GC_{(1)}^{\minitadp}:=(\GC_{(1)}')^z
\]
to denote the version with tadpoles.
\end{defi}

\begin{rem}
We view $\GCg$ as a complete dg Lie algebra by endowing it with the complete descending filtration
\[
F_p\GCg=\{\gamma\ | \ 2\# \text{ edges } - \text{total degree of all decorations} \geq p\}.
\]
Each $F_p\GCg$ defines a subcomplex, and it satisfies 
\[
[F_p\GCg,F_q \GCg]\subset F_{p+q}\GCg.
\]
\end{rem}

\subsection{Extension by $\spp(H^*)$}
In what follows, we adopt the rather unconventional notation and write $\{\partial_{f_q}\}_q$ for the basis of the linear dual space $(H^*)^*$ of $H^*$. Let $\osp(H^*)$ denote the Lie algebra of graded linear transformations $\sigma:H^*\rightarrow H^*$ which preserve the induced pairing on $H^*$. Denote by $\spp(H^*)$ its Lie subalgebra given by
\[
\spp(H^*)=\{\sigma \in \osp(H^*) \ \vert  \ 1^*\not \in \mathrm{im}(\sigma) \}.
\]
A basis for $\spp(H^*)\subset \gl(H^*)\cong H^*\otimes (H^*)^*$ is given by
\begin{align*}
\text{degree }0: \ & \ \alpha_i \partial_{\beta_i}, \ \beta_i \partial_{\alpha_i}, \ \alpha_i\partial_{\alpha_j}-\beta_j\partial_{\beta_i} \text{ for } i,j=1,\dots g  \text{ and } \alpha_i\partial_{\beta_j} + \alpha_j \partial_{\beta_i}, \ \beta_i\partial_{\alpha_j}+\beta_j \partial_{\alpha_i} \text{ for } i\neq j\\
\text{degree } -1: \ & \ \omega\partial_{\alpha_i} + \beta_i\partial_1, \ \omega\partial_{\beta_i} - \alpha_i\partial_1 \text{ for } i,j=1,\dots g
\end{align*}
In particular, note that for $2n=2g+2$, $\spp(H^*)$ is $2n^2-n-1$ dimensional. It acts as a graded Lie algebra on the dg Lie algebras $\GCg'$ and $\GCg''$ on the decorations (on one at the time, and we take the sum over all decorations). For $g=1$, it is compatible with the twisting differential, since $\sigma.z=0$ for all $\sigma \in \spp(H)$. This is not the case for $g\geq 2$, however, since for any $i$, we have
\begin{align*}
(\omega\partial_{\alpha_i} + \beta_i\partial_1). z\neq & 0\\
(\omega\partial_{\beta_i} - \alpha_i\partial_1). z\neq & 0.
\end{align*}
In contrast, for all other $\sigma \in \spp(H^*)$ we have $\sigma . z=0$. In any event, we twist the semi-direct product of Lie algebras $\spp(H^*)\ltimes \GCg'' $ and, in the case $g=1$ additionally $\spp(H^*) \ltimes \GC_{(1)}'$ by the Maurer-Cartan element $(0,z)$. By abuse of notation, we will denote the twisted complexes by
\begin{align*}
\spp(H^*)\ltimes \GCg:=(\spp(H^*)\ltimes \GCg'')^{(0,z)} \ \text{ for } g\geq 1\\
\spp(H^*)\ltimes \GC^{\minitadp}_{(1)}:=(\spp(H^*)\ltimes \GC_{(1)}')^{(0,z)} \ \text{ for } g=1.
\end{align*}
The extension by $\spp(H^*)$ only affects the cohomology in degrees $-1,0$ non-trivially.

\begin{lemma}\label{lemma:extensioncommutes}
For $g\geq 2$, we have
\[
H^{i}(\spp(H^*) \ltimes \GCg)=H^{i}(\GCg)
\]
for all $i\neq 0$. For $g=1$, 
\begin{align*}
H^i(\spp(H^*)\ltimes \GC_{(1)})=&\spp_i(H^*)\ltimes H^i(\GC_{(1)})\\
H^i(\spp(H^*)\rtimes \GC^{\minitadp}_{(1)})=&\spp_i(H^*)\ltimes H^i(\GC^{\minitadp}_{(1)})
\end{align*}
where $\spp_i(H^*)$ denotes the degree $i$ part of $\spp(H^*)$. It is non-zero only for $i=-1,0$.
\end{lemma}

\begin{proof}
Since for any $g\geq 1$ the graded Lie algebra $\spp(H^*)$ is concentrated in degrees $-1$ and $0$, we have
\[
H^i(\spp(H^*)\ltimes \GCg)\cong H^i(\GCg) \text{ for } i \neq -1,0,1.
\]
A priori, on homogeneous elements of degree $-1$ and $0$ the differential reads
\begin{equation}\label{eq:dsigma}
d(\sigma,\Gamma)=(0, (d_{s}+d_{p}) \Gamma +(-1)^{|\sigma|}([z, \Gamma] - \sigma. z) ).
\end{equation}
Since for all $\sigma\in \spp(H^*)$ of degree zero, $\sigma.z=0$, the differential is not affected by the twist, and
\[
d(\sigma,\Gamma)=(0,d\Gamma).
\]
Thus also in degree one passing to cohomology commutes with extending by $\spp(H^*)$. For $g\geq 2$, consider the descending filtration on $\GCg$ given by
\[
\# \text{ of edges } + \#\text{ of vertices } \geq p.
\]
For $\sigma \in \spp_{-1}(H^*)$, $d(\sigma,\Gamma)$ is as in equation \eqref{eq:dsigma}, but while $\sigma. z\in F_1\GCg$ has only terms with exactly one vertex and no edges, $(d_{s}+d_{p}) \Gamma +(-1)^{|\sigma|}[z, \Gamma]\in F_2\GCg$. Since $\sigma.z\neq 0$, the cohomology classes of degree $-1$ are all of the form $(0,\Gamma)$ with $|\Gamma|=-1$ and
\[
H^{-1}(\spp(H^*) \ltimes \GCg)=H^{-1}(\GCg).
\]
For $g=1$, $\sigma.z=0$ also for $|\sigma|=-1$, and the second result follows. 
\end{proof}

\subsection{Tripods - special elements of degree zero for $g\geq 2$.}

Let $g\geq 2$. Consider the graphs $t_{\alpha,\beta,\gamma}$ with one vertex and three decorations $\alpha,\beta,\gamma\in H^{*}_{-1}$.
\[
t_{\alpha,\beta,\gamma}=\triabc
\]
We refer to such diagrams as tripods. They satisfy $d t_{\alpha,\beta,\gamma}=0$ (the splitting differential and the twisting term cancel each other out, while the pairing leads to a tadpole). The extension by $\spp(H^*)$ leads to the following relation in $H^0(\spp(H^*)\ltimes \GCg)$. Set $\sigma_{\alpha_j}:=\omega \partial_{\beta_j}+\alpha_j\partial_1$, and consider
\[
d (\sigma_{\alpha_j},0)=(0,-\sigma_{\alpha_j}.z)=(0,-\sum_{i\neq j} t_{\alpha_i,\beta_i,\alpha_j}).
\]
Similarly for $\sigma_{\beta_j}:=\omega \partial_{\alpha_j}+\beta_j\partial_1$, we find
\[
d (\sigma_{\beta_j},0)=(0,-\sigma_{\beta_j}.z)=(0,\sum_{i\neq j} t_{\alpha_i,\beta_i,\beta_j}).
\]
Let $\frakh'$ be the Lie subalgebra of $\GCg$ generated by all tripods $t_{\alpha,\beta,\gamma}$ for $ \alpha,\beta,\gamma \in H^{*}_{-1}$. Quotient $\frakh'$ by the Lie ideal generated by the relations coming from above, i.e.
\[
\forall j\in 1,\dots,g \ : \ \sum_{i\neq j} t_{\alpha_i,\beta_i,\alpha_j}=0 \ \text{ and } \ \sum_{i\neq j} t_{\alpha_i,\beta_i,\beta_j}=0
\]
and denote by $\frakh$ the quotient Lie algebra. Moreover, notice that the action of the Lie algebra $\spp_0(H^*)$ on $\GCg$ preserves both the tripods as well as the Lie ideal. This allows us to define the product $\spp_0(H^*)\ltimes \frakh$. We then have the following conjecture the author learned from Anton Alekseev. It is inspired by work of Hain \cite{Hain97}.

\begin{conj}\label{conj:Tripods}
The morphism induced by the inclusion of $\frakh'\rightarrow  \GCg$ induces an isomorphism
\[
\spp_0(H^*)\ltimes \frakh\cong H^0(\spp(H^*)\ltimes \GCg).
\]
\end{conj}

\subsection{Lie algebra action}

For $r\geq 1$, the dg Lie algebra $\GCg'$ acts on $\Graphsg'(r)$ through the pairing operation. That is, for $\gamma\in \GCg'$ and $\Lambda \in \Graphsg'(r)$
\[
\gamma\cdot \Lambda \in \Graphsg'(r)
\]
is defined by summing over all pairs of decorations, one in $\Lambda$, the other in $\gamma$,  replacing them by an edge connecting the vertices they were attached to and finally multiplying the result by the value of the pairing on the decorations. The action is compatible with all the respective structures on both $\GCg'$ and $\Graphsg'$. In particular, it defines a coderivation for the dg cocommutative coalgebra structure in each arity, and a derivation with respect to the operadic module structure.

\begin{figure}[h]
\centering
\begin{tikzpicture}
[baseline=-.55ex,scale=.7, every loop/.style={}]
 \node[circle,draw,inner sep=1.5pt] (a) at (0,0) {$1$};
  \node[circle,draw,inner sep=1.5pt] (b) at (1,0) {$2$};
  \node[circle,draw,inner sep=1.5pt] (c) at (2,0) {$3$};
  \node[circle,draw,fill,inner sep=1pt] (d) at (0.5,1) {};
  \node[circle,draw,fill,inner sep=1pt] (e) at (1.5,1) {};
  \draw (a) edge[] (d);
\draw (a) edge[] (d);
\draw (b) edge[] (d);
\draw (c) edge[] (e);
\draw (e) edge[] (d);
\draw (b) edge[] (e); 
  \node[circle,draw,densely dotted,inner sep=.5pt] (f) at (0,1.75) {$f_j$};
  \draw (d) edge[densely dotted] (f);
\node [] (g2) at (-5.5,1) {$\gamma =$};
\node [] (g) at (-0.75,1) {$\Lambda=$};
  \node[circle,draw,fill,inner sep=1pt] (a2) at (-3,1) {};
  \node[circle,draw,fill,inner sep=1pt] (b2) at (-4.5,1) {};
    \node[circle,draw,fill,inner sep=1pt] (c2) at (-3.75,1.75) {};
  \node[circle,draw,fill,inner sep=1pt] (d2) at (-3.75,0.25) {};
  \node[circle,draw,densely dotted,inner sep=.5pt] (e2) at (-2.5,0.25) {$f_i$};
  \draw (a2) edge[densely dotted] (e2); 
\draw (a2) edge[] (b2);
\draw (a2) edge[] (c2);
\draw (a2) edge[] (d2);
\draw (b2) edge[] (d2);
\draw (b2) edge[] (c2);
\draw (c2) edge[] (d2);
\node [] (f2) at (5,1) {$\gamma\cdot \Lambda=\langle e_i,e_j\rangle$};
 \node[circle,draw,inner sep=1.5pt] (a3) at (9,0) {$1$};
  \node[circle,draw,inner sep=1.5pt] (b3) at (10,0) {$2$};
  \node[circle,draw,inner sep=1.5pt] (c3) at (2+9,0) {$3$};
  \node[circle,draw,fill,inner sep=1pt] (d3) at (0.5+9,1) {};
  \node[circle,draw,fill,inner sep=1pt] (e3) at (1.5+9,1) {};
  \draw (a3) edge[] (d3);
\draw (a3) edge[] (d3);
\draw (b3) edge[] (d3);
\draw (c3) edge[] (e3);
\draw (e3) edge[] (d3);
\draw (b3) edge[] (e3); 
  \node[circle,draw,fill,inner sep=1pt] (a23) at (4.5+2.5,1) {};
  \node[circle,draw,fill,inner sep=1pt] (b23) at (6+2.5,1) {};
    \node[circle,draw,fill,inner sep=1pt] (c23) at (5.25+2.5,1.75) {};
  \node[circle,draw,fill,inner sep=1pt] (d23) at (5.25+2.5,0.25) {};
\draw (a23) edge[] (b23);
\draw (a23) edge[] (c23);
\draw (a23) edge[] (d23);
\draw (b23) edge[] (d23);
\draw (b23) edge[] (c23);
\draw (c23) edge[] (d23);
\draw (d3) to (b23);
\end{tikzpicture}
\caption{The action of $\GCg'$ on $\Graphsg'$ is defined via the pairing operation. It induces an action of $\GCg''$ on $\BVGraphsg'$, and an action of $\GCg$ on $\BVGraphsg$.}\label{fig:actionGCGraphs}
\end{figure}

Additionally, note that the action of $\GCg'$ descends to the quotient $\BVGraphsg'(r)$ - and that moreover any graph in $\GCg'$ containing at least one tadpole acts trivially on $\BVGraphsg'$, which implies that we have a well-defined action of the quotient dg Lie algebra $\GCg''$ of tadpole-free diagrams on the right operadic dg Hopf module $\BVGraphsg'$. For $r\geq 1$, we twist its differential by the action of $z$ to obtain the complex
\[
\BVGraphsg(r):=(\BVGraphsg'(r))^{z}.
\]
The differential on $\BVGraphsg(r)$ is thus given by 
\[
d\Gamma = d_s \Gamma + d_p \Gamma + z \cdot \Gamma.
\]
In particular, this yields an action of the dg Lie algebra $\GCg$ on the right operadic dg Hopf module $\BVGraphsg=\{\BVGraphsg(r)\}_{r\geq 1}$. For $g=1$ similar considerations hold without having to quotient out tadpoles, and in this case we twist $\Graphs_{(1)}'$ by the Maurer-Cartan element $z\in \GC_{(1)}'$ and set
\[
\Graphs_{(1)}(r):=(\Graphs_{(1)}'(r))^{z}.
\]
The corresponding right operadic dg Hopf module comes equipped with an action of the dg Lie algebra $\GC^{\minitadp}_{(1)}$.

In both cases, notice that the Lie algebra $\spp(H^*)$ acts on $\Graphsg'$ via the decorations. The argument above extends to an action of the twisted dg Lie algebras $\spp(H^*)\ltimes \GCg$ and $\spp(H^*)\ltimes \GC^{\minitadp}_{(1)}$ and thus, in summary, we obtain 
\begin{align*}
\spp(H^*)\ltimes \GCg \ \actson \ & \BVGraphsg \ \text{ for } g\geq 1\\
\spp(H^*)\ltimes \GC^{\minitadp}_{(1)}  \ \actson \ & \Graphs_{(1)} \ \text{ for } g=1.
\end{align*}

\subsection{Dual spaces}
In the dual case, we have the analogous construction. The right cooperadic dg Hopf comodule $\Graphsg'^*$ is obtained by considering only finite linear combinations of graphs satisfying the same connectivity conditions as in $\Graphsg'$, and for which the decorations lie in $H$. The differential on $\Graphsg'^*(r)$ is given by the sum of two terms. First, we have the same edge contraction operation as in $\pdGraphs(r)$, that is, two internal vertices connected by an edge are merged into one internal vertex, while an internal and an external vertex connected by an edge are merged to an external vertex. Edges between external vertices are not contracted. The second term $d_p^*$ is given by replacing edges (one at the time and taking the sum over all the edges) by the diagonal element as sketched below. We require that the resulting graph still lies in $\Graphsg'^*(r)$, i.e. that each connected component contains at least one external vertex, and only replace the edges for which this condition is satisfied.

\begin{figure}[h]
\centering
{{
\begin{tikzpicture}[baseline=-.55ex,scale=.7, every loop/.style={}]
\node (aa) at (-1.5,0) {$d_p^*$};
\node[circle,draw,fill,inner sep=1.5pt] (a) at (0,0) {};
\node (b) at (-0.75,0.75) {};
\node (c) at (-1,0) {};
\node (d) at (-0.75,-0.75) {};
\draw (a) to (b);
\draw (a) to (c);
\draw (a) to (d);
\node[circle,draw,fill,inner sep=1.5pt] (a2) at (1,0) {};
\draw (a) to (a2);
\node (b2) at (1.75,0.75) {};
\node (c2) at (2,0) {};
\node (d2) at (1.75,-0.75) {};
\draw (a2) to (b2);
\draw (a2) to (c2);
\draw (a2) to (d2);
\node (aaaa) at (3,0) {$= \sum\limits_{q}$};
\node[circle,draw,densely dotted,inner sep=.5pt] (f) at (5.5,0.75) {$e_q$};
\node[circle,draw,densely dotted,inner sep=.5pt] (f2) at (5.75,-0.75) {$e_q^*$};
\node[circle,draw,fill,inner sep=1.5pt] (a3) at (5,0) {};
\node (b3) at (-0.75+5,0.75) {};
\node (c3) at (0+4,0) {};
\node (d3) at (-0.75+5,-0.75) {};
\draw (a3) to (b3);
\draw (a3) to (c3);
\draw (a3) to (d3);
\node[circle,draw,fill,inner sep=1.5pt] (a23) at (2.25+4,0) {};
\node (b23) at (2+5,-0.75) {};
\node (c23) at (3+4,0) {};
\node (d23) at (7,0.75) {};
\draw (a23) to (b23);
\draw (a23) to (c23);
\draw (a23) to (d23);
\draw[densely dotted] (a3) to (f);
\draw[densely dotted] (a23) to (f2);
\end{tikzpicture}}}
\caption{The operation dual to the pairing differential replaces edges by the diagonal class.}\label{fig:dualtopairing}
\end{figure}
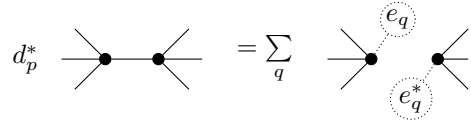

In this way, $\Graphsg'^*(r)$ allows an action of the dg Lie algebra $\GCg'$ given by replacing those edges by the diagonal class for which one of the two components contains no external vertices. The cut-off graph may be viewed as an element in the dual of $\GCg'$ which we evaluate on the element that acts. For $\gamma \in \GCg'$, $\Gamma \in \Graphsg'^*(r)$, and denoting $\lambda$ the cut-off graph and $\Gamma/\lambda$ the remaining graph in $\Graphsg'^*(r)$, we have
\[
\gamma. \Gamma= \sum\limits_{\lambda} \lambda(\gamma)\cdot \Gamma/\lambda.
\]
The action is compatible with the graded commutative product, defined by identifying external vertices, and moreover, with the right cooperadic comodule structure over $\pdGraphs$ which is given by subgraph extractions. Denote by $\BVGraphsg'^*(r)$ the dg subalgebra of $\Graphsg'^*(r)$ generated by diagrams with no tadpoles at the internal vertices. Notice that by the same reasoning as in $\pdBVGraphs(r)$ we cannot produce tadpoles at internal vertices by contracting edges. Clearly, this holds also for $d_p^*$. Moreover, $\BVGraphsg'^*$ assembles to form a right cooperadic comodule over $\pdBVGraphs$. 

\begin{figure}[h]
\centering
{{
\begin{tikzpicture}[baseline=-.55ex,scale=.7, every loop/.style={}]
 \node[circle,draw,inner sep=1.5pt] (a) at (0,0) {$1$};
 \node[circle,draw,inner sep=1.5pt] (b) at (1,0) {$2$};
 \node[circle,draw,inner sep=1.5pt] (c) at (2,0) {$3$};
 \node[circle,draw,fill,inner sep=1.5pt] (d1) at (0,1) {};
 \node[circle,draw,fill,inner sep=1.5pt] (d2) at (1,1) {};
\node[circle,draw,densely dotted,inner sep=.5pt] (h) at (1,2) {$\alpha_3$};
\draw (a) to (d1);
\draw (b) to (d1);
\draw (b) to (d2);
\draw (d1) to (d2);
\draw[densely dotted] (d2) to (h);
\node (aa) at (3,0) {$\wedge$};
 \node[circle,draw,inner sep=1.5pt] (a2) at (0+4,0) {$1$};
 \node[circle,draw,inner sep=1.5pt] (b2) at (1+4,0) {$2$};
 \node[circle,draw,inner sep=1.5pt] (c2) at (2+4,0) {$3$};
 \node[circle,draw,fill,inner sep=1.5pt] (d22) at (1+4.5,1) {};
\node[circle,draw,densely dotted,inner sep=.5pt] (f2) at (-0.5+4,1) {$\alpha_1$};
\node[circle,draw,densely dotted,inner sep=.5pt] (j2) at (2+3.5,2) {$\omega^*$};
\draw (b2) to (d22);
\draw (c2) to (d22);
\draw[densely dotted] (a2) to (f2);
\draw[densely dotted] (d22) to (j2);
\node (aa) at (7,0) {$=$};
 \node[circle,draw,inner sep=1.5pt] (a3) at (8,0) {$1$};
 \node[circle,draw,inner sep=1.5pt] (b3) at (9,0) {$2$};
 \node[circle,draw,inner sep=1.5pt] (c3) at (10,0) {$3$};
 \node[circle,draw,fill,inner sep=1.5pt] (d13) at (8,1) {};
 \node[circle,draw,fill,inner sep=1.5pt] (d23) at (9,1) {};
  \node[circle,draw,fill,inner sep=1.5pt] (d233) at (9.5,1) {};
\node[circle,draw,densely dotted,inner sep=.5pt] (f3) at (-0.5+8,1) {$\alpha_1$};
\node[circle,draw,densely dotted,inner sep=.5pt] (h3) at (1+8,2) {$\alpha_3$};
\node[circle,draw,densely dotted,inner sep=.5pt] (j3) at (2+8,2) {$\omega^*$};
\draw (a3) to (d13);
\draw (b3) to (d13);
\draw (b3) to (d23);
\draw (d13) to (d23);
\draw (b3) to (d233);
\draw (c3) to (d233);
\draw[densely dotted] (a3) to (f3);
\draw[densely dotted] (d23) to (h3);
\draw[densely dotted] (d233) to (j3);
\end{tikzpicture}}}
\caption{The graded commutative product on $\pdBVGraphsg(3)$ is given by identifying external vertices.}\label{fig:dgcag}
\end{figure}

Notice that the action of $\GCg'$ descends to an action on $\BVGraphsg'^*$ and induces an action of the tadpole-free version $\GCg''$. Again, we twist by the action of the Maurer-Cartan element $z\in \GCg''$ to obtain the right cooperadic dg Hopf comodule 
\[
\pdBVGraphsg=(\BVGraphsg'^*)^z
\]
equipped with the action of the twisted dg Lie algebra $\GCg$. Additionally, define an action of $\spp(H^*)$ on $H$ by specifying the values on the basis elements as follows - if $\sigma(f_q)=f_p$, then $\sigma(e_p)=e_q$. It induces an action of $\spp(H^*)$ on $\pdBVGraphsg$, and ultimately, an action of $\spp(H^*)\ltimes \GCg$ on $\pdBVGraphsg$.

\begin{rem}
When $g=1$, we may perform precisely the same construction as above, but twist $\Graphs_{(1)}'^*$ by the Maurer-Cartan element $z\in \GC_{(1)}'$, i.e. we allow tadpoles at internal vertices. We obtain a right cooperadic dg Hopf comodule over $\pdGraphs$ which we denote by $\pdGraphs_{(1)}$.
\end{rem}

\begin{rem}
Equip $\pdBVGraphsg$ with the increasing filtration
\[
F^p\pdBVGraphsg=\{\Gamma \ | \ 2\# \text{ edges } + \text{total degree of all decorations} \leq p\}
\]
and $\pdBVGraphs$ with filtration defined by twice the number of edges (see Remark \ref{rem:graphsfintype}). In this way, $(\pdBVGraphsg, \pdBVGraphs)$ defines a filtered right cooperadic dg Hopf comodule in the sense of Section \ref{sec:filtered}. Furthermore, for fixed arity and degree, fixing the number of edges and the total degree of all decorations, leaves us with only finitely many possible diagrams. This implies in particular that the subquotients
\[
F^s \pdBVGraphsg(r)/F^{s-1} \pdBVGraphsg(r)
\]
are of finite type. The same holds for $\pdBVGraphs$. Moreover, the action of the complete dg Lie algebra $\GCg$ on $\pdBVGraphsg$ is compatible (in the sense of Section \ref{sec:Lieaction}) with the respective filtrations.
\end{rem}

\begin{rem}
The notation is sometimes (for instance in \cite{CIW}) chosen such that $\pdBVGraphsg$ (and similarly $\pdGraphs$, $\pdBVGraphs$) denotes a right operadic dg Hopf module (or a dg Hopf operad, respectively), instead of a \emph{co}module, and $\BVGraphsg$ denotes a right operadic dg Hopf comodule (instead of a module) - while in \cite{CW} it is the other way around. We choose to adhere to the second version.
\end{rem}

The generalization to the surface case of Kontsevich's result then reads as follows. The framed Fulton-MacPherson-Axelrod-Singer compactification of configuration spaces of points on $\Sigma_g$, $\FFM_{(g)}$ defines a right operadic module over the framed Fulton-MacPherson-Axelrod-Singer operad, $\FFM_2$. Intuitively, the compactification $\FFM_{(g)}(r)$ is as a manifold with corners, whose interior is given by $\Conf_r(\Sigma_g)$, and its boundary consists of configurations of points on $\Sigma_g$ which become infinitesimally close, together with a trivialization of the tangent bundle (a framing) of $\Sigma_g$ at each point of the configuration. For $g\neq 1$, the surface $\Sigma_g$ is not parallelizable, and indeed this forces us to work with the framed versions. In the case $g=1$, we can also consider the (unframed version of the) Fulton-MacPherson-Axelrod-Singer compactification, $\FM_{(1)}$, which in this case forms a right operadic module over the Fulton-MacPherson-Axelrod-Singer operad $\FM_2$. Through techniques similar to the ones introduced by Kontsevich, Campos and Willwacher were able to prove the following result. Notice that the framed version also exists for $g=1$.

\begin{thm}[\cite{CIW}]\label{prop:model}
There are explicit quasi-isomorphism of right cooperadic dg Hopf comodules over $\Q$
\begin{align*}
&(\pdGraphs_{(1)},\pdGraphs) \xrightarrow{\sim} (\Omega_{\mathrm{PA}}(\FM_{(1)}),\Omega_{\mathrm{PA}}(\FM_2))\\
&(\pdBVGraphsg,\pdBVGraphs) \xrightarrow{\sim} (\Omega_{\mathrm{PA}}(\FFM_{(g)}),\Omega_{\mathrm{PA}}(\FFM_2))
\end{align*}
\end{thm}

\begin{rem}
In other words, Theorem \ref{prop:model} states that
\begin{itemize}
\item{for $g=1$, $(\pdGraphs,\pdGraphs_{(1)})$ is a rational model for the pair $(\FM_2,\FM_{(1)})$}
\item{for $g \geq 1$, $(\pdBVGraphs,\pdGraphsg)$ is a rational model for the pair $(\FFM_2,\FFM_{(g)})$.}
\end{itemize}
\end{rem}

\subsection{The right cooperadic $\bv^c$-comodule $\Mog$}\label{sec:Mog}

In \cite{CIW}, Campos, Idrissi and Willwacher adapt Bezrukavnikov's model $\Mog$ \cite{Bezru94} for configuration spaces of points on surfaces to encode also a right cooperadic $\bv^c$-comodule structure. In this section, we recall some of their results. In part, these already appeared in some form in their separate works (\cite{Idrissi19}, \cite{CW}).

\begin{defi}
Let $\Mog(r)$ denote the dg commutative algebra generated by the symbols 
\[
w^{(ij)}, \ a_l^{(i)}, \ b_l^{(i)}, \nu^{(i)} \text{ for } 1\leq i,j \leq r \text{ and } 1\leq l\leq g
\]
subject to the relations 
\begin{align*}
&|w^{(ij)}|=|a_l^{(i)}|=|b_l^{(i)}|=1, \ |\nu^{(i)}|=2,\\
&w^{(ij)}=w^{(ji)},\ a_l^{(i)}b_l^{(i)}=\nu^{(i)} \text{ for all }  l,\\
&a_l^{(i)} w^{(ij)}=a_l^{(j)} w^{(ij)}, \ b_l^{(i)} w^{(ij)}=b_l^{(j)} w^{(ij)} \\
&w^{(ij)}w^{(jk)}+w^{(jk)}w^{(ki)}+w^{(ki)}w^{(ij)}=0 \text{ for } i,j,k \text{ all distinct}
\end{align*}
and equipped with the differential which is zero on all generators but $w^{(ij)}$, for which
\begin{align*}
d w^{(ij)}&=\nu^{(i)}+\nu^{(j)}-\sum\limits_{l=1}^g (a_l^{(i)}b_l^{(j)}-b_l^{(i)}a_l^{(j)}) \text{ for } i\neq j\\
d w^{(ii)} &=(2-2g)\nu^{(i)}.
\end{align*}
\end{defi}

Notice that $\Mog$ is a right cooperadic comodule over $\bv^c$. The cooperadic module structure on $\Mog$ is then given on generators by (\cite{Idrissi19}, \cite{CIW})
\begin{align*}
\circ^*_u:\Mog(U \setminus \{u\} \sqcup W )&\rightarrow \Mog(U)\otimes \bv^c(W)\\
w^{(ij)}&\mapsto 
\begin{cases}
w^{(uu)}\otimes 1 + 1\otimes w^{(ij)} & \text{ if } i,j \in W\\
w^{(uj)}\otimes 1 & \text{ if } i\in W,\  j\in U\setminus \{u\}\\
w^{(iu)}\otimes 1 & \text{ if } i \in U\setminus\{u\},\  j\in W\\
w^{(ij)}\otimes 1 & \text{ if } i,j \in U\setminus \{u\}
\end{cases}\\
w^{(ii)}&\mapsto
\begin{cases}
w^{(uu)}\otimes 1 +1 \otimes w^{(ii)} & \text{ if } i\in W\\
w^{(ii)} \otimes 1 &\text{ if } i\in U\setminus \{u\}
\end{cases}\\
a_l^{(i)}&\mapsto
\begin{cases}
a_l^{(u)}\otimes 1 & \text{ if } i\in W\\
a_l^{(i)} \otimes 1 &\text{ if } i\in U\setminus \{u\}
\end{cases}\\
b_l^{(i)}&\mapsto
\begin{cases}
b_l^{(u)}\otimes 1 & \text{ if } i\in W\\
b_l^{(i)} \otimes 1 &\text{ if } i\in U\setminus \{u\}
\end{cases}
\end{align*}

\begin{rem}\label{rem:Mog}
We may view elements of $\Mog(r)$ as diagrams with only labelled vertices which possibly carry decorations in $H$. The differential acts by replacing edges by the diagonal class. The generators are of the form
\[
\lsek \ = \ e_q^{(k)} \hspace{2cm} \ \lstkl \  = w^{(kl)} \hspace{2cm}  \lstkk = w^{(kk)}
\]
and the product is given by identifying vertices. A general element in $\Mog(3)$ is depicted in Figure \ref{fig:dgcaMog}.
\begin{figure}[h]
\centering
{{
\begin{tikzpicture}[baseline=-.55ex,scale=.7, every loop/.style={}]
\node (p) at (-3,0) {$w^{(12)}w^{(23)} a_2^{(1)} b_1^{(2)} w^{(33)}=$};
 \node[circle,draw,inner sep=1.5pt] (a) at (0,0) {$1$};
 \node[circle,draw,inner sep=1.5pt] (b) at (1,0) {$2$};
 \node[circle,draw,inner sep=1.5pt] (c) at (2,0) {$3$};
 \node[circle,draw,densely dotted,inner sep=.5pt] (g) at (0,1) {$a_2$};
\node[circle,draw,densely dotted,inner sep=.5pt] (h) at (1,1) {$b_1$};
\draw (a) to (b);
\draw (b) to (c);
\draw [densely dotted] (a) to (g);
\draw[densely dotted] (b) to (h);
\draw [loop] (c) to (c);
\end{tikzpicture}}}
\caption{An element in $\Mog(3)$.}\label{fig:dgcaMog}
\end{figure}
\end{rem}

\begin{rem}
There is a morphism of dg commutative algebras
\[
\Phi:\pdBVGraphsg(r)\rightarrow \Mog(r)
\]
which sends all graphs with at least one internal vertex to zero, and, under the identification of Remark \ref{rem:Mog}, is the identity on all other graphs. It follows from the work of Idrissi (\cite{Idrissi19}, Section 6) and Campos and Willwacher (\cite{CW}, Appendix A) that $\Phi$ is a quasi-isomorphism. Moreover, it is compatible with the cooperadic comodule structures which are related through the quasi-isomorphism
\[
\phi:\pdBVGraphs\xrightarrow{\sim} \bv^c.
\]
\end{rem}

\begin{rem}
Notice that the symmetric sequences underlying the cooperadic comodule $\Mog$ may be identified with
\[
\Mog\cong \Com^c\circ H\circ \Lie^c\{-1\}\circ \K[\Delta]/(\Delta^2).
\]
where $H$ and $\K[\Delta]/(\Delta^2)$, the algebra of dual numbers, are viewed as symmetric sequences concentrated in arity one. 
\end{rem}

\begin{rem}
Notice that both $\Mog(r)$ and $\bv^c(r)$ are of finite type for all $r\geq 1$. The filtrations on $\Mog$ and $\bv^c$ which are compatible under the morphisms $(\Phi,\phi)$ are defined by setting all $w^{(ij)}$ to have weight $2$, while assigning the generators $a_l^{(i)}$, $b_l^{(i)}$, $\nu^{(i)}$ to have weight equal to their degree (i.e. weight $1$,$1$ and $2$).
\end{rem}

\subsection{The unframed case for $g=1$}
Let $g=1$ and consider the ideal $I_1(r):=(w^{(11)},\dots, w^{(nn)})$ in the dg commutative algebra $\Mo_{(1)}(r)$. Since $d w^{(ii)}=0$, it is preserved by the differential. The quotient
\[
\Mo(r)=\Mo_{(1)}(r)/ I_1(r)
\]
forms a dg commutative algebra which additionally assembles to a right cooperadic comodule $\Mo$ over the cooperad $\e_2^c$. The cooperadic comodule structure on $\Mo$ corresponds to the one of $\Mo_{(1)}$ over $\bv^c$ (omitting all $w^{(ii)}$). Furthermore, the result above can be adapted to this case. There is a quasi-isomorphism of dg commutative algebras
\[
\Phi: \pdGraphs_{(1)}(r)\rightarrow \Mo(r).
\]
which sends all graphs with at least one internal vertex to zero, and is the identity on all other graphs. Moreover, the map is compatible with the cooperadic comodule structures which are related through Kontsevich's quasi-isomorphism
\[
\phi:\pdGraphs \xrightarrow{\sim} \e_2^c.
\]

\begin{rem}
We may identify the symmetric sequences
\[
\Mo\cong \Com^c\circ H \circ \Lie^c\{-1\}.
\]
\end{rem}

\subsection{An auxiliary complex -- hairy graphs}\label{subsec:hairy}

Let $\uniGraphsg\subset (\BVGraphsg',d_s)$ be the subcollection of internally connected diagrams (i.e. the graphs remain connected if all external vertices are deleted) for which all external vertices are univalent. Notice that here we consider $\BVGraphsg'$ equipped only with the vertex splitting differential $d_s$. The hairy graph complex $\HGCg'$ is defined as the following dg vector space of invariant elements
\[
\HGCg'=\prod\limits_{r\geq 1} (\uniGraphsg(r)\otimes H^{\otimes r})^{S_r}.
\]
The elements of this space may be identified with formal series of internally connected graphs $\Gamma \in \BVGraphsg'$ whose external vertices have only one incident edge (to which we refer as the ``hairs''), are made indistinguishable and additionally carry a decoration in the cohomology $H$. In our pictures, we omit the univalent external vertices, and only draw the hairs with the decoration sitting at the ends (see Figure \ref{fig:HGC}). For reasons that will become apparent shortly, we identify $H\cong (H^*)^*$ and write $\partial_{f_q}\in (H^*)^*$ rather than $e_q\in H$ for the decorations.

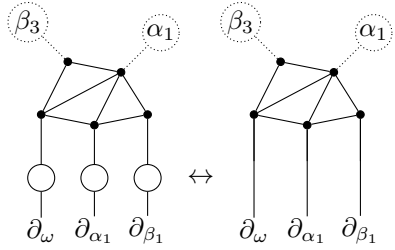
\begin{figure}[h]
\centering
\begin{tikzpicture}
[baseline=-.55ex,scale=.7, every loop/.style={}]
 \node[circle,draw,inner sep=1.5pt] (a) at (0,0) {$ \ \ $};
  \node[circle,draw,inner sep=1.5pt] (b) at (1,0) {$ \ \ $};
  \node[circle,draw,inner sep=1.5pt] (c) at (2,0) {$ \ \ $};
  \node[circle,draw,fill,inner sep=1pt] (d) at (1,1) {};
  \node[circle,draw,fill,inner sep=1pt] (e) at (2,1.2) {};
\node[circle,draw,fill,inner sep=1pt] (f) at (0,1.2) {};
\node[circle,draw,fill,inner sep=1pt] (g) at (0.5,2.2) {};
\node[circle,draw,fill,inner sep=1pt] (h) at (1.5,2) {};
\draw (b) edge[] (d);
\draw (c) edge[] (e);
\draw (e) edge[] (d);
\draw (f) edge[] (d);
\draw (f) edge[] (g);
\draw (f) edge[] (h);
\draw (g) edge[] (h);
\draw (h) edge[] (d);
\draw (h) edge[] (e);
\draw (f) edge[] (a);
\node[circle,draw,densely dotted,inner sep=.5pt] (i) at (-0.25,2.95) {$\beta_3$};
\node[circle,draw,densely dotted,inner sep=.5pt] (j) at (2.25,2.75) {$\alpha_1$};
\draw (g) edge[densely dotted] (i);
\draw (h) edge[densely dotted] (j);
\node[inner sep=.5pt] (i1) at (0,-1) {$\partial_{\omega}$};
\node[inner sep=.5pt] (i2) at (1,-1) {$\partial_{\alpha_1} $};
\node[inner sep=.5pt] (i3) at (2,-1) {$\partial_{\beta_1}$};
\draw (i1) edge[] (a);
\draw (i2) edge[] (b);
\draw (i3) edge[] (c);
\node (aa) at (3,0) {$\leftrightarrow$};
\node[circle,draw,fill,inner sep=1pt] (d2) at (5,1) {};
\node[circle,draw,fill,inner sep=1pt] (e2) at (6,1.2) {};
\node[circle,draw,fill,inner sep=1pt] (f2) at (4,1.2) {};
\node[circle,draw,fill,inner sep=1pt] (g2) at (4.5,2.2) {};
\node[circle,draw,fill,inner sep=1pt] (h2) at (5.5,2) {};
\draw (b2) edge[] (d2);
\draw (c2) edge[] (e2);
\draw (e2) edge[] (d2);
\draw (f2) edge[] (d2);
\draw (f2) edge[] (g2);
\draw (f2) edge[] (h2);
\draw (g2) edge[] (h2);
\draw (h2) edge[] (d2);
\draw (h2) edge[] (e2);
\draw (f2) edge[] (a2);
\node[circle,draw,densely dotted,inner sep=.5pt] (i24) at (-0.25+4,2.95) {$\beta_3$};
\node[circle,draw,densely dotted,inner sep=.5pt] (j24) at (2.25+4,2.75) {$\alpha_1$};
\draw (g2) edge[densely dotted] (i24);
\draw (h2) edge[densely dotted] (j24);
\node[inner sep=.5pt] (i12) at (4,-1) {$\partial_{\omega}$};
\node[inner sep=.5pt] (i22) at (5,-1) {$\partial_{\alpha_1} $};
\node[inner sep=.5pt] (i32) at (6,-1) {$\partial_{\beta_1}$};
\draw (i12) edge[] (f2);
\draw (i22) edge[] (d2);
\draw (i32) edge[] (e2);
\end{tikzpicture}
\caption{We omit the univalent external vertices as above. Still the edges connecting vertices to the decorations count towards the degree of a graph, just as it would in the diagram on the left.}\label{fig:HGC}
\end{figure}

The differential is induced by the one of $(\BVGraphsg',d_s)$, i.e. it is given by splitting internal vertices, with no restriction on the valence. A diagram with $k$ edges, $N$ internal vertices, $m$ decorations $\{f_{i_1},\dots,f_{i_m}\}$ and $\{\partial_{f_{j_1}}, \dots, \partial_{f_{j_l}}\}$ decorating the hairs is of degree
\[
-k+2N+\sum\limits_{k=1}^m |f_{i_k}|+\sum_{k=1}^l |\partial_{f_{j_k}}|.
\]
Recall that the $f_{i_k} \in \overline{H}^*$ lie in negative degrees, while the $\partial_{f_{j_k}} \in (H^*)^*\cong H$ live in non-negative degrees. In particular, we include diagrams with no internal vertices. These consist of a single edge connecting two decorations for which either one is $\overline{H}^*$ (being decorated by $1$ would violate the univalence condition of the external vertex) and the other is in $ (H^*)^*\cong H$, or both decorations are in $(H^*)^*\cong H$, i.e.

\begin{center}
\begin{tikzpicture}[baseline=-.55ex,scale=.7, every loop/.style={}]
 \node (a) at (-4,0) {for $x\neq 1^*$:};
 \node [circle,draw,densely dotted,inner sep=.5pt] (b) at (-2,1) {$x$};
 \node [circle,draw,inner sep=1.5pt] (c) at (-2,0) {$ \ \ $};
 \node [inner sep=.5pt](d) at (-2,-1) {$\partial_{y}$};
 \draw[densely dotted] (b) to[] (c);
 \draw[] (d) to[] (c);
 \node (a2) at (0,0) {$x$};
 \node (b2) at (2,0) {$\partial_{y}$};
 \draw[] (a2) to[] (b2);
 \node (aaaa) at (-1,0) {$\leftrightarrow$};
 \node (aaaaaa) at (3,0) {and};
 \node [circle,draw,inner sep=1.5pt] (aaa2) at (4,0) {$ \ \ $};
 \node [circle,draw,inner sep=1.5pt] (aaa3) at (5,0) {$ \ \ $};
 \node [inner sep=.5pt](bb2) at (4,-1) {$\partial_{x}$};
 \node [inner sep=.5pt](bb3) at (5,-1) {$\partial_{y}$};
 \draw[bend left=30] (aaa2) to[] (aaa3);
 \draw[] (aaa2) to[] (bb2);
 \draw[] (aaa3) to[] (bb3);
 \node (aaaa2) at (6,0) {$\leftrightarrow$};
  \node (a4) at (7,0) {$\partial_{x}$};
 \node (b4) at (9,0) {$\partial_{y}$};
 \draw[] (a4) to[] (b4);
\end{tikzpicture}
\end{center}

\begin{rem}
Notice that diagrams in $\HGCg'$ have at least one hair. Also, the newly added (odd) decorations are added to the linear order on the set of edges and odd decorations (attached to internal vertices), while the edge part of the hair still contributes to this ordering.
\end{rem}

%

Furthermore, $\HGCg'$ forms a dg Lie algebra. For two diagrams $\Gamma_1$, $\Gamma_2 \in \HGCg'$ consider the sum over all ways of connecting hairs of $\Gamma_1$ to the vertices of $\Gamma_2$ carrying a decoration, while multiplying with the product of the respective actions of the hair decorations in $H\cong {(H^{*})^*}$ on the decorations in $H^*$ attached to the internal vertices of $\Gamma_2$.

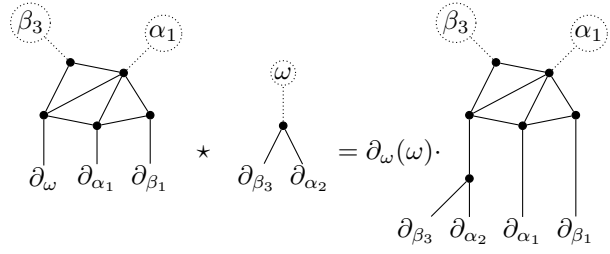
\begin{figure}[h]
\centering
\begin{tikzpicture}
[baseline=-.55ex,scale=.7, every loop/.style={}]
\node[circle,draw,fill,inner sep=1pt] (d2) at (5,1) {};
\node[circle,draw,fill,inner sep=1pt] (e2) at (6,1.2) {};
\node[circle,draw,fill,inner sep=1pt] (f2) at (4,1.2) {};
\node[circle,draw,fill,inner sep=1pt] (g2) at (4.5,2.2) {};
\node[circle,draw,fill,inner sep=1pt] (h2) at (5.5,2) {};
\draw (e2) edge[] (d2);
\draw (f2) edge[] (d2);
\draw (f2) edge[] (g2);
\draw (f2) edge[] (h2);
\draw (g2) edge[] (h2);
\draw (h2) edge[] (d2);
\draw (h2) edge[] (e2);
\node[circle,draw,densely dotted,inner sep=.5pt] (i24) at (-0.25+4,2.95) {$\beta_3$};
\node[circle,draw,densely dotted,inner sep=.5pt] (j24) at (2.25+4,2.75) {$\alpha_1$};
\draw (g2) edge[densely dotted] (i24);
\draw (h2) edge[densely dotted] (j24);
\node[inner sep=.5pt] (i12) at (4,0) {$\partial_{\omega}$};
\node[inner sep=.5pt] (i22) at (5,0) {$\partial_{\alpha_1} $};
\node[inner sep=.5pt] (i32) at (6,0) {$\partial_{\beta_1}$};
\draw (i12) edge[] (f2);
\draw (i22) edge[] (d2);
\draw (i32) edge[] (e2);
\node (a) at (7,0.5)  {$\star$};
\node[circle,draw,fill,inner sep=1pt] (b) at (8.5,1) {};
\node[circle,draw,densely dotted,inner sep=.5pt] (c) at (8.5,2) {$\omega$};
\node[inner sep=.5pt] (d) at (8,0) {$\partial_{\beta_3}$};
\node[inner sep=.5pt] (e) at (9,0) {$\partial_{\alpha_2} $};
\draw (b) edge[densely dotted] (c);
\draw (e) edge[] (b);
\draw (d) edge[] (b);
\node (a2) at (10.5,0.5)  {$=\partial_\omega(\omega)\cdot $};
\node[circle,draw,fill,inner sep=1pt] (d23) at (12+1,1) {};
\node[circle,draw,fill,inner sep=1pt] (e23) at (13+1,1.2) {};
\node[circle,draw,fill,inner sep=1pt] (f23) at (11+1,1.2) {};
\node[circle,draw,fill,inner sep=1pt] (g23) at (11.5+1,2.2) {};
\node[circle,draw,fill,inner sep=1pt] (h23) at (12.5+1,2) {};
\draw (e23) edge[] (d23);
\draw (f23) edge[] (d23);
\draw (f23) edge[] (g23);
\draw (f23) edge[] (h23);
\draw (g23) edge[] (h23);
\draw (h23) edge[] (d23);
\draw (h23) edge[] (e23);
\node[circle,draw,densely dotted,inner sep=.5pt] (i243) at (-0.25+11+1,2.95) {$\beta_3$};
\node[circle,draw,densely dotted,inner sep=.5pt] (j243) at (2.25+11+1,2.75) {$\alpha_1$};
\draw (g23) edge[densely dotted] (i243);
\draw (h23) edge[densely dotted] (j243);
\node[inner sep=.5pt] (i223) at (12+1,-1) {$\partial_{\alpha_1} $};
\node[inner sep=.5pt] (i323) at (13+1,-1) {$\partial_{\beta_1}$};
\draw (i223) edge[] (d23);
\draw (i323) edge[] (e23);
\node[circle,draw,fill,inner sep=1pt] (q) at (11+1,0) {};
\node[inner sep=.5pt] (p) at (10+1,-1) {$\partial_{\beta_3}$};
\node[inner sep=.5pt] (r) at (11+1,-1) {$\partial_{\alpha_2} $};
\draw (f23) edge[] (q);
\draw (q) edge[] (p);
\draw (q) edge[] (r);
\end{tikzpicture}
\caption{The pre-Lie product on $\HGCg'$ inducing the Lie bracket. The hair labelled by $\partial_\omega$ acts on the $\omega$ decoration, and the result is multiplied by the pairing $\partial_{\omega}(\omega)=1$ of $\partial_\omega$ on its linear dual.}\label{fig:HGCpreLie}
\end{figure}

This operation defines a pre-Lie product whose graded commutator induces the Lie bracket on $\HGCg'$. Since the differential acts on internal vertices only, it defines a derivation for the bracket. Next, consider the two elements
\begin{align*}
m_1&:=\hmcvertomega \ + \ \hmcvertone \ + \  \sum\limits_{i=1}^g \hmcvertab \ + \ \sum\limits_{i=1}^g \hmcverta \ + \ \hmcvertb \\
m_2&:=\hmclineomega \ + \ \sum\limits_{i=1}^g \hmclineab
\end{align*}

\begin{lemma}
All of $m_1$, $m_2$ and $m:=m_1+m_2$ define Maurer-Cartan elements in $\HGCg'$.
\end{lemma}

\begin{proof}
By direct computation, we find $[m_1,m_1]=2 d_{s} m_1$ and $[m_1,m_2]=0$ (because we quotient out tadpoles at internal vertices), whereas $[m_2,m_2]=d_s m_2=0$, trivially. 
\end{proof}

Consider
\[
m=m_1+m_2.
\]
We twist the differential by this Maurer-Cartan element to obtain the dg Lie algebra $\HGCg:=(\HGCg')^m$. 
%
%
Notice that $\GCg$ acts on $\HGCg'$ via the pairing operation. In particular, this action induces a Lie algebra morphism
\begin{align*}
F_1:\GCg & \rightarrow \HGCg\\
\Gamma&\mapsto \sum\limits_{x \in H } \Gamma . (\hairyactsonx) .
\end{align*}
In this instance, we allow the graph $\hairyactsonone$ to be acted on. 
%
%
Additionally, we may define a morphism of Lie algebras
\begin{align*}
F_2:\spp(H^*)&\rightarrow \HGCg\\
\sigma &\mapsto \sum\limits_{x\in H} \hairysigma.
\end{align*}
We readily check that the morphisms are compatible with the $\spp(H^*)$-action on $\GCg$ in the sense that
\[
F_1(\sigma.\Gamma)=[F_2(\sigma),F_1(\Gamma)].
\]
Thus, they assemble to form a Lie algebra morphism
\begin{align*}
F:\spp(H^*)\ltimes \GCg &\rightarrow \HGCg\\
(\sigma,\Gamma)&\mapsto F_2(\sigma)+F_1(\Gamma).
\end{align*}

Remark the following small technical lemmas.

\begin{lemma}
The Maurer-Cartan element $m_1\in \HGCg$ satisfies $m_1=F_1(z)$.
\end{lemma}

\begin{proof}
This follows via direct computation.
\end{proof}

In particular, since $F((0,z))=F_1(z)$, and $F$ is a Lie algebra morphism, it follows that for all $(\sigma,\Gamma)\in \spp(H^*)\ltimes \GCg$
\[
[m_1, F((\sigma,\Gamma))]=[F((0,z)),F((\sigma,\Gamma))]=F([(0,z),(\sigma,\Gamma)]).
\]

\begin{lemma}\label{lemma:m2F1}
For all $\Gamma\in \GCg$, we have $F_1(d_{p} \Gamma)= [m_2,F_1(\Gamma)]$. 
\end{lemma}

\begin{proof}
By definition $F_1(d_p \Gamma)$ consists of a linear combination of diagrams all of which have a single hair. It coincides with the one-hair part of $[m_2,F_1(\Gamma)]$. Additionally, $[m_2,F_1(\Gamma)]$ will produce diagrams with two hairs. This is the case when the diagrams in $m_2$ act with only one end on the decorations of $F_1(\Gamma)$ - the other end becoming a hair. However, we note that each diagram with two hairs appears twice (and with opposite signs) in the linear combination $[m_2,F_1(\Gamma)]$. Indeed, for each pair of decorations in $\Gamma$, $F_1(\Gamma)$ produces two diagrams with one hair (one for each decoration), while acting with $m_2$ turns the respective second decoration (of the chosen pair) into a hair. The identical diagrams appear with opposite signs, and thus cancel. The statement follows.
\end{proof}

\begin{lemma}\label{lemma:m2F2}
For all $\sigma \in \spp(H^*)$, we have $[m_2,F_2(\sigma)]=0$.
\end{lemma}

\begin{proof}
By bilinearity, it is enough to check the statement on basis elements. In the case of $\alpha_i\partial_{\beta_i}$ (and similarly $\beta_i\partial_{\alpha_i}$), this is clear, since
\[
[m_2, F_2(\alpha_i\partial_{\beta_i})]=[m_2, \adbi]=\pm \hmclinebb=0
\]
by symmetry. Let us perform the computation for $\sigma=\alpha_i\partial_{\alpha_j}-\beta_j\partial_{\beta_i}$. Adopt the following sign convention,
\[
[\hmcabii, \adaijorder]=\dbdaij.
\]
In that case, we also have
\[
[\hmcabjj,\bdbjiorder]=-[\hmcbajj,\bdbjiorder]=-\dadbji=\dbdaij.
\]
Clearly, thus,
\[
[m_2, F_2(\sigma)]=[m_2,\adaijorder - \bdbjiorder]=0.
\]
All other cases follow from similar computations.
\end{proof}

Lemma \ref{lemma:m2F1} and \ref{lemma:m2F2} imply that 
\[
[m_2, F((\sigma,\Gamma))]=F_1(d_{p}\Gamma)=F((0,d_{p} \Gamma))=F(d_{p}(\sigma,\Gamma)).
\]

\begin{lemma}
The map
\[
F:\spp(H^*)\ltimes\GCg\rightarrow \HGCg
\]
is a map of complexes.
\end{lemma}

\begin{proof}
This readily follows from the computation
\begin{align*}
d F((\sigma,F))&= d_s F((\sigma,\Gamma))+[m_1,F((\sigma,\Gamma))]+[m_2, F((\sigma,\Gamma))]\\
&=F(d_s(\sigma,\Gamma))+F([(0,z),(\sigma,\Gamma)])+F(d_p(\sigma,\Gamma))=F(d(\sigma,\Gamma)).
\end{align*}

\end{proof}

To prove the next proposition we adapt the proof by Fresse and Willwacher (\cite{FW20}, Proposition 2.2.9.) to our setting. In their case, they establish a link between Kontsevich's graph complex $\GC$ and its ``hairy'' variant $\HGC$, and find that, up to one class, the two complexes are quasi-isomorphic (see Remark \ref{rmk:GCHGCzero} below).

\begin{prop}\label{prop:GCHGC}
The map $F:\spp(H^*) \ltimes \GCg\rightarrow \HGCg$ defines a quasi-isomorphism in all degrees except 3 where it is injective on cohomology with a 1-dimensional cokernel.
\end{prop}

Leading up to the proof, we need a technical result within the following setup. Denote by $\HGCg^\bullet\subset \HGCg$ the subcomplex whose diagrams have at least one internal vertex. Additionally, denote by $\HGCg^0$ the subspace of $\HGCg$ spanned by diagrams with no internal vertices. It does not define a subcomplex. As graded vector spaces we have the decomposition
\[
\HGCg=\HGCg^0\oplus \HGCg^\bullet.
\]
Notice that $F_1$ maps $\GCg$ to $\HGCg^\bullet$. Next, denote by $f\HGCg$ the mapping cone of $F_1$. The underlying vector space is thus $\GCg[1]\oplus \HGCg^\bullet$ (i.e. diagrams with at least one internal vertex, but possibly without hairs), whereas the differential is induced by $F_1$. It may be thought of as the natural extension of the differential on $\HGCg$ to diagrams without hairs.

\begin{lemma}\label{lemma:fHGCgacyclic}
The complex $f\HGCg$ is acyclic.
\end{lemma}

\begin{proof}
Let us refer to the edges connecting two internal vertices as internal edges. Equip $f\HGCg$ with the complete descending filtration given by
\[
\# \text{ internal edges } \geq p.
\]
On the associated graded complex $\gr f \HGCg=\gr (\GCg[1]\oplus \HGCg^{\bullet})$ the differential acts on the $\GCg[1]$ part via $F_1$ (and thus is given by the sum over all ways of attaching a hair), while on $\gr \HGCg$ we only see the part of $[m,-]$ which is given by bracketing with
\[
 m_2=\hmclineomega \ + \ \sum\limits_{i=1}^g \hmclineab
\]
and for which only one of two ends pairs with a decoration of the element in $\gr \HGCg$ (see Figure \ref{fig:dgrHGC}).

\begin{figure}[h]
\centering
\begin{tikzpicture}
[baseline=-.55ex,scale=.7, every loop/.style={}]
\node (a) at (3,0.5)  {$\gr d$};
\node[circle,draw,fill,inner sep=1pt] (d2) at (5,1) {};
\node[circle,draw,fill,inner sep=1pt] (e2) at (6,1.2) {};
\node[circle,draw,fill,inner sep=1pt] (f2) at (4,1.2) {};
\node[circle,draw,fill,inner sep=1pt] (g2) at (4.5,2.2) {};
\node[circle,draw,fill,inner sep=1pt] (h2) at (5.5,2) {};
\draw (e2) edge[] (d2);
\draw (f2) edge[] (d2);
\draw (f2) edge[] (g2);
\draw (f2) edge[] (h2);
\draw (g2) edge[] (h2);
\draw (h2) edge[] (d2);
\draw (h2) edge[] (e2);
\node[circle,draw,densely dotted,inner sep=.5pt] (j24) at (2.25+4,2.75) {$\alpha_3$};
\draw (h2) edge[densely dotted] (j24);
\node[inner sep=.5pt] (i12) at (4,0) {$\partial_{\omega}$};
\node[inner sep=.5pt] (i22) at (5,0) {$\partial_{\alpha_1} $};
\node[inner sep=.5pt] (i32) at (6,0) {$\partial_{\beta_2}$};
\draw (i12) edge[] (f2);
\draw (i22) edge[] (d2);
\draw (i32) edge[] (e2);
\node (a) at (7,0.5)  {$=$};
\node[circle,draw,fill,inner sep=1pt] (d23) at (9,1) {};
\node[circle,draw,fill,inner sep=1pt] (e23) at (10,1.2) {};
\node[circle,draw,fill,inner sep=1pt] (f23) at (8,1.2) {};
\node[circle,draw,fill,inner sep=1pt] (g23) at (8.5,2.2) {};
\node[circle,draw,fill,inner sep=1pt] (h23) at (9.5,2) {};
\draw (e23) edge[] (d23);
\draw (f23) edge[] (d23);
\draw (f23) edge[] (g23);
\draw (f23) edge[] (h23);
\draw (g23) edge[] (h23);
\draw (h23) edge[] (d23);
\draw (h23) edge[] (e23);
\node[inner sep=.5pt] (i123) at (8,0) {$\partial_{\omega}$};
\node[inner sep=.5pt] (i223) at (9,0) {$\partial_{\alpha_1} $};
\node[inner sep=.5pt] (i323) at (10,0) {$\partial_{\beta_2}$};
\node[inner sep=.5pt] (q) at (11,0) {$\partial_{\beta_3}$};
\draw (i123) edge[] (f23);
\draw (i223) edge[] (d23);
\draw (i323) edge[] (e23);
\draw (q) edge[in=60,out=90] (h23);
\end{tikzpicture}
\caption{On $\gr f\HGCg$, we only see the part of $[m_2,-]$ for which only one of the two ends pairs with a decoration of the element in $\gr f\HGCg$.}\label{fig:dgrHGC}
\end{figure}
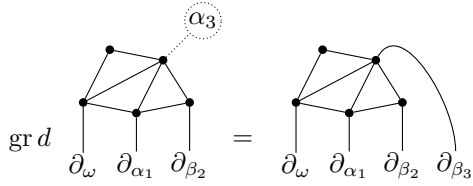

For any diagram $\Gamma$ in $f\HGCg$ let us refer to the graph in $f\HGCg$ obtained by replacing all hairs by decorations in the following way as the core $C(\Gamma)$ of $\Gamma$: If a hair is decorated by $\partial_x$, replace it with the Poincar\'e dual decoration $x^*$ of $x$ at the corresponding vertex (see Figure \ref{fig:coreHGC}).

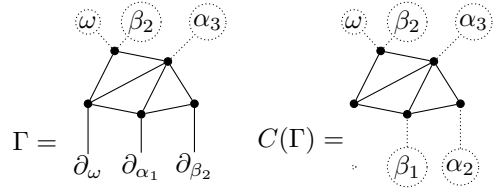
\begin{figure}[h]
\centering
\begin{tikzpicture}
[baseline=-.55ex,scale=.7, every loop/.style={}]
\node (a) at (3,0.5)  {$\Gamma=$};
\node[circle,draw,fill,inner sep=1pt] (d2) at (5,1) {};
\node[circle,draw,fill,inner sep=1pt] (e2) at (6,1.2) {};
\node[circle,draw,fill,inner sep=1pt] (f2) at (4,1.2) {};
\node[circle,draw,fill,inner sep=1pt] (g2) at (4.5,2.2) {};
\node[circle,draw,fill,inner sep=1pt] (h2) at (5.5,2) {};
\draw (e2) edge[] (d2);
\draw (f2) edge[] (d2);
\draw (f2) edge[] (g2);
\draw (f2) edge[] (h2);
\draw (g2) edge[] (h2);
\draw (h2) edge[] (d2);
\draw (h2) edge[] (e2);
\node[circle,draw,densely dotted,inner sep=.5pt] (j24) at (2.25+4,2.75) {$\alpha_3$};
\node[circle,draw,densely dotted,inner sep=.5pt] (j24a) at (5,2.75) {$\beta_2$};
\node[circle,draw,densely dotted,inner sep=.5pt] (j24b) at (4,2.75) {$\omega$};
\draw (h2) edge[densely dotted] (j24);
\draw (g2) edge[densely dotted] (j24a);
\draw (g2) edge[densely dotted] (j24b);
\node[inner sep=.5pt] (i12) at (4,0) {$\partial_{\omega}$};
\node[inner sep=.5pt] (i22) at (5,0) {$\partial_{\alpha_1} $};
\node[inner sep=.5pt] (i32) at (6,0) {$\partial_{\beta_2}$};
\draw (i12) edge[] (f2);
\draw (i22) edge[] (d2);
\draw (i32) edge[] (e2);
\node (a) at (8,0.5)  {$C(\Gamma)=$};
\node[circle,draw,fill,inner sep=1pt] (d23) at (10,1) {};
\node[circle,draw,fill,inner sep=1pt] (e23) at (11,1.2) {};
\node[circle,draw,fill,inner sep=1pt] (f23) at (9,1.2) {};
\node[circle,draw,fill,inner sep=1pt] (g23) at (9.5,2.2) {};
\node[circle,draw,fill,inner sep=1pt] (h23) at (10.5,2) {};
\draw (e23) edge[] (d23);
\draw (f23) edge[] (d23);
\draw (f23) edge[] (g23);
\draw (f23) edge[] (h23);
\draw (g23) edge[] (h23);
\draw (h23) edge[] (d23);
\draw (h23) edge[] (e23);
\node[circle,draw,densely dotted,inner sep=.5pt] (j24a) at (2.25+4+5,2.75) {$\alpha_3$};
\node[circle,draw,densely dotted,inner sep=.5pt] (j24aa) at (5+5,2.75) {$\beta_2$};
\node[circle,draw,densely dotted,inner sep=.5pt] (j24ba) at (4+5,2.75) {$\omega$};
\draw (h23) edge[densely dotted] (j24a);
\draw (g23) edge[densely dotted] (j24aa);
\draw (g23) edge[densely dotted] (j24ba);
\node[circle,draw,densely dotted,inner sep=.5pt] (i123a) at (9,0) {};
\node[circle,draw,densely dotted,inner sep=.5pt] (i223a) at (10,0) {$\beta_1 $};
\node[circle,draw,densely dotted,inner sep=.5pt] (i323a) at (11,0) {$\alpha_2$};
\draw (i223a) edge[densely dotted] (d23);
\draw (i323a) edge[densely dotted] (e23);
\end{tikzpicture}
\caption{The core of a graph $\Gamma$.}\label{fig:coreHGC}
\end{figure}

Notice that here we need to allow bi- and univalent vertices, since hairs decorated by $\partial_\omega$ will produce decorations labeled by $1$ which do not contribute to the valence. Notice that constructing the core from a graph preserves the number of internal edges, and thus it descends to the associated graded complex. The differential $\gr d$ on $\gr f\HGCg$ does not change the core of a graph (that is, $\Gamma$ and $\gr d (\Gamma)$ have the same core). Therefore, $\gr f \HGCg$ splits into a direct sum of subcomplexes $V_{C(\Gamma)}$, one for each core. Next, assume $\Gamma \in \gr f \HGCg$ has $l$ vertices. We may identify the complex $V_{C(\Gamma)}$ with
\[
\big(\bigoplus\limits_{x\in H_1} \K[x]/(x^2) \oplus \K [1] \oplus \K [\omega] \oplus \bigoplus\limits_{x\in H_1} \K[\partial_x]/(\partial_x^2)\oplus \K [\partial_1]/(\partial_1^2)\oplus \K[\partial_\omega]/(\partial_\omega^2) \big)^{\otimes l}.
\]
The differential acts as a derivation, and on basis vectors i is given by
\begin{align*}
x&\mapsto \partial_{x^*} \text{ for } x\in H^1 \\
1&\mapsto \partial_\omega\\
\omega&\mapsto \partial_1
\end{align*}
and the rest to zero. The morphism mapping
\begin{align*}
\partial_x&\mapsto x^* \text{ for } x \in H^1 \\
\partial_1&\mapsto \omega\\
\partial_\omega&\mapsto 1
\end{align*}
defines a homotopy for (a scalar multiple - the number of hairs plus the number of vertices plus the number of decorations in $\overline H^*$ - of) the identity. Thus, $V_{C(\Gamma)}$ is acyclic, and since as a complex
\[
\gr f \HGCg= \bigoplus_{\Gamma} V_{C(\Gamma)}
\]
the complex $\gr f\HGCg$ is acyclic, from which the statement follows.
\end{proof}
%
%

\begin{proof}[Proof of Proposition \ref{prop:GCHGC}]
Filter $\spp(H^*)\ltimes \GCg$ by
\[
\# \text{ edges } + \#\text{ vertices in the }\GCg \text{ part }\geq p 
\]
and $\HGCg$ by
\[
\# \text{ internal edges }+ \# \text{ internal vertices } \geq p.
\]
The morphism is compatible with the respective complete descending filtrations. On $\gr(\spp (H^*)\ltimes \GCg)$ the differential vanishes, whereas on $\gr\HGCg$ we only see the part of $[m,-]$ which is given by bracketing with
\[
 m_2=\hmclineomega \ + \ \sum\limits_{i=1}^g \hmclineab
\]
and for which only one of two ends pairs with a decoration of the element in $\HGCg$ (see Figure \ref{fig:dgrHGC}). Moreover, the associated graded complex $\gr \HGCg$ splits into the direct sum of subcomplexes
\[
\gr\HGCg\cong \gr \HGCg^0 \oplus \gr \HGCg^\bullet.
\]
Recall that $\HGCg^0$ is spanned by diagrams with no internal vertices, while $\HGCg^\bullet$ is generated by diagrams with at least one internal vertex. On the level of the associated graded these are indeed subcomplexes since the differential $\gr d$ preserves the number of vertices. Let us first consider $\gr \HGCg^0$. It contains two types of diagrams,
\begin{align*}
&\hairyactson \text{ for } x\neq 1, \text{ and }\\
&\novertexhairy.
\end{align*}
The differential maps diagrams shaped as in the first line to the ones in the second. In fact, up to one element, the differential is surjective, that is there is an exact sequence of the form
\[
0\rightarrow \ker (\gr d)\rightarrow   \langle \hairyactson \ \vert \ x\neq 1 \rangle \xtwoheadrightarrow{\gr d} \langle \novertexhairy \ \vert \ (\partial_x,\partial_y)\neq (\partial_\omega,\partial_\omega) \rangle \rightarrow \langle \hgcomegaomega\rangle \rightarrow 0.
\]
To see the surjectivity of the map above, notice that for any $x\neq \omega \in H$ , any $y\in H$, and $x^*$ the Poincar\'e dual of $x$,
\[
[m_2, \hairyactsonpoincare]=[\novertexpoincare, \hairyactsonpoincare]=\novertexhairy. 
\]
In particular, the only element $\gr d$ misses is $\hgcomegaomega$, since that would require $m_2$ to act on $\hairyactsonpoincareomega$ with $x^*=1$. Such diagrams are however not allowed. Additionally, the above implies 
\[
H(\gr \HGCg^0,\gr d )=\ker (\gr d) \oplus \K \hgcomegaomega.
\]
Moreover, for $2n=2g+2$,
\begin{align*}
\dim \left(\langle \hairyactson \ \vert \ x\neq 1 \rangle\right) =& 2n(2n-1)\\
\dim \left(\langle \novertexhairy \ \vert \ (\partial_x,\partial_y)\neq (\partial_\omega,\partial_\omega) \rangle\right)=& n(2n+1)-1-\underbrace{(2n-2)}_{\text{symmetry}}=2n^2-n+1.
\end{align*}
Notice that due to symmetry, diagrams of the form $\novertexhairyxx$ are zero when $|x|$ is odd. The exactness of the sequence above then implies that
\[
\dim \ker (\gr d)= 2n^2-n-1.
\]
At this point, we check that the $2n^2-n-1$ diagrams
\begin{align*}
&\adbi\\
&\bdai\\
&\adaij \ - \ \bdbji\\
&\adbij \ + \ \adbji\\
&\bdaij \ + \ \bdaji\\
&\odai \ + \ \bdone\\
&\odbi \ - \ \adone\\
\end{align*}
all lie in the kernel (this is precisely the computation of the proof of Lemma \ref{lemma:m2F2}). The space spanned by these diagrams corresponds precisely to the image of $\gr F\vert_{(\gr \spp(H),0)}=F_2$ restricted to $\spp(H)$. We conclude that
\[
H(\gr \HGCg^0,\gr d )\cong F_2(\spp(H))\oplus \K \hgcomegaomega .
\]
Notice that in addition $\hgcomegaomega$ lifts to a cohomology class (for instance $[m_1,\hgcomegaomega]=0$ by symmetry reasons) of degree three in $\HGCg$. 

On the other hand, since $f\HGCg$ is acyclic by Lemma \ref{lemma:fHGCgacyclic}, we have that $F_1:\GCg\rightarrow \HGCg^\bullet$ is a quasi-isomorphism. In particular, this implies that 
\[
\gr F:\gr(\spp(H^*)\ltimes \GCg)=\spp(H^*)\oplus \gr \GCg \rightarrow \gr \HGCg=\gr \HGCg^0\oplus \gr \HGCg^\bullet
\]
is a quasi-isomorphism in all degrees except degree 3. The statement of the proposition follows.
\end{proof}

\begin{rem}
The hairy graph complex $\HGCg'$ also acts on $(\BVGraphsg',d_s)$. The action generalizes the pre-Lie product on $\HGCg'$, and for $\gamma\in \HGCg$ and $\Gamma\in \BVGraphsg'$ it is given by the sum over all possible ways of connecting the hairs $\gamma$ to the vertices of $\Gamma$ carrying a decoration, while multiplying with the product of the respective actions of the hair decorations in $H\cong {(H^{*})^*}$ on the decorations in $H^*$ attached to the internal vertices of $\Gamma$. Notice that here we require that all hairs of $\gamma$ are connected to vertices of $\Gamma$. If this cannot be done, the action is zero. We check that it is compatible with the dg Lie algebra structure on $\HGCg'$. At this point, we may twist $(\BVGraphsg',d_s)$ by the Maurer-Cartan element $m\in \HGCg'$. We obtain an action of the twisted dg Lie algebra $\HGCg$ on ${(\BVGraphsg',d_s)}^m$. In particular, notice that twisting by $m$ reproduces precisely the pairing differential $d_p$ as well as the $z$-twist of the differential on $\BVGraphsg$. Furthermore, since the underlying spaces coincide, we conclude that $\BVGraphsg = {(\BVGraphsg',d_s)}^m$.
\end{rem}


\begin{rem}\label{rem:F}
Notice that the map $F:\spp(H^*)\ltimes \GCg \rightarrow \HGCg$ is compatible with the respective actions on $\BVGraphsg$, in the sense that, for $(\sigma,\gamma)\in \spp(H^*)\ltimes \GCg$ and $\Gamma$ in $\BVGraphsg$, we have
\[
F(\sigma,\gamma).\Gamma=(\sigma,\gamma).\Gamma.
\]
\end{rem}

\begin{rem}\label{rmk:GCHGCzero}
Both dg Lie algebras $\GCg$ and $\HGCg$ may be viewed as natural generalizations of Kontsevich's graph complex $\GC$ and its ``hairy'' version $\HGC$. Recall that elements of Kontsevich's graph complex $\GC$ are given by formal series of (isomorphism classes of) connected graphs with a single type of undistinguishable vertices of valency at least three (see Figure \ref{fig:GC}). The differential is defined by the same vertex splitting operation (on internal vertices) as in $\Graphs(r)$, i.e. when reconnecting the loose edges we require that the valency condition is respected (for a more precise definition see \cite{Willwacher15}). Moreover, and in addition to a combinatorially defined grading, $\GC$ may be equipped with a combinatorial Lie bracket. It was introduced by Kontsevich in his work on the formality conjecture \cite{Kontsevich97}. It holds numerous applications, many of which are related to Willwacher's result of computing its zero cohomology and identifying it with the Grothendieck-Teichm\"uller Lie algebra \cite{Willwacher15}. Its hairy version first appeared in \cite{AT15} in order to compute the rational homotopy of spaces of long embeddings, and within a similar setting it is used extensively to express the rational homotopy type of the mapping spaces of the little disks operads \cite{FTW17}. As already mentioned before the proof of Proposition \ref{prop:GCHGC}, by a result of Fresse and Willwacher, up to one class, $\GC$ and $\HGC$ are quasi-isomorphic \cite{FW20}.
\end{rem}

\begin{figure}[h]
\centering
\begin{tikzpicture}
[baseline=-.55ex,scale=.7, every loop/.style={}]
\node[circle,draw,fill,inner sep=1pt] (d2) at (0,0) {};
\node[circle,draw,fill,inner sep=1pt] (e2) at (1,0) {};
\node[circle,draw,fill,inner sep=1pt] (f2) at (0.31,0.95) {};
\node[circle,draw,fill,inner sep=1pt] (g2) at (-0.81,0.59) {};
\node[circle,draw,fill,inner sep=1pt] (h2) at (-0.81,-0.59) {};
\node[circle,draw,fill,inner sep=1pt] (i2) at (0.31,-0.95) {};
\draw (e2) edge[] (d2);
\draw (f2) edge[] (d2);
\draw (g2) edge[] (d2);
\draw (h2) edge[] (d2);
\draw (i2) edge[] (d2);
\draw (e2) edge[] (f2);
\draw (f2) edge[] (g2);
\draw (g2) edge[] (h2);
\draw (h2) edge[] (i2);
\draw (i2) edge[] (e2);
\node (a) at (2,0)  {$+ \ \frac52$};
\node[circle,draw,fill,inner sep=1pt] (d23) at (3,1) {};
\node[circle,draw,fill,inner sep=1pt] (e23) at (3,-1) {};
\node[circle,draw,fill,inner sep=1pt] (f23) at (5,1) {};
\node[circle,draw,fill,inner sep=1pt] (g23) at (5,-1) {};
\node[circle,draw,fill,inner sep=1pt] (h23) at (3.66,0) {};
\node[circle,draw,fill,inner sep=1pt] (i23) at (4.33,0) {};
\draw (e23) edge[] (d23);
\draw (f23) edge[] (d23);
\draw (f23) edge[] (g23);
\draw (f23) edge[] (i23);
\draw (g23) edge[] (i23);
\draw (h23) edge[] (d23);
\draw (h23) edge[] (e23);
\draw (h23) edge[] (i23);
\draw (e23) edge[] (i23);
\draw (g23) edge[] (e23);
\end{tikzpicture}
\caption{An element of Kontsevich's graph complex $\GC$.}\label{fig:GC}
\end{figure}
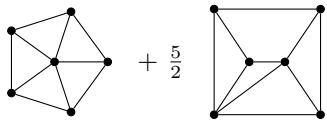

\subsection{The non-framed case for $g=1$}

In the non-framed case for $g=1$, we simply drop the condition that internal vertices need to be tadpole-free. More concretely, as a starting point we take $\uniGraphs_{(1) }\subset (\Graphs_{(1)}',d_s)$, the subcollection of internally connected graphs whose external vertices are univalent. All subsequent constructions go through analogously, and lead to the definition of the tadpole variant of $\HGC_{(1)}$ which we denote by $\HGC^{\minitadp}_{(1)}$. Ultimately, we obtain a quasi-isomorphism (except in degree three) of dg Lie algebras
\[
F: \spp(H^*)\ltimes \GC^{\minitadp}_{(1)}\rightarrow \HGC^{\minitadp}_{(1)}.
\]
which moreover is compatible with the respective actions on $\Graphs_{(1)}$.

\section{The deformation complex $\Def(\pdBVGraphsg\xrightarrow{_\Phi} \Mog)$}

Set $(\op M, \op C)=(\pdBVGraphsg,\pdBVGraphs)$ and $(\op N,\op D)=(\Mog,\bv^c)$, and consider the quasi-isomorphism of cooperadic dg Hopf comodules
\[
(\Phi:\pdBVGraphsg\rightarrow \Mog \ , \ \phi: \pdBVGraphs\rightarrow \bv^c)
\]
from Remark \ref{rem:Mog}. Our first main result establishes a link between the deformation complex 
\[
\Def(\pdBVGraphsg \xrightarrow \Phi \Mog)
\]
and the graph complex $\HGCg$. By Proposition \ref{prop:Lieaction}, the action of $\HGCg$ on $\BVGraphsg$ induces a morphism of complexes
\begin{align*}
\widetilde \Psi: \HGCg[-1] &\rightarrow \Def(\pdBVGraphsg \xrightarrow \Phi \Mog)\\
s^{-1}\Gamma &\mapsto  \widetilde \Phi\circ D_\Gamma
\end{align*}
where $\widetilde \Phi$ is as in Proposition \ref{prop:Lieaction}, namely the composition,
\[
\widetilde\Phi:\Harr \pdBVGraphsg \xrightarrow {s^{-1}\circ\pi} \pdBVGraphsg \xrightarrow \Phi \Mog \xrightarrow \iota \Omega \Mog.
\]

In this section, we proof the following result.

\begin{thm}\label{thm:defHGC}
The action of $\HGCg$ on $\pdBVGraphsg$ induces a quasi-isomorphism
\begin{align*}
\widetilde \Psi:\HGCg[-1] &\rightarrow \Def(\pdBVGraphsg \xrightarrow \Phi \Mog).
\end{align*}
\end{thm} 

Notice that similarly, the action of $\spp(H^*)\ltimes \GCg$ on $\pdBVGraphsg$ induces the morphism of complexes
\begin{align*}
\Psi:( \spp(H^*) \ltimes \GCg)[-1] &\rightarrow \Def(\pdBVGraphsg \xrightarrow \Phi \Mog)\\
s^{-1}(\sigma,\gamma) &\mapsto  \widetilde \Phi\circ D_{(\sigma,\gamma)}
\end{align*}
which, in addition yields the following commutative diagram
\begin{center}
 \begin{tikzcd}
  \HGCg[-1] \arrow[r,"\widetilde \Psi"] & \Def(\pdBVGraphsg\xrightarrow \Phi \Mog)  \\
   (\spp(H^*)\ltimes\GCg)[-1] \arrow[ru, "\Psi"] \arrow[u,"F"]& 
 \end{tikzcd}
 \end{center}

Admitting Theorem \ref{thm:defHGC} thus provides us with the following result.

\begin{thm}\label{thm:defGC}
The morphism 
\[
\Psi: (\spp(H^*)\ltimes \GCg)[-1]\rightarrow \Def(\pdBVGraphsg \xrightarrow \Phi \Mog)
\]
induces a quasi-isomorphism in all degrees except in degree $4$. In degree $4$, the map on cohomology $H^4((\spp(H^*)\ltimes \GCg)[-1])\rightarrow H^4( \Def(\pdBVGraphsg\xrightarrow \Phi \Mog))$ is injective with one-dimensional cokernel.
\end{thm}

The proof of Theorem \ref{thm:defHGC} will require a few technical results on the deformation complex. We aim to replace the deformation complex by a simpler quasi-isomorphic complex. The strategy is to use a ``Koszul'' operadic module instead of the cobar construction $\Omega \Mog$. 

\begin{rem}\label{rmk:HGCFW}
Notice that Theorem \ref{thm:defHGC} and Theorem \ref{thm:defGC} are further adaptations of the work \cite{FW20} of Fresse and Willwacher to the case of cooperadic comodules. They define a similar quasi-isomorphism from the hairy variant of Kontsevich's graph complex to the deformation complex from Section \ref{sec:defforcoops} for $\op C=\pdGraphs$ and $\op D=\e_2^c$, i.e.
\[
\HGC[-1]\xrightarrow \sim \Def(\pdGraphs\xrightarrow \phi \e_2^c)
\]
thus also relating $\GC$ to the cooperadic deformation complex. Moreover, a similar result for $\op C=\pdBVGraphs$ and $\op D=\bv^c$ is discussed in \cite{Brun20}.
\end{rem}

\subsection{The Koszul operadic module $\Mogg$}\label{sec:OmegaMo}

Parallel to the cobar construction for $\Mog$ we consider the free right operadic $\bvk$-module $\Mogg$
\[
\Mogg:=\Mog\circ \bvk
\]
equipped with the differential
\[
d_{\Mogg}=d_{\Mog}\circ' \id_{ \bvk}+\id_{\Mog }\circ' d_{\bvk}+(\id_{\Mog}\circ  \gamma_{\bvk})((\id_{\Mog}\circ (\kappa\circ \iota))\circ \Delta_{\Mog} )\circ \id_{\bvk})
\]
where in this case $\gamma_{\bvk}$ denotes the operadic composition in $\bvk$. One way to describe elements of $\Mogg=\Mog \circ \bvk$ is as in Figure \ref{fig:diagramMog}. Each vertex of the $\Mog$ part contains an element in $\bvk\cong \K[\delta^*] \circ \Com\{-2\} \circ \Lie \{-1\}$, i.e. a product of Lie trees and powers of $\delta^*$. We refer to the vertices of the $\Mog$ part as \emph{clusters} and depict them as dotted half-circles.

\begin{figure}[h]
\centering
\includegraphics[width=0.6\textwidth]{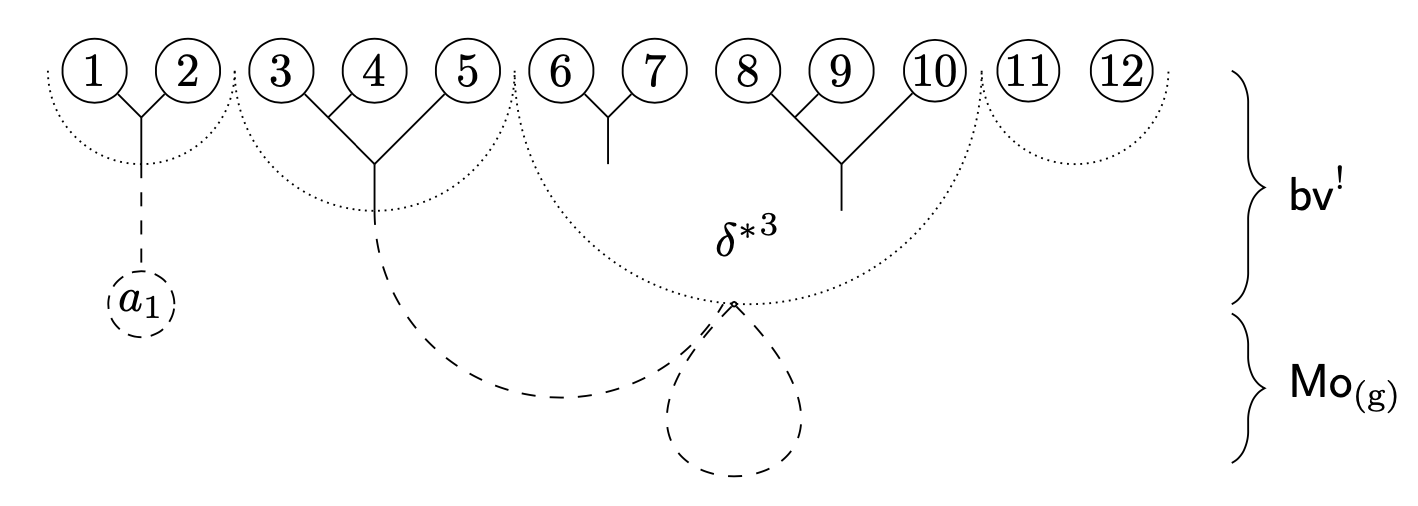}
\caption{An element of $\Mogg(12)$. We refer to the dotted half-circles as clusters. They correspond to the external vertices of the $\Mog$ part, and each contains an element in $\bvk$.}\label{fig:diagramMog}
\end{figure}

On such diagrams, the differential on $\Mogg=\Mog\circ \bvk$ is given (using a slightly abbreviated notation) by
\[
d_{\Mogg}=d_{\Mog}+d_{\bvk}+d_\mathrm{coop}
\]
where
\begin{itemize}
\item the internal differential on $\Mog$, $d_{\Mog}$, replaces dashed edges (and tadpoles) by decorating the corresponding clusters with the diagonal class.
\item the internal differential on $\bvk$, $d_{\bvk}$, takes the bracket between Lie words within a cluster and raises the power of $\delta^*$ by one.
\end{itemize}
and 
\[
d_\mathrm{coop}=(\id_{\Mog }\circ  \gamma_{\bvk})((\id_{\Mog }\circ \kappa \circ \iota)\circ \Delta_{\Mog} )\circ \id_{\bvk})
\]
defined through the cooperadic $\bv^c$ comodule structure of $\Mog$ and the morphism $\kappa$, produces the three terms for which
\begin{itemize}
\item a dashed edge is removed and the clusters it connected are fused into one larger cluster,
\item a dashed tadpole is removed and the power of $\delta^*$ in the corresponding cluster is raised by one,
\item two clusters which are not connected by a dashed edge are fused together to form one cluster, and we sum over all ways of taking Lie brackets between the Lie words that were previously in distinct clusters - this operation is sometimes denoted by $d_{\mathrm{fuse}}$.
\end{itemize}

\begin{figure}[h]
\centering
\includegraphics[width=0.6\textwidth]{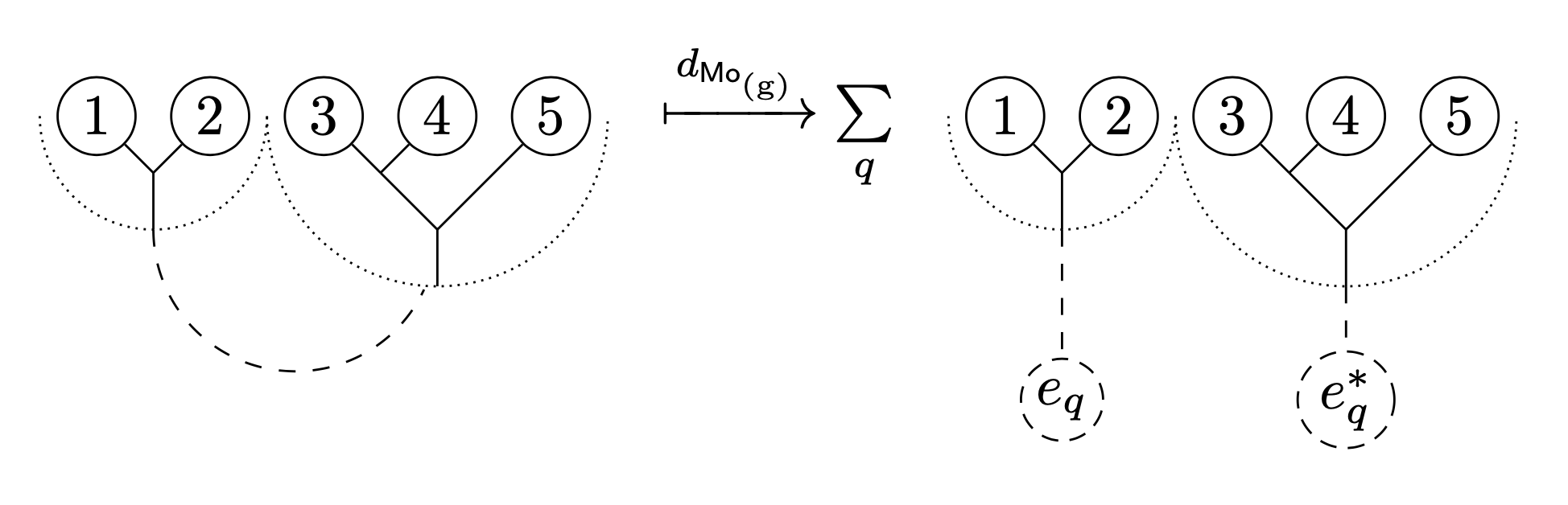}
\caption{Schematic description of the internal differential on $\Mog$.}\label{fig:diffMog}
\end{figure}

\begin{figure}[h]
\centering
\includegraphics[width=0.6\textwidth]{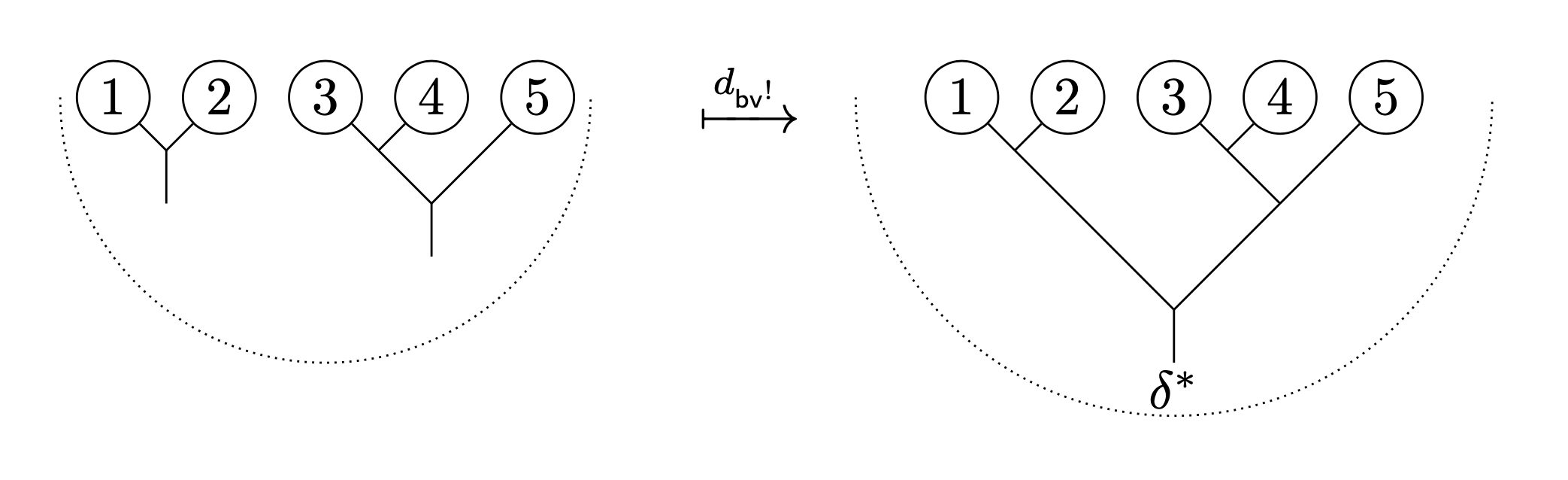}
\caption{Schematic description of the internal differential on $\bvk$.}\label{fig:diffbvk}
\end{figure}

\begin{figure}[h]
\centering
\includegraphics[width=0.6\textwidth]{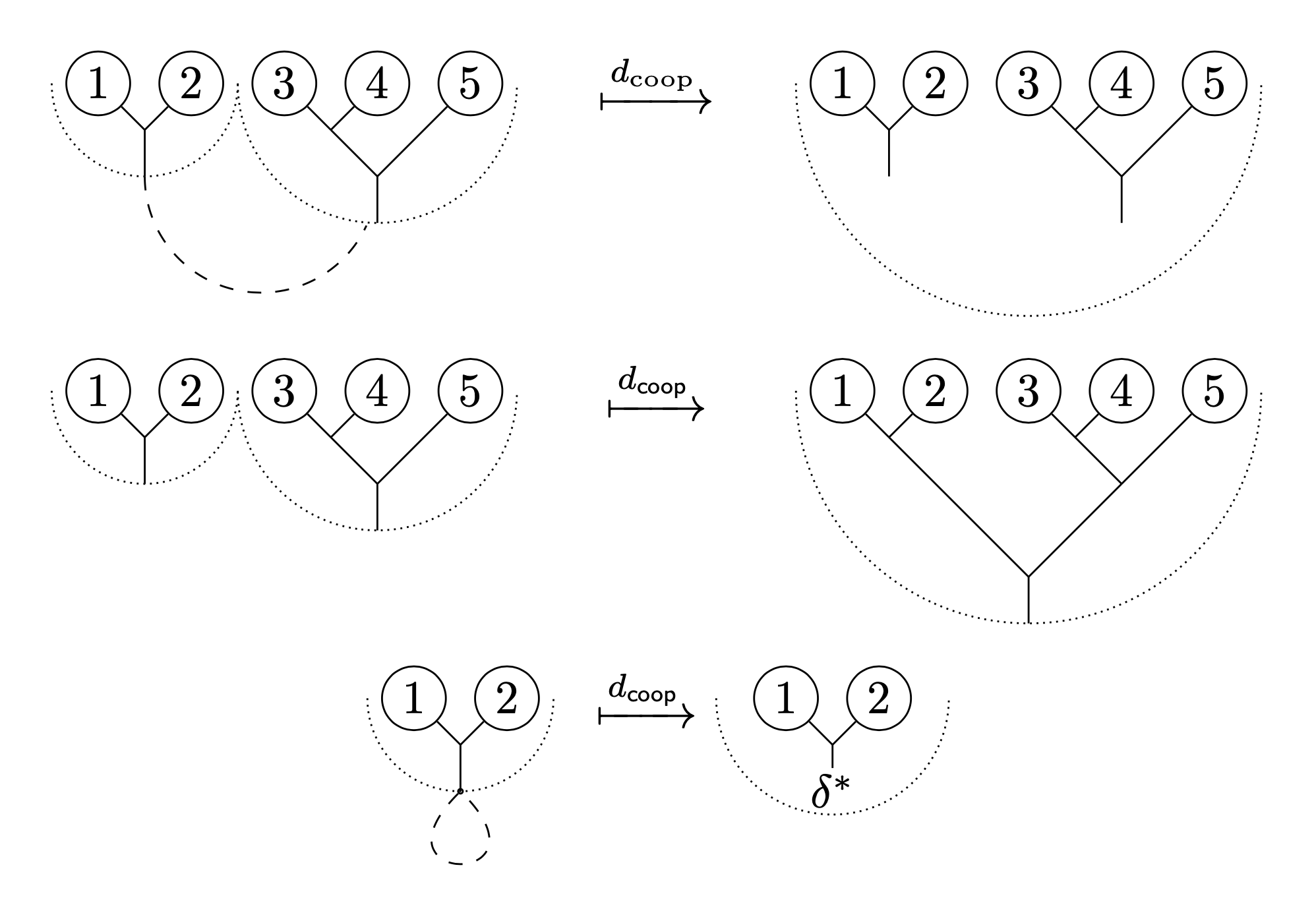}
\caption{Schematic description of the differential induced by the cooperadic comodule structure -  the operation on the second line is sometimes denoted by $d_\mathrm{fuse}$. }\label{fig:diffcoop}
\end{figure}

The induced morphism
\[
\tilde\kappa:=\id_{\Mog}\circ \kappa: \Omega\Mog=\Mog\circ \Omega\bv^c\rightarrow \Mogg=\Mog\circ \bvk
\]
is a map of complexes.

\begin{lemma}\label{lemma:kappa}
The morphism $\tilde\kappa$ describes a quasi-isomorphism $\Omega\Mog\xrightarrow{\sim} \Mogg$.
\end{lemma}

\begin{proof}
Since the differential preserves the total arity, it suffices to check that for any $r\geq 1$, the map
\[
\tilde \kappa(r)=(\id_{\Mog}\circ \kappa)(r):\Omega\Mog(r) \rightarrow \Mogg(r)
\]
is a quasi-isomorphism. To see this, filter both sides by the arity of the $\Mog$ part, plus the number of tadpoles in $\Mog$, i.e.
\begin{align*}
\mF^p (\Omega \Mog(r))=&\{\gamma\circ b \in (\Mog \circ  \Omega\bv^c )(r) \ \vert \ \mathrm{arity}(\gamma)+ \# \text{ of tadpoles in } \gamma \leq p\}\\
\mF^p (\Mogg(r))=&\{\gamma\circ b \in (\Mog\circ\bvk)(r) \ \vert \ \mathrm{arity}(\gamma)+ \# \text{ of tadpoles in } \gamma \leq p\}.
\end{align*}
These define bounded ascending filtrations. The differentials on the associated graded complexes are given by
\begin{align*}
&\gr (d_{\Mog}) \circ' \id_{\Omega \bv^c} +\id_{\Mog} \circ' d_{\Omega\bv^c}\\
&\gr (d_{\Mog}) \circ' \id_{\bvk} + \id_{\Mog}\circ' d_{\bvk}.
\end{align*}
Here $\gr(d_{\Mog})$ describes the part of the differential which does not operate on tadpoles (i.e. it splits edges of the form $w^{ij}$ for $i\neq j$, and replaces it by the diagonal class). From Section \ref{sec:BV} we know that the associated graded complexes are quasi-isomorphic. The statement follows.
\end{proof}

Notice furthermore that the right action of $\bv^c(r)$ on $(\Omega \bv^c)(r)$ and $\bvk(r)$ for any $r\geq 1$ induces a right action of $\Mog(r)$ on both $(\Omega \Mog)(r)$ and $\Mogg(r)$, which is compatible with the morphism $\tilde \kappa(r)$. For instance, an element in $(\Omega\Mog)(r)$ is of the form
\[
\Mog(s)\otimes((\Omega\bv^c)(r_1)\otimes \dots\otimes (\Omega\bv^c)(r_s))
\]
with $\sum_i r_i=r$. In this case, the action of $\Mog(r)$ is obtained by first applying the cooperadic comodule structure
\[
\Mog(r)\xrightarrow {\Delta_{\Mog}} \Mog(s)\otimes (\bv^c(r_1)\otimes \dots \otimes \bv^c(r_s))
\]
before using the algebra structure on $\Mog(s)$ and the respective actions of $\bv^c(r_i)$ on $(\Omega\bv^c)(r_i)$, tensor-wise. The action of $\Mog(r)$ on $\Mogg(r)$ is defined analogously.

\subsubsection{The non-framed case for $g=1$}

In the non-framed case for $g=1$, we find in an analogous way that $\Omega\Mo=\Mo\circ \Omega \e_2^c$ is quasi-isomorphic to $\Mo^{!}:=\Mo\circ \e_2^{!}=\Mo \circ \e_2\{-2\}$. Moreover, the arity-wise right actions of $\e_2^c$ on $\Omega\e_2^c$ and $\e_2^{!}$ induce compatible arity-wise right actions of $\e_2^c$ on both $\Omega \Mo$ and $\Mo^{!}$.

\subsection{``Koszul'' deformation complex}

Notice that the morphism $\tilde \kappa:\Omega \Mog\rightarrow \Mogg$ induces a surjective morphism of complexes
\begin{align*}
\tilde \kappa_*:(\Hom_S(\Harr\pdBVGraphsg,\Omega \Mog),d)&\rightarrow (\Hom_S(\Harr\pdBVGraphsg, \Mogg),d)\\
f&\mapsto \tilde \kappa \circ f
\end{align*}
Moreover, observe that the subspace $\Hom_S(\Harr\pdBVGraphsg, \ker\tilde \kappa)$ forms a subcomplex of the deformation complex, i.e. it is preserved by the differential $d+\partial_\mathsf{op}+\partial_{\mathsf{mod}}$. To see this, notice that it follows from Section \ref{sec:BV} that the kernel of the quotient map $\tilde \kappa: \Omega \Mog \rightarrow \Mogg$ is given by
\[
\ker \tilde \kappa=\mathcal{T}^{(\geq 4 )} \otimes (1\otimes d_{\Omega \bv^c}) ({\mathcal{T}^{(\leq 2,3)}})
\]
where $\mathcal{T}^{(\geq 4)}\subset \Mog\circ \Omega\bv^c$ is spanned by rooted trees for which at least one of the vertices labelled by elements in $\bv^c[-1]$ in the $\Omega\bv^c$ part is of valence at least four (i.e. the element in $ \bv^c[-1]$ is of arity at least three), and $\mathcal{T}^{(\leq 2,3)} \subset \Mog\circ \Omega \bv^c$ consists of linear combinations of rooted trees whose vertices which are labelled by elements in $ \bv^c[-1]$ are of arity less or equal to two, except for one, which is of arity three. Next, notice that since the action of $\Mog$ on $\Omega \Mog$ preserves the arity of the respective vertices in the tree representation of $\Omega \Mog$ and it is compatible with the differential of the cobar construction, $\partial_{\mathsf{mod}}$ preserves $\Hom_S(\Harr\pdBVGraphsg, \ker\tilde \kappa)$. Finally, observe that $\partial_\mathsf{op}$ acts on $\Hom_S(\Harr\pdBVGraphsg, \ker\tilde \kappa)$ through the operadic module structure of $\Omega\Mog$ over $\Omega \bv^c$, and as such it is given by grafting of trees which thus preserves the defining conditions of $\ker \tilde \kappa$ (i.e. attaching additional vertices at the leafs does not decrease the number of $\geq 4$-valent vertices, and if there a unique 4-valent vertex on which the differential $d_{\Omega \bv^c}$ was applied, attaching tri- and bivalent vertices will preserve this property).

It follows that we may endow $\Hom_S(\Harr(\pdBVGraphsg,\Mogg)$ with the induced differential $D_{\tilde \kappa}=d+\partial'_\mathsf{op}+\partial'_\mathsf{mod}$, defined through the equation $D_{\tilde \kappa}(\tilde \kappa \circ f)=\tilde \kappa \circ (Df)$. We adopt the notation
\[
\Def_\tk (\pdBVGraphsg\xrightarrow \Phi \Mog):=(\Hom_S(\Harr(\pdBVGraphsg,\Mogg), D_\tk).
\]
In the next section, we shall proof the following result.
\begin{prop}\label{prop:koszuldef}
The morphism $\tk_*$ induces a quasi-isomorphism
\[
\Def (\pdBVGraphsg\xrightarrow \Phi \Mog)\rightarrow \Def_\tk (\pdBVGraphsg\xrightarrow \Phi \Mog)
\]
\end{prop}

\subsection{Diagrammatic description of $\Def_\tk (\pdBVGraphsg\xrightarrow \Phi \Mog)$}\label{section:diagdescr}

We identify 
\[
\Def_\tk (\pdBVGraphsg\xrightarrow \Phi \Mog)\cong \Mogg\hotimes_S \Harr \BVGraphsg
\]
and represent elements of $\Mogg\hotimes_S \Harr\BVGraphsg$ by diagrams as in Figure \ref{fig:diagram1} where the upper part describes the $\Harr\BVGraphsg$ part.

\begin{figure}[h]
\centering
\includegraphics[width=0.6\textwidth]{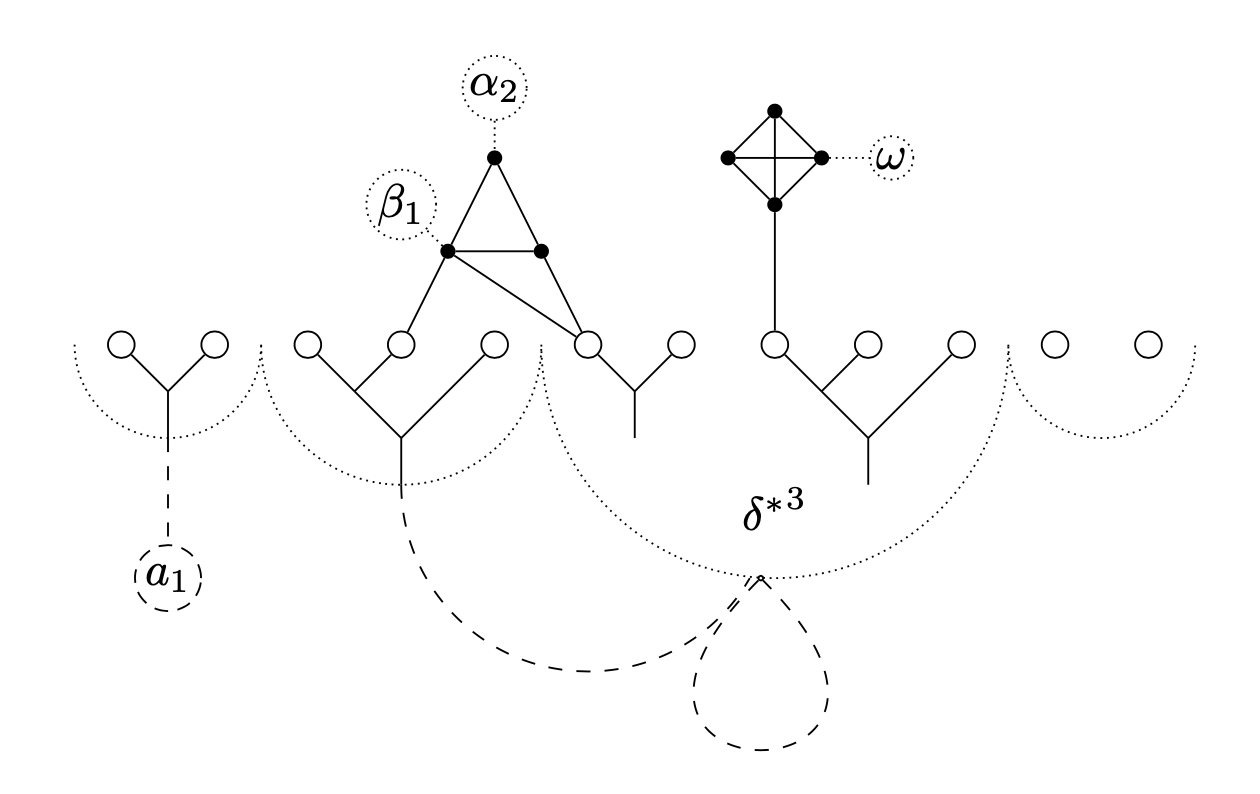}
\caption{An element of $\Mogg(12)\hotimes_{S_{12}} \BVGraphsg(12)[-1]\subset \Mogg(12)\hotimes_{S_{12}} \Harr (\BVGraphsg(12)) $.}\label{fig:diagram1}
\end{figure}

Let us sketch the different terms of the induced differential on $\Mogg \hotimes_S \Harr\BVGraphsg$. For this, observe that $\partial_\mathsf{op}$ is defined through the action of the element whose leading terms (i.e. up to arity two) are
\[
m_\mathsf{op}=s^{-1} \Delta^c \otimes \deltabvgraphs+s^{-1}\lambda^c\otimes \brabvgraphs +s^{-1} \mu^c\otimes \prodbvgraphs  + \dots \in \Omega \bv^c\hotimes_S \BVGraphs
\]
i.e. we find for $y\otimes x\in \Omega\Mog\hotimes_S \Harr \BVGraphsg$
\[
m_\mathsf{op} \cdot (y\otimes x)= \sum_i \pm y\circ_i s^{-1}\lambda^c\otimes (x \circ_i \brabvgraphs) +\pm s^{-1}y\circ_i \mu^c\otimes (x\circ_i \prodbvgraphs) +\pm s^{-1} y \circ_i \Delta^c \otimes  (x  \circ_i \deltabvgraphs)+\dots.
\]
Thus, for a representative in the quotient $\tilde \kappa (y) \otimes x \in \Mogg \hotimes_S \Harr \BVGraphsg$ the induced differential $\partial'_\mathsf{op}$ equals
\begin{align*}
&\partial'_\mathsf{op} (\tilde \kappa (y) \otimes x)=(\tilde \kappa\otimes 1) (m_\mathsf{op} \cdot (y \otimes x)) \\
&=  \sum_i \pm \tilde\kappa (y\circ_i s^{-1}\lambda^c) \otimes (x \circ_i \brabvgraphs) +\pm \tilde \kappa( y\circ_i s^{-1}\mu^c) \otimes (x\circ_i \prodbvgraphs) +\pm \tilde \kappa (y \circ_i s^{-1}\Delta^c) \otimes (x  \circ_i \deltabvgraphs)+\dots \\
&=\sum_i \pm \tilde\kappa (y)\circ_i \kappa( s^{-1}\lambda^c) \otimes (x \circ_i \brabvgraphs) +\pm \tilde \kappa( y) \circ_i \kappa (s^{-1} \mu^c) \otimes (x\circ_i \prodbvgraphs) +\pm \tilde \kappa (y) \circ_i \kappa (s^{-1}  \Delta^c) \otimes  (x  \circ_i \deltabvgraphs)+0\\
&=\sum_i \pm \tilde\kappa (y)\circ_i  s^{-2}\mu \otimes (x \circ_i \brabvgraphs) +\pm \tilde \kappa( y) \circ_i s^{-2} \lambda \otimes (x\circ_i \prodbvgraphs) +\pm \tilde \kappa (y) \circ_i \delta^* \otimes  (x  \circ_i \deltabvgraphs)
\end{align*}
since, in particular $\kappa$ vanishes on terms of arity $\geq 3$. The first terms are thus given by inserting an edge at the $i$-the external vertex of the $\Harr \BVGraphsg$ part through the bimodule structure of $\Harr \BVGraphsg$ over $\BVGraphs$, while composing with a product on the $\Mogg$ side (recall that $\Mogg$ is an operadic module over $\bvk$). In the second case, we insert an empty graph at the $i$-th external vertex, which amounts to splitting the external vertex, and compose with a Lie bracket in the respective cluster of $\Mogg$. For the arity one terms a tadpole is added to the $i$-th external vertex, while the corresponding $\bvk$ cluster is multiplied by $\delta^*$.

\begin{rem}
We use the reduced version here and omit the $s^{-1}\id \otimes \id$ term (and thus do not write $\partial_\mathsf{op}$), since by Proposition \ref{prop:subcomplex} its action does not contribute to the total differential.
\end{rem}

On the other hand, recall that $\partial_\mathsf{mod}$ is defined through the action of the Maurer-Cartan element describing the composition
\[
\Phi\circ s^{-1} \circ \pi :\Harr \pdBVGraphsg(r) \xrightarrow \pi \pdBVGraphsg (r) [1] \xrightarrow {s^{-1}} \pdBVGraphsg(r) \xrightarrow \Phi \Mog.
\]
Observe that in this case, we may encode this morphism as a linear combination in $\Mog \hotimes_S \BVGraphsg[-1]$ of elements of the form
\begin{equation}\label{eq:examplePhimc}
\examplePhi
\end{equation}
i.e. diagrams for which
\begin{itemize}
\item the upper part is a graph without any internal vertices, and
\item the lower part consists of the identical graph (with dual decorations). Notice that here we omit the $\bvk$ clusters since the lower part lies in $\Mog$ (and not in $\Mogg$).
\end{itemize}

Recall that diagrams of this sort act on $\Mogg \hotimes_S \Harr \BVGraphsg$ via the Lie bracket on the Harrison complex and through the action of $\Mog$ on $\Mogg$ induced by $\tilde \kappa$ and the action of $\Mog$ on $\Omega \Mog$ (as described in Section \ref{sec:OmegaMo}).

\begin{proof}[Proof of Proposition \ref{prop:koszuldef}]
We show the equivalent statement that the morphism $\tilde \kappa \otimes 1 : \Omega\Mog \hotimes_S \Harr \BVGraphsg\rightarrow \Mogg \hotimes_S \Harr \BVGraphsg$ defines a quasi-isomorphism. Equip both sides with the complete descending filtrations
\begin{align*}
\mF^p &= \prod\limits_{N\geq p} \Omega\Mog(N) \hotimes_{S_N} \Harr \BVGraphsg[1](N)\\
\mF^p &= \prod\limits_{N\geq p} \Mogg(N) \hotimes_{S_N} \Harr \BVGraphsg[1](N)
\end{align*}
Furthermore, filter both associated graded complexes by
\[
\# \text{ of Lie brackets in the Harrison part } + \text{ total } \# \text{ of edges in all } \BVGraphsg \text{-expressions }\geq p.
\]
It defines a complete descending filtration. While $d_{\Harr}$ raises the number of brackets, $d_{\BVGraphsg}$ creates an additional edge. Moreover, $\partial_\mathsf{mod}$ and $\partial'_\mathsf{mod}$ also increase the number of brackets in the Harrison complex, whereas $\partial_\mathsf{op}$ and $\partial'_\mathsf{op}$ raise the arity by one (or the number of edges, by adding a tadpole) in the respective complexes, and thus does not contribute to the differential on the associated graded. We obtain
\[
(\gr\gr(\Omega \Mog\hotimes_S\Harr \BVGraphsg  ), d_{\Omega\Mog}\otimes 1 )\rightarrow (\gr\gr(\Mogg \hotimes_S\Harr \BVGraphsg ),d_{\Mogg}\otimes 1)
\]
which describes a quasi-isomorphism by Lemma \ref{lemma:kappa}.
\end{proof}

\begin{rem}
The morphism from the hairy graph complex $\HGCg[-1]$ to the deformation complex extends via the map $\tilde \kappa_* $ to the morphism
\begin{align*}
\widetilde \Psi_\tk:=\tilde \kappa_* \circ \widetilde \Psi : \HGCg[-1] &\rightarrow \Def_\tk (\pdBVGraphsg\xrightarrow \Phi \Mog)\\
s^{-1}\Gamma &\mapsto \tilde \kappa \circ \widetilde \Phi\circ D_{\Gamma}
\end{align*}
The preferred element in $\Def_\tk (\pdBVGraphsg\xrightarrow \Phi \Mog)$ corresponds to
\[
\widetilde \Phi_{\tilde \kappa}:=\tilde \kappa \circ \widetilde \Phi:\Harr \pdBVGraphsg \xrightarrow \pi \pdBVGraphsg\xrightarrow \Phi \Mog \xrightarrow { i  } \Omega \Mog\xrightarrow {\tilde \kappa} \Mogg.
\]
Under the identification with $\Mogg \hotimes_S\Harr \BVGraphsg$ we may represent it by a linear combination of diagrams as in \eqref{eq:examplePhi}, i.e. diagrams of the form
\begin{equation}\label{eq:examplePhi}
\examplePhi
\end{equation}
that is, diagrams for which
\begin{itemize}
\item the upper part is a graph without any internal vertices, and
\item the lower part consists of the identical graph (with dualized decorations). In contrast to the diagrams in \eqref{eq:examplePhimc} the lower part lies in $\Mogg$. However, the clusters contain only one vertex (and are thus omitted in our diagrams).
\end{itemize}
%
%
\end{rem}

\subsection{Proof of Theorem \ref{thm:defHGC}}

We break up the proof in several smaller claims. First, filter $\HGCg[-1]$ by 
\[
\# \text{internal  edges } + \# \text{hairs } \geq p
\]
and similarly the complex $\Mogg\hotimes_S \Harr \BVGraphsg$ by 
\[
\text{ total } \# \text{ of edges in all } \BVGraphsg \text{-expressions in the free Lie algebra} \geq p.
\]
These define complete descending filtrations which are compatible with the morphism $\tilde \Psi$. 


\begin{rem}
Let us describe the morphism $\gr^{(1)}\widetilde \Psi$. Every hair is transformed into an edge by acting on $\PK$. Moreover, $\gr^{(1)}\widetilde \Psi$ maps a graph with, say $\#\text{edges} +\#\text{hairs}=p$, to a diagram with exactly $p$ edges. This is only obtained by acting on the part of the preimage of $\PK$ in $\Mogg \hotimes_S \Harr\BVGraphsg$ whose $\Harr \BVGraphsg$ part lies in $\BVGraphsg[1]$ and has no edges, but possibly decorations attached to the external vertices. That is, linear combinations of diagrams of the form,
\[
\examplePhii \ .
\]
\end{rem}

\begin{rem}
Notice that the differential on $\HGCg$ raises
\[
\# \text{ internal edges}+\# \text{hairs}
\]
by one and thus, it vanishes on the level of the associated graded, and 
\[
(\gr^{(1)}\HGCg[-1],d)\cong(\HGCg[-1],0).
\]
On the other hand, applying $d_{\BVGraphsg}$ as well as the part of $\partial'_\mathsf{mod}$ given by forming the bracket with the terms of the Maurer-Cartan element containing at least one edge increases the number of edges in the $\BVGraphsg$ part by at least one. In addition, the only term in $\partial'_\mathsf{op}$ which is visible on the associated graded complex is given by the operadic insertion operations of 
\[
s^{-2} \lambda \otimes \prodbvgraphs \ = \ \mcbvtwo
\]
since the other two terms in $m_\mathsf{op}$ both raise the number of edges by one. The differential we find on the associated graded $\gr^{(1)}(\Mogg\hotimes_S \Harr \BVGraphsg )$ is hence
\[
\gr^{(1)} D_\tk= 1\otimes d_{\Harr}+d_{\Mogg}\otimes 1+\gr \partial'_\mathsf{mod}+\gr \partial'_\mathsf{op}.
\]
\end{rem}

Next, notice that any diagram in $\BVGraphsg(r)$ decomposes as a (co)product of its internally connected components (i.e. the subgraphs which remain connected if all external vertices are deleted, see Section \ref{subsec:hairy}). Denote by $\icgg[1]\subset \BVGraphsg$ the (degree shifted) subcollection of internally connected graphs. For any $r\geq 1$, we may write $\BVGraphsg(r)=S^c(\icgg(r)[1])$ (see also \cite{SevWill11} for a similar construction in the local case). Since the differential on $\BVGraphsg(r)$ is compatible with the coalgebra structure, $\icgg(r)$ comes equipped with an $L_\infty$-structure. However, this is not important for our purposes, and we ignore both the differential on $\BVGraphsg(r)$ and the $L_\infty$-structure on $\icgg(r)$. Consider the underlying graded vector spaces $\BVGraphsg^0(r):=(\BVGraphsg(r),0)$ and $\icgg^0(r):=(\icgg(r),0)$ instead. We then have the following result.

\begin{lemma}\label{lemma:harr}
The projection $\pi:(\Harr \BVGraphsg^0(r),d_{\Harr} )\rightarrow \icgg^0(r)$ defines a quasi-isomorphism, i.e.
\[
H(\Harr \BVGraphsg^0(r),d_{\Harr})= \icgg^0(r).
\]
\end{lemma}

\begin{proof}
The cohomology of the Harrison complex of a cofree cocommutative coalgebra is for instance computed in Dolgushev and Rogers' article (\cite{DR12}, Theorem B.1).
\end{proof}

\begin{lemma}
The projection $\pi:(\Harr \BVGraphsg^0,d_{\Harr}) \rightarrow \icgg^0$ above induces a quasi-isomorphism
\[
1\otimes \pi: (\gr^{(1)} \Mogg \hotimes_S\Harr \BVGraphsg, \gr D_{\tk})\xrightarrow \sim ( \Mogg\hotimes_S \icgg^0, d_{\Mogg}\otimes 1  +\gr \partial'_\mathsf{op}).
\]
\end{lemma}

\begin{proof}
We need to check that $( 1\otimes \pi) \circ \gr \partial'_\mathsf{mod}=0$ in order for $1 \otimes \pi$ to be a map of complexes. This follows readily from the fact that $\pi$ vanishes on terms of weight at least two in the Harrison complex, and $\gr \partial'_\mathsf{mod}$ raises the number of Lie brackets by one. For the quasi-isomorphism, filter first both sides by the arity, and the corresponding associated graded by the cohomological degree of $\Mogg$. These define complete descending filtrations. Notice that only $1\otimes d_{\Harr}$ does not raise the cohomological $\Mogg$-degree, and that thus the statement follows by Lemma \ref{lemma:harr}.
\end{proof}

Next, endow $\Mogg\hotimes_S \icgg^0$ with the complete descending filtration
\[ 
\text{ degree of } \Mog \text{ part }+\text{ total arity } - \# \text{ clusters } +\#\delta^*  \geq p
\]
and consider the spectral sequence associated to it.

\begin{lemma}\label{lemma:ComHLie}
For the zeroth page $E_0^{(2)}$ of the spectral sequence we have
\[
H(E_0^{(2)},d_0):=H(\gr^{(2)} \Mogg\hotimes_S \icgg^0, \gr^{(2)}(d_{\Mogg}\otimes 1+\gr^{(1)}\partial'_\mathsf{op}))\cong \Com^c\circ H \circ \Lie\{-1\}\hotimes_S \icgg^0.
\]
\end{lemma}

\begin{proof}
Recall that as symmetric sequences
\[
\Mogg =\Mog \circ \bv^!=\Com^c\circ H \circ \Lie^c\{-1\} \circ \K[\Delta]/(\Delta^2)\circ \K[\delta^*]\circ \Com\{-2\}\circ \Lie\{-1\}.
\]
The subspace $\Com^c\circ H \circ \Lie \{-1\} \hotimes_S \icgg^0$ forms a subcomplex of $(\Mogg\hotimes_S \icgg^0, d_{\Mogg}\otimes 1  +\gr^{(1)} \partial'_\mathsf{op})$. Observe that on the level of the associated graded the differential 
\[
\gr^{(2)}(d_{\Mogg}\otimes 1+\gr^{(1)}\partial'_\mathsf{op})
\]
vanishes on this subcomplex, while on $\gr^{(2)} \Mogg\hotimes_S \icgg^0$ it is equal to the part of $d_{\Mogg}\otimes 1$ which
\begin{itemize}
\item removes a dashed edge and fuses the clusters it connected, and the part which
\item removes a tadpole and increases the power of the corresponding $\delta^*$ by one
\end{itemize}
(see also the first and the third line in Figure \ref{fig:diffcoop}). Both operations decrease the degree on $\Mog$ by one, while either decreasing the number of clusters or raising a power of $\delta^*$ by one. In particular, they correspond precisely to the differentials $d_{\kappa_1}$ and $d_{\kappa_2}$ induced by the Koszul morphisms
\begin{align*}
&\kappa_1:\Lie^c\{-1\}=\Com^{\text{!`}}\{-2\}\rightarrow \Com\{-2\}\\
&\kappa_2:\K[\Delta]/(\Delta^2)\rightarrow \K[\delta^*]=(\K[\Delta]/(\Delta^2))^{\text{!`}}.
\end{align*}
In this case it is well-known (see for instance \cite{LodayVallette12}, Chapter 7) that the complexes
\begin{align*}
&(\Lie^c\{-1\} \circ \Com\{-2\},d_{\kappa_1})\\
&(\K[\Delta]/(\Delta^2)\circ \K[\delta^*],d_{\kappa_2})
\end{align*}
are acyclic (i.e. $\K$ in degree zero). The differential on the associated graded reduces thus to $(d_{\kappa_1}+d_{\kappa_2})\otimes 1$. Moreover, the acyclicity of the two aforementioned complexes implies that the cohomology of
\[
(\Com^c\circ H \circ \Lie^c\{-1\} \circ \K[\Delta]/(\Delta^2)\circ \K[\delta^*]\circ \Com\{-2\}\circ \Lie\{-1\},d_{\kappa_1}+d_{\kappa_2})
\]
is $\Com^c\circ H \circ \Lie\{-1\}$, and the inclusion 
\[
\Com^c\circ H \circ \Lie \{-1\} \hotimes_S \icgg^0\rightarrow \Mogg\hotimes_S \icgg^0
\]
is indeed a quasi-isomorphism.
\end{proof}

\begin{rem}
Notice that elements in $ \Com^c\circ H\circ \Lie\{-1\}$ consist of diagrams with no dashed edges, and for which each cluster contains exactly one (possibly trivial) Lie tree. 

\begin{figure}[h]
\centering
\includegraphics[width=0.4\textwidth]{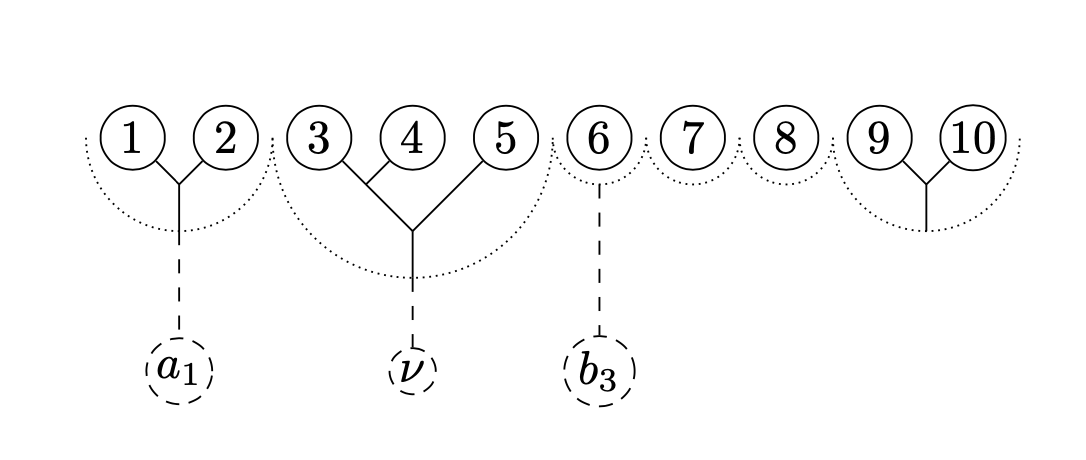}
\caption{An element in $\Com^c\circ H \circ \Lie\{-1\} (10)$.}\label{fig:comHlie}
\end{figure}

On the first page the differential is given by the sum of $\gr^{(1)} \partial'_\mathsf{op}$ and remaining the part of $d_{\Mogg}\otimes 1$ which fuses disconnected clusters and pairs the two Lie trees by adding a bracket (see the second line of Figure \ref{fig:diffcoop}). Denote this second term by $d_\mathrm{fuse}\otimes 1$. We need to compute the cohomology of the complex
\[
E_1^{(2)}=(H(E_0^{(2)},d_0), d_{\Mogg}\otimes 1+\gr^{(1)}\partial'_\mathsf{op} )\cong ( \Com^c\circ H\circ \Lie^c\{-1\}\hotimes_S \icgg^0,d_\mathrm{fuse}\otimes 1+\gr^{(1)}\partial'_\mathsf{op} ).
\]
\end{rem}

\begin{defi}
A small cluster is either a cluster of arity one (with trivial Lie word) carrying no decoration and for which the external vertex in $\icgg^0$ is zero-valent, or a cluster of arity two with one Lie bracket, no decoration and zero-valent external vertices in $\icgg^0$. All clusters which are not small clusters will be called big clusters (see Figure \ref{fig:smallclusters}).
\end{defi}

\begin{figure}[h]
\centering
\includegraphics[width=0.5\textwidth]{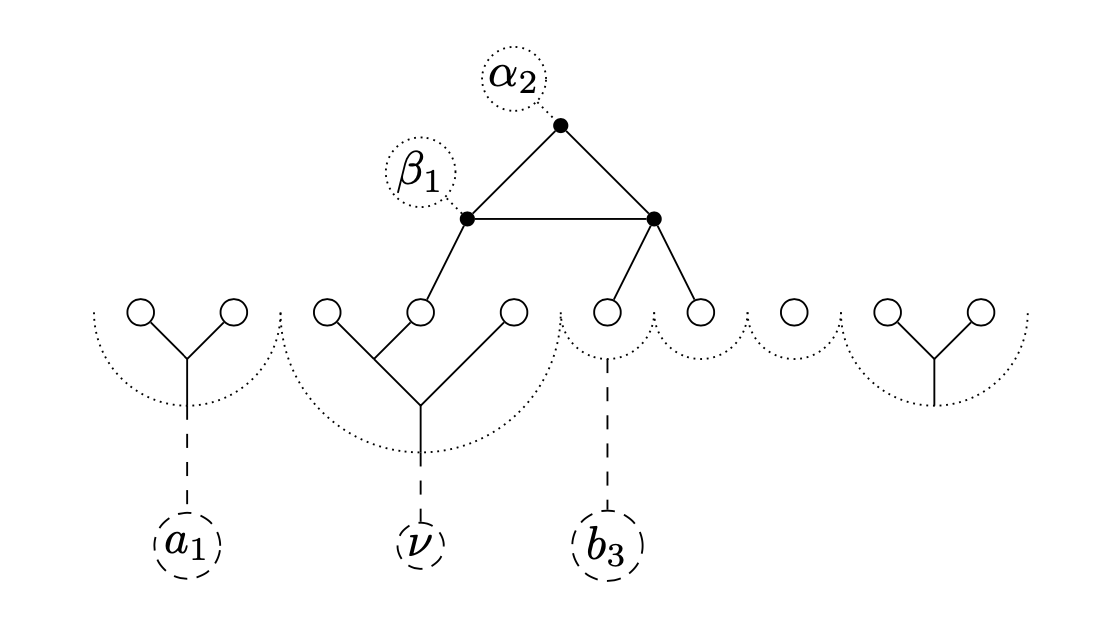}
\caption{The two right-most clusters (of arity one and two) are called small clusters. We refer to the other clusters as big clusters. Notice in particular that the fourth cluster from the left is also a big cluster, since the corresponding external vertex is of valence one in $\icgg^0$.}\label{fig:smallclusters}
\end{figure}

At this point filter $\ \Com^c\circ H\circ \Lie^c\{-1\}\hotimes_S \icgg^0$ by
\[
\#\text{big clusters}\leq p.
\]
The ascending filtration is complete. Observe that the differential on the associated graded consists of $\gr^{(1)} \partial'_\mathsf{op}$ plus the part of $d_\mathrm{fuse}\otimes 1$ which fuses either two small clusters or one small and one big cluster. We equip the associated graded with the descending complete filtration given by the total number of vertices in the big clusters. We are thus considering
\[
\gr^{(4)}\gr^{(3)} E_1^{(2)}.
\]
The induced differential acts on small clusters only:
\begin{itemize}
\item $\gr^{(1)} \partial'_\mathsf{op}$ splits a small cluster of arity one into a cluster of arity two , 
\item $d_\mathrm{fuse}\otimes 1 $ fuses two small clusters of arity one,
\end{itemize}
and in both cases a Lie bracket is added.

\begin{lemma}\label{lemma:grgr}
The cohomology of $\gr^{(4)}\gr^{(3)} E_1^{(2)}$ is given by the direct sum
\[
V_0\oplus V_S
\]
where
\begin{itemize}
\item $V_0$ is spanned by diagrams without small clusters, and
\item $V_S=V_0\otimes S(\smcluster)$ is spanned by formal series of diagrams for which the small cluster part is fixed by the formal series,
\[
S(\smcluster):=\sum\limits_{k\geq 1} \frac{ (-2)^{k-1}}{k!} \underbrace{\smclusterseries}_{k}
\]
\end{itemize}
\end{lemma}

\begin{proof}
Since the differential acts on the small cluster part only, we may view $\gr^{(4)}\gr^{(3)} E_1^{(2)}$ as the tensor product
\[
V_0\oplus V_0\otimes V_{small}
\]
equipped with a differential which acts on $V_{small}$, the subcomplex spanned by small clusters, only. First of all, it is clear that diagrams in $V_0$ form non-zero cohomology classes in $\gr^{(4)}\gr^{(3)} E_1$. As for $V_{small}$, notice that having two small clusters of arity two within a diagram produces an odd symmetry which forces it to be zero. Moreover, any diagram with exactly one such small cluster of arity two is exact. To see this, we argue by induction on the arity. The diagram of arity two, with exactly one small cluster (of arity two) is clearly exact, in fact
\[
\smclusterbra= d(\smcluster).
\]
Now consider the diagram in $V_{small}$ of arity $k+1$, with one small cluster of arity two. It satisfies
\[
(k+1) \underbrace{\smclusterbra \ \smclusterseries}_{k+1} = d(\underbrace{\smclusterseries}_{k}) - \binom{k}{2} \underbrace{\smclusterbra \ \smclusterseries}_{k}.
\]
By the induction hypothesis the right hand side is exact. Next, assume we are given a formal series of diagrams of the form
\[
\lambda=\sum\limits_{k\geq 1} \lambda_k \underbrace{\smclusterseries}_{k}.
\]
Applying the differential yields
\[
\sum\limits_{k\geq 2}(\lambda_k \binom{k}{2} + \lambda_{k-1}(k-1)) \underbrace{\smclusterbra \ \smclusterseries}_{k}.
\]
For $d\lambda=0$, we thus have to require that for $k\geq 2$
\[
\lambda_k=-\lambda_{k-1}\frac{2}{k}=\dots=\lambda_1\frac{(-2)^{k-1}}{k!}.
\]
Thus, we find
\[
\lambda=\lambda_1 \sum\limits_{k\geq 1} \frac{(-2)^{k-1}}{k!}\underbrace{\smclusterseries}_{k}
\]
from which we conclude that $H(\gr^{(4)}\gr^{(3)} E_1^{(2)})\cong V_0\oplus (V_0\otimes S(\smcluster))\cong V_0\oplus V_S$.
\end{proof}

Lemma \ref{lemma:grgr} implies $E_1^{(4)}(\gr^{(3)} E_1^{(2)})\cong V_0\oplus V_S$. It comes furthermore equipped with part of the differential on $\gr^{(3)} E_1^{(2)}$ which raises the total number of vertices in the big clusters by one. This consists of the part of $\gr^{(1)} \partial'_\mathsf{op}$ which splits a vertex in a big cluster and adds a bracket, and the part of $d_\mathrm{fuse}\otimes 1$ which fuses a big cluster with a small cluster of arity one while adding a bracket. Notice however, that $V_0$ forms a subcomplex of $E_1^{(4)}(\gr^{(3)} E_1^{(2)})$, while $V_S$ does not (the diagram with one small cluster may be mapped to a diagram with no small clusters). This is fixed by the following slight change of basis. Set 
\[
\tilde S(\smcluster):=-\frac12+S(\smcluster)
\]
and $\tilde V_S=V_0 \otimes \tilde S(\smcluster)$. Still, we have $E_1^{(4)}(\gr^{(3)} E_1^{(2)})\cong V_0\oplus \tilde V_S$, but now this is a direct sum of subcomplexes. 

\begin{lemma}\label{lemma:coh}
The cohomology of the first page $E_1^{(4)}(\gr^{(3)} E_1^{(2)})$ is spanned by diagrams 
\[
v\otimes \tilde S(\smcluster) \in \tilde V_S= V_0\otimes \tilde S(\smcluster)
\]
for $v$ as in Figure \ref{fig:v0}, i.e.
\begin{itemize}
\item all external vertices are univalent, and
\item all clusters are of arity one.
\end{itemize}
\end{lemma}

\begin{figure}[h]
\centering
\includegraphics[width=0.3\textwidth]{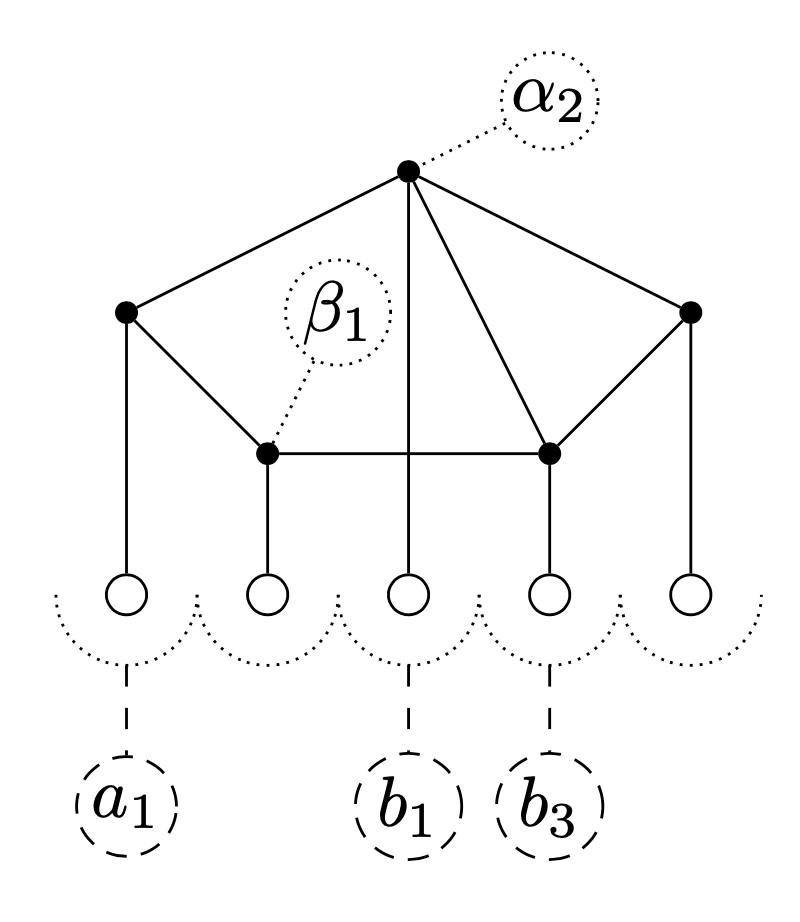}
\caption{A diagram in $V_0$, for which all all external vertices are of valence one, and all clusters are of arity one.}\label{fig:v0}
\end{figure}

\begin{proof}
Consider first elements of $\tilde V_S=V_0\otimes \tilde S(\smcluster)\subset E_1^{(4)}(\gr^{(3)} E_1^{(2)})$. Let the character of a graph $\Gamma \in \icgg^0$ be the isomorphism class of the graph obtained by deleting all external vertices, while keeping the dangling edges. The differential on $E_1^{(4)}(\gr^{(3)} E_1^{(2)})$ does not act on the character of a graph. Thus, for each character, say $c$, with $k$ edges connecting internal and external vertices, we obtain a subcomplex $\tilde V_S(c)$ of $\tilde V_S$. The symmetric group $S_k$ acts on the dangling edges of the character by permuting their order. Furthermore, introduce formal variables $t_1,\dots,t_k$ of degree zero and set $P_k=\vspan\{t_1, \dots, t_k\}$. Consider the following version of the Harrison complex of the cofree coaugmented cocommutative coalgebra $S^c(P_k[1])$,
\[
\Harr'(S^c (P(k))):=(L(S^c(P_k)[-1]), d_{\Harr}^{\mathrm{red}})
\]
where $d_{\Harr}^{\mathrm{red}}$ is induced by the reduced coproduct on $S^c(P_k)$. Since the differential does not annihilate or create any of the formal variables, $\Harr'(S^c(P_k))$ comes with a $\mathbb{N}^k_0$-grading counting the number of $t_1,\dots , t_k$ appearing in any Lie monomial. Moreover, its (degree shifted) symmetric tensor powers
\[
S^p(\Harr'(S^c(P_k))[1])
\]
inherit the $\mathbb{N}^k_0$-grading. We are interested in the degree $(1,\dots,1)$ part, i.e. those monomials in which every formal variable occurs exactly once. In fact, as computed in (\cite{Willwacher15}, Appendix A; see also \cite{DR12}, Appendix B for a more detailed account),
\[
H^j(S^p(\Harr'(S^c(P_k))[1])^{(1,\dots,1)})=\begin{cases} \K\cdot [t_1\otimes \dots \otimes t_k] & j=0 \text{ and } p=k\\ 0 & \text{otherwise.} \end{cases}
\]
Returning to $\tilde V_S(c)$, observe that we may decompose it as
\[
\tilde V_S(c)\cong \bigoplus\limits_{\substack{p=1 \\ \alpha_1,\dots, \alpha_p}}^k (c \otimes_{S_k} (S^p(\Harr'(S^c(P_k))[1])^{(1,\dots,1)})\otimes_{S_p} \vspan_\K\{\alpha_1,\dots,\alpha_p\} 
\]
where the direct sum runs also over all $p$-tuples $(\alpha_1,\dots,\alpha_p)\in H^{\times p}$. Figure \ref{fig:Harrargument} sketches the correspondence. 

\begin{figure}[h]
\centering
\includegraphics[width=1\textwidth]{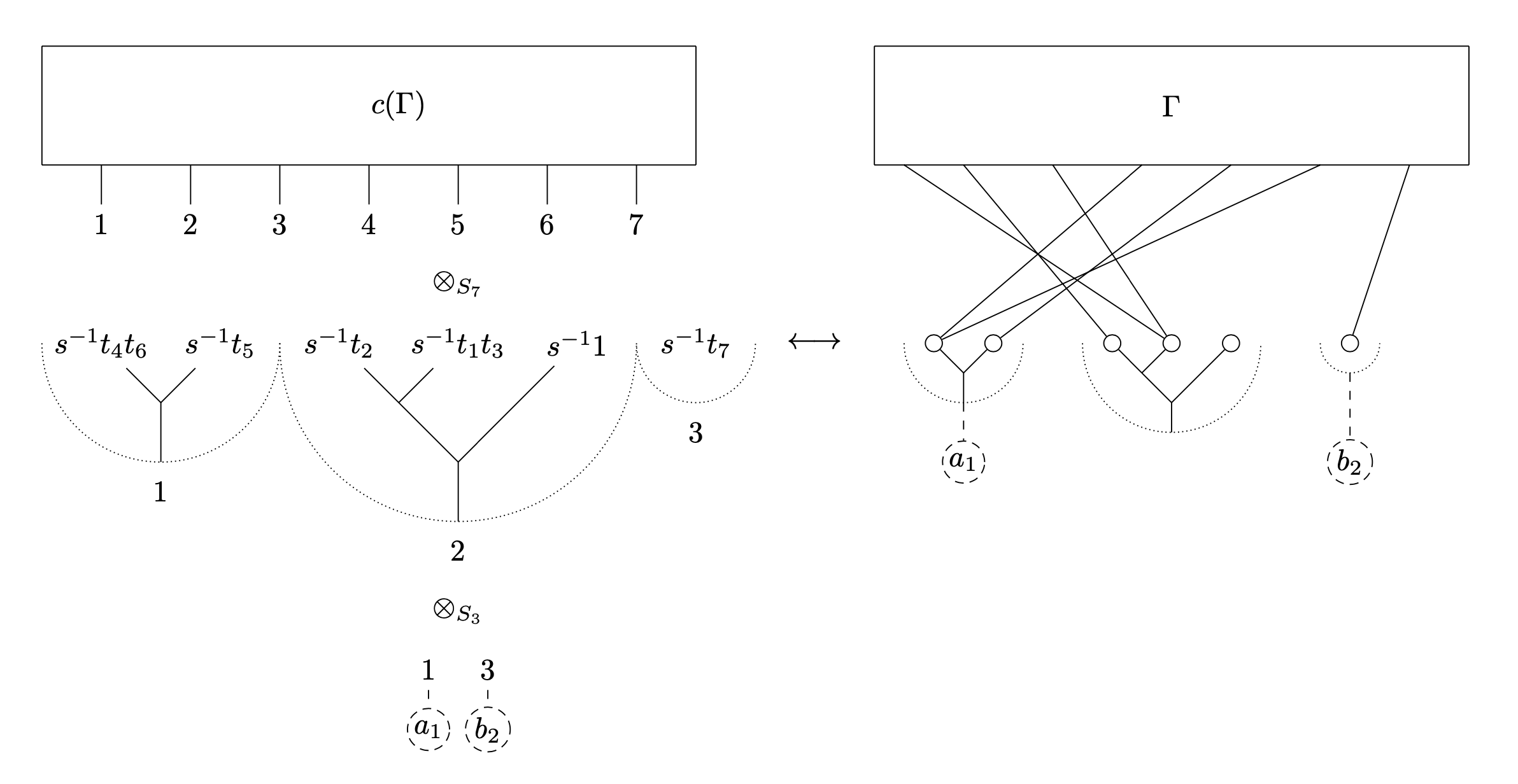}
\caption{The decomposition of a diagram in $\tilde V_S(c)$. The middle piece on the left hand side is identified with the element $s\left([s^{-1}t_4t_6,s^{-1}t_5]\wedge [[s^{-1}t_2,s^{-1}t_1t_3],s^{-1}1] \wedge s^{-1}t_7\right)$ in $S^3(\Harr ' (S^c(P_7))[1] )^{(1,\dots,1)}$. }\label{fig:Harrargument}
\end{figure}

Roughly speaking, we are cutting the $\icgg^0$ part off its external vertices, and detaching the decorations from the clusters. The remaining piece in the middle corresponds precisely to the complex $S^p(\Harr'(S^c(P_k))[1])^{(1,\dots,1)}$. Since taking cohomology commutes with taking coinvariants
\[
H(\tilde V_S(c))\cong\bigoplus\limits_{\alpha_1,\dots, \alpha_k} ( c \otimes_{S_k} \K\cdot [t_1\otimes \dots \otimes t_k] ) \otimes_{S_k} \vspan_\K\{\alpha_1,\dots,\alpha_k\}
\]
Translated into the language of graphs, this means that the cohomology classes in $\tilde V_S(c)$ are given by diagrams all of whose big clusters are (possibly decorated and) of arity one, whose character is $c$ and for which each external vertex is univalent.  

We conclude by arguing that the subcomplex $V_0\subset E_1(\gr E_1)$ is acyclic. Recall that it is spanned by diagrams without small clusters. Consequently, the differential only splits external vertices while adding a Lie bracket. We may perform a similar identification as above. Let $c$ be the character of a graph with $k$ edges connecting internal and external vertices. Since also in this case, the differential does not alter the character, we obtain subcomplexes $V_0(c)$ of $V_0$, one for each character. We have,
\[
V_0(c)\cong \bigoplus\limits_{\substack{p=1 \\ \alpha_1,\dots, \alpha_p}}^k (c \otimes_{S_k} (S^p(\Harru(S^c(P_k))[1])^{(1,\dots,1)})\otimes_{S_p} \vspan_\K\{\alpha_1,\dots,\alpha_p\}
\]
where $\Harru (S^c(P_k))$ denotes the unreduced Harrison complex. It takes into account that we are missing the term which fuses a big cluster with a small cluster while adding a bracket. We refer to Lemma \ref{lemma:unredHarr} below where we show that the cohomology of $\Harru (S^c(P_k))^{(1,\dots,1)}$ vanishes. The statement follows.
\end{proof}

\begin{lemma}\label{lemma:unredHarr}
The cohomology of $\Harru(S^c(P_k))^{(1,\dots,1)}$ vanishes.
\end{lemma}

\begin{proof}
Basis elements of $L(S^c(P_k)[-1])^{(1,\dots,1)}$ are of the form
\begin{equation}\label{eq:basis}
b = [\ad_{s^{-1}1}^{k_1}(s^{-1}c_1),[\ad_{s^{-1}1}^{k_2}(s^{-1}c_2),[\ad_{s^{-1}1}^{k_3}(s^{-1}c_3),\dots, [\ad_{s^{-1}1}^{k_{l-1}}(s^{-1}c_{l-1}),\ad_{s^{-1}1}^{k_l}(s^{-1}c_l)]\cdots ]
\end{equation}
for $k_i\geq 0$ and $c_i\neq 1 \in S^c(P(k))$. Additionally, include the elements $s^{-1}1$ and $[s^{-1}1,s^{-1}1]$, but notice that for $k\geq 2$, $\ad_{s^{-1}1}^k (s^{-1}1)=0$. Filter $\Harru (S^c(P_k))^{(1,\dots,1)}$ by the number of non-(degree shifted-)unit elements. The filtration is bounded degree-wise. Since on $\ad_{s^{-1}1}^k(s^{-1}c)$, the differential on the associated graded is given by
\[
\gr d_{\Harr} (\ad_{s^{-1}1}^k(s^{-1}c))=
\begin{cases}
2 \ \ad_{s^{-1}1} ^{k+1} (s^{-1}c) & \text{ if } k \text{ even}\\
0 & \text{ if } k \text{ odd}
\end{cases}
\]
we find that $\gr d_{\Harr}$ yields
\[
\gr d_{\Harr} (b)=\sum\limits_{\substack{1\leq j\leq l \\ k_j \ \text{even}}} (-1)^{\epsilon(j)} 2 [\ad_{s^{-1}1}^{k_1}(s^{-1}c_1),\dots ,[\ad_{s^{-1}1}^{k_j+1}(s^{-1}c_j),\dots, [\ad_{s^{-1}1}^{k_{l-1}}(s^{-1}c_{l-1}),\ad_{s^{-1}1}^{k_l}(s^{-1}c_l)]\cdots ]
\]
on elements as in \eqref{eq:basis} and where the sign ${\epsilon(j)}=k_1+\dots + k_{j-1}+j-1$ is determined by the Koszul rule. Set
\[
h(b)=\frac{1}{2l} \sum\limits_{\substack{1\leq j\leq l \\ k_j \ \text{odd}}} (-1)^{\epsilon(j)} [\ad_{s^{-1}1}^{k_1}(s^{-1}c_1),\dots ,[\ad_{s^{-1}1}^{k_j-1}(s^{-1}c_j),\dots, [\ad_{s^{-1}1}^{k_{l-1}}(s^{-1}c_{l-1}),\ad_{s^{-1}1}^{k_l}(s^{-1}c_l)]\cdots ]
\]
for $\epsilon(j)=k_1+\dots + k_{j-1}+j-1$ (which is again determined by a Koszul type rule). It defines a homotopy, i.e. it satisfies
\[
h\gr d_{\Harr}+ \gr d_{\Harr} h =\id
\]
from which the statement follows.
\end{proof}

\begin{proof}[Proof of Theorem \ref{thm:defHGC}]
The series of nested spectral sequences implies that
\[
\gr^{(1)} \widetilde \Psi: (\HGCg[-1],0)\rightarrow \gr^{(1)} (\Mogg \hotimes_S \Harr\BVGraphsg)
\]
is a quasi-isomorphism, since 
\[
H(\gr^{(1)} (\Mogg \hotimes_S \Harr\BVGraphsg))\cong H(\Mogg \hotimes_S \icgg^0)\cong H(\tilde V_S)
\]
and by Lemma \ref{lemma:coh} representatives in the cohomology of $\tilde V_S$ are completely determined by diagrams whose external vertices are univalent and all clusters are of arity one - i.e. (up to degree shift) diagrams in $\HGCg$.
\end{proof}

\subsection{The non-framed case for $g=1$}

In the non-framed case for $g=1$, the set-up is as follows. We consider $(\op M,\op C)=(\pdGraphsone,\pdGraphs)$ and $(\op N,\op D)=(\Mo,e_2^c)$, and the quasi-isomorphism of dg Hopf comodules
\[
(\Phi:\pdGraphsone\rightarrow \Mo \ , \ \phi: \pdGraphs\rightarrow \e_2^c).
\]
The results and proofs carry over analogously. Indeed, some of the arguments are simplified slightly since $\Mo$ is tadpole-free and as opposed to $\bvk$, and when considering $\Mo^{!}=\Mo\circ e_2^!=\Mo \circ e_2\{-2\}$, $e_2^!=e_2\{-2\}$ has no internal differential. In this case the results read as follows.

\begin{thm}\label{thm:defHGCone}
The respective action of $\spp(H^*)\rtimes \GC^{\minitadp}_{(1)}$ and $\HGC^{\minitadp}_{(1)}$ on $\pdGraphsone$ induce morphisms
\begin{align*}
\HGC^{\minitadp}_{(1)}[-1] &\rightarrow \Def(\pdGraphsone \xrightarrow \Phi \Mo)\\
(\spp(H^*)\ltimes \GC^{\minitadp}_{(1)})[-1]&\rightarrow \Def(\pdGraphsone \xrightarrow \Phi \Mo)
\end{align*}
While the first defines a quasi-isomorphism, the second gives a quasi-isomorphism in all degrees except in degree $4$. In degree $4$, the map on cohomology $H^4((\spp(H^*)\ltimes \GC^{\minitadp}_{(1)})[-1])\rightarrow H^4( \Def(\pdGraphsone\xrightarrow \Phi \Mo))$ is injective with one-dimensional cokernel.
\end{thm}

\section{Bezrukavnikov's model}\label{sec:tg}
We introduce a further model for the configuration spaces of points on surfaces based on work of Bezrukavnikov \cite{Bezru94}. More concretely, we extend his work to the framed setting, and additionally include an operadic module structure.

\subsection{The Lie algebras $\fraktg(n)$}

Consider the following collection of Lie algebras.

\begin{defi}
Fix $n\geq 1$. Let $\fraktg(n)$ be the Lie algebra with generators of degree zero,
\begin{equation*}
x_l^{(i)},\ y_l^{(i)}, \ t_{ij} \text{ for } 1\leq i, j\leq n, \ 1\leq l \leq g
\end{equation*}
subject to the relations,
\begin{align*}
&t_{ij}=t_{ji}, \ [t_{ij},t_{kl}]=0 \text{ for } i,j,k,l \text{ all distinct},\  [t_{ik}+t_{kj},t_{ij}]=0 \text{ for } i,j,k \text{ all distinct.}\\
&[x_l^{(i)},x_k^{(j)}]=[y_l^{(i)},y_k^{(j)}]=0 \ (i\neq j), \ [x_k^{(i)},y_l^{(j)}]=0 \ (i\neq j, \ k\neq l)\\
&[x_k^{(i)},y_k^{(j)}]=[x_l^{(i)},y_l^{(j)}]=t_{ij} \ (i\neq j)\\
&\sum\limits_{k=1}^g [x_k^{(i)},y_k^{(i)}]=-\sum\limits_{j\neq i}^{n} t_{ij}-(2-2g)t_{ii}.
\end{align*}
and for which all $t_{ii}$ are central. 
\end{defi}

It follows from the Jacobi identity that for $i\neq j$, $[x_l^{(i)}+x^{(j)}_l,t_{ij}]=[y_l^{(i)}+y_l^{(j)},t_{ij}]=0$. For $g=1$, $\frakt_{(1)}(1)$ is abelian. For $g\geq 2$, it is probably well-known to experts that in arity one the center of $\fraktg(1)$ is one-dimensional with generator $t_{11}$. It is a consequence of (Proposition A', \cite{AsadaKaneko87}). Moreover, Enriquez gives a variant of Asada and Kaneko's proof \cite{AsadaKaneko87} in his unpublished notes \cite{Enriquezprivate}. The proof we present here was suggested by F. Naef.

\begin{lemma}\label{lemma:center}
For $g\geq 2$ the center of $\fraktg(1)$ is one-dimensional with generator $t_{11}$.
\end{lemma}

\begin{proof}
Let $A:=\mathrm{Ass}(x_l,y_l,z)_{l=1,\dots,g}$ be the free associative algebra with generators $x_l$, $y_l$ of degree zero, and $z$ of degree $-1$. Equip it with the differential
\[
dz=\sum_{l=1}^g x_l y_l - y_l x_l , \ dx_l=0 , \ dy_l=0.
\]
Moreover, consider the double complex $C$ defined through the mapping cone of the morphisms
\[
\ad_{x_1}:A\rightarrow A, \ c \mapsto x_1c-cx_1.
\]
Denote by $A_i$ the degree $i$ part of $A$. Notice that $C=A\oplus A[-1]$ looks as follows.
\begin{center}
\begin{tikzcd}[row sep=0.4cm ]
  0 & 0 & 0 & 0\\
  0 & A_0 \arrow[r,"\ad_{x_1}"] & A_0[-1] & 0  \\
  0 & A_{-1}  \arrow[r,"\ad_{x_1}"] \arrow[u,"d"]& A_{-1}[-1]  \arrow[u,"d"]& 0\\
  0 & A_{-2} \arrow[r,"\ad_{x_1}"] \arrow[u,"d"]& A_{-2}[-1] \arrow[u,"d"]& 0\\
  0 & A_{-3} \arrow[r,"\ad_{x_1}"] \arrow[u,"d"]& A_{-3}[-1] \arrow[u,"d"]& 0\\
   &  \arrow[u,"d"]&  \arrow[u,"d"]& \\
\end{tikzcd}
\end{center}
Using the standard spectral sequence techniques, we find that the spectral sequence of the total complex of $C$ abuts on the second page, where it is concentrated in total degrees zero and one. On the one hand, filtering such that on the first page we have the cohomology with respect to the horizontal differential, yields
\[
E_1^{0,i}=\begin{cases}
\K[x_1] & i=0\\
0 & i>0
\end{cases}
\]
on the first column. Here we used that the kernel of $\ad_{x_1}$ corresponds to the centralizer of $x_1$ in $A$ and is thus given by the subalgebra of polynomials $\K [x_1]$ (see for instance \cite{Bourbaki2006}, Exercise 3 of Chapter 3). On the second column we find
\[
E_1^{1,i}\cong\bigoplus_{k\geq 0} x_1^k \tilde A_{i}[-1]
\]
where $\tilde A_i$ consists of words in $A_i$ which do not begin or end with the varible $x_1$. We show that the second column of the second page is concentrated in degree one, i.e. $E_2^{1,i}=0$ for $i\neq 0$. For this, equip the second column by the filtration given by the number of $x_i\neq x_2$. On the associated graded the differential is induced by 
\[
z \mapsto x_2y_2-y_2x_2.
\]
Filtering again by the number of times the expression $y_2x_2$ appears in any monomial in that order reduces the differential to $z\mapsto x_2y_2$ on the associated graded complex. The homotopy $x_2y_2\mapsto z$ ensures that the cohomology is concentrated in $E_2^{1,0}$. In particular, since $E_2^{1,-1}=0$, the cohomology of the total complex in degree zero is
\[
E_\infty^{0,0}\oplus E_\infty^{1,-1}=E_2^{0,0}\oplus E_2^{1,-1}=E_2^{0,0}=\K [x_1].
\]
On the other hand, filtering such that the first page computes the cohomology with respect to the vertical differential gives
\[
{\tilde E}_\infty^{0,0}\oplus \tilde E_\infty^{0,0}={\tilde E}_2^{0,0}=Z_{x_1}(A_0/B)
\]
where $Z_{x_1}(\dots)$ denotes the centralizer of $x_1$, and $B$ is the ideal in $A_0$ generated by 
\[
\sum_{l=1}^g x_l y_l - y_l x_l.
\]
In this case we used that $H^0(A,d)=A_0/B$ and the fact the cohomology of the columns is concentrated in $z$-degree zero (by a similar argument as above). Since the spectral sequences converges to the cohomology of the total complex, we have $E_\infty^{0,0}\oplus E_\infty^{1,-1}\cong {\tilde E}_\infty^{0,0}\oplus \tilde E_\infty^{1,-1}$ and thus $\K[x_1]\cong Z_{x_1}(A_0/B)$. We deduce that the center $Z(A_0/B)=\K 1$ is one-dimensional. Notice that $A_0/B$ contains a subalgebra isomorphic to $\fraktg(1)/( t_{11})$. More precisely, the image of $\mathbb{L}(x_l,y_l)_{l=1,\dots g}$ under the composition
\[
\mathbb{L}(x_l,y_l)_{l=1,\dots g}\hookrightarrow A_0=\mathrm{Ass}(x_l,y_l)_{l=1,\dots g}\twoheadrightarrow A_0/B
\]
is isomorphic $\fraktg(1)/( t_{11})$. Its center is contained in the center of $A_0/B$, and since $1$ is not in the image of the composition, it is equal to zero. Finally, let $\theta=\theta(x_l,y_l,t_{11}) \in Z(\fraktg(1))$. Since $t_{11}$ is central in $\fraktg(1)$, $\theta(x_l,y_l,t_{11}=0) \in Z(\fraktg(1))$ still lies in the center, and its projection onto the quotient $\fraktg(1)/(t_{11})$ yields an element in the center $Z(\fraktg(1)/(t_{11}))=0$. Thus $\theta(x_l,y_l,t_{11}=0)=0$, and $\theta=k\cdot t_{11}$ for $k\in \K$. 
\end{proof}

Recall further that for any $m,n\geq 1$ and any map of sets $\phi:\{1,\dots,m\}\rightarrow \{1,\dots,n\}$, there is a unique Lie algebra morphism
\begin{align*}
\fraktg(n)&\rightarrow \fraktg(m)\\
\theta&\mapsto \theta^\phi
\end{align*}
which on the generators is given by
\begin{align*}
\text{for } i\neq j: \ t_{ij} &\mapsto (t_{ij})^{\phi}:= \sum\limits_{i', j'\in \phi^{-1}(j)} t_{i'j'}\\
t_{ii}&\mapsto (t_{ii})^{\phi} := \sum\limits_{i'\in\phi^{-1}(i)} t_{i'i'} + \frac12 \sum\limits_{i'\neq j' \in\phi^{-1}(i)} t_{i'j'}\\
x_l^{(i)}&\mapsto (x_l^{(i)})^{\phi}:=\sum\limits_{i'\in \phi^{-1}(i)} x^{(i')}_l\\
y_l^{(i)}&\mapsto (y_l^{(i)})^{\phi}:=\sum\limits_{i'\in \phi^{-1}(i)} y^{(i')}_l.
\end{align*}
For a more concise notation, we write $\theta^{\phi}=\theta^{I_1,\dots,I_n}$ where $I_i=\phi^{-1}(i)$. Moreover, we define the following collection of maps of sets for $n\geq 1$ and $1\leq k \leq n+1$,
\begin{align*}
\delta^{n}_k:\{1,\dots,n+1\}&\rightarrow \{1,\dots,n\}\\
j&\mapsto 
\begin{cases}
j & \text{ if } j\leq k\\
j-1 & \text{ if } j>k
\end{cases}
\end{align*}
which, by the above, induce Lie algebra morphisms $\fraktg(n)\rightarrow \fraktg(n+1)$ mapping
\[
\theta \mapsto \theta^{\delta_k^n}=\theta^{1,\dots,k \ k+1, \dots, n+1}.
\]
A simple calculation shows that these satisfy the relations (here for $\theta \in \fraktg(n-1)$ and $l\leq k$)
\[
(\theta^{\delta_l^{n-1}})^{\delta_{k+1}^{n}}=\theta^{\delta_l^{n-1} \circ \delta_{k+1}^{n}}=\theta^{\delta_{k}^{n-1} \circ \delta_{l}^n}=(\theta^{\delta_{k}^{n-1}})^{\delta_l^n}.
\]
In particular, the case $l=k$, which will be important later on, reads
\[
\theta^{\delta_k^{n-1} \circ \delta_{k+1}^{n}}=\theta^{\delta_{k}^{n-1} \circ \delta_{k}^n}.
\]
Notice that in the alternative notation, both expressions equal 
\[
\theta^{\delta_k^{n-1} \circ \delta_{k+1}^{n}}=\theta^{1,\dots, k \ k+1 \ k+2, \dots, n+1}
\]
whereas the previous relation yields
\[
\theta^{\delta_l^{n-1} \circ \delta_{k+1}^{n}}=\theta^{1,\dots, l \ l+1, \dots, k \ k+1, \dots, n+1}.
\]

Similarly, for $n\geq 1$, consider the abelian extension $\frakt_\bv(n)$ of the Drinfeld-Kohno Lie algebra $\frakt(n)$ (\cite{Kohno87}, \cite{Drinfeld90}). It is generated by the symbols $t_{ij}$ for $1\leq i,j\leq n$ of degree zero, subject to the same relations as above, i.e.
\[
t_{ii} \text{ central, } t_{ij}=t_{ji}, \ [t_{ij},t_{kl}]=0 \text{ for } i,j,k,l \text{ all distinct},\  [t_{ik}+t_{kj},t_{ij}]=0 \text{ for } i,j,k \text{ all distinct.}
\]
Denote by $\fraks(n)\subset \frakt_\bv(n)$ the Lie ideal spanned by all $t_{ii}$. There is a short exact sequence of Lie algebras $0\rightarrow \fraks(n)\rightarrow \frakt_\bv(n)\rightarrow \frakt(n)\rightarrow 0$. Notice that we may define the Lie algebra morphisms $\frakt_\bv(n)\rightarrow \frakt_\bv(m)$ induced by maps of sets $\phi:\{1,\dots,m\}\rightarrow \{1,\dots,n\}$ analogously to the $\fraktg$ case by using the formula for the elements $t_{ij}$.

\subsection{The Chevalley-Eilenberg cochain complex $C(\fraktg)$}

As with the Drinfeld-Kohno Lie algebras, the collection of Lie algebras $\frakt_\bv=\{\frakt_\bv(n)\}_{n\geq 1}$ assembles to form an operad in Lie algebras, and indeed, the collection of Lie algebras $\fraktg:=\{\fraktg(n)\}_{n\geq 1}$ yields a right operadic module over the operad in Lie algebras $\frakt_\bv$. More precisely, we have Lie algebra morphisms $\circ_u: \fraktg(U) \oplus \frakt_{\bv} (W) \rightarrow \fraktg(U \setminus \{ u\}) \sqcup W)$, for finite sets $U$ and $W$, and $u \in U$, satisfying the additive analogues of the equivariance, associativity and unit relations. In particular, these maps are completely determined by their respective projections on the first and second component. On generators, they are defined by
\begin{align*}
\circ_u (t_{ij},-)&=\begin{cases}
t_{ij} & \text{ if } i\neq j\in U\setminus \{ u\}\\
\sum\limits_{w\in W} t_{iw} & \text{ if } i\in U\setminus \{ u\}, \ j= u
\end{cases}\\
\circ_u (x_l^{(i)},-)&=\begin{cases}
x_l^{(i)} & \text{ if } i\in U\setminus \{ u\}\\
\sum\limits_{w\in W} x^{(w)}_l & \text{ if } i=u
\end{cases}\\
\circ_u (y_l^{(i)},-)&=\begin{cases}
y_l^{(i)} & \text{ if } i\in U\setminus \{ u\}\\
\sum\limits_{w\in W} y^{(w)}_l & \text{ if } i=u
\end{cases}\\
\circ_u(t_{ii},-)&=\begin{cases}
t_{ii} & \text{ if } i\in U\setminus \{ u\}\\
\sum\limits_{w\in W} t_{ww} + \frac12\sum\limits_{v\neq w\in W} t_{vw} & \text{ if } i=u 
\end{cases}\\
\circ_u (-,t_{ij})&=t_{ij} \text{ for all } i,j\in W.\\
\end{align*}

For the right operadic $\frakt_\bv$-module $\fraktg$ the Chevalley-Eilenberg cochain complex forms a right cooperadic dg Hopf $C(\frakt_\bv)$-comodule $C(\fraktg):=\{C(\fraktg(r))\}_{r\geq 1}=\{(S(\fraktg^*[-1](r)),d_{\CE})\}_{r\geq 1}$. In this section, we relate this dg Hopf comodule to the pair $(\Mog,\bv^c)$. Indeed, on the level of dg Hopf cooperads the quasi-isomorphism $C(\frakt)\xrightarrow \sim e_2^c$ relating the Drinfeld-Kohno Lie algebras to the Gerstenhaber cooperad (\cite{Tamarkin03}, \cite{FresseBook17}) induces a quasi-isomorphism of dg Hopf cooperads \cite{Severa10}
\[
C(\frakt_\bv) \xrightarrow \sim \bv^c. 
\]
To check the result for the respective comodules, denote, for $r\geq 1$, by
\[
E(r):=\vspan_\K(t_{ij}, x_l^{(i)}, y_l^{(i)})_{1\leq i,j \leq r, \ 1\leq l \leq g}
\]
the linear space spanned by the generators of $\fraktg(r)$. Notice that the linear inclusion map $i:E(r)\rightarrow \fraktg(r)$ induces a linear map $\fraktg(r)^*\rightarrow E(r)^*$. Denote by 
\[
\{  w^{(ij)}=s^{-1} (t_{ij})^*, \ a_l^{(i)}=s^{-1}( x_l^{(i)})^*, \ b_l^{(i)}=s^{-1} (y_l^{(i)})^* \}_{1\leq i,j \leq r, \ 1\leq l \leq g}
\]
the basis of the degree shifted dual vector space $E(r)^*[-1]$. We view these as the generators of the dgca $\Mog(r)$. Notice that the linear map $\fraktg^*(r)\rightarrow E(r)^*$ induces the commutative diagram
\begin{center}
 \begin{tikzcd}
  S(\fraktg^*(r)[-1]) \arrow[r,"\Xi(r)"] & \Mog(r) \\
   \fraktg^*(r)\arrow[u]\arrow[r] \arrow[ru,"\xi(r)"] & E(r)^* \arrow[u]
 \end{tikzcd}
 \end{center}
where the vertical maps are the desuspension and inclusions of the generators in their respective dgca's. In particular, $\Xi(r)$ is the unique morphism of graded commutative algebras extending $\xi(r)$. We still need to check that $\Xi(r)$ is indeed a morphism of complexes.

\begin{lemma}\label{lemma:MCtMO)}
The morphism $\xi(r)$ describes a Maurer-Cartan element in $\Hom_{S_r}(\fraktg^*(r),\Mog(r))$. Equivalently, $\Xi(r)$ is a morphism of dgca's.
\end{lemma}

\begin{proof}
We identify $\xi(r)$ with the following element in $ \Mog(r)\otimes_{S_r}\fraktg(r)$,
\[
m_r=\sum\limits_{1\leq i<j \leq r} \  w^{(ij)}\otimes t_{ij}+ \sum\limits_{i=1}^r  \  w^{(ii)}\otimes t_{ii}+ \sum\limits_{i=1}^r\sum\limits_{l=1}^g  \  a_l^{(i)}\otimes x_l^{(i)}+\sum\limits_{i=1}^r\sum\limits_{l=1}^g  \ b_l^{(i)}\otimes y_l^{(i)}.
\]
For this we find
\begin{align*}
dm_r&=\sum\limits_{1\leq i<j \leq r} \  dw^{(ij)}\otimes t_{ij}+ \sum\limits_{i=1}^r  \  dw^{(ii)}\otimes t_{ii}\\
&=\sum\limits_{1\leq i<j \leq r} \  \left(\nu^{(i)}+\nu^{(j)}-\sum\limits_{l=1}^g (a_l^{(i)}b_l^{(j)}-b_l^{(i)}a_l^{(j)})\right)\otimes t_{ij}+ \sum\limits_{i=1}^r  \ (2-2g)\nu^{(i)} \otimes t_{ii}
\end{align*}
and
\begin{align*}
[m_r,m_r]&= \sum\limits_{i<j,k<l} 2w^{(ij)}w^{(kl)}  \otimes [t_{ij},t_{kl}]+\sum\limits_{i\neq j,l} 2 a_l^{(i)}b_l^{(j)} \otimes [x_l^{(i)},y_l^{(j)}]+\sum\limits_{i,l} 2 a_l^{(i)}b_l^{(i)}\otimes [x_l^{(i)},y_l^{(i)}]\\
&+\sum\limits_{i<j,k,l}2 \left( w^{(ij)}a_l^{(k)} \otimes [t_{ij},x_l^{(k)}]+ w^{(ij)}b_l^{(k)}\otimes [t_{ij},y_l^{(k)}]\right)
\end{align*}
We regroup the relevant terms with the Maurer-Cartan equation 
\[
dm_r+\frac12 [m_r,m_r]=0.
\]
For instance, for $i,j,k,l$ all distinct, $[t_{ij},t_{kl}]=0$. On the other hand, whenever two indices coincide, we may split the sum as
\[
\left(\sum\limits_{i=k<j<l}+\sum\limits_{i=k<l<j}+\sum\limits_{i<j=k<l}+\sum\limits_{k<i=l<j}+\sum\limits_{k<i<j=l}+\sum\limits_{i<k<j=l}\right)  w^{(ij)}w^{(kl)} \otimes [t_{ij},t_{kl}].
\]
Renaming the variables yields
\begin{align*}
&\sum\limits_{i<j<l} \Big( w^{(ij)}w^{(il)} \otimes [t_{ij},t_{il}]+ w^{(il)}w^{(ij)}\otimes [t_{il},t_{ij}]+ w^{(ij)}w^{(jl)}\otimes [t_{ij},t_{jl}]\\
&+ w^{(jl)}w^{(ij)} \otimes [t_{jl},t_{ij}]+w^{(jl)}w^{(il)} \otimes [t_{jl},t_{il}]+ w^{(il)}s^{jl)} \otimes [t_{il},t_{jl}]\Big)\\
&=2 \sum\limits_{i<j<l} \Big( w^{(ij)}w^{(il)}\otimes [t_{ij},t_{il}]+ w^{(ij)}w^{(jl)}\otimes [t_{ij},t_{jl}]+ w^{(jl)}w^{(il)} \otimes [t_{jl},t_{il}]\Big)\\
&=2 \sum\limits_{i<j<l} \Big( w^{(ij)}w^{(il)}\otimes [t_{ij},t_{il}]- w^{(ij)}w^{(jl)}\otimes [t_{ij},t_{il}]-w^{(jl)}w^{(il)} \otimes [t_{ji},t_{il}]\Big)\\
&=2 \sum\limits_{i<j<l}  \Big(w^{(ij)}w^{(il)}-w^{(ij)}w^{(jl)}-w^{(jl)}w^{(il)}\Big)\otimes  [t_{ij},t_{il}]\\
&=2 \sum\limits_{i<j<l}  \Big(w^{(ji)}w^{(il)}+w^{(lj)}w^{(ji)}+w^{(il)}w^{(lj)}\Big)\otimes [t_{ij},t_{il}]=0
\end{align*}
where to go from the third to the fourth line we used the relation $[t_{ij},t_{il}+t_{lj}]=0$ twice. The last expression vanishes by Arnold's relation. Moreover, notice that for $i,j,k$ all distinct, $[t_{ij},x_l^{(k)}]=0$ for all $l$. Thus, for each $l$ fixed, we find
\[
\sum\limits_{i < j,k}  w^{(ij)}a_l^{(k)} \otimes [t_{ij},x_l^{(k)}]=\sum\limits_{i<j=k} w^{(ij)}a_l^{(j)}\otimes [t_{ij},x_l^{(j)}]+ \sum\limits_{i=k<j} w^{(ij)}a_l^{(i)}\otimes [t_{ij},x_l^{(i)}]=0
\]
since $[t_{ij},x_l^{(j)}]=-[t_{ij},x_l^{(i)}]=0$ and $w^{(ij)}a_l^{(j)}=w^{(ij)}a_l^{(i)}$. Next, consider
\begin{align*}
&-\sum\limits_{i<j,l}  (a^{(i)}_l b^{(j)}_l - b^{(i)}_l a^{(j)}_l) \otimes t_{ij}+\frac12 \sum\limits_{i\neq j, l} 2 a^{(i)}_l b^{(j)}_l \otimes [x_l^{(i)},y_l^{(j)}]\\
=& -\sum\limits_{i<j,l}  a^{(i)}_l b^{(j)}_l \otimes t_{ij}+ \sum\limits_{i<j,l}  b^{(i)}_l a^{(j)}_l \otimes t_{ij}+\sum\limits_{i< j, l}   a^{(i)}_l b^{(j)}_l \otimes [x_l^{(i)},y_l^{(j)}]+\sum\limits_{j < i , l}  a^{(i)}_l b^{(j)}_l \otimes  [x_l^{(i)},y_l^{(j)}]\\
=& -\sum\limits_{i<j,l}  a^{(i)}_l b^{(j)}_l \otimes t_{ij}+ \sum\limits_{i<j,l}  b^{(i)}_l a^{(j)}_l \otimes t_{ij}+\sum\limits_{i< j, l}  a^{(i)}_l b^{(j)}_l  \otimes  t_{ij}-\sum\limits_{j < i , l}  b^{(j)}_l a^{(i)}_l \otimes  t_{ij}=0.
\end{align*}
For the remaining terms, notice first that
\[
\sum\limits_{1\leq i<j\leq r}(\nu^{(i)}+\nu^{(j)})   \otimes t_{ij}=\sum\limits_{i=1}^r\sum\limits_{j\neq i}  \nu^{(i)}  \otimes t_{ij}.
\]
We then find
\begin{align*}
&\sum\limits_{1\leq i<j\leq r} (\nu^{(i)}+\nu^{(j)}) \otimes t_{ij}+ \sum\limits_{i=1}^r  \ (2-2g)\nu^{(i)}  \otimes t_{ii}+\frac12 \sum\limits_{i,l} 2  a_l^{(i)}b_l^{(i)} \otimes [x_l^{(i)},y_l^{(i)}]\\
=&\sum\limits_{i=1}^r\left( \sum\limits_{j\neq i}  \nu^{(i)} \otimes t_{ij}+(2-2g)\nu^{(i)} \otimes  t_{ii}+\sum\limits_{l=1}^g  a^{(i)}_l b^{(i)}_l \otimes [x_l^{(i)},y_l^{(i)}]\right)\\
=&\sum\limits_{i=1}^r\nu^{(i)}\otimes  \left(\sum\limits_{j\neq i} t_{ij}+(2-2g) t_{ii} +\sum\limits_{l=1}^g [x_l^{(i)},y_l^{(i)}]\right)=0
\end{align*}
by identifying $a_l^{(i)}b_l^{(i)}=\nu^{(i)}$ in $\Mog(r)$ and the defining relation in $\fraktg(r)$.
\end{proof}

\begin{lemma}
The collection of morphisms $\{\Xi(r)\}_{r\geq 1}$ assembles to form a morphism of right cooperadic comodules $\Xi:C(\fraktg)\rightarrow \Mog$.
\end{lemma}

\begin{proof}
We need to check that diagrams of the form
\begin{center}
 \begin{tikzcd}
  C(\fraktg)(U\setminus \{u\} \sqcup W) \arrow[r]\arrow[d,"\circ_u^*"] & \Mog(U\setminus \{u\} \sqcup W) \arrow[d,"\circ_u^*"] \\
   C(\fraktg)(U)\otimes C(\frakt_{\bv})(W) \arrow[r] & \Mog(U)\otimes \bv^c(W).
 \end{tikzcd}
 \end{center}
commute.
For any $r\geq 1$, we may equip $\fraktg(r)$ (and similarly on $\frakt_{\bv}(r)$) with the weight grading defined by $|x_l^{(i)}|=|y_l^{(i)}|=1$, $|t_{ij}|=|t_{ii}|=2$. The grading is compatible with the Lie algebra structure by construction, and moreover also with the $\frakt_{\bv}$-module structure. It induces a weight grading on the dual Lie coalgebras $\fraktg(r)^*$ and $\frakt_{\bv}(r)^*$ by setting $|\alpha^*|:=|\alpha|$ for any element of the dual basis. This enables us to define a weight grading on $C(\fraktg)$ and $C(\frakt_{\bv})$ by putting $|v_1\cdots v_n|:=|v_1|+\dots+|v_n|$. Since both $\Xi$ and $\circ_u^*$ are morphisms of commutative algebras (for any fixed arity), it is enough to check the claim on $\fraktg^*[-1]$. Moreover, since the morphism $\xi:\fraktg^*\rightarrow \Mog$ maps elements of weight greater than two to zero, and the cooperadic comodule structure is compatible with the weight grading, it suffices to check the statement for $s^{-1} (t_{ij})^*$, $s^{-1}(x_l^{(i)})^*$, $s^{-1}(y_l^{(i)})^*$. For these elements, however, the cooperadic comodule structure is precisely given by the analogous formulas we found for the generators $w^{(ij)}$, $a_l^{(i)}$, $b_l^{(i)}$ of $\Mog$ (see Section \ref{sec:Mog}), yielding the statement.
\end{proof}

\begin{prop}\label{prop:Ctmodel}
The morphism $\Xi:C(\fraktg)\rightarrow \Mog$ is a quasi-isomorphism of right cooperadic dg Hopf comodules.
\end{prop}

\begin{proof}
We only need to check the statement about each $\Xi(r)$ being a quasi-isomorphism. It follows from the work of Bezrukavnikov \cite{Bezru94}. To see this, denote by $\fraksg(r)\subset \fraktg(r)$ the Lie ideal spanned by all $t_{ii}$'s. Additionally, consider the short exact sequence of Lie algebras
\[
0\rightarrow \fraksg(r)\rightarrow \fraktg(r)\rightarrow \fraktg(r)/\fraksg(r)\rightarrow 0.
\]
Recall that the Lie algebra Bezrukavnikov (\cite{Bezru94}, Main Theorem) considers is isomorphic to the quotient $\fraktg(r)/\fraksg(r)$. Denote by $\pi:\fraktg(r)^*\twoheadrightarrow \fraksg (r)^*$ the restriction map which is dual to the inclusion. Define the following filtration on $S(\fraktg^*(r)[-1])$,
\[
\mF^p S(\fraktg^*(r)[-1])=\{v_1 \cdots v_n \ | \ \# \text{ of } v_i  \text{ for which } \pi(v_i)\neq 0 \leq p\}.
\]
This defines a complete, bounded below ascending filtration on $S(\fraktg^*(r)[-1])$. Moreover, notice that for any $c\in \fraktg^*[-1]=S^1(\fraktg^*(r)[-1])$, $d_{\CE}(c)=\sum c'c''$ satisfies $\pi(c')=\pi(c'')=0$, since $\fraksg(r)$ is central. In particular, this implies that on the level of the associated graded
\[
\gr S(\fraktg^*(r)[-1])=\bigoplus\limits_{p\geq 0} \mF^{p+1} S(\fraktg^*(r)[-1])/\mF^{p} S(\fraktg^*(r)[-1])
\]
the induced differential $d_0$ acts on homogeneous elements as 
\[
d_0(v_1\cdots v_n)=\sum\limits_{i :  \pi(v_i)=0} \pm v_1\cdots d_{\CE} (v_i) \cdots v_n
\]
since applying it to a $v_i$ for which $\pi(v_i)\neq 0$ decreases the number defining the filtration above by one. In particular, this implies
\[
(\gr C(\fraktg(r)),d_0)\cong (C(\fraktg(r)/\fraksg(r))\otimes S(\fraksg(r)^*[-1]),d_{\CE}\otimes 1).
\]
On homogeneous elements the map is given by
\[
v_1\cdots v_n\mapsto \pm\prod\limits_{i:\pi(v_i)=0} v_i \otimes \prod\limits_{j:\pi(v_j)\neq 0} v_j.
\]
On the other hand, let $T(r)\subset \Mog(r)$ be the graded subalgebra generated by all $w^{(ii)}$. It is not a subcomplex. Furthermore, denote by $B(r)\subset \Mog(r)$ the dg commutative subalgebra generated by $\{a^{(i)}_l,b^{(i)}_l,w^{(ij)}\}_{l,i\neq j}$ subject to the same relations as in $\Mog(r)$. Recall that $B(r)$ is precisely the dgca considered by Bezrukavnikov \cite{Bezru94}. Filter $\Mog(r)$ by
\[
\mF^p\Mog(r)=\{c_1\dots c_n \ | \ \# \text{ of } c_i\in T(r) \leq p\}.
\]
This defines a complete, bounded below ascending filtration. On the level of the associated graded, the differential acts on homogeneous elements as
\[
d_0(c_1\cdots c_n)=\sum\limits_{i:c_i\notin T(r)} \pm c_1\cdots d (c_i) \cdots c_n.
\]
This implies in particular that
\[
(\gr \Mog,d_0)\cong (B(r) \otimes T(r),d\otimes 1)
\]
where on homogeneous elements the morphism is given by
\[
c_1\cdots c_n\mapsto \pm\prod\limits_{i:c_i\notin T(r)} c_i \otimes \prod\limits_{j:c_j\in T(r)} c_j.
\]
The morphism $\Xi(r)$ is compatible with the respective filtrations, and moreover $\gr \Xi(r)$ induces a map between the respective identifications, i.e.
\[
C(\fraktg(r)/\fraksg(r))\otimes S(\fraksg(r)^*[-1])\rightarrow B(r)\otimes T(r).
\]
Since $S(\fraksg(r)^*[-1])\cong T(r)$, the result now follows by the fact that $\gr \Xi(r)$ describes precisely the quasi-isomorphism 
\[
C(\fraktg(r)/\fraksg(r))\xrightarrow\simeq B(r)
\]
appearing in Bezrukavnikov's work (\cite{Bezru94}, Section 4, Corollary 2).
\end{proof}

\begin{rem}
Proposition \ref{prop:Ctmodel} implies that the pair $(C(\fraktg),C(\frakt_\bv))$ defines a real model for the framed configuration spaces of points on surfaces.
\end{rem}

\subsubsection{The non-framed case for $g=1$}\label{sec:nonfrt}

In the non-framed case for $g=1$, we denote by $\fraktone(n)$ the Lie algebra with generators of degree zero,
\begin{equation*}
x^{(i)},\ y^{(i)}, \ t_{ij} \text{ for } 1\leq i\neq j\leq n
\end{equation*}
subject to the relations,
\begin{align*}
&t_{ij}=t_{ji}, \ [t_{ij},t_{kl}]=0 \text{ for } i,j,k,l \text{ all distinct},\  [t_{ik}+t_{kj},t_{ij}]=0 \text{ for } i,j,k \text{ all distinct.}\\
&[x^{(i)},y^{(j)}]=t_{ij} \ (i\neq j) , \ [x^{(i)},y^{(i)}]=-\sum\limits_{j\neq i}^{n} t_{ij}.
\end{align*}
Notice that for the quotient Lie algebra $\fraktg(n)/\fraksg(n)$ appearing in the proof above is isomorphic to $\fraktone(n)$ for $g=1$. This is the genus one case of the Lie algebra introduced by Bezrukavikov \cite{Bezru94}. The collection of Lie algebras $\fraktone=\{\fraktone(n)\}_{n\geq 1}$ assembles to form an operadic module over the operad of Drinfeld-Kohno Lie algebras $\frakt$. The Chevalley-Eilenberg complex $C(\frakt)$ forms a dg Hopf cooperad and $C(\fraktone)$ a right $C(\frakt)$-cooperadic dg Hopf comodule. By similar arguments as in the section above, there is a quasi-isomorphism of cooperadic dg Hopf comodules (also called $\Xi$, by abuse of notation)
\[
\Xi:C(\fraktone)\rightarrow \Mo
\]
compatible with the quasi-isomorphism of dg Hopf cooperads
\[
C(\frakt)\rightarrow \e_2^c.
\]

\subsection{Deformation complexes revisited}

The zig-zag of quasi-isomorphisms
\[
\pdBVGraphsg\xrightarrow \Phi \Mog\xleftarrow \Xi C(\fraktg)
\]
induces quasi-isomorphisms
\begin{equation}\label{eq:zigzagHarr}
\Harr \pdBVGraphsg \xrightarrow {\Harr \Phi} \Harr \Mog \xleftarrow {\Harr \Xi} \Harr C(\fraktg).
\end{equation}

In turn, they induce a zig-zag of dg vector spaces
\begin{equation}\label{eq:zigzagDef}
\Def_\tk(\pdBVGraphsg \xrightarrow \Phi \Mog)\leftarrow \Def_\tk(\Mog\xrightarrow \id \Mog) \rightarrow  \Def_\tk(C(\fraktg)\xrightarrow \Xi\Mog)
\end{equation}
given by $f\mapsto f\circ \Harr (\Phi)$ and $f\mapsto f\circ \Harr(\Xi)$, respectively.

\begin{lemma}\label{lemma:zigzagDef}
The morphisms in the zig-zag \eqref{eq:zigzagDef} are quasi-isomorphisms.
\end{lemma}

\begin{proof}
Identify
\begin{align*}
\Def_\tk(\pdBVGraphsg \xrightarrow \Phi \Mog)&\cong \Mogg \hotimes_S \Harr \BVGraphsg\\
\Def_\tk(\Mog \xrightarrow \id \Mog)&\cong \Mogg \hotimes_S\Harr \Mog^*\\
\Def_\tk(C(\fraktg) \xrightarrow \Xi \Mog)&\cong \Mogg \hotimes_S\Harr C^c(\fraktg) \\
\end{align*}
and equip first all complexes with the complete descending filtration defined by the arity. Moreover, filter the associated graded by 
\[
\text{ degree of } \Mogg \text{ part } \geq p.
\]
Since $\Mogg$ is non-negatively graded, these define bounded descending filtrations on the respective complexes in each arity. On the level of the associated graded complexes, we only see the part of the differential $1\otimes (d_{\Harr}+ d_{\BVGraphsg})$, $1\otimes (d_{\Harr}+ d_{\Mog})$ and $1\otimes (d_{\Harr}+ d_{\CE})$, respectively. The quasi-isomorphisms \eqref{eq:zigzagHarr} between the Harrison complexes now imply the statement of the lemma.
\end{proof}

\subsubsection{The non-framed case for $g=1$}
The results of the section above carry over word for word to the non-framed case for $g=1$. As a result, we find that the deformation complexes
\[
\Def_\tk(\pdGraphsone \xrightarrow \Phi \Mo)\leftarrow \Def_\tk(\Mo\xrightarrow \id \Mo) \rightarrow  \Def_\tk(C(\fraktone)\xrightarrow \Xi\Mo)
\]
are all related by quasi-isomorphisms.

\section{Higher genus Grothendieck-Teichm\"uller Lie algebras}\label{sec:grtg}

In \cite{Enriquez14}, Enriquez defines the elliptic Grothendieck-Teichm\"uller Lie algebra $\grt_1^{ell}$ by extending Drinfeld's construction of the Grothendieck-Teichm\"uller Lie algebra $\grt_1$ \cite{Drinfeld90}. Recall that $\grt_1$ may be viewed as a linear subspace of the completed free Lie algebra in two variables, equipped with a different Lie bracket. In the elliptic case, $\grt_1^{ell}$ decomposes into a semi-direct product
\[
\grt_1^{ell}=\grt_1\ltimes \frakr^{ell}
\]
where $\frakr^{ell}=\sll_2\ltimes \frakr_{ell}$, and $\frakr_{ell}$ is the kernel of a natural projection $\grt_1^{ell}\rightarrow \grt_1$. Explicitly, it is given by
\begin{equation}\label{eq:defrell}
\arraycolsep=0.75pt\def\arraystretch{1.5}
\frakr_{ell}=\left\lbrace (U_a,U_b) \in \fraktone (2)^{\times 2} \ \Bigg \vert \ \begin{array}{l}
U_a^{1,23}+U_a^{2,13}+U_a^{3,12}=0 \\
U_b^{1,23}+U_b^{2,13}+U_b^{3,12}=0 \\
\noindent [U_a^{1,23},y^{(2)}]-[U_b^{2,13},x^{(1)}]=0 \\
\noindent [U_a^{1,23},x^{(2)}]-[U_a^{2,13},x^{(1)}]=0  \\
\noindent [U_b^{1,23},y^{(2)}]-[U_b^{2,13},y^{(1)}]=0  \\
\noindent [U_a^{1,23},y^{(1)}]-[U_b^{1,23},x^{(1)}]=0 
\end{array} \right\rbrace.
\end{equation}
While Enriquez' work generalizes Drinfeld's approach by constructing a genus one analogue of a theory of braided monoidal categories in order to define the elliptic versions of the Grothendieck-Teichm\"uller groups and Drinfeld associators - and ultimately identifying $\grt_1^{ell}$ as the graded Lie algebra of the appropriate Grothendieck-Teichm\"uller group - our work is mostly inspired by the following analogy. In \cite{Willwacher15}, Willwacher established (among others) two important results relating the Grothendieck-Teichm\"uller Lie algebra $\grt_1$, Kontsevich's graph dg Lie algebra $\GC$ and homotopy derivations of the Gerstenhaber operad $e_2$. They read as follows.
\begin{thm}(\cite{Willwacher15})\label{thm:Willwachermain}
There are isomorphisms of Lie algebras
\[
H^0(\GC)\cong \grt_1 \ \text{ and } \ H^0(\Der^h(e_2))\cong  \K \ltimes \grt_1.
\]
\end{thm}

Additionally, as already mentioned in Remark \ref{rmk:HGCFW}, later work by Fresse and Willwacher \cite{FW20} establishes the link between $\HGC$, a hairy variant of Kontsevich's graph dg Lie algebra $\GC$, and the operadic deformation complex from Section \ref{sec:defforcoops} by defining a quasi-isomorphism
\[
\HGC[-1]\xrightarrow \sim \Def(\pdGraphs\xrightarrow \phi e_2^c)
\]
while also identifying the deformation complex to be quasi-isomorphic to the complex of homotopy biderivations of the Hopf cooperad $e_2^c$. Furthermore, they give a quasi-isomorphism $\GC \rtimes \K \xrightarrow \sim \HGC$. While Proposition \ref{prop:GCHGC} and Theorem \ref{thm:defHGC} generalize these constructions to the higher genus case, our next aim is to extend the first part of Willwacher's Theorem \ref{thm:Willwachermain} to surfaces. It entails computing $H^0(\spp(H^*)\ltimes \GCg)$. For this, let
\begin{equation}\label{eq:defZg}
\arraycolsep=0.75pt\def\arraystretch{1.5}
Z_{(g)}=\left\lbrace (U_c)_{c\in H^1} \in \fraktg(2)^{\times 2g} \ \Bigg \vert \ \begin{array}{l}
\forall \alpha \in H^1 \ \exists u_\alpha \in \fraktg(1) \ : \ u_\alpha^{12}=U_\alpha^{1,2}+U_\alpha^{2,1} \\
\forall \alpha \in H^1 \ : \ U_\alpha^{12,3}-U_\alpha^{1,23}-U_\alpha^{2,13}=0 \\
 \forall i, j \ : \ [U_{a_i}^{1,23},y_j^{(2)}]-[U_{b_j}^{2,13},x_i^{(1)}]=0   \\
i<  j \ : \ [U_{a_i}^{1,23},x_j^{(2)}]-[U_{a_j}^{2,13},x_i^{(1)}]=0  \\
 i< j \ : \ [U_{b_i}^{1,23},y_j^{(2)}]-[U_{b_j}^{2,13},y_i^{(1)}]=0  \\
 \sum_{i=1}^g [U_{a_i}^{1,23},y_i^{(1)}]-[U_{b_i}^{1,23},x_i^{(1)}]=0 
\end{array} \right\rbrace
\end{equation}
and $B_{(g)}$ be the subspace of $Z_{(g)}$ spanned by
\begin{equation}\label{eq:Bg}
 B_{(g)}=\{(V_c)_{c} \in \fraktg(2)^{\times 2g} \ | \ \exists v\in \fraktg(1)  \ :   \ V_{a_i}=[v^{12},x_i^{(1)}], \ V_{b_i}=[v^{12},y_i^{(1)}]  \ \forall i \}
\end{equation}
\begin{rem}
To see that $B_{(g)}$ is indeed a subspace of $Z_{(g)}$, we compute for instance
\[
V_{a_i}^{12,3}-V_{a_i}^{1,23}-V_{a_i}^{2,13}=[v^{123},x_i^{(1)}+x_i^{(2)}-x_i^{(1)}-x_i^{(2)}]=0
\]
and
\[
V_{a_i}^{1,2}+V_{a_i}^{2,1}=[v^{12},x_i^{(1)}+x_i^{(2)}]=[v,x_i^{(1)}]^{12}.
\]
Moreover, the last four relations in the definition of $Z_{(g)}$ follow by the Jacobi identity and the relations in $\fraktg(3)$.
\end{rem}

Set furthermore 
\[
\frakr_{(g)}:= Z_{(g)}/B_{(g)}.
\]
The main result of this section is the following theorem.

\begin{thm}\label{thm:GCgvsp}
There is an isomorphism of vector spaces
\[
H^0(\spp(H^*)\ltimes \GCg) \cong \frakr_{(g)}.
\]
Moreover, 
\[
H^{-1}(\spp(H^*)\ltimes \GCg)=
\begin{cases}
\spp_{-1}(H^*)=\K [1]\oplus \K[1] & \text{ for } g=1\\
0 & \text{ for } g\geq 2
\end{cases}
\]
and for all $i<-1$, 
\[
H^{i}(\spp(H^*)\ltimes \GCg)=H^{i}(\GCg)=0.
\]
In particular, $H^{i}(\GCg)=0$ for all $i<0$ and $g\geq 1$.
\end{thm}

Notice in particular that the last statement uses Lemma \ref{lemma:extensioncommutes}. Before turning to the set-up of the proof, let us address the Lie algebra structure $\frakr_{(g)}$ may be equipped with.

\subsection{Lie algebra structure on $\frakr_{(g)}$}

Given $U=(U_\alpha)_\alpha \in Z_{(g)}$, we may construct an $S_r$-equivariant derivation of $\fraktg(r)$, for any $r\geq 1$. For $r\geq 2$, define
\begin{align*}
Z_{(g)}&\rightarrow \Der_{S_r}(\fraktg(r))\\
U=(U_c)_c&\mapsto U=\left( U(x_i^{(1)})=U^{1,2\dots r}_{a_i}, \ U(y_i^{(1)})=U^{1,2\dots r}_{b_i} \ \text{ for } i=1,\dots g, \ U(t_{kl})=U(t_{kk})=0 \right) .
\end{align*}
The derivation is completely determined by specifying its value on the generators $x_i^{(1)}$, $y_i^{(1)}$. Notice that by equivariance
\[
U(x_i^{(l)})=U^{l,1\dots \hat{l}\dots r}_{a_i}, \ U(y_i^{(l)})=U^{l,1\dots \hat{l}\dots r}_{b_i}.
\]
That this construction yields a well-defined derivation of the Lie algebra $\fraktg(r)$ (i.e. that it preserves the relations among the generators) follows directly from the defining equations of $Z_{(g)}$. For instance,
\[
0=U([x_i^{(k)},y_j^{(l)}])=[U(x_i^{(k)}),y_j^{(l)}]+[x_i^{(k)},U(y_j^{(l)})]=[U_{a_i}^{k,1\dots r},y_j^{(l)}]+[x_i^{(k)},U_{b_j}^{l,1\dots r}]
\]
is the image of the relation $[U_{a_i}^{1,23},y_j^{(2)}]+[x_i^{(1)},U_{b_j}^{2,13}]=0$ under the Lie algebra morphism $\fraktg(3)\rightarrow \fraktg(r)$ induced by the map of sets $\phi:\{1,\dots r\}\rightarrow \{1,2,3\}$ mapping $1\mapsto k$, $2\mapsto l$, $3,\dots, r\mapsto 3$. The other compatibility relations follow from similar calculations. For $r=1$, notice first that from $U_c$ we can recover the arity one component $u_c$ by setting the variables $x_i^{(2)}=y_i^{(2)}=t_{12}=t_{22}=0$ for all $i$ in the expression $U_c^{1,2}+U_c^{2,1}$. This projection onto the arity one component $u=(u_c)_c$ produces a derivation of $\fraktg(1)$, i.e.
\[
Z_{(g)}\rightarrow \Der(\fraktg(1)), \ U=(U_c)_c\mapsto u=\left( u(x_i^{(1)})=u_{a_i}, \ u(y_i^{(1)})=u_{b_i}, u(t_{11})=0 \right).
\]
For us, the most important application of this fact will be in the case $r=2$. We use the natural commutator Lie bracket on $\Der_{S_2}(\fraktg(2))$ to endow $Z_{(g)}$ with a Lie algebra structure.

\begin{lemma}\label{lemma:bracket}
For $U,V\in Z_{(g)}$, setting
\[
\{U,V\}_c:=U(V_c)-V(U_c)
\]
and $\{U,V\}=(\{U,V\}_c)_c$ yields a Lie bracket on $Z_{(g)}$. The subspace $B_{(g)}$ forms a Lie ideal in $Z_{(g)}$.
\end{lemma}

\begin{proof}
The Lie bracket is the natural commutator Lie bracket for derivations, hence we only need to check that $Z_{(g)}$ is closed under the Lie bracket. For $U,V\in Z_{(g)}$, let $u=(u_c)_c$ and $v=(v_c)_c$ (satisfying $u_c^{12}=U_c^{1,2}+U_c^{2,1}$, and the respective equation for the pair $v,V$) be the derivations of $\fraktg(1)$. Denote their Lie bracket also by $\{u,v\}$. Additionally, notice that
\[
(u(v_c))^{12}=U(v_c^{12}).
\]
This readily implies
\[
(U(V_c))^{1,2}+(U(V_c))^{2,1}=U((V_c)^{1,2})+U((V_c)^{2,1})=U(v_c^{12})
\]
and thus
\[
\{U,V\}_c^{1,2}+\{U,V\}_c^{2,1}=U(V_c^{1,2}+V_c^{2,1})-V(U_c^{1,2}+U_c^{2,1})=U(v_c^{12})-V(u_c^{12})=(u(v_c)-v(u_c))^{12}=\{u,v\}_c^{12}.
\]
Similarly, denote by $\tilde U$ the $S_3$-equivariant derivation of $\fraktg(3)$ defined by the formulas above, i.e.
\[
\tilde U(x_i^{(1)})=U_{a_i}^{1,23}, \ \tilde U(y_i^{(1)})=U_{b_i}^{1,23}.
\]
Moreover, notice that for any map of sets $\phi:\{1,2,3\}\rightarrow \{1,2\}$, and all $i=1\dots g$ and $l=1,2$, 
\[
(U(x_i^{(l)}))^{\phi}=\tilde U((x_i^{(l)})^{\phi}), \ (U(y_i^{(l)}))^{\phi}=\tilde U((y_i^{(l)})^{\phi})
\]
and in particular, for any $c\in H^1$,
\[
\{U,V\}_c^{\phi}=U(V_c^\phi)-V(U_c^{\phi}).
\]
Next, observe that thus
\[
\{U,V\}_c^{12,3}-\{U,V\}_c^{1,23}-\{U,V\}_c^{2,13}=\tilde U(V_c^{12,3}-V_c^{1,23}-V_c^{2,13})-\tilde V(U_c^{12,3}-U_c^{1,23}-U_c^{2,13})=0.
\]
Furthermore, consider
\begin{align*}
&0=\tilde U([V_{a_i}^{1,23},y_j^{(2)}]-[V_{b_j}^{2,13},x_i^{(1)}])=[\tilde U (V_{a_i}^{1,23}),y_j^{(2)}]+[V_{a_i}^{1,23}, U_{b_j}^{2,13}] - [\tilde U (V_{b_j}^{2,13}),x_i^{(1)}]-[V_{b_j}^{2,13},U_{a_j}^{1,23}]\\
&0=\tilde V([U_{a_i}^{1,23},y_j^{(2)}]-[U_{b_j}^{2,13},x_i^{(1)}])=[\tilde V (U_{a_i}^{1,23}),y_j^{(2)}]+[U_{a_i}^{1,23},V_{b_j}^{2,13}] - [\tilde V (U_{b_j}^{2,13}),x_i^{(1)}]-[U_{b_j}^{2,13},V_{a_i}^{1,23}].
\end{align*}
Their difference yields
\begin{align*}
0&=[\tilde U (V_{a_i}^{1,23}),y_j^{(2)}]-[\tilde U (V_{b_j}^{2,13}),x_i^{(1)}]-[\tilde V (U_{a_i}^{1,23}),y_j^{(2)}]+[\tilde V (U_{b_j}^{2,13}),x_i^{(1)}]\\
&=[\tilde U (V_{a_i}^{1,23})-\tilde V (U_{a_i}^{1,23}),y_j^{(2)}]-[\tilde U (V_{b_j}^{2,13})-\tilde V (U_{b_j}^{2,13}),x_i^{(1)}]=[\{U,V\}_{a_i}^{1,23},y_j^{(2)}]-[\{U,V\}_{b_j}^{2,13},x_i^{(1)}].
\end{align*}
A similar computation shows that the Lie bracket $\{U,V\}$ also satisfies the last three defining equation of $Z_{(g)}$. To show that $B_{(g)}$ is a Lie ideal in $Z_{(g)}$, let $V=(V_{a_i}=[v^{12},x_i^{(1)}], V_{b_i}=[v^{12},y_i^{(1)}])_{i=1,\dots,g} \in B_{(g)}$ and $U=(U_c)_c\in Z_{(g)}$. Notice that by Jacobi, $V=\ad_{v^{12}}$. Then,
\begin{align*}
\{U,V\} (x_i^{(1)})= U([v^{12},x_i^{(1)}]) - V(U_{a_i})&=[ U(v^{12}),x_i^{(1)}] + [v^{12}, U(x_i^{(1)})]- V(U_{a_i})\\
&=[ U(v^{12}),x_i^{(1)}] +  [v^{12}, U_{a_i}] -  [v^{12}, U_{a_i}]=[ U(v^{12}),x_i^{(1)}] .
\end{align*}
The analogous calculations yields the respective result for all $y_i^{(1)}$. Finally, take $u=(u_c)_c$, the derivation of $\fraktg(1)$ we considered at the beginning of this proof and note that
\[
U(v^{12})=(u(v))^{12}
\]
from which we conclude that $\{U,V\}\in B_{(g)}$.
\end{proof}

In the light of Willwacher's Theorem \ref{thm:Willwachermain} and Theorem \ref{thm:GCgvsp}, we refer to $\frakr_{(g)}$ as the Grothendieck-Teichm\"uller Lie algebra of a surface of genus $g$.

\begin{defi}
The Grothendieck-Teichm\"uller Lie algebra of a surface of genus $g$ is
\[
\frakr_{(g)}:=(Z_{(g)}/ B_{(g)}, \{ \ , \ \})
\]
\end{defi}

We conjecture moreover that the isomorphism of Theorem \ref{thm:GCgvsp} is compatible with the respective Lie algebra structures.

\begin{conj}\label{thm:GCgrt}
The zeroth cohomology of the complex $\spp(H^*)\ltimes \GCg$, considered as a Lie algebra, is isomorphic to the Grothendieck-Teichm\"uller Lie algebra of a surface of genus $g$, i.e.
\[
H^0(\spp(H^*)\ltimes \GCg)\cong \frakr_{(g)}.
\]
\end{conj}

In the next sections we first proof our main result, before exploring further consequences of Theorem \ref{thm:GCgvsp}.

\section{Proof of Theorem \ref{thm:GCgvsp}}

Recall that Theorem\ref{thm:defGC}, together with the zig-zag of quasi-isomorphisms from Lemma \ref{lemma:zigzagDef} imply that
\[
H^{i-1}(\spp(H^*)\ltimes \GCg)\cong H^i (\Def_\tk (C(\fraktg)\xrightarrow \Xi \Mog)) \text{ for all } i\neq 4.
\]
It turns out that we may compute the right hand side for non-positive degrees and in degree one by introducing the auxiliary complex below.

\subsection{An auxiliary complex}

Consider the dg vector space
\[
\Qc:=(\Mogg\hotimes_S \fraktg, d_{\Mogg}\otimes 1).
\]
The collection of dg Lie algebras
\[
\Mog \hotimes_S \fraktg:=\prod\limits_{r\geq 1} \Mog(r)\hotimes_{S_r} \fraktg(r)
\]
acts arity-wise on $\Qc$ via the Lie bracket on $\fraktg(r)$ and the $\Mog(r)$-module structure of $\Mogg(r)$. Recall from Lemma \ref{lemma:MCtMO)} that 
\[
m_r=\sum\limits_{1\leq i<j \leq r}  w^{(ij)}\otimes t_{ij}+ \sum\limits_{i=1}^r   w^{(ii)}\otimes t_{ii}+ \sum\limits_{i=1}^r\sum\limits_{l=1}^g  \ a_l^{(i)} \otimes x_l^{(i)}+\sum\limits_{i=1}^r\sum\limits_{l=1}^g  \ b_l^{(i)} \otimes y_l^{(i)}
\]
defines a Maurer-Cartan element in $\Mog(r) \hotimes_{S_r} \fraktg(r) $. Formally set
\[
m_{\fraktg}:=\sum_{r\geq 1} m_r .
\]
Next, twist $\Qc$ by this Maurer-Cartan element to obtain the complex
\[
\mathsf{Q}_\mathrm{(g)}^{m_{\fraktg}}=(\Mogg\hotimes_S \fraktg, d_{\Mogg}\otimes 1+\partial_{m_{\fraktg}}).
\]

Define the morphism of degree one
\begin{align*}
\partial_2:\Mogg(r)\hotimes_{S_r} \fraktg (r)&\rightarrow \Mogg(r+1)\hotimes_{S_{r+1}} \fraktg(r+1)\\
l\otimes  \theta &\mapsto \sum\limits_{i=1}^r (-1)^{|l|} l\circ_i s^{-2}\lambda \otimes  \theta^{\delta_{i}^r}. 
\end{align*}

By abuse of notation we drop the dependence on the arity and write $\partial_2$ for all such operations.

\begin{lemma}
The morphism above squares to zero, i.e. $\partial_2^2=0$.
\end{lemma}

\begin{proof}
By direct computation, we find
\begin{align*}
\partial_2^2( l\otimes \theta)=&\sum\limits_{i,j} (-1)^{2|l|+1} (l\circ_i s^{-2}\lambda )\circ_j s^{-2}\lambda \otimes (\theta^{\delta_i^r})^{\delta_j^{r+1}}\\
= &-\sum\limits_{i+1<j} (l\circ_i s^{-2}\lambda )\circ_j s^{-2}\lambda \otimes \theta^{\delta_i^r \circ \delta_j^{r+1}}- \sum\limits_{j<i} (l\circ_i s^{-2}\lambda )\circ_j s^{-2}\lambda \otimes \theta^{\delta_i^r \circ \delta_j^{r+1}}\\
&- \sum\limits_{i=j} (l\circ_i s^{-2}\lambda )\circ_i s^{-2}\lambda \otimes \theta^{\delta_i^r \circ \delta_i^{r+1}} -  \sum\limits_{i+1=j} (l\circ_i s^{-2}\lambda )\circ_{i+1} s^{-2}\lambda \otimes \theta^{\delta_i^r \circ \delta_{i+1}^{r+1}}\\
=&-\sum\limits_{i+1<j} (-1)(l\circ_{j-1} s^{-2}\lambda )\circ_i s^{-2}\lambda \otimes \theta^{\delta_{j-1}^r \circ \delta_i^{r+1}} -  \sum\limits_{j<i} (l\circ_i s^{-2}\lambda )\circ_j s^{-2}\lambda \otimes \theta^{\delta_i^r \circ \delta_j^{r+1}}\\
&-\sum\limits_{i} \big((l\circ_i s^{-2}\lambda )\circ_{i} s^{-2}\lambda +(l\circ_i s^{-2}\lambda )\circ_{i+1} s^{-2}\lambda\big) \otimes \theta^{\delta_i^r \circ \delta_{i}^{r+1}}\\
=&-\sum\limits_{i} l\circ_i \big(s^{-2}\lambda \circ_{1} s^{-2}\lambda + s^{-2}\lambda \circ_{2} s^{-2}\lambda\big) \otimes \theta^{\delta_i^r \circ \delta_{i}^{r+1}}=0
\end{align*}
since $l\circ_i (s^{-2}\lambda \circ_1 s^{-2}\lambda)\otimes \theta^{\delta_i^r  \circ \delta_{i}^{r+1}}=0$ (and similarly for the $\circ_2$-case) in the space of coinvariants by the Jacobi identity.
\end{proof}

\begin{lemma}
The sum $d_{\Mogg}\otimes 1 + \partial_{m_{\fraktg}}+ \partial_2$ defines a differential on $\Mogg\hotimes_S \fraktg$.
\end{lemma}

\begin{proof}
We need to show that $\partial_2$ anti-commutes with the other two terms. For this we compute
\begin{align*}
&\partial_{m_{\fraktg}}\partial_2( l\otimes \theta)= \partial_{m_{\fraktg}} \sum_i  (-1)^{|l|} l\circ_i s^{-2}\lambda \otimes  \theta^{\delta_i^r}\\
=& \sum\limits_{j<k,i}  (-1)^{|l|} w^{(jk)} \wedge (l\circ_i s^{-2} \lambda) \otimes [t_{jk} ,\theta^{\delta_i^r}]\\
+& \sum\limits_{j,m,i}  (-1)^{|l|} a^{(j)}_m \wedge (l\circ_i s^{-2} \lambda) \otimes [x_m^{(j)} ,\theta^{\delta_i^r}]
+\sum\limits_{j,m,i} (-1)^{|l|} b^{(j)}_m \wedge (l\circ_i s^{-2} \lambda) \otimes [y_m^{(j)} ,\theta^{\delta_i^r}]\\
=&\big(\sum\limits_{\substack{j<k,i \\ i,i+1,j,k \text{ all distinct}}}+ \sum_{j=i<k}+ \sum_{j=i+1< k}+ \sum_{j<i=k}+ \sum_{j<i+1= k} \big) \big( (-1)^{|l|} w^{(jk)} \wedge (l\circ_i s^{-2} \lambda) \otimes [t_{jk} ,\theta^{\delta_i^r}]\big)\\
+&\big(\sum\limits_{\substack{j\neq i,i+1 \\ m}}+ \sum\limits_{\substack{j=i \\ m}}+ \sum\limits_{\substack{j=i+1 \\ m}} \big)  \big((-1)^{|l|} a^{(j)}_m \wedge (l\circ_i s^{-2} \lambda) \otimes [x_m^{(j)} ,\theta^{\delta_i^r}] + (-1)^{|l|}  b^{(j)}_m \wedge (l\circ_i s^{-2} \lambda) \otimes [y_m^{(j)} ,\theta^{\delta_i^r}]\big)
\end{align*}
Here we already used the fact that all $t_{jj}$ are central. Next, notice that
\[
\sum\limits_{\substack{j<k,i \\ i,i+1,j,k \text{ all distinct}}} (-1)^{|l|} w^{(jk)} \wedge (l\circ_i s^{-2} \lambda) \otimes [t_{jk} ,\theta^{\delta_i^r}]=
\sum\limits_{\substack{j<k,i \\ i,i+1,j,k \text{ all distinct}}}  (-1)^{|l|} (w^{(jk)} \wedge l) \circ_i s^{-2} \lambda \otimes [t_{jk} ,\theta]^{\delta_i^r} 
\]
and
\begin{align*}
&\sum_{j=i<k} w^{(ik)} \wedge (l\circ_i s^{-2} \lambda) \otimes [t_{ik} ,\theta^{\delta_i^r}]+ \sum_{j=i+1< k} w^{(i+1 \ k)} \wedge (l\circ_i s^{-2} \lambda) \otimes [t_{i+1 \ k} ,\theta^{\delta_i^r}]\\
=&\sum_{j=i<k} w^{(ik)} \wedge (l\circ_i s^{-2} \lambda) \otimes [t_{ik}+t_{i+1 \ k} ,\theta^{\delta_i^r}]= \sum_{j=i<k} (w^{(ik)} \wedge l) \circ_i s^{-2} \lambda \otimes [t_{ik} ,\theta]^{\delta_i^r}
\end{align*}
where in particular we used that $w^{(ik)} \wedge (l\circ_i s^{-2} \lambda)=w^{(i+1 \ k)} \wedge (l\circ_i s^{-2} \lambda) $ and $[t_{i \ i+1}, \theta^{\delta_i^r}]=0$. The sums over $j<k=i$ and $j<k=i+1$ work similarly. Moreover,
\[
\sum\limits_{\substack{j\neq i,i+1 \\ m}} a^{(j)}_m \wedge (l\circ_i s^{-2} \lambda) \otimes [x_m^{(j)} ,\theta^{\delta_i^r}] = \sum\limits_{\substack{j\neq i,i+1 \\ m}} (a^{(j)}_m \wedge l) \circ_i s^{-2} \lambda \otimes [x_m^{(j)} ,\theta]^{\delta_i^r} 
\]
and
\begin{align*}
&\sum\limits_{\substack{j=i \\ m}} a^{(i)}_m \wedge (l\circ_i s^{-2} \lambda) \otimes [x_m^{(i)} ,\theta^{\delta_i^r}] + \sum\limits_{\substack{j=i+1 \\ m}} a^{(i+1)}_m \wedge (l\circ_i s^{-2} \lambda) \otimes [x_m^{(i+1)} ,\theta^{\delta_i^r}]\\
=&\sum\limits_{i,m} (a^{(i)}_m \wedge l)\circ_i s^{-2} \lambda \otimes [x_m^{(i)}+x_m^{(i+1)} ,\theta^{\delta_i^r}] =\sum\limits_{i,m} (a^{(i)}_m \wedge l)\circ_i s^{-2} \lambda \otimes [x_m^{(i)} ,\theta]^{\delta_i^r}.
\end{align*}
The terms involving the $b$'s may be treated analogously. In summary, these manipulations lead to
\[
\partial_{m_{\fraktg}} \partial_2 ( l\otimes \theta)=-\partial_2 \partial_{m_{\fraktg}} ( l \otimes  \theta).
\]
For $d_{\Mogg}$ the diagrammatic interpretation of $d_{\Mogg}$ (see Section \ref{sec:OmegaMo}) readily gives
\begin{align*}
(d_{\Mogg}\otimes 1) \partial_2( l \otimes  \theta)=&\sum_i d_{\Mogg} (-1)^{|l|} (l\circ_i s^{-2}\lambda) \otimes \theta^{\delta_i^r}=\sum_{i}(-1)^{|l|} d_{\Mogg}(l) \circ_i s^{-2}\lambda \otimes \theta^{\delta_i^r}\\
=&(-1)^{2|l|+1}\partial_2 ( d_{\Mogg}(l) \otimes \theta ) =-\partial_2 (d_{\Mogg} \otimes 1)( l\otimes  \theta).
\end{align*}
\end{proof}

Finally, abusing notation again, set
\[
(\Qc, d_\Qc):= (\Mogg \hotimes_S \fraktg, d_{\Mogg}\otimes 1 + \partial_{m_{\fraktg}}+ \partial_2).
\]

\subsection{Relation to the deformation complexes}

Recall the natural morphism of dg Lie algebras
\[
\Harr C^c(\fraktg)\rightarrow \fraktg.
\]
It is defined through the composition which on, say elements of weight $k$ is given by
\[
\Lie_k(IS^c(\fraktg[1])[-1])\xrightarrow {\pi^{\otimes k}} \Lie_k(\fraktg[1][-1])\rightarrow \fraktg
\]
where $\pi:IS^c(\fraktg[1])[-1]\rightarrow \fraktg[1][-1]=\fraktg$ is the natural projection and the second morphism is the unique Lie algebra map extending the identity $\fraktg\rightarrow \fraktg$.

\begin{lemma}\label{lemma:Harr}
The map $T:\Harr C^c(\fraktg)\rightarrow \fraktg$ is a quasi-isomorphism of dg Lie algebras.
\end{lemma}

\begin{proof}
Filter both sides by the total number of brackets in $\fraktg$. Remark that $\Harr C^c(\fraktg)$ has two types of Lie brackets - the internal bracket on $\fraktg$ and the one coming from taking the Harrison complex. Here the filtration is taken with respect to the internal Lie bracket. This defines complete descending filtrations on the respective complexes. On the level of the associated graded, we find that
\[
\gr T:(\gr \Harr C^c(\fraktg),d_{\Harr} ) \rightarrow \gr \fraktg
\]
vanishes on elements of weight greater than one in the free Lie algebra. On elements of weight one, i.e. elements in $C^c(\fraktg)[-1]$, it is by definition non-zero only on elements in $\fraktg[1][-1]$ (i.e. elements which are not a coproduct of elements in $\fraktg[1]$) on which it is the identity. Additionally, the Chevalley-Eilenberg differential vanishes on the associated graded. Moreover, as computed in (\cite{DR12}, Theorem B.1.), we have 
\[
H(\Harr (S^c(\fraktg[1])),d_{\Harr})=\fraktg
\] 
and thus $\gr T$ defines a quasi-isomorphism. The statement follows.
\end{proof}

\begin{lemma}
Denote the differential on $\Def_\tk (C(\fraktg)\xrightarrow \Xi \Mog) \cong \Mogg \hotimes_S \Harr C^c(\fraktg)$ by
\[
D_\tk=d_{\Mogg}\otimes 1 + 1\otimes d_{\Harr}+ 1\otimes d_{\CE} + \partial_{\mathsf{op}}'+\partial_{\mathsf{mod}}'.
\]
The map above induces a morphism of complexes
\[
E:=1\otimes T: ( \Mogg\hotimes_S \Harr C^c(\fraktg),D_\tk) \rightarrow (\Qc,d_\Qc)
\]
\end{lemma}

\begin{proof}
Since $T:\Harr C^c(\fraktg)\rightarrow \fraktg[1]$ is a map of complexes, we have
\[
E(d_{\Harr} \otimes 1+d_{\CE}\otimes 1)=0.
\]
As for $\partial'_{\mathsf{mod}}$, notice that the Maurer-Cartan element in $\Hom_S((\Harr C(\fraktg))(r),\Mog(r))$ 
\[
\Xi\circ s^{-1} \circ \pi:(\Harr C(\fraktg))(r)\rightarrow C(\fraktg)(r)[1]\rightarrow \Mog(r)
\]
is non-zero only on products of the cogenerators of $\fraktg^*(r)[-1][1]=\fraktg^*(r)$ (i.e any bracket in the Harrison complex is mapped to zero). Additionally, since the differential $\partial'_{\mathsf{mod}}$ is defined through the action of 
\[
m_\mathsf{mod}(r)=\sum\limits_{z} (\Xi\circ s^{-1} \circ \pi) (z^*)\otimes z \in \Mog(r)\hotimes_{S_r} \Harr C^c(\fraktg)(r)
\]
and moreover, $T$ vanishes on non-trivial coproducts of generators of $\fraktg$, we find that, for any $u \otimes  v  \in \Mogg(r)\hotimes_{S_r} (\Harr C^c(\fraktg))(r)$
\begin{align*}
&E(\partial'_\mathsf{mod}(u\otimes v))=\sum \pm E( (\Xi\circ s^{-1} \circ \pi) (z^*)\wedge u \otimes  [z,v])\\
=&\sum \pm  (\Xi\circ s^{-1} \circ \pi) (z^*)\wedge u \otimes  [T(z),T(v)])\\
=&
\begin{cases}
\sum\pm (\Xi\circ s^{-1}\circ \pi) (z^*)\wedge u \otimes   [z,T(v)]  & \text{ if } z \in \{t_{ij},  t_{ii}, x_l^{(i)}, y_l^{(i)}\}_{i,j,l}\\
0 & \text{ otherwise.}
\end{cases}
\end{align*}
Now, on cogenerators $\Xi\circ s^{-1}\circ \pi$ takes the values
\[
(\Xi \circ s^{-1} \circ \pi)( (t_{ij})^*)=w^{(ij)}, \ \ \ (\Xi \circ s^{-1}\circ  \pi)( (x_l^{(i)})^*)=a_l^{(i)}, \ \ \ (\Xi\circ s^{-1}\circ  \pi)( (y_l^{(i)})^*)=b_l^{(i)}
\]
from which we conclude that
\begin{align*}
E(\partial'_\mathsf{mod} (u\otimes v))& =\sum_{1\leq i<j\leq r} w^{(ij)}\wedge u \otimes [ t_{ij},T(v)] +\sum_{i=1}^r w^{(ii)} \wedge u \otimes [ t_{ii},T(v)] \\
&+\sum_{i=1}^r \sum_{l=1}^g a_l^{(i)}\wedge u \otimes [x_l^{(i)},T(v)] + b_l^{(i)} \wedge u \otimes [ y_l^{(i)},T(v)]\\
&=m_r \cdot E(u\otimes  v) = \partial_{m_{\fraktg}} E(u\otimes v).
\end{align*}

On the other hand, notice that in this case by a similar reasoning as in Section \ref{section:diagdescr}, the differential $\partial'_\mathsf{op}$ is induced by the action of the Maurer-Cartan element $m_\mathsf{op}\in \Omega \bv^c \hotimes_S C^c(\frakt_{\bv})$ whose leading order terms are
\[
m_\mathsf{op}=s^{-1}\Delta^c \otimes s t_{11} + s^{-1}\lambda^c \otimes s t_{12}+ s^{-1}\mu^c\otimes 1_{12}+ \dots.
\]
Here $1_{12}\in C^c(\frakt_{\bv})(2)=\K 1_{12}\oplus I S^c(\frakt_{\bv}[1](2))$. Additionally, (again, as in Section \ref{section:diagdescr}) since $\kappa$ vanishes on elements of arity greater than three, we only need to consider these terms when we apply $\partial'_\mathsf{op}$ on $\Mogg\hotimes_S \Harr C^c(\fraktg)$. Thus, we have
\begin{align*}
E(\partial'_\mathsf{op} ( u\otimes   v))&=E(m_\mathsf{op} \cdot (u\otimes  v))\\
&= \sum_{i=1}^r \pm E\big( u \circ_i \delta^* \otimes v\circ_i s t_{11} + \pm u\circ_i s^{-2}\mu \otimes v \circ_i s t_{12}+ \pm u \circ_i s^{-2} \lambda \otimes v\circ_i 1_{12}  \big) \\
&= \sum_{i=1}^r \pm u\circ_i s^{-2} \lambda \otimes  T(v\circ_i 1_{12})= \sum_{i=1}^r \pm u\circ_i s^{-2} \lambda \otimes  T(v)^{\delta_i^r}= \partial_2 E(u\otimes v)
\end{align*}
since $v\circ_i 1_{12}=v^{\delta^r_i}$ is the only term on which the morphism $T$ does not vanish (in the other two cases, the operadic bimodule structure of $\Harr C^c(\fraktg)$ over $C^c(\frakt_\bv)$ produces least one non-trivial coproduct within the brackets, and on such elements $T$ vanishes). Notice also that $T(v^{\delta^r_i})=T(v)^{\delta^r_i}$, since $T$ is a Lie algebra morphism.

\end{proof}

\begin{lemma}\label{lemma:defot}
The map $E: \Mogg\hotimes_S \Harr C^c(\fraktg) \rightarrow \Qc$ is a quasi-isomorphism.
\end{lemma}

\begin{proof}
Filter both sides by
\[
\text{ degree on } \Mogg  \geq p.
\]
Since $\Mogg$ is non-negatively graded, this defines descending complete bounded above filtrations on the respective complexes. On the level of the associated graded, we are left with $1\otimes (d_{\Harr}+d_{\CE})$ as a differential on $\gr (\Mogg\otimes_S \Harr C(\fraktg))$, while on $\gr \Qc$ the differential vanishes. The statement now follows by Lemma \ref{lemma:Harr}.
\end{proof}

\begin{rem}\label{rmk:crucialzigzag}
In particular, Lemma \ref{lemma:defot}, Theorem \ref{thm:defGC} and the zig-zag of quasi-isomorphisms from Lemma \ref{lemma:zigzagDef} imply that
\[
H^{i-1}(\spp(H^*)\ltimes \GCg)\cong H^i( \Mogg\hotimes_S \fraktg) \text{ for all } i\neq 4.
\]
Since $\Mogg$ is non-negatively graded and $\fraktg$ is concentrated in degree zero, we obtain the following result as a corollary. Indeed this also proves one part of Theorem \ref{thm:GCgvsp}. Notice that the elements of $\spp(H^*)$ live in degrees $-1$ and $0$.
\end{rem}

\begin{cor}\label{cor:negativedegrees}
For $i<-1$, we have $H^i (\spp(H^*)\ltimes \GCg)=H^i(\GCg)=0$.
\end{cor}

\begin{rem}
In the non-framed case for $g=1$ the complex $\mathsf{Q}^{\mathrm{non-fr}}_\mathrm{(g)}:=\Mo^{!} \hotimes_S \fraktone$ is set up analogously, with all occurrences of $\Mo_{(1)}$ and $\frakt_{(1)}$ replaced by $\Mo$ and $\fraktone$, respectively. We recover the results that
\[
H(\Def_\tk(C(\fraktone)\rightarrow \Mo))\cong H(\mathsf{Q}^{\mathrm{non-fr}}_\mathrm{(g)}) 
\]
and that, since $\Mo^{!}$ is non-negatively graded, for $i<-1$
\[
H^i(\spp(H^*)\ltimes \GC_{(1)}^{\minitadp})=H^i(\GCg)=0.
\]
\end{rem}

\subsection{The cohomology of $\Mogg \hotimes_S \fraktg$ in degrees zero and one}

We aim to prove the following result.

\begin{prop}\label{thm:Qcohomology}
The cohomology of $\Qc$ in degrees zero and one is
\begin{align*}
H^{0}(\Qc)&=\begin{cases}
\K [1]\oplus \K[1] & \text{ for } g=1\\
0 & \text{ for } g\geq 2
\end{cases} \\
H^1(\Qc)&\cong Z_{(g)}/B_{(g)}.
\end{align*}
\end{prop}

Filter the complex $(\Qc,d_{\Qc})=(\Mogg\hotimes_S \fraktg, D)$ by
\[
\text{total degree of the decorations in } \Mog \geq p.
\]
On the level of the associated graded $\gr \Qc$ notice that
\[
\# \text{ of edges in } \Mog +(\# \text{ Lie elements } - \# \text{ clusters }) + \# \delta^*  \geq p
\]
defines a filtration. Here, we count the trivial Lie word in $\Lie\{-1\}(1)$ also as a Lie element. Moreover, this number is always non-negative, and the subspace
\[
(\gr \Qc)_{\geq 1} \subset \gr \Qc
\]
spanned by elements for which this number is at least one is a subcomplex of $\gr \Qc$. Consider the quotient space
\[
W:=\gr \Qc / (\gr \Qc)_{\geq 1}.
\]
Notice that elements in $W$ may be represented by elements in $ \Com^c \circ H \circ \Lie \{-1\}\hotimes_S \fraktg$, since these are the diagrams for which the number above is zero (i.e. there are no dashed edges or $\delta^*$'s, and each cluster contains precisely one Lie word). It is subject to the following claim.

\begin{lemma}\label{lem:lemma2}
The quotient map 
\[
\gr \Qc \twoheadrightarrow W
\]
is a quasi-isomorphism.
\end{lemma}

\begin{proof}
Filter both sides by
\[
\text{ arity } - \# \text{ clusters } + \# \delta^* + \Mog \text{ degree } \geq p.
\]
The induced differential on the associated graded complex $\gr W$ vanishes, whereas on $\gr\gr \Qc[-1]$, it is given by the sum of the Koszul differentials (see the proof of Lemma \ref{lemma:ComHLie}), i.e.
\[
(d_{\kappa_1}+d_{\kappa_2})\otimes 1.
\]
Its cohomology is isomorphic to
\[
H(\gr\gr \Qc, (d_{\kappa_1}+d_{\kappa_2})\otimes 1)\cong \Com^c \circ H\circ \Lie\{-1\}\hotimes_S \fraktg\cong \gr W.
\]
\end{proof}

Thus, in order to compute $H(\gr \Qc)$, it suffices to compute the cohomology of the quotient space $H(W,d_W)$. To set up the computation of the latter, filter $W$ by
\begin{equation}\label{eq:specseqarity}
\text{arity }  \geq p.
\end{equation}
Notice that the differential splits as $d_W=d_0+d_1$ where $d_i$ raises the arity by $i$. More precisely,
\begin{align*}
d_0&= d_\mathrm{fuse}\otimes 1\\
d_1&= \partial_2
\end{align*}

Consider the spectral sequence associated to the complete filtration above. For the associated graded complex, we have the following result. The proof is given further below.

\begin{lemma}\label{lemma:grW}
The cohomology of the associated graded complex in degrees zero, one and two is
\begin{align*}
H^0(\gr W, d_0)&\cong \fraktg(1)\\
H^1(\gr W,d_0)&\cong  \fraktg(1)[-1]^{\times 2g} \oplus \{(U_c)_{c} \in \fraktg(2)[-1]^{\times 2g} \ | \ U_c^{1,2}+ U_c^{2,1}=0 \}\\
H^2(\gr W,d_0)&\cong \fraktg(1)[-2] \oplus  \fraktg(2)[-2]^{\times 2g^2-2g} \\
&\oplus \{(U_i)_i \in \fraktg(2)[-2]^{\times g} \ | \ U_i^{1,2}=-U_i^{2,1} \} \oplus \{U_\nu \in \fraktg(2)[-2] \ | \ U_\nu^{1,2}=-U_\nu^{2,1} \}\\
 &\oplus  \{(V_i)_i \in \fraktg(2)[-2]^{\times g} \ | \ V_i^{1,2}=V_i^{2,1} \}\\
& \oplus \{(U_c)_c \in \fraktg(3)^{\times 2g} \ | \ U_c^{1,2,3}=U_c^{\sigma(1),\sigma(2),\sigma(3)} \ \forall \sigma \in S_3\}\\
&\oplus \{ (U_{p,q})\in \fraktg(3)^{g(2g-2)} \ | \ U_{p,q}^{1,2,3}=(-1)^{|\sigma|} U_{p,q}^{\sigma(1),\sigma(2),\sigma(3)} \ \forall \sigma \in S_3\}.
\end{align*}
\end{lemma}

To compute the cohomology of $\gr W$ (for which the induced differential is $d_0$) in low degrees, notice that $d_0$ only acts on $V:=\Com^c\circ H\circ \Lie\{-1\}$. Recall that we may represent elements in $V$ by diagrams as in Figure \ref{fig:comHlie}, and that $d_0$ is the operation described in the second line of Figure \ref{fig:diffcoop}.

\subsubsection{The cohomology of $(\Com^c\circ H\circ \Lie \{-1\},d_0)$ in degrees zero, one and two}

The computation yields the following list of representatives.

\begin{lemma}\label{lemma:E1}
For the cohomology of $V=(\Com^c\circ H^* \circ \Lie \{-1\},d_0)$ we find
\begin{align*}
H^0( \Com^c\circ H^* \circ \Lie \{-1\},d_0)&\cong \K \ \overset{1}{\clu}\\
H^1( \Com^c\circ H^* \circ \Lie \{-1\},d_0)&\cong \vspan\{ \overset{1}{\cclu} \ | \ c \in H^1\}\oplus \vspan\{ \overset{1}{\cclu} \ \overset{2}{\clu} \  - \  \overset{2}{\cclu} \ \overset{1}{\clu} \ | \ c \in H^1\}\\
H^2( \Com^c\circ H^* \circ \Lie \{-1\},d_0)&\cong W_1\oplus W_2\oplus W_3 \oplus \K \ \overset{1}{\oclu} \ 
\end{align*}
with explicit bases for the spaces $W_i$ given in the proof.
\end{lemma}

\begin{proof}
We say a cluster is decorated if it carries decorations in $H^{\geq 1}$. All other clusters are called undecorated. Filter $V$ by
\[
\# \text{ decorated clusters }+ \text{ arity of undecorated clusters }\leq p.
\]
The associated filtration is ascending, exhaustive and bounded below. On the level of the associated graded, the induced differential, denoted $\td_0$, fuses two nondecorated clusters and pairs the two Lie words by adding a bracket. Furthermore, if we denote by $V_1$ the sub-collection of $V$ for which all clusters are decorated, we may write
\[
\gr V=V_1\otimes \Com^c\circ \Lie \{-1\}.
\]
Moreover, the differential acts only on the second factor, that is, $\Com^c\circ \Lie\{-1\}$. It acts as the Koszul differential. It is well-known (as already mentioned in the proof of Lemma \ref{lemma:ComHLie}) that this complex is acyclic (i.e. $\K$ in arity one). Thus, we find
\[
H(\gr V, \td_0)\cong V_1\oplus (V_1\otimes (\Com^c\circ \Lie\{-1\})(1))=V_1\oplus (V_1\otimes \K).
\]
In particular, we may give a list of representatives in low degrees 
\begin{align*}
H^0(\gr V,\td_0)&\cong \K \ \overset{1}{\clu}  \\
H^1(\gr V,\td_0)&\cong \vspan\{ \ \overset{1}{\cclu} \ | \ c \in H^1\}\oplus \vspan\{ \ \overset{\sigma(1)}{\cclu} \ \overset{\sigma(2)}{ \clu} \ | \ c \in H^1 , \sigma \in S_2 \}\\
H^2(\gr V,\td_0)&\cong \vspan\{ \overset{1}{\pclu} \  \overset{2}{\qclu}  \ | \ p,q \in H^1\}\oplus \vspan\{ \overset{\sigma(1)}{\pclu} \ \overset{\sigma(2)}{\qclu} \ \overset{\sigma(3)}{ \clu} \ | \ p,q \in H^1, \sigma\in S_3 \}\oplus \K \ \overset{1}{ \oclu} \oplus \K \ \overset{1}{\oclu } \ \overset{2}{ \clu} \\
&\oplus \K \ \overset{1}{\clu } \ \overset{2}{ \oclu} \oplus \vspan\{ \overset{1 \ \ 2}{\cbraclu}  \ | \ c \in H^1\}\oplus \vspan\{ \overset{\sigma(1) \ \sigma(2)}{\cbraclu} \ \overset{\sigma(3)}{\clu} \ | \ c \in H^1, \sigma\in S_3\}
\end{align*}

\begin{rem}\label{rmk:convergence}
Notice that on the first page of the spectral sequence, $H(\gr V, \td_0)\cong V_1\oplus (V_1\otimes \K)$, the induced differential is $\td_{-1}$ which fuses either two decorated clusters, or a decorated with the unique undecorated cluster (if there is one), and adds a Lie bracket. Moreover, by writing out the first page explicitly, we find that in degrees zero, one and two, the spectral sequence abuts on the second page. In particular, since the filtration is exhaustive and bounded below, we find for $i=0,1,2$,
\[
H^i(V,d_0)\cong H^i(H(\gr V,\td_0),\td_{-1}). 
\]
\end{rem}

Thus, we compute
\[
H^0(H(\gr V,\td_0),\td_{-1})\cong \K \ \overset{1}{ \clu } \ .
\]
Moreover, in degree one
\[
H^1(H(\gr V,\td_0),\td_{-1})\cong \vspan\{ \overset{1}{\cclu} \ | \ c \in H^1\}\oplus \vspan\{ \overset{1}{\cclu} \ \overset{2}{\clu} \  - \  \overset{2}{\cclu} \ \overset{1}{\clu} \ | \ c \in H^1\}.
\]
In degree two, consider first $\td_{-1}$ on the representatives of the form
\[
\cbraclu \ \clu
\]
for a fixed decoration $c \in H^1$. Notice that there are three such elements, namely
\[
\overset{1 \ \ 2}{\cbraclu} \ \overset{3}{\clu} \ \ ; \ \ \overset{2 \ \ 3}{\cbraclu} \ \overset{1}{\clu} \ \ ; \ \ \overset{3 \ \ 1}{\cbraclu} \ \overset{2}{\clu} \ .
\]
We may identify these with the expressions in $\Lie\{-1\}(2)\otimes \Lie \{-1\}(1)$ given by
\[
[x_1,x_2]\otimes x_3 \ ; \ [x_2,x_3]\otimes x_1 \ ; \ [x_3,x_1]\otimes x_2.
\]
These are mapped by $\td_{-1}$ to
\[
[[x_1,x_2], x_3] \ ; \ [[x_2,x_3], x_1] \ ; \ [[x_3,x_1], x_2]
\]
or alternatively, to diagrams of the form
\[
\cbrabraclu \ .
\]
Notice that $\td_{-1}$ may thus be viewed as a surjective map from the three-dimensional space $\Lie\{-1\}(2)\otimes \Lie \{-1\}(1)$ to the two-dimensional space $\Lie\{-1\}(3)$. The one-dimensional kernel is determined by the graded Jacobi identity. Set
\[
W_1:=\vspan\{\overset{1 \ \ 2}{\cbraclu} \ \overset{3}{\clu} \ \ + \ \ \overset{2 \ \ 3}{\cbraclu} \ \overset{1}{\clu} \ \ + \ \ \overset{3 \ \ 1}{\cbraclu} \ \overset{2}{\clu} \ | \ \alpha\in H^1\}.
\]
Then $W_1\subset \ker \td_{-1}$. Next, consider $\td_{-1}$ restricted to
\[
\vspan\{ \overset{1}{\pclu} \ \overset{2}{\qclu}  \ | \ p,q \in H^1\}\oplus \K \ \overset{1}{\oclu} \ \overset{2}{\clu} \oplus \K \ \overset{1}{\clu} \ \overset{2}{\oclu} \ .
\]
We are thus considering a $(4g^2+2)$-dimensional space. Moreover, the image under $\td_{-1}$ is one-dimensional, with generator
\[
\overset{1 \ \ 2}{\obraclu} \  .
\]
The kernel, denoted $W_2$, is thus $4g^2+1$-dimensional. A choice for a basis is given by the $2g(2g-1)$ diagrams
\[
\overset{1}{\pclu} \ \overset{2}{\qclu} 
\]
for which the product $pq=0$, together with the $2g+1$ combinations
\[
 \underset{1}{\overset{1}{\aiclu}} \ \underset{2}{\overset{2}{\biclu}} \ - \  \underset{1}{\overset{1}{\biclu}} \ \underset{2}{\overset{2}{\aiclu}}  \ -( \ \overset{1}{\oclu} \ \overset{2}{\clu} \ + \ \ \overset{1}{\clu} \ \overset{2}{\oclu} \ ) \ \ \text{ and } \ \ \overset{1}{\aiclu} \ \overset{2}{\biclu} \ + \ \overset{1}{\biclu} \ \overset{2}{\aiclu} \ \ \text{ and } \ \ \overset{1}{\oclu} \ \overset{2}{\clu} \ - \ \overset{1}{\clu} \ \overset{2}{\oclu}
\]
for $i=1,\dots,g$. We further need to consider $\td_{-1}$ restricted to
\[
\vspan\{\overset{\sigma(1)}{\pclu} \ \overset{\sigma(2)}{\qclu} \ \overset{\sigma(3)}{\clu} \ | \ p,q \in H^1,\sigma \in S_3\}.
\]
Notice that this is a $3\cdot 4g^2$-dimensional space. For fixed $p$ and $q$ the image under $\td_{-1}$ is spanned by diagrams of the form
\[
\pqbraclu \ \clu \ + \ \pbraclu \ \qclu \ + \ \pclu \ \qbraclu \ .
\]
These linear combinations are subject to the relation
\[
\sum\limits_{\sigma \in S_3} (-1)^{|\sigma|} \big( \ \overset{\sigma(1) \sigma(2)}{\pqbraclu} \ \overset{\sigma(3)}{\clu} \ + \ \overset{\sigma(1) \sigma(3)}{\pbraclu} \ \overset{\sigma(2)}{\qclu} \ + \ \overset{\sigma(1)}{\pclu} \ \overset{\sigma(2)\sigma(3)}{\qbraclu} \ \big) =0.
\]
The kernel of $\td_{-1}$ restricted to such diagrams, denoted $W_3$, is thus spanned by the completely anti-symmetric linear combinations
\[
\sum\limits_{\sigma\in S_3} (-1)^{| \sigma |} \overset{\sigma(1)}{\pclu} \ \overset{\sigma(2)}{\qclu} \ \overset{\sigma(3)}{\clu}
\]
for $p,q \in H^1$. Notice further that elements of the form
\[
\overset{1 \ \ 2}{\cbraclu}
\]
are exact under $\td_{-1}$. In summary, we thus find
\[
H^2(H(\gr V, \td_0),\td_{-1})\cong W_1\oplus W_2\oplus W_3 \oplus \K \ \overset{1}{\oclu} \ .
\]
The calculations above, together with Remark \ref{rmk:convergence} imply the result.

\end{proof}

\subsubsection{The cohomology of $\gr W\cong (\Com^c\circ H\circ \Lie \{-1\})\hotimes_S \fraktg$ in degrees zero, one and two}

We are now in a position to prove the technical result from Lemma \ref{lemma:grW}.

\begin{proof}[Proof of Lemma \ref{lemma:grW}]
Lemma \ref{lemma:E1} computes certain degrees of the first page of the spectral sequence associated to the arity filtration \eqref{eq:specseqarity}. Recall further that $(\gr W,d_0)\cong (V\hotimes_S \fraktg,d_0)$ with $d_0$ acting only on $V$. Moreover, since taking cohomology commutes with taking coinvariants, we find
\[
H(\gr W ,d_0)=H(V,d_0)\hotimes_S \fraktg.
\]
In degree zero, we clearly find $H^0(\gr W,d_0) \cong \fraktg(1)$. In degree one, we find on the one hand representatives of the form 
\[
\overset{1}{\cclu} \ \otimes \theta_c
\]
yielding the $2g$ copies of $\fraktg(1)[-1]$. Moreover, in $H^1(V,d_0)\hotimes_{S_2} \fraktg(2)$, we have the following identification for $c\in H^1$
\[
( \ \overset{1}{\cclu} \ \overset{2}{\clu} \   - \  \overset{2}{\cclu} \ \overset{1}{\clu} \ )\otimes \theta_c^{1,2} = \ \overset{1}{\cclu} \ \overset{2}{\clu} \ \otimes (\theta_c^{1,2}-\theta_c^{2,1})
\]
which, by setting $U_c^{1,2}:= \theta_c^{1,2}-\theta_c^{2,1}$, readily implies
\begin{align*}
 H^1(V(2),d_0) \hotimes_{S_2} \fraktg(2)&\cong  \bigoplus_{c\in H^1} \K \ ( \ \overset{1}{\cclu} \ \overset{2}{\clu} \  - \  \overset{2}{\cclu} \ \overset{1}{\clu} \ )\hotimes_{S_2} \fraktg(2)\\
&\cong \bigoplus_{c\in H^1}   \K \ \overset{1}{\cclu} \ \overset{2}{\clu} \ \hotimes_{S_2}\{ \theta_{c}^{1,2} -\theta_c^{2,1} \ | \ \theta_c \in \fraktg(2) \} \\
& \cong \{ (U_c)_c \in \fraktg(2)[-1]^{\times 2g} \ | \ U_c^{1,2}=-U_c^{2,1} \}.
\end{align*}
In degree two, notice first that for $c \in H^1$
\[
( \ \overset{1 \ \ 2}{\cbraclu} \ \overset{3}{\clu} \ \ + \ \ \overset{2 \ \ 3}{\cbraclu} \ \overset{1}{\clu} \ \ +\ \ \overset{3 \ \ 1}{\cbraclu} \ \overset{2}{\clu} \ ) \otimes \theta_c^{1,2,3}= \frac12 \sum_{\sigma\in S_3}  \ \overset{1 \ \ 2}{\cbraclu} \ \overset{3}{\clu} \ \otimes \theta_c^{\sigma(1),\sigma(2),\sigma(3)}  .
\]
This enables us to define the isomorphisms
\begin{align}\label{eq:basisW1}
 \nonumber W_1\hotimes_{S_3} \fraktg(3)&\cong \bigoplus_{c\in H^1} \K ( \ \overset{1 \ \ 2}{\cbraclu} \ \overset{3}{\clu} \ \ + \ \ \overset{2 \ \ 3}{\cbraclu} \ \overset{1}{\clu} \ \ +\ \ \overset{3 \ \ 1}{\cbraclu} \ \overset{2}{\clu} \ ) \hotimes_{S_3}  \fraktg(3)\\ 
&\cong  \bigoplus_{c\in H^1}  \K \  \overset{1 \ \ 2}{\cbraclu} \ \overset{3}{\clu} \ \hotimes_{S_3} \{ U_{c}\in \fraktg(3) \ | \ U_{c}^{\sigma(1),\sigma(2),\sigma(3)} = U_c^{1,2,3} \ \forall \sigma \in S_3\}  \\
\nonumber &\cong \{(U_c)_c \in \fraktg(3)[-2]^{\times 2g} \ |  \ U_c^{1,2,3}=U_c^{\sigma(1),\sigma(2),\sigma(3)} \ \forall \sigma \in S_3\}
\end{align}
Furthermore, in the case of $ W_2 \hotimes_{S_2}\fraktg(2)$, remark that for all $p\neq q\in H^1$ with $pq=0$
\[
 \overset{1}{\pclu} \ \overset{2}{\qclu} \ \otimes   \theta_{p,q}^{1,2} +  \ \overset{1}{\qclu} \ \overset{2}{\pclu} \ \otimes \theta_{q,p}^{1,2}= \ \overset{1}{\pclu} \ \overset{2}{\qclu} \ \otimes (\theta_{p,q}^{1,2} -\theta_{q,p}^{2,1} ) 
\]
and that
\[
\overset{1}{\pclu} \ \overset{2}{\pclu} \ \otimes \theta=0
\]
because of the odd symmetry in the diagram. Additionally,
\begin{align*}
( \ \overset{1}{\aiclu} \ \overset{2}{\biclu} \ + \ \overset{1}{\biclu} \ \overset{2}{\aiclu} \ ) \otimes \theta_{i}^{1,2}&= \ \overset{1}{\aiclu} \ \overset{2}{\biclu} \ \otimes  (\theta_i^{1,2} - \theta_i^{2,1} )\\
 ( \ \overset{1}{\oclu} \ \overset{2}{\clu} \ - \ \overset{1}{\clu} \ \overset{2}{\oclu} \ ) \otimes \theta_{\nu}^{1,2} &= \ \overset{1}{\oclu} \ \overset{2}{\clu} \ \otimes  (\theta_\nu^{1,2} - \theta_\nu^{2,1} )  \\
 ( \ \underset{1}{\overset{1}{\aiclu}} \ \underset{2}{\overset{2}{\biclu}} \ - \  \underset{1}{\overset{1}{\biclu}} \ \underset{2}{\overset{2}{\aiclu}}  \ -( \ \overset{1}{\oclu} \ \overset{2}{\clu} \ + \ \ \overset{1}{\clu} \ \overset{2}{\oclu} \ ))\otimes \theta^{1,2} &=  (  \ \underset{1}{\overset{1}{\aiclu}} \ \underset{2}{\overset{2}{\biclu}} \ - \ \overset{1}{\oclu} \ \overset{2}{\clu} \ )\otimes (\theta^{1,2}+\theta^{2,1}) .
\end{align*}
In particular, this implies 
\begin{align*}
 W_2 \hotimes_{S_2} \fraktg(2)& \cong \ \bigoplus_{1\leq i < j \leq g} \ \overset{1}{\aiclu} \ \overset{2}{\ajclu} \  \hotimes_{S_2}  \fraktg(2)\\
&\oplus \bigoplus_{1\leq i <j \leq g}   \ \overset{1}{\biclu} \ \overset{2}{\bjclu} \ \hotimes_{S_2} \fraktg(2) \\
&\oplus \bigoplus_{1\leq i \neq j \leq g}   \ \overset{1}{\aiclu} \ \overset{2}{\bjclu} \ \hotimes_{S_2} \fraktg(2) \\
&\oplus \bigoplus_{i=1,\dots,g}  \K \ \overset{1}{\aiclu} \ \overset{2}{\biclu} \ \hotimes_{S_2} \{U_i \in \fraktg(2) \ | \ U_i^{1,2}=-U_i^{2,1}\}  \\
& \oplus   \K \ \overset{1}{\oclu} \ \overset{2}{\clu} \ \hotimes_{S_2} \{U_\nu \in \fraktg(2) \ | \ U_\nu^{1,2}=-U_\nu^{2,1}\} \\
&\oplus \bigoplus_{i=1,\dots,g}  \K ( \ \overset{1}{\aiclu} \ \overset{2}{\biclu} \ - \ \overset{1}{\oclu} \ \overset{2}{\clu} \ )\hotimes_{S_2} \{ V_i \in \fraktg(2) \ | \ V_i^{1,2}=V_i^{2,1} \} \\
 \cong \fraktg(2)[-2]^{\times 2g^2-2g} &\oplus \{(U_i)_i \in \fraktg(2)[-2]^{\times g} \ | \ U_i^{1,2}=-U_i^{2,1} \} \oplus \{U_\nu \in \fraktg(2)[-2] \ | \ U_\nu^{1,2}=-U_\nu^{2,1} \}\\
 &\oplus  \{(V_i)_i \in \fraktg(2)[-2]^{\times g} \ | \ V_i^{1,2}=V_i^{2,1} \}.
\end{align*}
Finally, in $ W_3 \hotimes_{S_3} \fraktg(3) $ and any $p\neq q \in H^1$
\[
\left( \sum\limits_{\sigma\in S_3} (-1)^{| \sigma |} \overset{\sigma(1)}{\pclu} \ \overset{\sigma(2)}{\qclu} \ \overset{\sigma(3)}{\clu} \ \right) \otimes \theta_{p,q} = \sum_{\sigma \in S_3} (-1)^{|\sigma|} \overset{1}{\pclu} \ \overset{2}{\qclu} \ \overset{3}{\clu} \ \otimes  \theta_{p,q}^{\sigma(1),\sigma(2),\sigma(3)} 
\]
which leads to
\begin{align*}
W_3 \hotimes_{S_3}  \fraktg(3) & \cong \bigoplus_{1\leq i<j \leq g}  \K \ \overset{1}{\aiclu} \ \overset{2}{\ajclu} \ \overset{3}{\clu} \ \hotimes_{S_3} \{ U_{a_i,a_j} \in \fraktg(3) \ | \ (-1)^{|\sigma|} U_{a_i,a_j}^{\sigma(1),\sigma(2),\sigma(3)} = U_{a_i,a_j}^{1,2,3} \  \forall \sigma \in S_3\} \\
&\oplus  \bigoplus_{1\leq i\neq j \leq g} \K \ \overset{1}{\aiclu} \ \overset{2}{\bjclu} \ \overset{3}{\clu} \ \hotimes_{S_3} \{ U_{a_i,b_j} \in \fraktg(3) \ | \ (-1)^{|\sigma|} U_{a_i,b_j}^{\sigma(1),\sigma(2),\sigma(3)} = U_{a_i,b_j}^{1,2,3} \  \forall \sigma \in S_3\}\\
&\oplus  \bigoplus_{1\leq i<j \leq g}  \K \ \overset{1}{\biclu} \ \overset{2}{\bjclu} \ \overset{3}{\clu} \ \hotimes_{S_3} \{ U_{b_i,b_j} \in \fraktg(3) \ | \ (-1)^{|\sigma|} U_{b_i,b_j}^{\sigma(1),\sigma(2),\sigma(3)} = U_{b_i,b_j}^{1,2,3} \  \forall \sigma \in S_3\} \\
&\cong \{ U \in \fraktg(3)[-2] \ | \  (-1)^{|\sigma|} U^{\sigma(1),\sigma(2),\sigma(3)}=U^{1,2,3} \ \forall \sigma \in S_3\}^{\times (2g^2-2g)}
\end{align*}
Since
\[
H^2(V,d_0)\hotimes_S \fraktg  = (W_1 \oplus W_2 \oplus W_3 \oplus \K \ \overset{1}{\oclu} \ ) \hotimes_S \fraktg.
\]
the result of Lemma \ref{lemma:grW} follows.
\end{proof}

\subsubsection{The cohomology of $W$ in degrees zero, one and two}

Unravelling the machinery of the spectral sequence, notice that Lemma \ref{lemma:grW} computes its first page which is of the form
\[
\begin{array}{ccccc}
0 & E_1^{1,2} & \cdots & &  \\
0 & E_1^{1,1} & E_1^{2,1} & \cdots & \\
0 & E_1^{1,0} & E_1^{2,0} & E_1^{3,0} & \cdots \\
0 & E_1^{1,-1} & E_1^{2,-1} & E_1^{3,-1} & E_1^{4,-1} \cdots\\
0 & 0 & 0 & 0 & 0 \cdots
\end{array}
\]
with horizontal differential $d_1$. We deduce that $E_2^{i,-1}=E_\infty^{i,-1}$ for all $i=1,2$, and  $E_3^{1,0}=E_\infty^{1,0}$. In particular,
\begin{align*}
H^0(W,d_W)&\cong E_2^{1,-1}\\
H^1(W,d_W)&\cong E_3^{1,0}\oplus E_2^{2,-1}.
\end{align*}
Moreover, in degree two, we have $E_\infty^{1,1}=E_4^{1,1} \subset E_0^{1,1}$ and $E_\infty^{2,0}=E_3^{2,0} \subset E_2^{2,0}$, whereas $E_\infty^{3,-1}=E_3^{3,-1}=E_2^{2,-1}/\mathrm{im}(d_2: E_2^{1,0}\rightarrow E_2^{3,-1})$. Additionally, rewriting Lemma \ref{lemma:grW} gives
\begin{align*}
E_1^{1,-1}&\cong \fraktg(1)\\
E_1^{1,0}&\cong \fraktg(1)[-1]^{\times 2g}\\
E_1^{2,-1}&\cong \{(U_c)_{c} \in \fraktg(2)[-1]^{\times 2g} \ | \ U_c^{1,2}+ U_c^{2,1}=0 \}\\
E_1^{1,1}&\cong  \ \overset{1}{\oclu} \ \hotimes_{S_1} \fraktg(1)\\
E_1^{2,0}&\cong W_2 \hotimes_{S_2} \fraktg(2)\\
E_1^{3,-1} &\cong  (W_1 \oplus W_3)\hotimes_{S_3} \fraktg(3) .
\end{align*}

More concretely, we may give explicit representatives in degree zero and one, and additionally describe the cohomology in degree two in a way that will become useful in the proof of Theorem \ref{thm:Qcohomology}.

\begin{lemma}\label{lem:lemma3}
There are isomorphisms of vector spaces in degrees zero, one and two,
\begin{align*}
H^0(W,d_W) & \cong \fraktg(1) \\
H^1(W,d_W) &\cong \bigoplus_{c \in H^1} \{  \ \overset{1}{\cclu} \ \overset{2}{\clu} \ \otimes U_c  \ | \ U_c \in \fraktg(2) \ : \ \exists u_c \in \fraktg(1)  : u_c^{12}=U_c^{1,2}+U_c^{2,1} \ ; \ U_c^{12,3}-U_c^{1,23}-U_c^{2,13}=0\}\\
H^2(W,d_W) & \cong E_3^{3,-1} \oplus E_3^{2,0} \oplus E_4^{1,1}
\end{align*}
where $E_4^{1,1}$ is a subspace of $E_1^{1,1}$, $E_3^{2,0}$ is a subspace of $E_1^{2,0}$ and $E_3^{3,-1}\cong A\oplus B$ is a direct sum of a quotient space $A$ of a subspace of $ W_1\hotimes_{S_3}\fraktg(3)$ and a subspace $B=\ker(d_1|_{ W_3\hotimes_{S_3}\fraktg(3) }:E_1^{3,-1}\rightarrow E_1^{4,-1})$ of $W_3\hotimes_{S_3}\fraktg(3)$.
\end{lemma}

\begin{proof}
To compute $E_2^{1,-1}$ consider $d_1:E_1^{1,-1}\rightarrow E_1^{2,-1}$. Since for any $u\in \fraktg(1)$
\[
d_1 ( \ \overset{1}{\clu} \ \otimes u)= \ \overset{1 \ \ 2}{\braclu} \ \otimes u^{12} =0 \in E_1^{2,-1}
\]
(i.e. it is exact under $d_0$), we conclude that $H^0(W,d_W)=E_\infty^{1,-1}=E_2^{1,-1}=\ker(d_1:E_1^{1,-1}\rightarrow E_1^{2,-1})\cong \fraktg(1)$. In degree one, we have already established that $H^1(W,d_W)\cong E_3^{1,0}\oplus E_2^{2,-1}$ at the beginning of this section. To  determine the right side, let us start by considering $d_1:E_1^{2,-1}\rightarrow E_1^{3,-1}$. For this, take $V_c \in \fraktg(2)$ with $V_c^{1,2}+V_c^{2,1}=0$ and compute
\[
d_1( \ \overset{1}{\cclu} \ \overset{2}{\clu} \ \otimes V_c)=  \ \overset{1 \ \ 2}{\cbraclu} \ \overset{3}{\clu} \ \otimes V_c^{12,3} +  \  \overset{1}{\cclu} \ \overset{2 \ \ 3}{ \braclu} \ \otimes V_c^{1,23} .
\]
In order to express the right hand side in the preferred basis of $E_1^{3,-1}$ (see identification \eqref{eq:basisW1}) we add the exact (under $d_0$) term
\[
-d_0 \ \overset{1}{\cclu} \ \overset{2}{\clu} \ \overset{3}{\clu} \ \otimes V_c^{1,23}  = - \left( \ \overset{1 \ \ 2}{\cbraclu} \ \overset{3}{\clu} \  + \ \overset{1 \ \ 3}{\cbraclu} \ \overset{2}{\clu} \  + \ \overset{1}{\cclu} \ \overset{2 \ \ 3}{\braclu} \ \right) \otimes V_c^{1,23} = - \left(2  \ \overset{1 \ \ 2}{\cbraclu} \ \overset{3}{\clu} \    + \ \overset{1}{\cclu} \ \overset{2 \ \ 3}{\braclu} \ \right) \otimes V_c^{1,23} 
\]
to obtain
\begin{align*}
\overset{1 \ \ 2}{\cbraclu} \ \overset{3}{\clu} \ \otimes ( V_c^{12,3}  - 2 V_c^{1,23})  = \ \overset{1 \ \ 2}{\cbraclu} \ \overset{3}{\clu} \ \otimes ( V_c^{12,3} - V_c^{2,13} - V_c^{1,23} ).
\end{align*}
Here we used that 
\[
\overset{1 \ \ 2}{\cbraclu} \ \overset{3}{\clu} \ \otimes V_c^{1,23}= \ \overset{1 \ \ 2}{\cbraclu} \ \overset{3}{\clu} \ \otimes V_c^{2,13}
\]
in the space of coinvariants. Notice also that 
\[
V_c^{12,3} - V_c^{2,13} - V_c^{1,23}=V_c^{12,3} + V_c^{13,2} + V_c^{23,1}
\]
(as $V_c^{1,2}+V_c^{2,1}=0$) is clearly $S_3$-invariant as required for the basis elements in $E_1^{3,-1}$. Since additionally, $\mathrm{im}(d_1:E_1^{1,-1}\rightarrow E_1^{2,-1})=0$, we conclude that 
\[
E_2^{2,-1}\cong\bigoplus_{c \in H^1} \{V_c \otimes \ \overset{1}{\aclu} \ \overset{2}{\clu}  \ | \ V_c \in \fraktg(2) \ : \ V_c^{1,2}+V_c^{2,1}=0 \ ; \ V_c^{12,3}-V_c^{1,23} - V_c^{2,13}=0\}.
\]
Next, let $u_c \in \fraktg(1)$. Notice that 
\[
d_1(\ \overset{1}{\cclu} \ \otimes u_c )=  \overset{1 \ \ 2}{\cbraclu} \  \otimes u_c^{12} = 0 \in E_1^{2,0}
\]
if and only if there exists $U_c \in \fraktg(2)$ satisfying $u_c^{12}=U_c^{1,2}+U_c^{2,1}$. Such an element always exists (for instance $U_c^{1,2}=\frac12 u_c^{12}$). This is used to compute 
\[
E_2^{1,0}\cong \bigoplus_{c \in H^1} \{ \ \overset{1}{\cclu} \ \otimes u_c \  | \ u_c \in \fraktg(1) \ : \ \exists U_c \in \fraktg(2) \ : \ u_c^{12}=U_c^{1,2} + U_c^{2,1} \} \cong \fraktg(1)[-1]^{\times 2g}.
\]
To determine $E_3^{2,0}$ we need to apply $d_1$ to the extension $ \overset{1}{\cclu} \ \overset{2}{\clu} \ \otimes U_c^{1,2}  $ of $\overset{1}{\cclu} \ \otimes u_c\in E_2^{1,0}$. This yields a similar calculation as above in the case of $E_2^{2,-1}$, from which we conclude
\[
E_3^{1,0}\cong \bigoplus_{c \in H^1} \{ \ \overset{1}{\cclu} \ \otimes u_c \  | \  u_c \in \fraktg(1) \ : \ \exists U_c \in \fraktg(2) \ : \ u_c^{12}= U_c^{1,2}+U_c^{2,1} \ ; \ U_c^{12,3}-U_c^{1,23}-U_c^{2,13}=0\}.
\]

To conclude the degree one case, consider the morphism of graded Lie algebras
\begin{align*}
\fraktg(2)& \rightarrow \fraktg(1)\\
\theta(x_l^{(1)}, x_l^{(2)},y_l^{(1)},y_l^{(2)},t_{11},t_{22},t_{12})&\mapsto \theta(x_l^{(1)}, x_l^{(2)}=0,y_l^{(1)},y_l^{(2)}=0,t_{11},t_{22}=0,t_{12}=0).
\end{align*}
Next, define
\[
G_c := \{  U_c \in \fraktg(2) \ | \  \exists u_c \in \fraktg(1) \ : \ u_c^{12}=U_c^{1,2}+ U_c^{2,1}; \ U_c^{12,3}-U_c^{1,23}-U_c^{2,13}=0\}.
\]
and
\begin{align*}
\pi\circ \mathrm{sym}: \bigoplus_{c \in H^1} G_c &\rightarrow \fraktg(2)^{\times 2g} \rightarrow  \fraktg(1)^{\times 2g} \\
 (U_c)_c &\mapsto  (U_c^{1,2}+U_c^{2,1} )_c \mapsto (\pi(U_c^{1,2}+U_c^{2,1}) )_c .
\end{align*}
While the image satisfies $\mathrm{im}(\pi\circ \mathrm{sym})\cong E_2^{2,-1}$, the kernel is given by $\ker(\pi \circ \mathrm{sym})=\ker(\mathrm{sym})\cong E_3^{1,0}$, since $\pi$ is injective when restricted to $\mathrm{im}(\mathrm{sym})$. We thus find
\[
H^1(W,d_W)\cong E_3^{1,0} \oplus E_2^{2,-1}\cong \bigoplus_{c\in H^1} G_c[-1].
\]
In degree two, it is an immediate consequence of the shape of the first page that $E_\infty^{1,1}=E_4^{1,1}$ is a subspace of $E_1^{1,1}$. Similarly, $E_\infty^{2,0}=E_2^{2,0}=\ker(d_1:E_1^{2,0}\rightarrow E_1^{3,0})\subset E_1^{2,0}$ since $\mathrm{im}(d_1:E_1^{1,0}\rightarrow E_1^{2,0})=0$. Lastly, for $E_1^{3,-1}$ consider first $W_3 \hotimes_{S_3} \fraktg(3) $. Its representatives are of the form 
\[
\overset{1}{\pclu} \ \overset{2}{\qclu} \ \overset{3}{\clu} \ \otimes  U^{1,2,3} 
\]
with $p\neq q \in H^1$ and $U^{1,2,3} =(-1)^{|\sigma|} U^{\sigma(1),\sigma(2),\sigma(3)}$ for all $\sigma \in S_3$. Notice in particular that all elements in $\mathrm{im}(d_1:E_1^{2,-1}\rightarrow E_1^{3,-1})$ and $\mathrm{im}(d_2:E_2^{1,0}\rightarrow E_2^{3,-1})$ contain at least one Lie bracket. Therefore, they cannot kill any of the representatives carrying two decorations. Thus, representatives of cohomology classes in $E_3^{3,-1}$ with two decorations span the space
\[
\ker(d_1|_{ W_3 \hotimes_{S_3} \fraktg(3)}:E_1^{3,-1}\rightarrow E_1^{4,-1})=:B
\] 
($d_2$ vanishes by degree reasons). Restricting to elements of $ W_1 \hotimes_{S_3} \fraktg(3)$, we have 
\begin{align*}
E_2^{3,-1}&\supset A':=\ker(d_1|_{ W_1 \hotimes_{S_3} \fraktg(3)} : E_1^{3,-1}\rightarrow E_1^{4,-1})/\mathrm{im}(d_1:E_1^{2,-1}\rightarrow E_1^{3,-1})\\
E_3^{3,-1}& \supset A:= A'/\mathrm{im}(d_2:E_2^{1,0}\rightarrow E_2^{3,-1})\\
E_3^{3,-1}&=A\oplus B.
\end{align*}
\end{proof}

By computing $H^i(W,d_W)=H^i(\gr \Qc,\gr d_\Qc)$ for $i=0,1,2$, we find that the first page $\tilde E_1$ of the spectral sequence induced by filtering by the degree of the decorations in $\Mog$ is
\begin{center}
\begin{tikzcd}[row sep=0.1cm]
0 \arrow[r,"d_1"] & \cdots &&\\
  0 \arrow[r,"d_1"] &\tilde E_1^{0,1}=0 \arrow[r,"d_1"] & \tilde E_1^{1,1} \arrow[r,"d_1"] & \cdots &  &  \\
   0 \arrow[r, "d_1"]& \tilde E_1^{0,0} \arrow[r,"d_1"]  &\tilde E_1^{1,0} \arrow[r,"d_1"]& \tilde E_1^{2,0}\arrow[r,"d_1"] & \cdots &\\
   0& 0 & 0 & 0 & 0\\
\end{tikzcd}
\end{center}
with $d_1$ induced by $\partial_{m_{\fraktg}}$. The row $\tilde E_0^{k,-1}=0$ vanishes by degree reasons already on the level of the associated graded. It follows that $\tilde E_2^{0,0}=\tilde E_\infty^{0,0}\cong H^0(\Qc)$ and $\tilde E_2^{1,0}=\tilde E_\infty^{1,0}\cong H^1(\Qc)$. Moreover, notice that, if we retain the notation from Lemma \ref{lem:lemma3}, $\tilde E_1^{2,0}\cong E_4^{1,1}\oplus E_3^{2,0}\oplus B$ and that $\tilde E_1^{1,1}\cong A$. In particular, for $w\in \tilde E_1^{1,0}$, $d_1(w)=0 \in \tilde E_1^{2,0}$ implies that $d_1(w)=0$ on the nose, i.e. as an element in $(\Com^c \circ H)\hotimes_S \fraktg $ since none of $ E_4^{1,1}$, $E_3^{2,0}$, $B$ are quotient spaces.  

\begin{lemma}\label{lem:lemma4}
In degrees zero and one, we find
\begin{align*}
\tilde E_2^{0,0}=& 
\begin{cases}
\K [1]\oplus \K[1] & \text{ for } g=1\\
0 & \text{ for } g\geq 2
\end{cases}\\
\tilde E_2^{1,0} \cong& Z_{(g)}/B_{(g)} .
\end{align*}
\end{lemma}

\begin{proof}
A class in $\tilde E_1^{0,0}\cong \fraktg(1)$ is completely determined its arity one component. Its leading terms are
\[
w_0= \ \overset{1}{\clu} \ \otimes  u+ \frac12 \ \overset{1}{\clu} \ \overset{2}{\clu} \ \otimes  u^{12} + \dots \ .
\]
Thus,
\[
d_1(w_0)=\sum_{i=1}^g  \ \overset{1}{\aiclu} \ \otimes [u,x_i^{(1)}]  +  \ \overset{1}{\biclu} \ \otimes [u,y_i^{(1)}] + \sum_{i=1}^g  \ \overset{1}{\aiclu} \ \overset{2}{\clu} \ \otimes [u^{12},x_i^{(1)}]+   \ \overset{1}{\biclu} \ \overset{2}{\clu} \ \otimes [u^{12},y_i^{(1)}] +\dots =0 \in \tilde E_1^{1,0}
\]
Observe that elements in $\tilde E_1^{1,0}\cong H^1(W,d_W)$ are determined by their components in arity one and two. For $g\geq 2$ the statement follows, since requiring $[u,x_i^{(1)}]=0=[u,y_i^{(1)}]$ for all $i$ in arity one, implies $u=k\cdot t_{11} \in Z(\fraktg(1))=\K t_{11}$ by Lemma \ref{lemma:center}. However, for $u=k\cdot t_{11}$
\[
[u^{12},x_i^{(1)}]=[t_{11}+t_{12}+t_{22},x_i^{(1)}]=[t_{12},x_i^{(1)}]\neq 0
\]
the arity two part does not vanish. Hence $u=0$. For $g=1$, $u=x^{(1)}$ and $u=y^{(1)}$ both solve the required equations in arity one (since $\frakt_{(1)}(1)$ is abelian) and two (for instance $[u^{12},y^{(1)}]=[x^{(1)}+x^{(2)},y^{(1)}]=-t_{12}+t_{12}=0$), whereas by the same reasoning as above, $u=t_{11}$ does not. Thus $d_1(w_0)=0 \in \tilde E_1^{1,0}$ for $u\in \{x^{(1)}, y^{(1)}\}$ and $\tilde E_2^{0,0}$ is two-dimensional in this case. In degree one, consider 
\[
w_1=\sum_{c \in H^1}  \ \overset{1}{ \cclu} \ \otimes u_c  +  \ \overset{1}{\cclu} \ \overset{2}{ \clu} \ \otimes U_c^{1,2} +  \ \overset{1}{\cclu} \ \overset{2}{\clu} \ \overset{3}{\clu} \ \otimes U_c^{1,23}  +  \dots \in W^1
\]
satisfying $d_W(w_1)=0$. In this case, 
\begin{align*}
d_1(w_1)&=\sum_{i=1}^{g}  \ \overset{1}{\oclu} \ \otimes ([u_{a_i},y_i^{(1)}] - [u_{b_i},x_i^{(1)}]) + \sum_{i=1}^g  \ \overset{1}{\oclu} \ \overset{2}{\clu} \ \otimes ([U_{a_i}^{1,2},y_i^{(1)}]-[U_{b_i}^{1,2},x_i^{(1)}]) \\
&+\sum_{i,j} \ \overset{1}{\aiclu} \  \overset{2}{\bjclu} \ \otimes ([U_{a_i}^{1,2},y_j^{(2)}] - [U_{b_j}^{2,1},x_i^{(1)}]) \\
&+\sum_{i\leq j}  \ \overset{1}{\aiclu} \  \overset{2}{\ajclu} \ \otimes ([U_{a_i}^{1,2},x_j^{(2)}] - [U_{a_j}^{2,1},x_i^{(1)}]) \\
&+\sum_{i\leq j}  \ \overset{1}{\biclu} \  \overset{2}{\bjclu} \ \otimes ([U_{b_i}^{1,2},y_j^{(2)}] - [U_{b_j}^{2,1},y_i^{(1)}]) \\
&+2 \sum_{i=1}^g \ \overset{1}{\oclu} \ \overset{2}{\clu} \ \overset{3}{\clu} \ \otimes ( [U_{a_i}^{1,23},y_i^{(1)}]-[U_{b_i}^{1,23} , x_i^{(1)}]) \\
&+ 2 \sum_{i,j} \  \overset{1}{\aiclu} \ \overset{2}{\bjclu} \ \overset{3}{\clu} \ \otimes ( [U_{a_i}^{1,23},y_j^{(2)}]-[U_{b_j}^{2,13} , x_i^{(1)}])\\
&+ 2 \sum_{i\leq j} \ \overset{1}{\aiclu} \ \overset{2}{\ajclu} \ \overset{3}{\clu} \ \otimes ( [U_{a_i}^{1,23},x_j^{(2)}]-[U_{a_j}^{2,13} , x_i^{(1)}]) \\
&+ 2 \sum_{i\leq j} \  \overset{1}{\biclu} \ \overset{2}{\bjclu} \ \overset{3}{\clu} \ \otimes ( [U_{b_i}^{1,23},y_j^{(2)}]-[U_{b_j}^{2,13} , y_i^{(1)}])  + \dots=0. \\
\end{align*}
Observe that elements in $\tilde E_1^{2,0}\cong H^1(W,d_W)\cong E_4^{1,1}\oplus E_3^{2,0}\oplus B $ are determined by their components in arity one, two and three. The arity three component gives precisely the relations from the statement. Moreover, by setting the variables $x_i^{(3)},y_i^{(3)}, t_{j3}$ equal to zero, the conditions for the arity three part to vanish readily imply that the arity two part is zero as well. Lastly, notice that assuming the relations in arity three (and thus that the arity two component is zero as well) implies
\begin{align*}
0&=\sum_{i=1}^g ([U_{a_i}^{1,2},y_i^{(1)}]-[U_{b_i}^{1,2},x_i^{(1)}])^{1,2+2,1}\\
&=\sum_{i=1}^g [U_{a_i}^{1,2}+U_{a_i}^{2,1},y_i^{(1)}+y_i^{(2)}]-[U_{b_i}^{1,2}+U_{b_i}^{2,1},x_i^{(1)}+x_i^{(2)}]\\
&=\sum_{i=1}^g [u_{a_i}^{12},y_i^{(1)}+y_i^{(2)}]-[u_{b_i}^{12},x_i^{(1)}+x_i^{(2)}]\\
&=\sum_{i=1}^g ([u_{a_i},y_i^{(1)}]-[u_{b_i},x_i^{(1)}])^{12}.
\end{align*}
By setting the variables $x_i^{(2)}$, $y_i^{(2)}$, $t_{12}$, $t_{22}$ equal to zero in this last line, we obtain
\[
\sum_{i=1}^{g} [u_{a_i},y_i^{(1)}] - [u_{b_i},x_i^{(1)}]=0
\]
i.e. the arity one terms are zero. For the exact elements in degree one, notice that $\tilde E_1^{0,0}\cong\fraktg(1)$ is completely determined by its arity one component. It is sent to
\[
d_1( \overset{1}{\clu} \ \otimes v)=\sum_{i=1}^g \ \overset{1}{\aiclu} \ \otimes [v,x_i^{(1)}] + \overset{1}{\biclu} \ \otimes [v,y_i^{(1)}] .
\]
Set $v_{a_i}=[v,x_i^{(1)}]$ and $V_{a_i}=[v^{12},x_i^{(1)}]$. They satisfy $v_{a_i}^{12}=V_{a_i}^{1,2}+V_{a_i}^{2,1}$ and
\[
V_{a_i}^{12,3}-V_{a_i}^{1,23}-V_{a_i}^{2,13}=[v_{a_i}^{123},x_i^{(1)}+x_i^{(2)}-x_i^{(1)}-x_i^{(2)}]=0.
\]
Thus, $\mathrm{im}(d_1:\tilde E_1^{0,0}\rightarrow \tilde E_1^{1,0})\cong \{(V_c)_{c} \in \fraktg(2)^{\times 2g} \ | \ \exists v\in \fraktg(1)  \ :   \ V_{a_i}=[v^{12},x_i^{(1)}], \ V_{b_i}=[v^{12},y_i^{(1)}]  \ \forall i \}$.
\end{proof}

The series of lemmas above now imply Proposition \ref{thm:Qcohomology}.

\begin{proof}[Proof of Proposition \ref{thm:Qcohomology}]
Since $H^0(\Qc)\cong \tilde E_2^{0,0}$ and $H^1(\Qc)\cong \tilde E_2^{1,0}$ Proposition \ref{thm:Qcohomology} follows.
\end{proof}

\begin{proof}[Proof of Theorem \ref{thm:GCgvsp}]
The results mentioned in Remark \ref{rmk:crucialzigzag} yield that as vector spaces
\[
H^0(\spp(H^*)\ltimes \GCg)\cong H^1(\Def(\pdGraphsg\xrightarrow \Phi \Mog))\cong H^1(\Qc,d_\Qc)\cong Z_{(g)}/B_{(g)}=\frakr_{(g)}
\]
and
\[
H^{-1}(\spp(H^*)\ltimes \GCg)\cong H^{0}(\Qc,d_\Qc)=\begin{cases}
\K [1]\oplus \K[1] & \text{ for } g=1\\
0 & \text{ for } g\geq 2.
\end{cases}.
\]
By Lemma \ref{lemma:extensioncommutes}, we have
\[
H^{-1}(\GCg)=H^{-1} (\spp(H^*)\ltimes \GCg)=0
\]
for $g\geq 2$, whereas for $g=1$
\[
H^{-1} (\spp(H^*)\ltimes \GC_\mathrm{(1)})=\spp_{-1}(H^*)\ltimes H^{-1}(\GC_{\mathrm{(1)}}).
\]
Since in this case $\spp_{-1}(H^*)$ is two-dimensional, we deduce that $H^{-1}(\GC_\mathrm{(1)})=0$. Corollary \ref{cor:negativedegrees} proves the result in degrees $i<-1$.
\end{proof}

\section{Corollaries}

In this section we unwrap some of the definitions and results given above in order to understand parts of the structure of $\frakr_{(g)}$ and the cohomology of the graph complex $\GCg$.

\subsection{The non-framed case for $g=1$}

The non-framed case for $g=1$ is particularly interesting since we may compare our findings to the literature - most notably the work of Enriquez \cite{Enriquez14}. First, observe that the results from above carry over to the non-framed case for $g=1$, with the collection of Lie algebras $\fraktone$ (see Section \ref{sec:nonfrt}) in place of $\frakt_{(1)}$. Explicitly, we have equivalent versions
\[
Z^{\mathrm{non-fr}}_{(1)}, B^{\mathrm{non-fr}}_{(1)}, \frakr^{\mathrm{non-fr}}_{(1)}=Z^{\mathrm{non-fr}}_{(1)}/ B^{\mathrm{non-fr}}_{(1)}
\]
defined by the analogous formulas (see equations \eqref{eq:defZg} and \eqref{eq:Bg}), but in which all occurrences of $\frakt_{(1)}$ are replaced by $\fraktone$. This may be simplified further.

\begin{lemma}
The subspace $B^{\mathrm{non-fr}}_{(1)}=0$.
\end{lemma}

\begin{proof}
This follows from the fact that $\fraktone(1)$ is abelian, and the following computation for the two generators. Indeed, for $v=x\in \fraktone(1)$, we have
\begin{align*}
&[v^{12},x^{(1)}]=[x^{(1)}+x^{(2)},x^{(1)}]=0\\
&[v^{12},y^{(1)}]=[x^{(1)}+x^{(2)},y^{(1)}]=-t_{12}+t_{12}=0
\end{align*}
while a similar computation holds for $v=y\in \fraktone(1)$.
\end{proof}

Hence, in the non-framed case for $g=1$ we have
\[
H^0(\spp(H^*)\ltimes \GC^{\minitadp}_{(1)})\cong Z^{\mathrm{non-fr}}_{(1)}=\frakr^{\mathrm{non-fr}}_{(1)}.
\]

Next, we find an interesting decomposition of $Z^{\mathrm{non-fr}}_{(1)}=\frakr^{\mathrm{non-fr}}_{(1)}$. 

\begin{lemma}\label{lemma:decompositionframed}
As Lie algebras
\[
\frakr^{\mathrm{non-fr}}_{(1)}\cong \sll_2\ltimes \frakr_{ell}.
\]
\end{lemma}

In \cite{Enriquez14}, Enriquez also considers this extension. He denotes it by $\frakr^{ell}:=\sll_2\ltimes \frakr_{ell}$.

\begin{proof}
Observe that the system of equations
\begin{align*}
0&=U_a^{1,2}+U_a^{2,1}\\
0&=U_a^{12,3}-U_a^{1,23}-U_a^{2,13} 
\end{align*}
is equivalent, via the fact that $U^{12,3}_a=-U^{3,12}_a$, to the single equation
\[
U_a^{1,23}+U_a^{2,13}+U_a^{3,12}=0.
\]
It follows that $\frakr_{ell}$ may be described by the subspace of $\frakrone$ spanned by those pairs $U=(U_a,U_b)\in \Zone$ for which $u=(u_a,u_b)=(0,0)\in \fraktone(1)^{\times 2}$. As we have seen in the proof of Lemma \ref{lemma:bracket}, the Lie bracket $\{U,V\} \in \frakrone$ respects the derivations of $\fraktone(1)$, i.e. if $u=(u_a,u_b)$ and $v=(v_a,v_b)$ satisfy the first relation in $\frakrone$ for $U$ and $V$, respectively, then $\{u,v\}$ solves the required equation for the Lie bracket $\{U,V\}$. Thus, $\frakr_{ell}$ clearly defines a Lie ideal in $\frakrone$. On the other hand, recall that $\fraktone(1)=\K x \oplus \K y$ is abelian, and therefore a basis for the elements in $\frakrone$ for which the first condition is satisfied for non-zero $u=(u_a,u_b)\in \fraktone(1)^{\times 2}$ is given by
\[
H=(x^{(1)},-y^{(1)}) \ , \ E=(0,x^{(1)}) \ , \ F=(y^{(1)},0).
\]
These satisfy
\[
\{E,F\}=H \ , \ \{H,E\}=2E \ , \ \{H,F\}=-2F.
\]
and thus form a Lie subalgebra of $\frakrone$ isomorphic to $\sll_2$.
\end{proof}

Furthermore, recall that in the case of genus one we have (see Lemma \ref{lemma:extensioncommutes})
\[
H^0(\spp(H^*)\ltimes \GC^{\minitadp}_{(1)})=\spp_0(H^*)\ltimes H^0(\GC^{\minitadp}_{(1)})
\]
where $\spp_0(H^*)$ denotes the transformations of degree zero. Notice that as Lie algebras $\spp_0(H^*)\cong \sll_2$. Moreover, by tracing through the zig-zag of quasi-isomorphisms, we find that 
\begin{align*}
a \partial_b &\mapsto  F=(y^{(1)},0)\\
b \partial_{a} &\mapsto E=(0,x^{(1)})\\
a\partial_{a}-b \partial_{b}&\mapsto H=(x^{(1)},-y^{(1)}).
\end{align*}

We conclude that the cohomology of the graph complex in degree zero is indeed isomorphic to Enriquez' $\frakr_{ell}$.

\begin{cor}\label{cor:GCrell}
In the non-framed case for $g=1$ the isomorphism of Theorem \ref{thm:GCgvsp}
\[
H^0(\spp(H^*)\ltimes \GC^{\minitadp}_{(1)})\cong \frakr^{ell}
\]
restricts to an isomorphism of vector spaces
\[
H^0(\GC^{\minitadp}_{(1)})\cong \frakr_{ell}.
\]
\end{cor}
%

\begin{rem}
In negative degrees the result corresponding to Theorem \ref{thm:GCgvsp} for the non-framed case in $g=1$ translates to
\[
H^{-1}(\spp(H^*)\ltimes \GC^{\minitadp}_{(1)})=\spp_{-1}(H^*)=\K[1]\oplus \K[1]
\]
and
\[
H^{i}(\GC^{\minitadp}_{(1)})=0 
\]
for all $i<0$.
\end{rem}

\subsubsection{Conjectural generators in degree zero}

In \cite{Enriquez14}, Enriquez defines the Lie subalgebra $\frakb_3\subset \frakr^{ell}$ generated by $\sll_2$ and the derivations $\delta_{2n}\in \frakr_{ell}$ for $n\geq 0$,
\[
\delta_{2n}(x^{(1)})=(\ad_{ x^{(1)}})^{2n+2}(y^{(1)}), \ \delta_{2n}(y^{(1)})=\frac12\sum_{\substack{0\leq p < 2n+1 \\ p+q=2n+1}} [(\ad_{x^{(1)})}^p (y^{(1)}),(\ad_{ x^{(1)})}^q(y^{(1)})].
\]

\begin{conj}(\cite{Enriquez14})\label{conj:deltas}
The elliptic Grothendieck-Teichm\"uller Lie algebra $\frakr^{ell}$ is generated by $\sll_2$ and the family of derivations $\{\delta_{2n}\}_{n\geq 0}$, i.e. 
\[
\frakr^{ell}=\frakb_3.
\]
\end{conj}

The existence of the elements $\delta_{2n}$ enables us to construct conjectural generators for the degree zero cohomology of the graph complex. To fix notation set
\[
t_{2n}:=\treegraph\in \GC^{\minitadp}_{(1)}
\]
be the graph with no loops, $2n+2$ vertices, each carrying a decoration by $\alpha$, and with the two outermost vertices additionally decorated by $\beta$. By adding correction terms which involve one and two loop graphs $\gamma^1_{2n}$, $\gamma^2_{2n}$ which can be explicitly calculated, we obtain diagrams
\[
\gamma_{2n}=t_{2n}+\gamma_{2n}^1+\gamma_{2n}^2 \in \GC^{\minitadp}_{(1)}
\]
which satisfy $d\gamma_{2n}=0$. In particular, $\gamma_{2n}$ is not exact, since the tree part $t_{2n}$ cannot be obtained through the vertex splitting operation (this would involve a vertex of valency four, decorated by two $\alpha$, but such diagrams are zero by symmetry). Moreover, it can be shown that $\GC^{\minitadp}_{(1)}$ is quasi-isomorphic to its subcomplex spanned by diagrams without $\omega$-decorations (see \cite{FN21}, note that this is only true in genus one), and thus this element cannot arise from applying the differential to a tree with one $\omega$-decoration. We have the following conjectures.

\begin{conj}
Under the isomorphism of Theorem \ref{thm:GCgvsp}, the element $\delta_{2n}\in \frakrone$ corresponds to a graph cohomology class, all of whose representatives have a nonvanishing coefficient in front of the graph $t_{2n}$.
\end{conj}

The equivalent of Conjecture \ref{conj:deltas} in terms of diagrams then reads as follows.

\begin{conj}
Consider the Lie subalgebra $\mathfrak b$ of $H^0(\spp(H^*)\ltimes \GC^{\minitadp}_{(1)})$ generated by $\spp_0(H^*)$ and the graph cohomology classes $[\gamma_{2n}]$. Then
\[
H^0(\spp(H^*)\ltimes \GC^{\minitadp}_{(1)})= \mathfrak b.
\]
\end{conj}
%

\subsection{The framed case for $g=1$}

In the framed case for $g=1$, we have the following result. Recall the definition of $B_{(1)}$ from equation \eqref{eq:Bg}.
\begin{lemma}
The subspace $B_{(1)}$ is one-dimensional with generator $V=([t_{12},x^{(1)}],[t_{12},y^{(1)}])$.
\end{lemma}

\begin{proof}
The Lie algebra $\frakt_{(1)}(1)$ is abelian, and a similar computation as in the framed case shows that for the generators $v=x,y\in \frakt_{(1)}(1)$, we have $[v^{12},x^{(1)}]=[v^{12},y^{(1)}]=0$. For $v=t_{11}$ on the other hand,
\begin{align*}
&[v^{12},x^{(1)}]=[t_{11}+t_{12}+t_{22},x^{(1)}]=[t_{12},x^{(1)}]\\
&[v^{12},y^{(1)}]=[t_{12},y^{(1)}]
\end{align*}
defines the generator of $B_{(1)}$.
\end{proof}

The result of Lemma \ref{lemma:decompositionframed} carries over to the framed case. For this denote by
\[
Z_{(1)}^0\subset Z_{(1)}
\]
the subspace of anti-symmetric derivations, i.e. we require that the pair $u=(u_a,u_b)=(0,0)$ or equivalently
\[
U_a^{1,2}+U_a^{2,1}=0 \text{ and } U_b^{1,2}+U_b^{2,1}=0.
\]
Notice that $Z_{(1)}^0$ takes the role of $\frakr_{ell}$ in the framed case. The analogous argument as in the proof of Lemma \ref{lemma:decompositionframed} implies that $Z_{(1)}^0$ defines a Lie ideal in $Z_{(1)}$. Moreover, since $\frakt_{(1)}(1)=\K x\oplus \K y \oplus \K t_{11}$ is abelian, we obtain the same basis $\{H,E,F\}$ for the subspace spanned by derivations for which the corresponding pair $u=(u_a,u_b)\in \frakt_{(1)}^{\times 2}$ is non-zero. Notice in particular, that for instance $u_a=t_{11}$ (or equivalently $u_b=t_{11}$) yields $U_a^{1,2}=\frac12 (t_{11}+t_{12}+t_{22})$ which does not satisfy the cyclic relation $U_a^{12,3}-U_a^{1,23}-U_a^{2,13}=0$. Thus, the Lie algebra $\frakr_{(1)}$ decomposes as
\[
\frakr_{(1)}\cong \sll_2\ltimes Z_{(1)}^0/B_{(1)}.
\]
Indeed, tracing through the zig-zag of quasi-isomorphisms, we find that the derivations $H,E,F$ correspond to the basis of $\spp_0(H^*)$ in the extension of the graph complex. Recall that also in this case
\[
H^0(\spp(H^*)\ltimes \GC_{(1)})\cong \spp_0(H^*)\ltimes H^0(\GC_{(1)}).
\]
This yields the equivalent statement to Corollary \ref{cor:GCrell} in the framed case.

\begin{cor}
In the framed case for $g=1$, the isomorphism from Theorem \ref{thm:GCgvsp} restricts to an isomorphism of vector spaces
\[
H^0(\GC_{(1)})\cong Z_{(1)}^0/B_{(1)}.
\]
\end{cor}
%
%
%

\subsection{The subspace $B_{(g)}$ for $g\geq 2$ - tripods revisited}

Let $g\geq 2$. The image of the $\spp_0(H^*)$-part under the zig-zag of quasi-isomorphisms leads to the following Lie subalgebra of $\frakr_{(g)}$. It maps
\begin{align*}
\alpha_i\partial_{\beta_i}&\mapsto V_{\alpha_i\partial_{\beta_i}}:=(0,\dots,0,(V_{\alpha_i\partial_{\beta_i}})_{b_i}=x_i^{(1)},0,\dots,0)\\
\beta_i\partial_{\alpha_i}&\mapsto V_{\beta_i\partial_{\alpha_i}}:=(0,\dots,0,(V_{\beta_i\partial_{\alpha_i}})_{a_i}=y_i^{(1)},0,\dots,0)\\
\alpha_i\partial_{\alpha_j}-\beta_j\partial_{\beta_i}&\mapsto V_{\alpha_i\partial_{\alpha_j}}:=(0,\dots,0,(V_{\alpha_i\partial_{\alpha_j}})_{a_i}=x_i^{(1)},0,\dots,0,(V_{\alpha_i\partial_{\alpha_j}})_{b_i}=-y_j^{(1)},0,\dots,0) \text{ for all } i,j\\
\alpha_i\partial_{\beta_j}+\alpha_j\partial_{\beta_i}&\mapsto V_{\alpha_i\partial_{\beta_j}}:=(0,\dots,0,(V_{\alpha_i\partial_{\beta_j}})_{b_i}=x_j^{(1)},0,\dots,0,(V_{\alpha_i\partial_{\beta_j}})_{b_j}=x_j^{(1)},0,\dots,0) \text{ for } i\neq j\\
\beta_i\partial_{\alpha_j}+\beta_j\partial_{\alpha_i}&\mapsto V_{\beta_i\partial_{\alpha_j}}:=(0,\dots,0,(V_{\beta_i\partial_{\alpha_j}})_{a_i}=y_j^{(1)},0,\dots,0,(V_{\beta_i\partial_{\alpha_j}})_{a_j}=y_i^{(1)},0,\dots,0) \text{ for } i\neq j\\
\end{align*}

\begin{lemma}\label{lemma:spinrg}
The elements 
\[
\{V_{\alpha_i\partial_{\beta_i}},V_{\beta_i\partial_{\alpha_i}},V_{\alpha_i\partial_{\alpha_j}},V_{\alpha_i\partial_{\beta_j}},V_{\beta_i\partial_{\alpha_j}}\}
\]
generate a Lie subalgebra of $\frakr_{(g)}$ which is isomorphic to $\spp_0(H^*)$.
\end{lemma}

\begin{proof}
A straightforward direct computation shows that all elements satisfy the relations of $\frakr_{(g)}$, that the subspace is closed under the Lie bracket and that the assignment above is compatible with the respective Lie algebra structure on $\spp_0(H^*)$ and $\frakr_{(g)}$.
\end{proof}

Next, we use the image of the tripods to define the generators of a second Lie subalgebra of $\frakr_{(g)}$. We first introduce a small piece of notation. For $\alpha \in H^1$, set $z^{(1)}_\alpha \in\{x_1^{(1)},\dots,x_g^{(1)},y_1^{(1)},\dots,y_1^{(g)}\}$ to be the corresponding analogous generator in $\fraktg(2)$. For instance, if $\alpha=a_i$, $z^{(1)}_\alpha=x_i^{(1)}$, and if $\alpha=b_j$, $z^{(1)}_\alpha=y_j^{(1)}$. Next let $\alpha, \beta, \gamma \in H^1$ be all distinct, and denote by $\alpha^*,\beta^*,\gamma^*$ their Poincar\'e duals. We define the $S_2$-equivariant derivation $V_{\alpha,\beta,\gamma}$ of $\fraktg(2)$ by
\[
(V_{\alpha,\beta,\gamma})_c=
\begin{cases}
-\langle \alpha,\alpha^*\rangle([z_\beta^{(1)},z_\gamma^{(1)}]+\langle \beta,\gamma\rangle \ 2 t_{11}) & \text{ if } c=\alpha^*\\
\langle \beta,\beta^*\rangle ([z_\alpha^{(1)},z_\gamma^{(1)}]+\langle \alpha,\gamma\rangle \ 2 t_{11})& \text{ if } c=\beta^*\\
-\langle \gamma,\gamma^*\rangle ([z_\alpha^{(1)},z_\beta^{(1)}]+\langle \alpha,\beta\rangle \ 2 t_{11})& \text{ if } c=\gamma^*\\
0 & \text{ otherwise.}
\end{cases}
\]

\begin{lemma}
For $\alpha,\beta,\gamma \in H^{1}$ all distinct, $V_{\alpha,\beta,\gamma}\in Z_{(g)}$.
\end{lemma}

\begin{proof}
It follows from a direct computation that $V_{\alpha,\beta,\gamma}$ verifies the defining equations.
\end{proof}

Next consider the derivations
\[
A_j=\sum_{i\neq j} V_{a_i,b_i,a_j} \ \text{ and } B_j=\sum_{i\neq j} V_{a_i,b_i,b_j}.
\]

\begin{lemma}
For all $j=1,\dots g$, we have $A_j\in B_{(g)}$ and $B_j\in B_{(g)}$.
\end{lemma}

\begin{proof}
For $v=x_j\in \frakt_{(g)}(1)$, we obtain the derivation $V$ defined by
\begin{align*}
V_{a_i}=&[x_j^{(1)}+x_j^{(2)},x_i^{(1)}]=[x_j^{(1)},x_i^{(1)}] \text{ for } i\neq j\\
V_{a_j}=&[x_j^{(1)}+x_j^{(2)},x_j^{(1)}]=[x_j^{(1)},x_j^{(1)}]=0 \\
V_{b_i}=&[x_j^{(1)}+x_j^{(2)},y_i^{(1)}]=[x_j^{(1)},y_i^{(1)}] \text{ for } i\neq j\\
V_{b_j}=&[x_j^{(1)}+x_j^{(2)},y_j^{(1)}]=[x_j^{(1)},y_j^{(1)}]+t_{12}
\end{align*}
On the other hand, we compute
\begin{align*}
A_j(a_i)=&\sum_{l\neq j} V_{a_l,b_l,a_j}(a_i)=[x_i^{(1)},x_j^{(1)}] \text{ for } i \neq j\\
A_j(a_j)=&\sum_{l\neq j} V_{a_l,b_l,a_j}(a_j)=0\\
A_j(b_i)=&\sum_{l\neq j} V_{a_l,b_l,a_j}(b_i)=[y_i^{(1)},x_j^{(1)}] \text{ for } i \neq j\\
A_j(b_j)=&\sum_{l\neq j} V_{a_l,b_l,a_j}(b_j)=\sum_{l\neq j} ([x_l^{(1)},y_l^{(1)}]-2t_{11})=-2(g-1)t_{11}+\sum_{l\neq j} [x_l^{(1)},y_l^{(1)}]=-[x_j^{(1)},y_j^{(1)}]-t_{12}
\end{align*}
This shows $V=-A_j$. A similar computation holds for $v=y_j$, and the respective derivation $V=-B_j$ in that case.
\end{proof}

Let $\fraka'$ be the Lie subalgebra of $Z_{(g)}$ generated by all $V_{\alpha,\beta,\gamma}$. Quotient $\fraka'$ by the Lie ideal generated by the $\{A_j,B_j\}_{j=1,\dots g}$ above and denote by $\fraka$ the quotient Lie algebra. Under the identification of Lemma \ref{lemma:spinrg}, the Lie algebra $\spp_0(H^*)$ acts by derivation on the Lie subalgebra $\fraka'$ via the Lie bracket on $Z_{(g)}$. Moreover, this action preserves the Lie ideal above. The following conjecture is then the equivalent of Conjecture \ref{conj:Tripods}.

\begin{conj}
The morphism induced by the inclusion of $\fraka'\rightarrow  Z_{(g)}$ induces an isomorphism
\[
\spp_0(H^*)\ltimes \fraka\cong \frakr_{(g)}.
\]
\end{conj}

\subsection{Outlook}

We end this section with a brief and slightly vague overview of open questions and potential improvements of our results. 
\subsubsection{Extension by Kontsevich's graph complex}
First, remark that we may define a combinatorial action of Kontsevich's dg Lie algebra $\GC$ on $\GCg'$ by insertion at the vertices. In order for the action to descend to the quotient $\GCg$ (where we quotient out tadpoles), we need to restrict to the Lie subalgebra $\GC^{div}\subset \GC$ of ``divergence''-free diagrams \cite{Willwacher17} spanned by diagrams whose Lie bracket with the diagram consisting of a single vertex and a tadpole vanishes. This allows us to define the dg Lie algebra
\[
\GC^{div} \ltimes (\spp(H^*)\ltimes \GCg).
\]
Remark that the ``correct'' differential is rather involved since one has to take the twists of several Maurer-Cartan elements into account. Pushing the analogy from above a bit further, we state the following conjecture.

\begin{conj}
There is an isomorphism of Lie algebras
\[
H^0(\GC^{div}\ltimes (\spp(H^*)\ltimes \GCg))\cong\grt_1 \ltimes \frakr_{(g)}.
\]
\end{conj}

\subsubsection{Higher genus Kashiwara-Vergne Lie algebras}

The Kashiwara-Vergne Lie algebra $\mathfrak{krv}$ is a Lie algebra closely related to the Grothendieck-Teichm\"uller Lie algebra $\grt_1$. It was introduced by Alekseev and Torossian \cite{AT12} and describes the symmetries of the Kashiwara-Vergne problem \cite{KV78} in Lie theory. In a series of works (\cite{AKKN17}, \cite{AKKN18}) Alekseev, Kuno, Kawazumi and Naef define the higher genus analogue of the Kashiwara-Vergne Lie algebra, which moreover easily translates to a diagrammatic language. Additionally, Raphael and Schneps established a similar definition for $g=1$ based on the theory of moulds \cite{SE17}. In genus zero, $\grt_1$ injects into $\mathfrak{krv}$ \cite{AT12}, and conjecturally the two Lie algebras coincide up to a one-dimensional Lie algebra. Moreover, the Kashiwara-Vergne Lie algebra also allows a combinatorial description due to \v Severa and Willwacher \cite{SevWill11} which is compatible with the identification of $\grt_1$ with $H^0(\GC)$. We refer to the upcoming work \cite{FN21} for the details on how to relate the higher genus Kashiwara-Vergne Lie algebra with $\GCg$ and $\frakr_{(g)}$. For the combinatorial description, the strategy reads as follows. The dg Lie algebra $\GCg$ has a descending filtration given by the number of loops. Both the vertex splitting differential $d_s$ and the twist by $z$ preserve this number, while the pairing differential raises the number of loops by one. Moreover, the action of $\spp(H^*)$ preserves this filtration. Additionally, the filtration restricts to the subcomplex spanned by diagrams with no $\omega$ decorations. On this subcomplex, consider the spectral sequence with respect to the loop filtration. The $E_1^{0,0}$ term consists of trivalent trees modulo the ``IHX''-relation. On the first page the differential is given by the pairing differential. In analogy to the local case \cite{SevWill11}, we expect the higher genus Kashiwara-Vergne Lie algebras from \cite{AKKN18} to be isomorphic to $E_2^{0,0}$.

\printbibliography

\end{document}